\documentclass[11pt]{amsart}
\usepackage{amsfonts}
\usepackage{amsmath,amssymb}
\usepackage{mathrsfs}
\usepackage{amssymb, amsmath, mathrsfs,graphicx}
\usepackage{subfigure,wrapfig}

\renewcommand{\theequation}{\arabic{section}.\arabic{equation}}
\newtheorem{theo}{Theorem}[section]
\newtheorem{pro}{Proposition}[section]
\newtheorem{lem}{Lemma}[section]
\newtheorem{cor}{Corollary}[section]
\newtheorem{defi}{Definition}[section]
\newcommand{\eop}{\hspace{\fill} $\;\;\;  \Box$}
\font\ui=cmcsc11

\font\uj=cmssbx10 scaled \magstep1
\DeclareMathDelimiter{\Norm}{\mathord}{largesymbols}{"3E}{largesymbols}{"3E}

\begin{document}
\title[Arnold diffusion in a priori stable systems]{Arnold diffusion
in nearly integrable\\ Hamiltonian systems}
\author[C.-Q. Cheng]{Chong-Qing CHENG}
\address{Department of Mathematics\\ Nanjing University\\ Nanjing 210093,
China\\ }
\email{chengcq@nju.edu.cn}

\begin{abstract} In this paper, Arnold diffusion is proved to be generic phenomenon in nearly integrable convex Hamiltonian systems with three degrees of freedom:
$$
H(x,y)=h(y)+\epsilon P(x,y), \qquad x\in\mathbb{T}^3,\ y\in\mathbb{R}^3.
$$
Under typical perturbation $\epsilon P$, the system admits ``connecting" orbit that passes through any two prescribed small balls in the same energy level $H^{-1}(E)$ provided $E$ is bigger than the minimum of the average action, namely, $E>\min\alpha$.
\end{abstract}
\maketitle \tableofcontents
\section{\ui Introduction}
\setcounter{equation}{0}

For nearly integrable Hamiltonian systems, the set of KAM tori has a relatively large Lebesgue measure in phase space. For systems with two degrees of freedom, it implies the dynamical stability: all orbits are stable, the variation of actions stays small for all the time as each 2-dimensional KAM torus separates the 3-dimensional energy level. However, this is a special property of lower-dimensional space, KAM torus of $n$-dimension does not separate $(2n-1)$-dimensional energy level if $n>2$. It is conceivable that the complement of all $n$-dimensional invariant tori forms dense and connected set in phase space. This would mean that by arbitrary small changes of the initial states one would find orbits along which the action variables ultimately escape. The underlying phenomenon is now called ``Arnold diffusion".

\noindent{\bf Conjecture} (\cite{Ar2,AKN}): {\it The typical case in a higher-dimensional problems is topological instability: through an arbitrarily small neighborhood of any point there passes a phase
trajectories along which the slow variables drift away from the initial value by a quantity of order 1.}

Since the celebrated example of Arnold \cite{Ar1} was published half a century ago, there are many works for the study of this problem. In recent years, it has become clear that diffusion is a typical phenomenon in so called {\it a priori} unstable  systems, refer to  \cite{Be3,CY1,CY2,DLS,LC,Tr}. The {\it a priori} unstable  condition  guarantees the existence of normally hyperbolic cylinder, from which one derives certain regularity of the barrier functions.  The genericity of the diffusion is obtained by using the regularity  \cite{CY1,CY2}.  There are also many works for the study of the problem, for instance, see  \cite{Bs,BCV,Fo,DH1,DH2,GL,GR1,GR2,KL1,KL2,X,Zha}.

General perturbation of integrable Hamiltonian is usually called {\it a priori} stable system. A bit away from strong complete resonance in such systems, some pieces of normally hyperbolic cylinder still exist and the method for {\it a priori} unstable system can also be applied \cite{BKZ,Be4}.  For {\it a priori} stable systems with three degrees of freedom, a notable difficulty occurs at the point of double resonance, around which the cylinder for prescribed single resonance may disappear. The averaged system has two homoclinic orbits associated with different classes in $H_1(\mathbb{T}^2,\mathbb{Z})$. As the energy decreases, the periodic orbit on the cylinder approaches these two homoclinic orbits simultaneously. Thus, the transition chain in $H^1(\mathbb{T}^2,\mathbb{R})$ for the single resonance may break. To solve this difficulty, Mather suggested a path in $H^1(\mathbb{T}^2,\mathbb{R})$ to cross double resonance, along which one moves the cohomology class in the channel determined by the prescribed homology class and switches it to the channel determined by one of these two classes when it is getting close to the double resonance \cite{Ma6}.

In this paper, the path we choose to construct transition chain is different from that suggested by Mather. We find an annulus surrounding the flat of double resonance in $H^1(\mathbb{T}^2,\mathbb{R})$, which has a foliation of circles. Each of these circles is actually a transition chains, possibly, of incomplete intersection. Although the annulus is not so thick, each single resonance path extends into. It allows us to use one of these circles connecting one single resonance path to another. In this way, we find a path of transition chain along which the diffusion orbits are constructed by variational method.

\subsection{Statement of the main result}
We consider nearly integrable Hamiltonian systems with 3 degrees of freedom:
\begin{equation}\label{introeq1}
H(x,y)=h(y)+\epsilon P(x,y), \qquad (x,y)\in\mathbb{T}^3\times\mathbb{R}^3,
\end{equation}
where $h$ is assumed to strictly convex, namely, the Hessian matrix $\partial^2h/\partial y^2$ is positive definite. It is also assumed that $\min h=0$, both $h$ and $P$  are $C^r$-function with $r\ge 8$.

For $E>0$, let $H^{-1}(E)=\{(x,y):H(x,y)=E\}$ denote the energy level set, $B\subset\mathbb{R}^3$ denote a ball in $\mathbb{R}^3$ such that $\bigcup_{E'\le E+1}h^{-1}(E')\subset B$. Let $\mathfrak{S}_a,\mathfrak{B}_a\subset C^r(\mathbb{T}^3\times B)$ denote a  sphere and a ball with radius $a>0$ respectively: $F\in\mathfrak{S}_a$ if and only $\|F\|_{C^r}=a$ and $F\in\mathfrak{B}_a$ if and only $\|F\|_{C^r}\le a$. They inherit the topology from $C^r(\mathbb{T}^3\times B)$.

For perturbation $P$ independent of $y$ (classical mechanical system) we use the same notation $\mathfrak{S}_a,\mathfrak{B}_a\subset C^r(\mathbb{T}^3)$ to denote a sphere and a ball with radius $a>0$

Let $\mathfrak{R}_a$ be a set residual in $\mathfrak{S}_a$, each $P\in\mathfrak{R}_a$ is associated with a set $R_P$ residual in the interval $[0,a_P]$ with $a_P\le a$. A set $\mathfrak{C}_a$ is said cusp-residual in $\mathfrak{B}_a$ if
$$
\mathfrak{C}_a=\{\lambda P:P\in\mathfrak{R}_a,\lambda\in R_P\}.
$$

Let $\Phi_H^t$ denote the Hamiltonian flow determined by $H$. Given an initial value $(x,y)$, $\Phi_H^t(x,y)$ generates an orbit of the Hamiltonian flow $(x(t),y(t))$.  An orbit $(x(t),y(t))$ is said to visit $B_{\delta}(y_0)\subset\mathbb{R}^3$ if there exists $t\in\mathbb{R}$ such that $y(t)\in B_{\delta}(y_0)$ a ball centered at $y_0$ with radius $\delta$.

\begin{theo}\label{mainthm}
Given any two balls $B_{\delta}(x_0,y_0),B_{\delta}(x_k,y_k)\subset\mathbb{T}^3\times\mathbb{R}^3$ and finitely many small balls $B_{\delta}(y_i)\subset \mathbb{R}^3$ $(i=0,1,\cdots,k)$, where $y_i\in h^{-1}(E)$ with $E>0$ and $\delta>0$ is small, there exists a cusp-residual set $\mathfrak{C}_{\epsilon_0}$ such that for each $\epsilon P\in\mathfrak{C}_{\epsilon_0}$, the Hamiltonian flow $\Phi_H^t$ admits orbits which, on the way between passing through $B_{\delta}(x_0,y_0)$ and $B_{\delta}(x_k,y_k)$, visit the balls $B_{\delta}(y_i)$ in any prescribed order.
\end{theo}

This is the main result of the paper. It is generic not only in usual sense, but also in the sense of Ma\~n\'e, namely, it is a typical phenomenon when the system is perturbed by potential.

The same result for time-periodic systems can be proved by using the same method. The statement of the result is: for typical time-periodic perturbations of integrable Hamiltonian with 2-degrees of freedom, the Hamiltonian flow admits orbits passing through any prescribed two balls $B_{\delta}(x_0,y_0)$ and $B_{\delta}(x_k,y_k)$ in the phase space and finitely many small balls $B_{\delta}(y_i)\subset \mathbb{R}^2$ $(i=0,1,\cdots,k)$ in the action variable space. The proof is easier from technical point of view, one can see it in the proof.

Similar result was announced by Mather earlier \cite{Ma4}. Recently, two other groups (Kaloshin and Zhang, Marco \cite{Mac}) announced, using different approach, a similar result for time-periodic systems with two degrees of freedom.

The result obtained here is stronger than what was formulated in \cite{Ar2}. By dropping the requirement that orbit passes two prescribed balls in the phase space and using the same construction, one can get an orbit that visits these balls $B_{\delta}(y_i)\subset \mathbb{R}^3$ $(i=0,1,\cdots,k)$ infinitely many times with  any prescribed order, as it was announced in \cite{Ma4}. Indeed, as finitely many balls are given, there exists a path with finite length passing through finitely many resonance layers and connecting any two of these balls directly. It does not damage the cusp-residual property.

\subsection{Brief introduction of Mather theory}
We use variational method to prove the result, which is based on Mather theory. This theory is established for Tonelli Lagrangian.

\begin{defi}
Let $M$ be a closed manifold. A $C^2$-function $L$: $TM\times\mathbb{T}\to\mathbb{R}$ is called Tonelli Lagrangian if it satisfies the following conditions:

\noindent{\ui Positive definiteness}. For each $(x,t)\in M\times\mathbb{T}$, the Lagrangian function is strictly convex in velocity: the Hessian $\partial_{\dot x\dot x}L$ is positive definite.

\noindent{\ui Super-linear growth}. We assume that $L$ has fiber-wise superlinear growth: for each $(x,t)\in M\times\mathbb{T}$, we have $L/\|\dot x\|\to\infty$ as $\|\dot x\|\to \infty$.

\noindent{\ui Completeness}. All solutions of the Lagrangian equations are well defined for the whole $t\in\mathbb{R}$.
\end{defi}

For autonomous systems, the completeness is automatically satisfied, since each orbit entirely stays in certain compact energy level set.

Let $\eta_c(x)$ denote a closed 1-form $\langle\eta_c(x),dx\rangle$ evaluated at $x$, with its first co-homology class $[\langle\eta_c(x),dx\rangle]=c\in H^1(M,\mathbb{R})$. We introduce a Lagrange multiplier $\eta_c=\langle\eta_c(x),\dot x\rangle$. Without danger of confusion, we call it closed 1-form also.

For each $C^1$ curve $\gamma$: $\mathbb{R}\to M$ with period $k$, there is unique probability measure $\mu_{\gamma}$ on $TM\times\mathbb{T}$ so that the following holds
$$
\int_{TM\times\mathbb{T}}fd\mu_{\gamma}=\frac 1k\int_0^kf(d\gamma(s),s)ds
$$
for each $f\in C^0(TM\times\mathbb{T},\mathbb{R})$, where we use the notation $d\gamma=(\gamma,\dot\gamma)$. Let
$$
\mathfrak{H}^*=\{\mu_{\gamma} |\ \gamma\in C^1(\mathbb{R},M)\ \text{\rm is periodic of}\ k\}.
$$
The set $\mathfrak{H}$ of holonomic probability measures is the closure of $\mathfrak{H}^*$ in the vector space of continuous linear functionals. One can see that $\mathfrak{H}$ is convex.

For each $\nu\in\mathfrak{H}$ the action $A_c(\nu)$ is defined as follows
$$
A_c(\nu)=\int (L-\eta_c)d\nu.
$$
It is proved in \cite{Ma1,Me} that for each co-homology class $c$ there exists at least one invariant probability measure $\mu_c$ minimizing the action over $\mathfrak{H}$
$$
A_c(\mu_c)=\inf_{\nu\in\mathfrak{H}}\int (L-\eta_c)d\nu,
$$
called $c$-minimal measure. Let $\mathfrak{H}_c\subset\mathfrak{H}$ be the set of $c$-minimal measures, the Mather set $\tilde{\mathcal {M}}(c)$ is defined as
$$
\tilde{\mathcal{M}}(c)=\bigcup_{\mu_c\in\mathfrak{H}_c}\text{\rm supp} \mu_c.
$$
The $\alpha$-function is defined as $\alpha(c)=-A_c(\mu_c): H^1(M,\mathbb{R})\to\mathbb{R}$, it is convex, finite everywhere with super-linear growth. Its Legendre transformation $\beta: H_1(M,\mathbb{R})\to\mathbb{R}$ is called $\beta$-function
$$
\beta(\omega)=\max_c (\langle\omega,c\rangle -\alpha(c)).
$$
It is also convex, finite everywhere with super-linear growth (see \cite{Ma1}).

Note that $\int\lambda d\mu_{\gamma}=0$ for each exact 1-form $\lambda$ and each $\mu_{\gamma} \in\mathfrak{H}^*$. Therefore, for each measure $\mu\in\mathfrak{H}$ one can define its rotation vector
$\omega(\mu)\in H_1(M,\mathbb{R})$ such that
$$
\langle[\lambda],\omega(\mu)\rangle=\int\lambda d\mu,
$$
for every closed 1-form $\lambda$ on $M$. For a closed curve $\gamma$: $[0,k]\to M$ its rotation vector is defined as
$$
[\gamma]=\frac{\bar{\gamma}(k)-\bar{\gamma}(0)}k,
$$
where $\bar{\gamma}$ stands for a curve in the lift of $\gamma$ to the universal covering $\mathbb{R}^n$. Let $\gamma_k$: $[0,k]\to M$ be a closed curve such that $[\gamma_k]=\omega_k$ and
$$
\frac 1kA(\gamma_k)=\inf_{[\gamma]=\omega_k}\frac 1k\int_0^kL(d\gamma(t),t)dt.
$$
The curve determines a periodic orbit $(\gamma_k,\dot\gamma_k)$ of $\phi^t_L$, the Lagrange flow determined by the Lagrangian $L$, it supports an invariant measure $\mu_k$ whose rotation vector is $\omega_k$. The measure $\mu_k$ is not necessarily minimal for $\beta(\omega_k)$. Nevertheless, if we choose a sub-sequence of closed curves $\{\gamma_{k_i}\}$ such that
$$
\lim_{k_i\to\infty}\frac 1{k_i}A(\gamma_{k_i})=\liminf_{k\to\infty}\inf_{[\gamma]=\omega_k}\frac
1k\int_0^kL(d\gamma(t),t)dt,
$$
and if $[\gamma_{k_i}]\to\omega$, then
$$
\lim_{k_i\to\infty}\frac 1{k_i}A(\gamma_{k_i})=\beta(\omega).
$$
Clearly, there is at least one invariant measure $\mu$ such that $\mu_{k_i}\rightharpoonup\mu$, and $\mu$ is a holonomic probability measure with the prescribed rotation vector $\omega(\mu)=\omega$. According to the definition of holonomic measure, and due to the work in \cite{Me}, we have
$$
\beta(\omega)=\inf _{\nu\in\mathfrak{H}_{\omega}}\int\ell d\nu
$$
where $\mathfrak{H}_{\omega}$ is a set of holonomic probability measures with the given rotation vector $\omega$, not necessarily invariant for $\phi_L^t$.

The Fenchel-Legendre transformation $\mathscr{L}_{\beta}$: $H_1(M,\mathbb{R})\to H^1(M,\mathbb{R})$ is defined by the following relation
$$
c\in\mathscr{L}_{\beta}(\rho)\ \ \iff \ \ \alpha(c)+\beta(\rho)=\langle c,\rho\rangle.
$$

The concept of semi-static curves is introduced by Mather and Ma\~n\'e (cf. \cite{Ma2,Me}). A curve
$\gamma$: $\mathbb{R}\to M$ is called $c$-semi-static if in time-1-periodic case we have
$$
[A_c(\gamma)|_{[t,t']}]=F_c((\gamma(t),t),(\gamma(t'),t'))
$$
where
\begin{equation}\label{introeq2}
[A_c(\gamma)|_{[t,t']}]=\int_{t}^{t'}\Big (L(d\gamma(t),t)-\eta_c(d\gamma(t))\Big ) dt+\alpha(c)(t'-t),\notag
\end{equation}
$$
F_c((x,t),(x',t'))=\inf_{\stackrel {\tau=t\ \text{\rm mod}\, 1}{\scriptscriptstyle \tau '=t'\text{\rm mod}\,1}} h_c((x,\tau),(x',\tau')),
$$
in which
\begin{equation}\label{introeq3}
h_c((x,\tau),(x',\tau'))=\inf_{\stackrel{\stackrel {\xi\in C^1}{\scriptscriptstyle \xi(\tau)=x}}{\scriptscriptstyle \xi(\tau')=x'}}[A_c(\xi)|_{[\tau,\tau']}].\notag
\end{equation}
In autonomous case, the period can be considered as any positive number. Consequently, the notation of semi-static curve in this case is somehow simpler
$$
[A_c(\gamma)|_{(t,t')}]=F_c(\gamma(t),\gamma(t')),
$$
where
\begin{equation}\label{introeq4}
F_c(x,x')=\inf_{\tau>0}h_c((x,0),(x',\tau)).\notag
\end{equation}

\noindent{\bf Convention}: Let $I\subseteq\mathbb{R}$ be an interval $($either bounded or unbounded$)$. A continuous map $\gamma$: $I\to M$ is called curve. If it is differentiable, the map $d\gamma=(\gamma,\dot\gamma)$: $I\to TM$ is called orbit. When the implication is clear without danger of confusion, we use the same symbol to denote the graph, $\gamma:=\cup_{t\in I}(\gamma(t),t)$ is called curve and $d\gamma:=\cup_{t\in I}(\gamma(t),\dot\gamma(t),t)$ is called orbit. In autonomous system, the terminology also applies to the image: $\gamma:=\cup_{t\in I}\gamma(t)$ is called curve and $d\gamma:=\cup_{t\in I}(\gamma(t),\dot\gamma(t))$ is called orbit.

A semi-static curve $\gamma\in C^1(\mathbb{R},M)$ is called $c$-static if, in addition, the relation
$$
[A_c(\gamma)|_{(t,t')}]=-F_c((\gamma(t'),\tau'),(\gamma(t),\tau))
$$
holds in time-1-periodic case and
$$
[A_c(\gamma)|_{(t,t')}]=-F_c(\gamma(t'),\gamma(t))
$$
holds in autonomous case. An orbit $X(t)=(d\gamma(t), t\, \text{\rm mod}\ 2\pi)$ is called $c$-static (semi-static) if $\gamma$ is $c$-static (semi-static). We call the Ma\~n\'e set $\tilde{\mathcal {N}}(c)$ the union of $c$-semi-static orbits
$$
\tilde{\mathcal{N}}(c)=\bigcup\{d\gamma:\gamma\ \text{\rm is}\ c\text{\rm -semi static}\}
$$
and call the Aubry set $\tilde{\mathcal {A}}(c)$ the union of $c$-static orbits
$$
\tilde{\mathcal{A}}(c)=\bigcup\{d\gamma :\gamma\ \text{\rm is}\ c\text{\rm -static}\}.
$$

We use $\mathcal {M}(c)$, $\mathcal {A}(c)$ and $\mathcal {N}(c)$ to denote the standard projection of $\tilde{\mathcal {M}}(c)$, $\tilde{\mathcal {A}}(c)$ and $\tilde{\mathcal {N}}(c)$ from $TM\times\mathbb{T}$ to $M\times\mathbb{T}$ respectively. They satisfy the inclusion relation
$$
\tilde{\mathcal{M}}(c)\subseteq\tilde{\mathcal{A}}(c)\subseteq\tilde{\mathcal{N}}(c).
$$
It is showed in \cite{Ma1,Ma2} that the inverse of the projection is Lipschitz when it is restricted to $\mathcal {A}(c)$ as well as to $\mathcal {M}(c)$. By adding subscript $s$ to $\mathcal{N}$, i.e. $\mathcal{N}_s$ we denote its time-$s$-section. This principle also applies to $\tilde{\mathcal{N}}(c)$, $\tilde{\mathcal{A}}(c)$, $\tilde{\mathcal{M}}(c)$, $\mathcal{A}(c)$ and $\mathcal{M}(c)$ to denote their time-$s$-section respectively. For autonomous systems, these sets are defined without the time component.

On the time-1-section of Aubry set a pseudo-metric $d_c$ is introduced by Mather in \cite{Ma2}, its definition relies on the quantity $h_c^{\infty}$. Let
\begin{equation}\label{introeq5}
h_c^{\infty}((x,s),(x',s'))=\liminf_{\stackrel{\stackrel{s=t\ \text{\rm mod}\ 1}{\scriptscriptstyle t'=s'\ \text{\rm mod}\ 1}}{\scriptscriptstyle t'-t\to\infty}}h_c((x,t),(x',t')),\notag
\end{equation}
\begin{equation}\label{introeq6}
h^{\infty}_c(x,x')=\liminf_{k\to\infty}h_c((x,0),(x',k)).\notag
\end{equation}
The pseudo-metric $d_c$ on Aubry set is defined as
$$
d_c((x,t),(x',t'))=h_c^{\infty}((x,t),(x',t'))+h_c^{\infty}((x',t'),(x,t)).
$$
With the pseudo-metric $d_c$ one defines equivalence class in Aubry set. The equivalence $(x,t)\sim (x',t')$ implies $d_c((x,t),(x',t'))=0$, with which one can define quotient Aubry set $\mathcal{A}(c)/\sim$. Its element is called Aubry class, denoted by $\mathcal{A}_i(c)$, its lift to $TM\times\mathbb{T}$ is denoted by $\tilde{\mathcal{A}}_i(c)$. Thus,
$$
\mathcal{A}(c)= \bigcup_{i\in\Lambda}\mathcal{A}_i(c), \qquad \tilde{\mathcal{A}}(c)=\bigcup_{i\in\Lambda}\tilde{\mathcal{A}}_i(c).
$$
In \cite{Ma5} Mather constructed an example with some quotient Aubry set homeomorphic to an interval. However, it is proved generic in \cite{BC} that, for system with $n$ degrees of freedom, each $c$-minimal measure contains not more than $n+1$ ergodic components. In this case, each Aubry set contains at most $n+1$ classes.

The definition of semi-static curve as well as of Ma\~n\'e set depends on which configuration manifold under our consideration. Let $\pi:\bar M\to M$ be a finite covering, a curve $\gamma$: $\mathbb{R}\to M$ is said semi-static in $\bar M$ if each curve in its lift $\bar\gamma$ is semi-static in $\bar M$. Accordingly, we define $\tilde{\mathcal{N}}(c,\bar M)$ ($\mathcal{N}(c,\bar M)$) as the set containing all $c$-semi-static orbits (curves) in $\bar M$. We use the symbol $\mathcal{N}(c)$ when $M$ is defaulted as the configuration manifold.

It is possible that $\pi\mathcal{N}(c,\bar M)\supsetneq\mathcal{N}(c,M)$. For instance, if $N\subset M$ is a open region such that $H_1(M,N,\mathbb{Z})\neq H_1(M,\mathbb{Z})$, $\mathcal{N}\subset N$
and the lift of $N$ in $\bar M$ has more than one connected component, then this phenomenon takes place.  But we have

\begin{pro} Let $\pi:\bar M\to M$ be a finite covering space, then
$$
\pi\mathcal{A}(c,\bar M)=\mathcal{A} (c,M).
$$
\end{pro}
\begin{proof}
Pick up any $\bar x\in\pi\mathcal{A}(c,\bar M)$ and any small $\delta>0$, by definition, there exists sufficiently large $T>0$ as well as a curve $\bar\xi:[0,T]\to\bar M$ such that $\bar\xi(0)=\bar\xi(T)=\bar x$ and $[A_c(\bar\xi)]<\delta$. Let $\xi$ denote the project of $\bar\xi$ down to $M$, clearly, we have
$[A_c(\xi)]=[A_c(\bar\xi)]<\delta$. Let $x=\pi\bar x$, clearly, $x\in\mathcal{A}(c)$.
\end{proof}

\subsection{Outline of the proof}
We use variational method to prove the result. Since the work of Mather \cite{Ma2,Ma3}, the variational method has become a powerful tool for the study of dynamical instability in positive definite Lagrange systems with multiple degrees of freedom.

In the study of Arnold diffusion in {\it a priori} stable systems with three degrees of freedom, mainly due to the work of Mather \cite{Ma6}, it has been widely known that the main difficulty takes place near double resonance. To describe what puzzled us and to explain the strategy of our proof, let us recall the example of Arnold and previous study on {\it a priori} unstable systems.

In the example of Arnold, there exists a 2-dimensional cylinder in the phase space, which is invariant and normally hyperbolic for the time-1-map determined by the Hamiltonian flow. This cylinder is foliated into a family of invariant circles, each of them has stable and unstable manifold which intersect each other transversally. Consequently, the unstable manifold of some circle intersects the stable manifold of other circles nearby, it implies the existence of a sequence of successively connected heteroclinic orbits. This structure is called transition chain by Arnold. Diffusion orbits are then constructed shadowing these heteroclinic orbits.

Such argument heavily depends on the geometric structure and variational method turns out to have wider range of application. Let us interpret the proof by variational language. Each invariant circle is the Aubry set for certain cohomology class, the stable as well as the unstable manifold is actually the graph of the differential of the weak KAM solution. Expressed as the difference of backward and forward weak KAM, the barrier function reaches its minimum at primary intersection points of these two manifolds. These homoclinic orbits and the Aubry set constitute the Ma\~n\'e set in certain finite covering space. For positive definite systems, the transversal intersection implies the minimality of the homoclinic orbits as well as local heteroclinic orbits, along which the Lagrange action reaches the minimum among all those curves with the same boundary conditions. The diffusion orbits are obtained by searching for global minimizers which generate orbits shadowing a sequence of local heteroclinic orbits.

The variational arguments still work even if there does not exist such nice geometric structure, provided the following conditions are satisfied for each cohomology class

{1, \it the Aubry set is lower dimensional: $H_1(M,\mathcal{A}(c),\mathbb{Z})\neq 0$};

{2, \it the stable ``manifold" intersects the unstable ``manifold" transversally}.

However, it turns out very difficult to verify whether the stable ``manifold" intersects the unstable ``manifold" transversally for each cohomology class along a path of first cohomology class. As uncountably many stable and unstable ``manifolds" have to be considered, one can not verify the genericity of the transversal intersection by taking the intersection of countably many open-dense sets. One possible way is to study some regularity of barrier functions with respect to some parameter, with which one obtains the finiteness of Hausdorff dimensions of the set of barrier functions. As such regularity is obtained in the case when a normally hyperbolic cylinder exists we are succeeded in solving the problem in {\it a priori} unstable case \cite{CY1,CY2,LC}. In fact, the regularity was obtained in \cite{CY1} only for those barrier functions for which the minimal measure is supported on an invariant circle. It was extended in \cite{Zho1} to all other barrier functions, which allows us to construct diffusion orbits of which the picture looks like what was constructed by Arnold, while in our previous work, the constructed orbits keep close to the cylinder when they pass through strong resonance (Birkhoff instability region). The ``gap" problem was then solved.

Intuitively, diffusion orbits in {\it a priori} stable systems may be constructed along some resonant path. In terms of rotation vector (first homology class), each point on this curve satisfies at least one resonant condition for the system with $3$ degrees of freedom. In integrable systems, each resonant path corresponds to an invariant cylinders without any hyperbolicity. Under generic perturbations, it breaks into many pieces of normally hyperbolic cylinder, but may disappear around double resonant points. It implies a bad consequence: we lost a handhold to get certain regularity of barrier functions in suitable parameter. It then becomes unclear whether there is a transition chain near double resonance.

In terms of first cohomology, strong double resonance corresponds to a convex disc with size $O(\sqrt{\epsilon})$ if the perturbation is of order $O(\epsilon)$, each piece of normally hyperbolic cylinder corresponds to a channel which extends to a small neighborhood of the disc. But it is unclear whether these channels are connected to the disc of double resonance.

The method we use to overcome this difficulty bases on following discoveries:

First, each double resonant disc is surrounded by an annulus foliated into a family of paths, along each of these paths, there is another invariant (a coordinate component of the cohomology class) besides the average action. For each class in this annulus, the intersection of stable ``manifold" with  unstable ``manifold" is nontrivial although it may not be transversal. This annulus has width of order $O(\epsilon)$.

Next,  incomplete intersection of the stable and unstable ``manifold" of an Aubry set does not implies that it can not be connected to any other Aubry set nearby. It does if they are $c$-equivalent.

Finally, the channels of normally hyperbolic cylinder reach to somewhere $\epsilon^{1+\delta}$-close to double resonant disc ($\delta>0$), i.e. it has overlap with the annulus. For each class in the channel, the relative homology of the Aubry set is non-trivial.

Therefore, we are able to find a path close to prescribed one, for each class on the path, the Aubry set is connected to another one nearby if the class is also on the path. All of these connecting orbits are minimal in local sense. The diffusion orbits are constructed shadowing these successively connected orbits.

We organize the proof in following way. Section 2 is used to establish the concept of elementary weak KAM, with which one has simpler expression of barrier function. It thus becomes easier to study genericity of transition chain. Section 3, 4 and 5 are devoted to study the structure of Ma\~n\'e set and of Aubry set. Since these sets are symplectic invariants \cite{Be2}, we do it by studying the normal form which is put into the appendix. The truncated Hamiltonian of the normal form is a system with two degrees of freedom. In Section 3, we study the dynamics around the double resonance, and the modulus continuity of the period on energy (average action). With these preliminary works, normally hyperbolic cylinder is shown to get very close to the double resonance in Section 4, the existence of annulus of $c$-equivalence is established in Section 5. Section 6 is devoted to establish two types of local connecting orbits. The local minimality of these local connecting orbits are naturally given, it enables us to construct global connecting orbit shadowing these local connecting orbits. It is obtained by searching for the minimizer of certain modified Lagrangian, which is done in Section 7. Finally in Section 8, we verify the cusp-residual property of the transition chain in nearly integrable Hamiltonian systems with three degrees of freedom. Consequently, the main result of this paper is proved.

\section{\ui Elementary Weak KAM and Barrier}
\setcounter{equation}{0}

The concept of elementary weak KAM solution is introduced for $c$-minimal measure with finitely many ergodic components, this condition has been shown to be generic in \cite{BC}. Each ergodic component $\mu^i_c$ determines a pair of elementary weak KAM $u^{\pm}_{c,i}$, with which we introduce a barrier function
$$
B_{c,i,j}=u^-_{c,i}-u^+_{c,j}.
$$
With this formula, it is easier to show that, generically, the set $\arg\min B_{c,i,j}\backslash(\mathcal{A}(c)+\delta)$ is totally disconnected. This property is crucial for the construction of diffusion orbits.

\subsection{Elementary weak KAM}
The concept of $c$-semi-static curves can be extended to the curves only defined on $\mathbb{R}^{\pm}$, which are called forward or backward $c$-semi-static curves respectively.  Usually one uses $\gamma^-_c(t,x,\tau)$: $(-\infty,\tau]\to M$ to denote backward $c$-semi-static curve such that $\gamma^-_c(\tau)=x$, and uses $\gamma^+_c(t,x,\tau)$: $[\tau,\infty)\to M$ to denote forward $c$-semi-static curve such that $\gamma^+_c(\tau)=x$. In autonomous case, one uses the notation $\gamma^{\pm}_c(t,x)$ such that $\gamma^{\pm}(0,x)=x$. Let
\begin{align}
\tilde{\mathcal{N}}^+(c)&=\{(x,\dot x,\tau)\in TM\times\mathbb{T}: \pi_x\phi^t_L (z,\tau)|_{[\tau,+\infty)}\ \text{\rm is \it c\rm-semi-static}\},\notag\\
\tilde{\mathcal{N}}^-(c)&=\{(x,\dot x,\tau)\in TM\times\mathbb{T}: \pi_x\phi^t_L (z,\tau)|_{(-\infty,\tau]}\ \text{\rm is \it
c\rm-semi-static}\},\notag
\end{align}
where $0\le\tau<1$, $\pi_x(x,\dot x)=x$ denotes the standard projection along the tangent fiber and $\phi^t_L(x,\dot x,\tau)$ denotes the orbit of the Lagrangian flow with the initial value $(x,\dot x)$ at the time $\tau$. The corrsponding orbits are called forward (backward) semi-static orbit set respectively. These two sets are upper semi-continuous for the cohomology class.
\begin{pro}\label{weakpro1}\text{\rm (see \cite{Bo})}
If the Lagrangian $L$ is of Tonelli type, for each point $(x,\tau)\in M\times\mathbb{T}$, there is at least one $\gamma^{\pm}_c(t,x,\tau)$ which is forward $($backward$)$ semi-static curve.
\end{pro}

As both the $\omega$-limit set of $d\gamma_c^+$ and the $\alpha$-limit set of $d\gamma_c^-$ are in the Aubry set one can define
$$
W^{\pm}_c=\bigcup_{(x,\tau)\in M\times\mathbb{T}}
\left\{x,\tau,\frac{d\gamma^{\pm}_c(\tau,x,\tau)}{dt}\right\},
$$
and call $W^+_c$ the stable set, $W^-_c$ the unstable set of the $c$-minimal measure respectively. If  $\dot\gamma^-(\tau,x,\tau)=\dot\gamma^+(\tau,x,\tau)$ holds for some $(x,\tau)\in M\times\mathbb{T}$, passing through the point $(x,\tau,\dot\gamma^-_c(\tau,x,\tau))$ the orbit is either in the Aubry set or homoclinic to this Aubry set.

When the Aubry set contains only one class, the stable as well as the unstable set has its own generating function $u_c^{\pm}$ such that $W^{\pm}_c=\text{\rm Graph}(du_c^{\pm})$ holds almost everywhere \cite{Fa1,E}. These functions are weak KAM solutions, which are the fixed points of so called Lax-Oleinik operator. We use $u_c^{\pm}$ to denote the weak KAM solution for the Lagrangian $L-\eta_c$, where $\eta_c$ is a closed form with $[\eta_c]=c$. These functions are Lipschitz, thus differentiable almost everywhere. At each differentiable point $(x,\tau)$, $(x,\tau,\partial_xu^-(x,\tau))$ uniquely determines backward $c$-semi static curve $\gamma^-_x$: $(-\infty, \tau]\to M$ such that $\gamma^-_x(\tau)=x$, $\dot\gamma^-_x(\tau)=\partial_yH(x,\tau,\partial_xu^-(x,\tau))$. Similarly, $(x,\tau,\partial_xu^+(x,\tau))$ uniquely determines forward $c$-semi static curve $\gamma^-_x$: $[\tau,\infty)\to M$ such that $\gamma^+_x(\tau)=x$, $\dot\gamma^+_x(\tau)=\partial_yH(x,\tau,\partial_xu^+(x,\tau))$.

Given a class $c\in H^1(M,\mathbb{R})$, we use $\{\tilde{\mathcal{A}}_c^i\}_{i\in\Lambda} \subset TM$ to denote the set of Aubry classes, use $\{\mathcal{A}^i_c\} _{i\in\Lambda}\subset M$ to denote the projected set along the tangent fibers, where $\Lambda$ is the subscript set: $\tilde{\mathcal{A}}(c)=\cup_{i\in\Lambda}\tilde{\mathcal{A}}_c^i$. We also use the notation $\tilde{\mathcal{M}}_c^i=\text{\rm supp}\mu_c^i$ where $\mu_c^i$ is an ergodic component of the $c$-minimal measure $\mu_c$. Let
$\mathcal{M}_c^i=\pi\tilde{\mathcal{M}}_c^i$.
\begin{pro}\label{weakpro2} {\rm (\cite{Fa2})}
Let $u^{\pm}_c$ and $u'^{\pm}_c$ be two weak-KAM solutions for $c$. Their difference keeps constant when they are restricted on an Aubry class $(u^{\pm}_c-u'^{\pm}_c)|_{\mathcal{A}_c^i}=\text{constant}$.
\end{pro}

Recall the definition of $h_c^{\infty}$ in the introduction. We use the symbol $h^{\infty}_L$ to denote the quantity defined in the same way for $L$ with $c=0$, and drop the subscript $L$ when it is clearly defined. The quantity $h_c^{\infty}(z,z')$ is a weak-KAM solution if we consider it as the function of $z$ or of $z'$. Let us consider the case that the $c$-minimal invariant measure has finitely many ergodic components. In this case, this function has some kind of continuity.
\begin{theo}\label{weakthm1}
Let $\{L_{\delta}\}$ be a sequence of Lagrangian, converging to $L$ in $C^2$-topology as $\delta\to 0$
when they are restricted on any bounded regions of $TM\times\mathbb{T}$. We assume that the minimal measure for $L$ consists of finitely many ergodic components $\mu^1,\mu^2$, $\cdots,\mu^m$ and the distance from $(x,\tau)\in\mathcal{M}^i$ to the Aubry set for $L_{\delta}$, $d((x,\tau), \mathcal{A}_{L_{\delta}})\to 0$ as $\delta\to 0$. Then
$$
\lim_{\delta\to 0}h_{L_\delta}^{\infty}((x,\tau),(x',\tau'))=h^{\infty}((x,\tau),(x',\tau')).
$$
\end{theo}
\begin{proof} We only need to prove it on the time-1-section, e.g.  for $\tau=\tau'=0$. So we omit the notation for the component $\tau$. For each $\epsilon>0$, there exists $k>0$ such that $|h^{\infty}(x,x') -h^{k}(x,x')|<\epsilon$. Let $\alpha$ and $\alpha_{\delta}$ denote the minimal average action of $L$ and $L_{\delta}$ respectively. Let $\gamma^k$: $[0,k]\to M$ be the curve such that $\gamma^k(0)=x$, $\gamma^k(k)=x'$ and
$$
[A_L(\gamma^k)]=\int_0^{k}L (d\gamma^{k}(t),t)dt+k\alpha=h^{k}(x,x').
$$
For any $k'>k$, let $\zeta$: $[0,k']\to M$ be an absolutely continuous curve such that $\zeta(0)=x$, $\zeta(t-k'+k)=\gamma^k(t)$ for $t\in[k'-k,k']$ and $[A_{L_{\delta}}(\zeta|_{[0,k'-k]})]= \int_0^{k'-k}L_{\delta} (d\zeta(t),t)dt+(k'-k)\alpha_{\delta}=h_{L_{\delta}}^{k'-k}(x,x)$. Therefore, we have
\begin{align*}
[A_{L_{\delta}}(\zeta)]&\le h_{L_\delta}^{k'-k}(x,x)+h^{k}(x,x')+k|\alpha-\alpha_{\delta}|\\
& +\left|\int_0^k(L-L_{\delta})(d\gamma^k(t),t)dt\right|.
\end{align*}
Since $d(x, \mathcal{A}_{L_{\delta}}|_{t=0})\to 0$ as $\delta\to 0$ we see that $\liminf_{k'-k}h_{\delta}^{k'-k}(x,x) \to 0$ as $\delta\to 0$. Since $\alpha$ is continuous in the Lagrangian and $\epsilon$ is arbitrarily small we see that
\begin{equation*}
\limsup_{\delta\to 0}h_{L_\delta}^{\infty}(x,x')\le h^{\infty}(x,x').
\end{equation*}

So, in order to complete the proof, we only need to show
\begin{equation}\label{weakeq2}
\liminf_{\delta\to 0}h_{L_\delta}^{\infty}(x,x')\ge h^{\infty}(x,x').
\end{equation}
Let $\gamma_{\delta}^{k_{\ell}}$: $[0,k_{\ell}]\to M$ be a curve such that $\gamma_{\delta}^{k_{\ell}}(0)=x$,
$\gamma_{\delta}^{k_{\ell}}(k_{\ell})=x'$ and
$$
[A_{\delta}(\gamma_{\delta}^{k_{\ell}})]=h_{L_\delta}^{k_{\ell}}(x,x')\to h_{L_\delta}^{\infty}(x,x'),
$$
where $k_{\ell}\to\infty$ is a sequence of integers. Let $O_{\epsilon}(S)$ denote the $\epsilon$-neighborhood of the set $S$. For small $\epsilon>0$, there exist some $j$ with $1\le j\le m$, an integer $k_j\in [0,k_{\ell}]$ and $x_j\in\mathcal {M}^j_{0}=\mathcal {M}^j|_{t=0}$ such that $\gamma_{\delta}^{k_{\ell}}(k_j)\in O_{\epsilon}(x_j)$ provided $k_{\ell}$ is sufficiently large.

Let us consider those ergodic components of minimal measure for $L$ of which the support is approached by $d\gamma^{k_{\ell}}_{\delta}$ as $\delta\to 0$: $\{d\gamma^{k_{\ell}}\}\cap O_{\epsilon} (\tilde{\mathcal{M}}^j)\neq\varnothing$.  We number some $j$ as $j_1$ if some $x_1\in\mathcal{M}^{j_1}_{0}$ exists such that $\gamma^{k_{\ell}}_{\delta}(k_1)\in O_{\epsilon}(x_1)$ and  for each $k<k_1$, $\gamma^{k_{\ell}}_{\delta}(k)$ does not fall into $\epsilon$-neighborhood of any $\mathcal{M}^{j}_{0}$. Let $k'_1\ge k_1$ be the integer such that $\gamma^{k_{\ell}}_{\delta}(k'_1)\in O_{\epsilon}(x_1)$ and $\gamma^{k_{\ell}}_{\delta}(k)\notin O_{\epsilon}(x_1)$ for all $k>k'_1$. We number some $j_2\ne j_1$ if some $k_2>k'_1$ and some $x_2\in\mathcal{M}^{j_2}_{0}$ exists such that $\gamma^{k_{\ell}}_{\delta}(k_2)\in O_{\epsilon}(x_2)$, let $k'_2\ge k_2$ be the integer such that $\gamma^{k_{\ell}}_{\delta}(k'_2)\in O_{\epsilon}(x_2)$ and $\gamma^{k_{\ell}}_{\delta}(k')\notin O_{\epsilon}(x_2)$ for all $k>k'_2$. Inductively,  one obtains $x_i\in\mathcal{M}^{j_i}_{0}$ ($i=1,2\cdots m'\le m$) and
$$
0\le k_1\le k'_1\le\cdots\le k_{m'}\le k'_{m'}\le k_{\ell}.
$$
Obviously, there exist small $\varepsilon=\varepsilon(\epsilon)>0$ and large integer $K=K(\epsilon)$ such that
$$
|k_j-k'_{j-1}|\le K, \qquad \forall\ k_{\ell}\to\infty
$$
provided $|L_{\delta}-L|<\varepsilon$. Otherwise, there would exist also an ergodic component $\nu$ of the minimal invariant measure such that $\nu\neq\mu^j$ for all $0\le j\le m$, but it is absurd.

Given small $\epsilon>0$, let $k_{\ell}$ be the integer such that $|h_{L_\delta}^{k_{\ell}}(x,x') -h_{L_\delta}^{\infty}(x,x')|<\epsilon$. Let $\bar x_j=\gamma_{\delta}^{k_{\ell}}(k_j)$, $\tilde x_j= \gamma_{\delta}^{k_{\ell}}(k'_j)$, we choose an absolutely continuous curve $\zeta_j$: $[0,k_{\ell}^j]\to M$ such that $\zeta_j(0)=\bar x_j$, $\zeta_j(k_{\ell}^j)=\tilde x_j$ and $[A(\zeta_j)]=h^{k_{\ell}^j}(\bar x_j,\tilde x_j)$. As $\bar x_j, \tilde x_j\in O_{\epsilon}(x_j)$ we can choose sufficiently large $k_{\ell}^j$ such that
$$
|h^{k_{\ell}^j}(\bar x_j,\tilde x_j)|< C\epsilon,
$$
where $C=C(L)$ is a constant depending on $L$ only. As $\|\bar x_j-\tilde x_j\|\le 2\epsilon$, for any positive integer $i$ we have
$$
h^{i}_{L_\delta}(\bar x_j,\tilde x_j)>-C\epsilon.
$$
For any large integer $k'\in\mathbb{Z}$ with $k'\ge k$, we construct an absolutely continuous curve $\zeta$: $[0,k']\to M$ joining $x$ with $x'$ such that
$$
\zeta(t)=\begin{cases}
\gamma_{\delta}^{k_{\ell}}(t-\tau_{j-1}), &\text{\rm if}\ k'_{j-1}+\tau_{j-1}\le t\le k_j+\tau_{j-1},\\
\zeta_j(t-k_j-\tau_{j-1}), & \text{\rm if}\ k_j+\tau_{j-1}\le t\le k'_j+\tau_j
\end{cases}
$$
where $\tau_j=\sum_{\jmath=1}^j(k_{\ell}^{\jmath}-k'_{\jmath}+k_{\jmath})$, $k'=k_{\ell}+\tau_{m'}$. The action of $L$ along this curve is easily estimated
\begin{eqnarray*}
h^{k'}(x,x')-h_{L_\delta}^k(x,x')&\le &[A(\zeta)]-h_{L_\delta}^k(x,x')\\
&\le & 2m'(C\epsilon +K|\alpha-\alpha_{\delta}|)\\
& &+\sum _{j=1}^{m'}\left|\int_{k'_{j-1}}^{k_j}(L-L_{\delta})(d\gamma_{\delta}^{\ell}(t),t)dt\right|.
\end{eqnarray*}
As $|k_j-k_{j-1}|\le K$, and $K$ is independent of $\delta$ when $\delta$ is sufficiently close to $0$, we see that the inequality (\ref{weakeq2}) holds. This completes the proof.
\end{proof}
\begin{cor}
Let $\{c_i\}$ be a sequence of cohomology classes such that $c_i\to c$. The $c$-minimal measure is assumed consisting of finitely many ergodic components $\mu_c^1$, $\mu_c^2$, $\cdots, \mu_c^m$, $(x,\tau)\in\mathcal{M}_{c}^j$ and $d((x,\tau),\mathcal{M}(c_i))\to 0$ for some $0\le j\le m$, as $c_i\to c$. Then
$$
\lim_{c_i\to c}h_{c_i}^{\infty}((x,\tau),(x',\tau'))=h_c^{\infty}((x,\tau),(x',\tau')).
$$
\end{cor}
From the proof one can see  that the function $h_c^{\infty}$ is lower semi-continuous in $c$ if the $c$-minimal measure is assumed to have finitely many ergodic components:
$$
\liminf_{c'\to c}h_{c'}^{\infty}(z,z')\ge h_c^{\infty}(z,z').
$$

Let us introduce the concept of {\it elementary} weak KAM solution if the $c$-minimal measure has finitely many ergodic components. One can choose finitely many non-negative functions $g_i$: $M\times\mathbb{T}\to\mathbb{R}$ such that its support has no intersection with a small neighborhood of $\mathcal{M}_c^i$ and the minimal measure for the Lagrangian $L_{c,i,\epsilon}=L_c+\epsilon g_i$ is uniquely supported on $\mathcal{M}_c^i$. By the theory of weak KAM (\cite{Fa2}), there is exactly one pair of weak KAM solutions denoted by $u_{c,i,\epsilon}^{\pm}$ and
$$
h^{\infty}_{L_{c,i,\epsilon}}(z,z')=u_{c,i,\epsilon}^{-}(z')-u_{c,i,\epsilon}^{+}(z).
$$
Let $z\in\mathcal{M}_c^i$, in virtue of Theorem \ref{weakthm1}, one has
$h^{\infty}_{L_{c,i,\epsilon}}(z,z')\to h^{\infty}_{c}(z,z')$ as $\epsilon\to 0$. Since $g_i=0$ in the neighborhood of $\mathcal{M}_c^i$, $u_{c,i,\epsilon}^{+}(z)$ remains unchanged as $\epsilon\to 0$. Thus, there is a Lipschitz function $u^{-}_{c,i}$ such that $u_{c,i,\epsilon}^{-} \to u^{-}_{c,i}$ as $\epsilon\to 0$. Clearly this $u^{-}_{c,i}$ is a weak-KAM solution for $L_c$. Similarly, we can see that $u_{c,i,\epsilon}^{+}\to u^{+}_{c,i}$ as $\epsilon\to 0$.
\begin{defi} {\ui (elementary weak-KAM solution)}.
Assume that the minimal measure for $L_c$ consists of finitely many ergodic components $\mu_c^1$, $\mu_c^2,\cdots, \mu_c^m$. A weak KAM solution $u_{c,i}^{\pm}$ of $L_c$ is called {\it elementary} for $\mu_c^i$ if $u_{c,i}^{\pm}=\lim_{\epsilon\to 0}u_{c,i,\epsilon}^{\pm}$ where $u_{c,i,\epsilon}^{\pm}$ is the weak KAM solution of $L_{c,i,\epsilon}$, of which the minimal measure $\mu=\mu_c^i$ is uniquely ergodic and $L_{c,i,\epsilon}\to L_c$ as $\epsilon\to 0$.
\end{defi}

It is not necessary that $u_{c,i}^{\pm}$ is a pair of conjugate weak KAM. Clearly, if $(x,t)\in\mathcal{M}_c^i$
\begin{align}
&u_{c,i}^{-}(x',t')=h_c^{\infty}((x,t),(x',t'))+u_{c,i}^{+}(x,t),\notag\\
&u_{c,i}^{+}(x',t')=u_{c,i}^{-}(x,t)-h_c^{\infty}((x',t'),(x,t)).\notag
\end{align}

These elementary weak KAM solutions generate all weak KAM solutions in the following sense.
\begin{pro}\label{weakpro3}
Assume the minimal measure consists of $m$ ergodic components. For each weak KAM solution $u^{\pm}$, there exist $m'$  $(m'\le m)$ constants $d^{\pm}_1, \cdots,d^{\pm}_{m'}$ and $m'$ open domains $D^{\pm}_1,\cdots, D^{\pm}_{m'}$ such that they do not overlap each other, $M=\cup_{1\le i\le m'}\bar{D}^{\pm}_i$ and
\begin{equation}\label{weakeq4}
u^{\pm}|_{D^{\pm}_i}=u_i^{\pm}+d^{\pm}_i, \qquad \forall 1\le i\le m'.
\end{equation}
\end{pro}
\begin{proof} It is deduced from the Lipschitz property of $u^-$ that it is differentiable almost every where. Let $x$ be a point where $u^-$ is differentiable, $du^-(x)$ determines a unique backward semi static orbit $d\gamma_c^i$: $(-\infty,0]\to M$ whose $\alpha$-limit set is in certain Aubry class $\tilde{\mathcal{A}}_c^i$. By definition we have
$$
u^-(x)-u^-(\gamma_c^i(-t))=\int_{-t}^0L_{c}(d\gamma_c^i(s),s)ds+\alpha(c)t.
$$
Let $t_k\to\infty$ such that $\gamma_c^i(-t_k)\to x'\in\mathcal{A}_c^i$, it follows from Proposition \ref{weakpro2} that
\begin{equation}\label{weakeq5}
u^-(x)=h_c^{\infty}(x',x)+u^+(x')=u_{c,i}^{-}(x)+d_i.
\end{equation}
If $x^*\in M$ is another point where $du^-(x^*)$ determines a backward semi-static orbits whose $\alpha$-limit set is also contained in $\tilde{\mathcal{A}}_c^i$, we then obtain (\ref{weakeq5}) for $u^-(x^*)$ with the same $d_i$. All these points constitute a set connected with $\mathcal{A}_c^i$. There are not more than $m$ connected sets such that (\ref{weakeq4}) holds.
\end{proof}
\begin{theo} Let $c_i\to c$ be a sequence of cohomology and assume that the minimal measure consists of finitely many ergodic components for each $c_i$ and $c$. Let $\tilde{\mathcal{M}}^j_{c_i}$, $\tilde{\mathcal{M}}^j_c$ be the support  for the ergodic minimal measure $\mu^j_{c_i}$ and $\mu^j_c$ respectively, let $u^-_{c_i}$ and $u^-_{c}$ be the corresponding elementary weak KAM solution. If $\mu^j_{c_i}\rightharpoonup\mu^j_c$ as $c_i\to c$, then
$u^-_{c_i}\rightarrow u^-_{c}$ in $C^0$-topology.
\end{theo}
\begin{proof} It follows from the continuity of $h_c^{\infty}(x,x')$ in $c$ shown in Theorem \ref{weakthm1} and the definition of the elementary weak KAM solution.
\end{proof}

In terms of conjugate pair of weak KAM solution, one has a definition of Ma\~n\'e set in \cite{Fa2}. For the purpose of this paper, we would like to use elementary weak KAM solution. Recall the definition of the barrier function in \cite{Ma2}:
$$
B_c^*(x)=\min_{\xi,\zeta\in\mathcal{M}_0(c)}\{h_c^{\infty}(\xi,x)-h_c^{\infty}(x,\zeta)+h_c^{\infty} (\xi,\zeta)\}.
$$
If the minimal measure consists of finitely many ergodic components, we introduce barrier functions in terms of elementary weak KAM solutions: given $z=(x,\tau)\in M\times\mathbb{T}$, we set
\begin{equation}\label{weakeq6}
B_{c,i,j}(z)=u^-_{c,i}(z)-u^+_{c,j}(z).
\end{equation}
which measures the minimum of the action along those curves passing through $z$ and joining $\mathcal{M}^i_c$ to $\mathcal{M}^j_c$. For autonomous systems, this barrier function is independent of time. Obviously, each $c$-semi static curve corresponds to a minimum of $u^-_{i,c}-u^+_{j,c}$ if its $\alpha$-limit set intersects $\tilde{\mathcal{M}}_c^i$ and its $\omega$-limit set intersects $\tilde{\mathcal{M}}_c^j$.

\subsection{Minimal homoclinic orbits to Aubry Set}
To extend the concept of elementary weak KAM solution to universal covering space, let us reveal some properties of minimal homoclinic orbit to Aubry set.

Given a curve $\gamma$: $\mathbb{R}\to M$, we call $d\gamma=(\gamma,\dot\gamma)$ a homoclinic orbit to some Aubry set $\tilde{\mathcal{A}}$ if it does not stay in the Aubry set, but its $\omega$-limit set as well as the $\alpha$-limit set is contained in the Aubry set:
$$
\alpha(d\gamma)\subseteq\tilde{\mathcal{A}}\ \ \ \ \ \text{\rm and}\
\ \ \ \ \omega(d\gamma)\subseteq\tilde{\mathcal{A}}.
$$
Correspondingly, we call $\gamma$ homoclinic curve. The existence of homoclinic orbits to Aubry sets has been studied in a few papers, see \cite{Bo,Be1,Cui,Zhe,Zho2}.

The existence of homoclinic orbits is closely related to the issue whether the $\check{\rm C}$ech homology group $H_1(M,\mathcal{A},\mathbb{R})$ is non-trivial ($H_1(M\times\mathbb{T},\mathcal{A},\mathbb{R})$ for time-periodically dependent Lagrangian). It is defined as the inverse limit $\lim_{\mathcal{A}\subset U}H_1(M,U,\mathbb{R})$, where $U$ is an open neighborhood of $\mathcal{A}$. There exists a small open neighborhood $U_0$ of $\mathcal {A}$ such that ${\rm rank}H_1(M,U,\mathbb{R})={\rm rank}H_1(M,\mathcal{A},\mathbb{R})$ provided $U\subseteq U_0$.

Let $\bar M$ be a covering of $M$ such that $\pi_1(\bar M)=\text{\rm ker}(\mathscr{H}: \pi_1(M)\to H_1(M,\mathbb{R}))$ where $\mathscr{H}$ denotes the Hurewicz homomorphism. The group of Deck transformation of this covering space is
$$
H=im(\mathscr{H}:\pi_1(M)\to H_1(M,\mathbb{R})).
$$
Let $U$ be an open neighborhood of $\mathcal{A}$ such that ${\rm rank}H_1(M,U,\mathbb{Z})={\rm rank} H_1(M,\mathcal{A},\mathbb{R})$. Let $K=i_*H_1(U,\mathbb{Z}) \subset H$ and $G=H/K$, then $G$ is a free Abel group. To each orbit $(\gamma,\dot\gamma)$: $\mathbb{R}\to M$ homoclinic to $\tilde{\mathcal{A}}$, an element $[\gamma]\in G$ is associated.

If the group $G$ is non-trivial, there is a flat $\mathbb{F}$ of the $\alpha$-function containing the cohomology class. A set $\mathbb{F}\subset H^1(M,\mathbb{R})$ is called flat if the function $\alpha$ is affine when it is restricted on $\mathbb{F}$, not affine for any set properly contains $\mathbb{F}$. The dimension of this flat is not smaller than $r={\rm rank} H_1(M,\mathcal{A},\mathbb{Z})$ and the Aubry set is the same for all classes in the interior of $\mathbb{F}$ (see \cite{Ms}).

In this paper, we are interested in so-called {\it minimal} homoclinic orbits. Let $\check{M}$ be a covering manifold of $M$ such that $\pi_1(\check{M})=\pi_1(U)$. A curve $\gamma$: $\mathbb{R}\to M$ is called $\check{M}$ semi-static if the lift of $\gamma$ to $\check{M}$, $\check{\gamma}$: $\mathbb{R}\to\check{M}$ is semi-static. A homoclinic orbit $d\gamma$ is called {\it minimal} if the lift $\check{\gamma}$: $\mathbb{R} \to\check{M}$ is semi-static.

\begin{theo}\label{homothm1}
If there is only one Aubry class and $\text{\rm rank}\,H_1(M,\mathcal{A},\mathbb{R})=r>0$, then there are at least $r+1$ minimal homoclinic orbits. If $\mathcal{M}(c)\supsetneq \mathcal{M}(c')$ for $c\in\partial\mathbb{F}$ and $c'\in int\mathbb{F}$, then there are infinitely many $c$-minimal homoclinic orbits.
\end{theo}

The existence of at least $r+1$ homoclinic orbits is proved in \cite{Be1}, they are actually minimal. The infinity of minimal homoclinic orbits are proved in \cite{Zhe,Zho2}. Let us briefly describe how to find these $r+1$ minimal homoclinic orbits. Given a point $x\in\mathcal{A}$ and $g\in G$, we denote by $\xi_k$: $[-k,k]\to M$ the minimizer of
$$
h_g^k(x)=\inf_{\stackrel {\xi_k(-k)=\xi_k(k)=x} {\scriptscriptstyle [\xi_k]=g}}\int_{-k}^kL(\xi_k(s),\dot\xi_k(s))ds+2k\alpha
$$
Obviously, the set $\{\|\dot\xi_k(t)\|:\, t\in [-k,k]\}$ is uniformly bounded for $k\in\mathbb{Z}_+$. Because of positive definiteness of $L$,  the set $\{\|\ddot{\xi}_k(t)\|:\, t\in [-k,k]\}$ is also uniformly bounded for each $k$. Let
$$
h_g^{\infty}(x)=\liminf_{k\to\infty}h_g^k(x),
$$
there exists a subsequence of $k_j$ such that $h_g^{k_j}(x)\to h_g^{\infty}(x)$. The quantity $h_g^{\infty}$ keeps constant on each Aubry class. By diagonal extraction argument we can find a subsequence of $\xi_{k_j}$ which $C^1$-uniformly converges, on each compact interval, to a $C^1$-curve $\gamma$: $\mathbb{R}\to M$. In this sense, $\gamma$: $\mathbb{R}\to M$ is called an accumulation point of $\{\xi_{k_j}\}$. Each accumulation point is $\check M$ semi-static and there is at least one accumulation point $\gamma_1$ with non-zero homology $[\gamma_1]\neq 0$.

As the relative homology of the Aubry set is non-trivial, some $a>0$ exists such that $h_g^{\infty}\ge a$ holds for each class $g\in G$ and $h_g\to\infty$ as $|g|\to\infty$. Therefore, for each $g\in G$, there are finitely many accumulation points of $\{\xi_{k_j}\}$ with non-zero homology, denoted by $\gamma_1,\cdots,\gamma_i$. Clearly $\sum_{j=1}^i[\gamma_j]=g$. As $G$ is $r$-dimensional, at least $r+1$ geometrically different minimal homoclinic orbits exist.

Let us look at these homoclinic orbits from another point of view. For certain finite covering manifold, the lift of the Aubry set has several connected components (several Aubry classes). These Aubry classes are connected by semi-static orbits \cite{CP}. The projection of these semi-static orbits are nothing else but minimal homoclinic orbits. For a finite covering manifold $\tilde\pi$: $\tilde M\to M$, the fiber $\tilde\pi^{-1}x$ contains finitely many points. For a closed curve $\phi$: $[0,1]\to M$ such that $\phi(0)=\phi(1)=x$, there is a lift of $\bar\phi$ such that $\bar\phi(0)=\bar x_0\in \tilde\pi^{-1}x$. By the monodromy theorem, $\bar\phi(1)\in\tilde\pi^{-1}x$ is uniquely determined by its class $[\phi]\in\pi_1(M)$. Let $g_1,g_2,\cdots,g_{r}$ be the generators of $G$, $\phi_1,\phi_2,\cdots,\phi_r$ be closed path so that $[\phi_i]=g_i$, $\phi_i(0)=x$ for $i=1,2,\cdots,r$. If $\tilde M$ is chosen so that $\bar\phi_i(0)=\bar x_0$ and $\bar\phi_i(1)\neq \bar\phi_j(1)$, there will be at least $2r$ Aubry classes for this covering manifold. Among the semi-static orbits connecting different Aubry classes for the covering manifold, there are at least $r+1$ orbits whose projection is different from each other.

Let $G_{m}\subset G$ be defined such that $g\in G_m$ if and only if some minimal homoclinic orbit $d\gamma$ exists such that $[\gamma]=g$. We say that there are $k$-types of minimal homoclinic orbits if $G_m$ contains exactly $k$ elements.

\begin{theo}\label{homothm2}
If $M=\mathbb{T}^n$, $H_1(M,\mathcal{A},\mathbb{Z})\neq 0$ and $\mathcal{A}$ contains a set homeomorphic to $\mathbb{T}^{n-1}$, there exist exactly two types of minimal homolcinic orbits.
\end{theo}
\begin{proof}
In this case $G=\mathbb{Z}$. By the condition, we can assume that each standard generator $e_i\in H_1(\mathbb{T}^n,\mathbb{Z})$ with $i>1$ can be represented by a closed curve in $\mathcal{A}$. Let $g=ke_1$ with $k>1$. If there is a minimal homoclinic orbits $(\gamma,\dot\gamma)$ such that $[\gamma]=g$, there must be some points $x=\gamma(t_0)\in\mathcal{A}$ but $(\gamma(t_0),\dot\gamma(t_0))\notin\tilde{\mathcal{A}}$.

As $x\in\mathcal{A}$, there is a unique vector $v$ such that $(x,v)\in\tilde{\mathcal{A}}$. Given any $\epsilon>0$, there is static curve $\xi$: $\mathbb{R}\to M$ and $s_0<s_1$ such that $\xi(s_0),\xi(s_1)$ are in $\epsilon$-neighborhood of $x$, $\|\dot\xi(s_0)-v\|<\epsilon$, $\|\dot\xi(s_1)-v\|<\epsilon$ and
$[A(\xi)|_{[s_0,s_1]}]<\epsilon$.

Let $\tau_1^-=t_0-t^->0$, $\tau_2^+=t^+-t_0>0$, $\tau_1^+=s_0^+-s_0>0$ and $\tau_2^-=s_1-s_1^->0$ be suitably small numbers. We join $\gamma(t^-)$ to $\xi(s_0^+)$ by the curve $\zeta_1$: $[-\tau_1^-,\tau_1^+]\to M$ which minimizes the action
$$
[A(\zeta_1)|_{[-\tau_1^-,\tau_1^+]}]=\inf_{\stackrel{\zeta(-\tau_1^-)=\gamma(t^-)} {\scriptscriptstyle \zeta(\tau_1^+)= \xi(s_0^+)}}\int_{-\tau_1^-}^{\tau_1^+}L(\zeta(s),\dot\zeta(s))ds +(\tau_1^++\tau_1^-)\alpha,
$$
and join $\xi(s_1^-)$ to $\gamma(t^+)$ by the curve $\zeta_2$: $[-\tau_2^-,\tau_2^+]\to M$ which minimizes the action
$$
[A(\zeta_2)|_{[-\tau_2^-,\tau_2^+]}]=\inf_{\stackrel{\zeta(-\tau_2^-)=\xi(s_1^-)} {\scriptscriptstyle \zeta(\tau_2^+)= \gamma(t^+)}}\int_{-\tau_2^-}^{\tau_2^+}L(\zeta(s),\dot\zeta(s))ds +(\tau_2^++\tau_2^-)\alpha.
$$
We define a continuous curve $\gamma'$: $\mathbb{R}\to M$ by
$$
\gamma'(t)=\begin{cases}
\gamma(t), &t\in(-\infty, t^-],\\
\zeta_1(t-\Delta_1),&t-\Delta_1\in [-\tau_1^-,\tau_1^+],\\
\xi(t-\Delta_2),&t-\Delta_2\in [s_0^+,s_1^-],\\
\zeta_2(t-\Delta_3),&t-\Delta_3\in [-\tau_2^-,\tau_2^+],\\
\gamma(t-\Delta_4),& t-\Delta_4\in [t^+,\infty),
\end{cases}
$$
where $\Delta_1=t^-+\tau_1^-$, $\Delta_2=t^-+\tau_1^-+\tau_1^+-s_0^+$, $\Delta_3=t^-+\tau_1^-+\tau_1^+ -s_0^++s_1^-+\tau_2^-$ and $\Delta_4=t^-+\tau_1^-+\tau_1^+ -s_0^++s_1^-+\tau_2^-+\tau_2^+-t^+$. By exploiting the {\it curve shorten} lemma in Riemannian geometry as did in \cite{Ma2} we find that
$$
[A(\gamma)|_{[t^-,t^+]}]+[A(\xi)|_{[s_0,s_0^+]\cup[s_1^-,s_1]}] >
[A(\zeta_1)|_{[-\tau_1^-,\tau_1^+]}]+[A(\zeta_2)|_{[-\tau_2^-,\tau_2^+]}]
$$
if $\xi(s_0)=\xi(s_1)=x$ and $\dot\xi(s_0)=\dot\xi(s_1)=v\neq\dot\gamma(t_0)$. As $x\in\mathcal{A}$, $(\xi(s_0),\dot\xi(s_0))$, $(\xi(s_1),\dot\xi(s_1))$ can be arbitrarily close to $(x,v)$ by choosing suitable $s_0$ and $s_1$, this inequality still hold in our case. Note that the quantity $[A(\xi)|_{[s_0,s_1]}]$ can be arbitrarily close to zero, we see that
$$
[A(\gamma)|_{[t_{-1},t_1]}]>[A(\gamma')|_{[t_{-1},t_1+\Delta_4]}]
$$
if $t_{-1}<t^-$ and $t_1>t^+$. As $[\gamma']=[\gamma]$, this property contradicts the fact that $\gamma$ is minimal. On the other hand, from Theorem \ref{homothm1}, we obtain the existence of 2 minimal homoclinic orbits. This completes the proof.
\end{proof}

For each class $g\in G$, we define
$$
h_g^k(x,x)=\inf_{\stackrel {\xi(0)=\xi(k)=x}{\scriptscriptstyle \xi\in C^1,[\xi]=g}}\int_{0}^{k} L(\xi(s),\dot\xi(s))ds +k\alpha,
$$
$$
h_g^{\infty}(x,x)=\liminf_{k\to\infty}h_g^k(x,x).
$$
It is easy to see that $h_g^{\infty}(x,x)\to\infty$ as $\|g\|\to\infty$. Indeed, it follows from the fact $H_1(\mathbb{T}^n,\mathcal{A},\mathbb{Z}) \neq 0$ that  $h^{\infty}_g(x,x)>0$ for any $g\ne 0$. If $h_g^{\infty}(x,x)$ remains bounded as $\|g\|\to\infty$, there would be a minimal measure whose support is obviously not contained in $\tilde{\mathcal{A}}$, but it is absurd.

If the Aubry set contains only one class, as a function of $x$, $h_g^{\infty}(x,x)$ keeps constant on the Aubry set. So it makes sense let $h_g^{\infty}=h_g^{\infty}(x,x)$ for $x\in\mathcal{A}$. Obviously, one has
$$
h_{g_1+g_2}^{\infty}\le h_{g_1}^{\infty}+h_{g_2}^{\infty}.
$$
and
\begin{pro}\label{homopro1}
If there is an infinite sequence $\{g_i\}\subset G$ such that
$$
h_{g_i+g_{i'}}^{\infty}< h_{g_i}^{\infty}+h_{g_{i'}}^{\infty},
$$
then $G_m$ contains infinitely many elements.
\end{pro}

The definition of $h_g^k(x,x)$ can be extended $h^k_g(x,x')$ for $x\ne x'$. Let us recall that the covering space  $\check{\pi}$: $\check{M}\to M=\mathbb{T}^n$ is defined such that $\pi_1(\check{M})=\pi_1(U)$, where $U$ is a open neighborhood of $\mathcal {A}\subset\mathbb{T}^n$ so that $H_1(M,U,\mathbb{R})=H_1(M,\mathcal{A}, \mathbb{R})$. Let $D=\{\bar x: \bar x_i\in [0,1)\}\subset\mathbb{R}^n$ be the fundamental domain for $\mathbb{T}^n$ and use the same symbol to denote its projection to $\check M$ as well. For each closed path $\phi$: $[0,1]\to M$, there is a unique curve $\check{\phi}$ in the lift of $\phi$ such that  $\check{\phi}(0)\in D$. Because of the monodromy theorem, $\check{\phi}(1)\in\check M$ is uniquely determined by the homological type $[\phi]\in H_1(\mathbb{T}^n,U,\mathbb{Z})$.

For a curve $\xi$: $[0,k]\to M$ with $\xi(0)=x$, $\xi(k)=x'$, we denote by $\check{\xi}$ the curve in the lift of $\xi$ such that $\check{\xi}(0)\in D$. We say $[\xi]=g$ if $[\phi]=g$ holds for any closed curve $\phi$ such that $\check\phi(0)\in D$ and $\check\phi(1)=\check\xi(k)$.  Therefore, the following is well-defined:
$$
h_g^{k}(x,x')=\inf_{\stackrel{\stackrel{\stackrel {\xi(0)=x}{\scriptscriptstyle \xi(k)=x'}}{\scriptscriptstyle
[\xi]=g}}{\scriptscriptstyle \xi\in C^1}}\int_{0}^{k}L(\xi(s),\dot\xi(s))ds +k\alpha,
$$
$$
h_g^{\infty}(x,x')=\liminf_{k\to\infty}h_g^k(x,x').
$$
Clearly, $h_g^{\infty}(x,x')\to\infty$ as $\|g\|\to\infty$.  Indeed, let $\check x,\check x'\in D$ such that $\check{\pi}\check x=x, \check{\pi}\check x'=x'$, let $\check{\zeta}$: $[0,1]\to\bar M$ be a straight line such that $\check{\zeta}(0)=x'$, $\check{\zeta}(1)=x$ and denoted by $\zeta$ the projection of
$\check{\zeta}$ down to $M$, we obviously have that
$$
h_g^{k+1}(x,x)\le h_g^k(x,x')+[A(\zeta)]
$$
holds for each class $g$. As $[A(\zeta)]$ is a finite number, we verify the claim.

\begin{pro}\label{homopro2}
There exists positive number $a>0$ such that $h^{\infty}_g (x,x')\ge \|g\|a$ holds for each
$(x,x')\in\mathbb{T}^n\times\mathbb{T}^n$ and for large $\|g\|$.
\end{pro}
\begin{proof}
Obviously, there exists a positive number $a'>0$ such that $h^{T}_g (x,x)\ge a'$ holds for each $g\ne 0$ and each $T>0$. If the proposition does not hold, for any small $\epsilon_i>0$ there would exists $g_i$ and $T_i$ such that
$$
h^T_{g_i}(x,x)\le\epsilon_i\|g_i\|, \qquad \forall\, T\ge T_i.
$$
Let $\gamma_i$: $[0,T_i]\to M$ be the minimizer of $h^{T_i}_{g_i} (x,x)$. Let $0=t_{i,0}<t_{i,1}<\cdots<t_{i,m_i}=T_i$ be a sequence so that $\gamma_i(t_{i,j})\in U$ and $H_1(\mathbb{T}^n,U,\mathbb{Z})\ni[\gamma_i|_{[t_{i,j},t_{i,j+1}]}]\ne 0$ and there does not exist $t'\in (t_{i,j},t_{i,j+1})$ such that $\gamma_i(t')\in U$, both $[\gamma_i|_{[t_{i,j},t']}]\ne 0$ and $[\gamma_i|_{[t',t_{i,j+1}]}]\ne 0$.

There are two possibilities for this sequence. Either some $t_{i,j}<t_{i,j+1}$ exists such that $t_{i,j+1}-t_{i,j}\to\infty$ as $i\to\infty$ or $t_{i,j+1}-t_{i,j}$ remains bounded for all $i,j$.

In the first case, let $\mu_i=d\gamma_i|_{[t_{i,j},t_{i,j+1}]}^*\nu_i$ where $\nu_i$ is a probability measure evenly distributed on the interval $[t_{i,j},t_{i,j+1}]$. By weak$^*$-compactness a probability measure $\mu$ exists such that $\mu_i\rightharpoonup\mu$. Clearly, $\mu$ is invariant for the Lagrange flow, $\int Ld\mu=0$ and the support of $\mu$ is not contained in the Aubry set. But it is absurd.

In the second case, one has $m_i\ge C\|g\|$. By choosing sufficiently small neighborhood $U$ of $\mathcal{A}$, the distance $d(\gamma_i(t_{i,j}),\mathcal{A})<\epsilon$ can be sufficiently small. As there is only one Aubry class, there is a closed curve $\zeta$ and a sequence of time $t'_j$ ($j=0,1,\cdots,m_i)$ such that $[A(\zeta)]<\epsilon$ and $d(\zeta(t'_j),\gamma_i(t_{i,m-j}))<\epsilon$. Let $\zeta'_j$ be the minimizer connecting $\gamma_i(t_{i,j})$ to $\gamma_j(t_{i,j-1})$, we have $|[A(\zeta'_j)]-[A(\zeta|_{[t'_{i,j-1}-t'_{i,j}]})]|\le C'\epsilon$, where the constant depends only on the Lagrangian. By construction, the curve $\gamma_i|_{[t_{i,j}-t_{i,j-1}]}\ast\zeta'_j$ is a closed curve with $[\gamma_i|_{[t_{i,j}-t_{i,j-1}]}\ast\zeta'_j]\ne 0$ and $[A(\gamma_i|_{[t_{i,j}-t_{i,j-1}]}\ast\zeta'_j)]>a'$. Therefore,
\begin{align*}
h^{T_i}_g(x,x)&\ge h^{T_i}_g(x,x)+[A(\zeta)]-\epsilon\\
&\ge\sum_{j} [A(\gamma_i|_{[t_{i,j}-t_{i,j-1}]}\ast\zeta'_j)]-(CC'\|g\|+1)\epsilon\\
&\ge C\|g\|a'-(CC'\|g\|+1)\epsilon.
\end{align*}
It contradicts the assumption. This proves the proposition in the case that $x=x'$.

For $x\ne x'$, we use a straight line connecting $x'$ to $x$. The action along this line is bounded. Therefore the proposition is also true for $x\ne x'$.
\end{proof}

\subsection{Globally elementary weak KAM solutions}
For the configuration space $\mathbb{T}^n$, each weak KAM solution is 1-periodic in $x_i$ for  $i=1,2,\cdots n$, where $(x_1,x_2,\cdots,x_n)=x$ denotes the configuration coordinate. If a finite covering of $\mathbb{T}^n$ is considered to be configuration space, weak KAM solution may not be 1-periodic for each coordinate.

We assume that the minimal measure contains finitely many ergodic components  $\mu_c^1$, $\mu_c^2,\cdots, \mu_c^m$ for the cohomology class $c$. In this case, the elementary weak KAM solution for each $\mu_c^i$ is well-defined. The lift of $\mu_c^i$ to a finite covering $k\mathbb{T}^n$ may contain several ergodic components. For instance, if $\mathcal {M}\subset\{|x_1|\le\delta\}\times \mathbb{T}^{n-1}$, then there are two ergodic components in the lift of $\mathcal {M}$ for $2\mathbb{T}\times \mathbb{T}^{n-1}$. However, there are cases that the minimal measure is always uniquely ergodic for any finite covering manifold, for instance, if the measure is supported on a KAM torus.

Given $k=(k_1,k_2,\cdots,k_n)\in\mathbb{Z}^n$ with $k_i\ge 1$ for each $i=1,2,\cdots,n$, we define an equivalence relation $\sim_k$ in $\mathbb{R}^n$: we say $x\sim_k x'$ if $x_i-x'_i=2jk_i$ for some $j\in\mathbb{Z}$ ($i=1,2,\cdots n$). Clearly, $\pi_k$: $M_k =\mathbb{R}^n/\sim_k\to\mathbb{T}^n$ is a finite covering of $\mathbb{T}^n$. In the following, we shall also use the symbols: $\pi_{\infty,k}$: $\mathbb{R}^n\to M_k$ and $\pi_{\infty}$: $\mathbb{R}^n\to\mathbb{T}^n$ to denote the projection. For a bounded domain $\Omega\subset\mathbb{R}^n$, if the topology of $\pi_{\infty,k}\Omega\subset M_k$ is trivial, we use the same symbol to denote its projection $\Omega:=\pi_{\infty,k}\Omega$.

Let $\mathcal{M}_{\infty}$ and $\mathcal{M}_{k}$ be the lift of Mather set $\mathcal{M}$ to the universal covering space as well as to $M_k$ respectively. The connected components are denoted by $\mathcal{M}_{\infty}^i$ and $\mathcal{M}_{k}^i$ correspondingly. Obviously, the unit cube $D=[0,1)^n$ intersects finitely many connected components of $\mathcal{M}_{\infty}$, denoted by $\mathcal{M}_{\infty}^i$ with $i=0,1,\cdots,i_m$ ($i_m\ge m$).

Let $d_k=\min\{k_1,k_2,\cdots k_n\}$. Some $R_D>0$ exists such that for any $k\in\mathbb{Z}^n$ with $d_k\ge R_D$, $\pi_{\infty,k}\mathcal{M}_{\infty}^i\neq \pi_{\infty,k}\mathcal{M}_{\infty}^j$ holds for $0\le i,j\le i_m$ and $i\neq j$. In this case, we use the notation $\mathcal{M}_{k}^j=\pi_{\infty,k} \mathcal{M}_{\infty}^j$ for $0\le j\le i_m$. Let $u_{k,j}^{\pm}$ denote the elementary weak KAM for $\mathcal{M}_{k}^j$ with respect to the configuration manifold $M_k$.

\begin{lem}\label{globallem1}
For each bounded region $\Omega\subset\mathbb{R}^n$, there exists $R_{\Omega}>0$ such that for any $k, k'\in\mathbb{Z}^n$ with $d_k,d_{k'}\ge\max\{R_{\Omega},R_D\}$,
$$
u^{\pm}_{k,j}|_{\Omega}=u_{k',j}^{\pm}|_{\Omega}+\text{\rm constant}
$$
holds for each $j=0,1,\cdots,i_m$.
\end{lem}
\begin{proof}
We only need to study the case that the minimal measure is uniquely ergodic. If there are finitely many ergodic components, we obtain this result by perturbing the Lagrangian so that it is uniquely ergodic and applying Theorem \ref{weakthm1}.

Each weak KAM solution for $\mathbb{T}^n$ is a weak KAM solution for any $M_k$. If the lift of the minimal measure to any finite covering space is still uniquely ergodic, the elementary weak KAM solution remains the same.

Let us consider the case that there are more than one connected component in $\mathcal{M}_k$ with $d_k\ge R_D$. Remember $\mathcal {M}_{k}^j=\pi_{\infty,k}\mathcal{M}_{\infty}^j$ for $0\le j\le m$ where $\mathcal{M}_{\infty}^j$ intersects the fundamental domain $[0,1)^n$. Considered as a function defined in $\mathbb{R}^n$, the elementary weak KAM solution $u^-_{k,j}$ determined by $\mathcal {M}_{k}^j$ is $k_i$-periodic in the $i$-th coordinate. By the definition of elementary weak KAM solution, a sequence of functions $u^-_{k,j,\epsilon}$ exists such that $u^-_{k,j,\epsilon}\to u^-_{k,j}$ as $\epsilon\to 0$, where $u^-_{k,j,\epsilon}$ is the weak KAM solution for the Lagrangian $L_{k,\epsilon}: TM_k\to\mathbb{R}$. This Lagrangian satisfies the following conditions:

1, it is the same as $L$ when it is restricted on the tangent bundle of a neighborhood $U$ of $\mathcal{M}_{k}^j$, i.e. $L_{k,\epsilon}|_{TU}=L|_{TU}$;

2, the minimal measure is uniquely ergodic whenever $\epsilon\neq 0$, supported on $\mathcal{M}_{k}^j$;

3, $L_{k,\epsilon}\to L$ as $\epsilon\to 0$.

Starting from each $x\in M_k$, there exists at least one backward semi-static curve for $L_{k,\epsilon}$, $\gamma_{k,x, \epsilon}$: $(-\infty,0]$ with $\gamma_{k,x,\epsilon}(0)=x$. Clearly, $\pi\alpha(d\gamma_{k,x,\epsilon})\cap \mathcal{M}_{k}^j\ne\varnothing$. Let $t_i\to\infty$ be the sequence so that $\gamma_{k,x,\epsilon}(-t_i)\to x_0\in\mathcal{M}_{k}^j$, let $\alpha$ stand for the average action, then we have
\begin{equation}\label{globaleq1}
u^-_{k,\epsilon}(x)-u^-_{k,\epsilon}(x_0)=\lim_{t_i\to\infty}\int_{-t_i}^0L_{k,\epsilon} (d\gamma_{k,x,\epsilon}(s))ds+t_i\alpha.
\end{equation}

Again, the lift of $\mathcal{M}_k^j$ to the universal covering space may contain many connected components, denoted by $\mathcal{M}_{\infty}^{j,\ell}$, among which only $\mathcal{M}_{\infty}^{j,0}$ intersects the fundamental domain $D$.

Let $D_k=\{\bar x: \bar x_i\in [-k_i,k_i)\}\subset\mathbb{R}^n$ so that $\pi_{\infty,k}$: $D_k\to M_k$ is an injection and $\pi_{\infty,k}D_k=M_k$. Let $\bar x\in D_k$ be the points such that $\pi_{\infty,k}\bar x=x$. Let $\bar\gamma_{k,\bar x,\epsilon}$ be the lift of $\gamma_{k,x,\epsilon}$ to the universal covering space so that $\bar\gamma_{k,\bar x,\epsilon}(0)=\bar x$. It is possible that $\pi\alpha(d\bar\gamma_{k,\bar x,\epsilon}) \cap\mathcal{M}_{\infty}^{j,0}=\varnothing$. The curve may approach to another connected component of $\mathcal{M}_{\infty}^{j,\ell}$. Let $\Omega_d=\{x:\max_i|x_i|\le d\}\subset\mathbb{R}^n$. Note that $L_{k,\epsilon}$ is a small perturbation of $L$. In virtue of Proposition \ref{homopro2} we claim that $\bar\gamma_{k,\bar x,\epsilon}$ approaches to $\mathcal{M}_{\infty}^{j,0}$ provided $\bar x\in\Omega_d$, $\epsilon$ is suitably small and $d_k$ is sufficiently large. Let us assume the contrary, i.e. $\gamma_{k,\bar x,\epsilon}$ approaches to another connected component of $\mathcal{M}_{\infty}$. In this case, $\|[\pi_k\gamma_{k,\bar x,\epsilon}]\|$ would be sufficiently large provided $d_k$ is sufficiently large. By Proposition \ref{homopro2} the action of $L$ along  $\pi_k\gamma_{k,\bar x,\epsilon}$
$$
\int L(d\pi_k\gamma_{k,\bar x,\epsilon}(t),t)dt\ge\|[\pi_k\gamma_{k,\bar x,\epsilon}]\|a
$$
with certain $a>0$. As $L_{k,\epsilon}$ is a small perturbation of $L$, the action of $L_{k,\epsilon}$ along $\pi_k\gamma_{k,\bar x,\epsilon}$ would approach infinity as $d_k\to\infty$. The absurdity verifies the claim.

The set $\{\dot{\gamma}_{k,x,\epsilon}(0)\}$ is compact as $\epsilon\to 0$. For each accumulation point $v$, there is a subsequence of $\epsilon\to 0$ such that $\dot\gamma_{k,x,\epsilon}(0) \to v$. The initial value $(x,v)$ uniquely determines an orbit $(\gamma_{k,x},\dot\gamma_{k,x})$ of $L$. The curve $\gamma_{k,x}$: $(-\infty,0]\to M_k$ is a backward semi-static curve for $L$ which may not approach to $ \mathcal{M}_{k,0}$. When $\epsilon\to 0$, $\gamma_{k,x,\epsilon}$ may approach not only one but a family of semi-static curves for $L$ including the curves connecting different connected components of $\mathcal{M}_k$. More precisely, there might be several connected components
$\mathcal{M}_{k}^{i_0}=\mathcal{M}_{k}^{j},\mathcal{M}_{k}^{i_1},\cdots,\mathcal{M}_{k}^{i_{\imath}}$ and semi-static curves $\gamma_{\ell,\ell+1}$ of $L$ for $M_k$ ($\ell=0,1,\cdots \imath-1$) such that $\pi\alpha(d\gamma_{\ell,\ell+1})\cap\mathcal{M}_{k}^{i_{\ell}}\ne\varnothing$,
$\pi\omega(d\gamma_{\ell,\ell+1})\cap\mathcal{M}_{k}^{i_{\ell+1}}\ne\varnothing$, $\pi\alpha(d\gamma_{k,x})\cap\mathcal{M}_{k}^{i_{\imath}}\ne\varnothing$ and each $\gamma_{\ell,\ell+1}$ is approached by $\gamma_{k,x,\epsilon}$ as $\epsilon\to 0$. These curves have their natural projection down to
$\mathbb{T}^n$, denoted by the same symbol.

We define the quantity $A_{i,j}$: $\mathcal{M}^i\times\mathcal{M}^j \to\mathbb{R}$
\begin{equation}\label{globaleq2}
[A_{i,j}(x_i,x_j)]=\inf_{\stackrel{\stackrel{\gamma(0)=x_i}{\scriptscriptstyle \gamma(k)=x_j}}{\scriptscriptstyle k\in\mathbb{Z}_+}}\int^{k}_{0} L(d\gamma(s))ds+k\alpha.
\end{equation}
By definition of weak KAM, for almost every point $x$, $(x,\partial _xu_{k,\epsilon}^-(x))$ uniquely determines a backward semi-static curve $\gamma_{k,x,\epsilon}$. Since this semi-static curve approaches to several curves: $\gamma_{k,x,\epsilon}\to\gamma_{1,2}\ast\cdots\ast \gamma_{\imath-1,\imath}\ast\gamma_{k,x}$ we obtain that
\begin{align}\label{globaleq3}
u^-_{k,0}(x)-u^-_{k,0}(x_0)=&\lim_{t_i\to\infty}\int_{-t_i}^0L(d\gamma_{x}(s))ds+t_i\alpha\\
&+\sum_{i=0}^{\imath-1}[A_{j_i,j_{i+1}}(x_i,x_{i+1})]\notag
\end{align}
where $t_i\to\infty$ is a sequence such that $\gamma_{k,x}(-t_i)\to x_{\imath}\in \mathcal{M}_{k,i_{\imath}}$, $x_j\in\mathcal{M}_{k,i_j}$.

For each $x\in\Omega_d$, the backward semi-static curve $\gamma_{k,x,\epsilon}$ approaches to $\mathcal{M}_{k}^j$ provided $d_k$ is sufficiently large. For different $k,k'$ satisfying this condition, $\gamma_{k,x,\epsilon}$ and $\gamma_{k',x,\epsilon}$ may converge to different curves, $\gamma_{k,x,\epsilon}\to\gamma_{0,1}\ast\cdots\ast\gamma_{\imath-1,\imath}\ast\gamma_{k,x}$ and
$\gamma_{k',x,\epsilon}\to\gamma'_{0,1}\ast\cdots\ast\gamma'_{\imath'-1,\imath'}\ast\gamma'_{k,x}$ as $\epsilon\to 0$. But the action of $L$ along $\gamma_{0,1}\ast\cdots\ast\gamma_{\imath-1,\imath}\ast\gamma_{k,x}$ is the same as along $\gamma'_{0,1}\ast\cdots\ast\gamma'_{\imath'-1,\imath'}\ast\gamma'_{k,x}$. Indeed, the action of $L_{k,\epsilon}$ along $\tilde\gamma_{k,x,\epsilon}$ is almost the same as the action of $L_{k',\epsilon}$ along $\tilde\gamma_{k',x,\epsilon}$ provided the perturbation is sufficiently small. Therefore, we obtain from the formula \ref{globaleq3} that
$$
\bar u_{k,0}|_{\Omega}=\bar u_{k',0}|_{\Omega}+\text{\rm constant}
$$
if both $d_k$ and $d_{k'}$ are sufficiently large.
\end{proof}

\begin{defi}
The function $\bar u_i:\mathbb{R}^n\to\mathbb{R}$ is called globally elementary weak KAM solution for $\mathcal{M}_{\infty}^{j}$ if for each bounded domain $\Omega\subset\mathbb{R}^n$, there exists $R_{\Omega}>0$ such that for any $k\in\mathbb{Z}^n$ with $d_k\ge R_{\Omega}$,
$$
\bar u_{k,j}|_{\Omega}=\bar u_j|_{\Omega}+\text{\rm constant}
$$
holds for each elementary weak KAM solution $\bar u_{k,j}$: $M_k\to\mathbb{R}$ for $\mathcal{M}_{k}^j$.
\end{defi}
From Lemma \ref{globallem1}, we obtain the existence of a globally elementary weak KAM solution for each $\mathcal{M}_{\infty}^j$.

To investigate the properties of globally elementary weak KAM solution, let us  consider a special case first, namely, the Mather set contains a connected component homeomorphic to $\mathbb{T}^{n-1}$. In this case, each $\mathcal{M}^i$ divided $\mathbb{R}^n$ into two parts, denoted by $R^-$ and $R^+$.

\begin{theo}\label{globalthm1}
If the Mather set contains a connected component contains a $\mathcal{M}_{\infty}^i$ homeomorphic to $\mathbb{T}^{n-1}$, then the globally elementary weak KAM solution $\bar u^{\pm}_i$: $\mathbb{R}^n\to\mathbb{R}$ has a decomposition
$$
\bar u^{\pm}_i=v^{\pm}_i+w^{\pm}_i,
$$
where $v^{\pm}_i$ is periodic and $w^{\pm}_i$ is affine when they are restricted in the half space $R^+$ as well as in another half space $R^-$.
\end{theo}
\begin{proof}
We only need to consider the case that the minimal measure is uniquely ergodic, as we did in the proof of Lemma \ref{globallem1}. According to Theorem \ref{homothm2}, there are exactly two types of minimal homoclnic orbits to the Aubry set, we pick up two representative elements $\gamma_-$, $\gamma_+$: $\mathbb{R}\to\mathbb{T}^n$. Let
$$
h_{\pm}=\liminf_{t_i^{\pm}\to\infty}\int_{-t_i^-}^{t_i^+}L(d\gamma_{\pm}(t),t)dt+ (t_i^-+t_i^+)\alpha,
$$
where $t_i^{\pm}$ is chosen such that $\gamma_{\pm}(t_i^+)\to 0$ and $\gamma_{\pm}(-t_i^-)\to 0$.

As the set $\mathcal{M}^i$ is co-dimension one, we are able to number all connected components by $\mathcal{M}^i_{\infty}$ ($i=\cdots -1,0,1,2,\cdots$) such that any path from $\mathcal{M}^{i-1}_{\infty}$ to $\mathcal{M}^{i+1}_{\infty}$ must pass through $\mathcal{M}^i_{\infty}$. Denote by $\Pi_i$ the strip bounded by $\mathcal{M}^i_{\infty}$ and $\mathcal{M}^{i+1}_{\infty}$. $\mathcal{M}^0_{\infty}$ separates $\mathbb{R}^n$ into two parts, denoted by $D^-$ and $D^+$ such that $\mathcal{M}_{\infty}^{-1}\subset D^-$ and $\mathcal{M}_{\infty}^1\subset D^+$.

Let $\bar\gamma_{\pm}$ denote a curve in the lift of $\gamma_{\pm}$ to $\mathbb{R}^n$ such that $\alpha(d\bar\gamma_{\pm})\subset\tilde{\mathcal {M}}_{\infty}^0$. Then, either $\omega(d\bar\gamma_{-})\subset\tilde{\mathcal{M}}_{\infty}^{-1}$, $\omega(d\bar\gamma_{+})\subset\tilde{\mathcal{M}}_{\infty}^1$, or $\omega(d\bar\gamma_{+})\subset\tilde{\mathcal{M}}_{\infty}^{-1}$,
$\omega(d\bar\gamma_{-})\subset\tilde{\mathcal{M}}_{\infty}^1$. We only need to study one case, let's say, the first case.

Given a bounded domain $\Omega\subset\mathbb{R}^n$. From the definition of globally elementary weak KAM solution, we see that
$$
\bar u^-_0|_{\Omega}=u^-_{k,0}|_{\Omega}
$$
whenever $d_k$ is suitably large. Clearly, the function $\bar u_0$ is periodic when it is restricted $\Omega\cap\Pi_i$, i.e. $u^-_{k,0}(x)=u^-_{k,0}(x')$ if $x'-x\in\mathbb{Z}^n$ and $x,x'\in\Omega\cap\Pi_i$.

For each $x\in\Omega\cap\Pi_i$ with $i>0$, there exists at least one point $x_0\in\Omega\cap\Pi_0$ such that $x-x_0\in\mathbb{Z}^n$. By definition, we find that $u^-_{k,0}(x)=u^-_{k,0}(x_0)+ih_+$. Obviously, $u^-_{k,0}(x)=u^-_{k,0}(x_0)+(1+i)h_-$ if $x\in\Omega\cap\Pi_i$ with $i<0$.

Pick up a point $x_0\in\mathcal{M}_{\infty}^0$. For each non-zero integer vector $k\in\mathbb{Z}^n$, the point $x=x_0+k$ stays in certain $\mathcal{M}_{\infty}^i$. Along the ray $x=x_0+tk$ with $t>0$, we define
$$
v^-_0(x)= \begin{cases} u^-_0(x)-u^-_0(x_0)-tih_+,\hskip 0.5 true cm \text{if} \ i>0;\\
u^-_0(x)-u^-_0(x_0)-tih_-, \hskip 0.5 true cm\text{if} \ i>0,
\end{cases}
$$
Clearly, $u^-_0-v^-_0$ is affine and $v^-_0$ is periodic when they are restricted in $D^-$ as well as in $D^+$.
\end{proof}

Given an ergodic component of a minimal measure with higher co-dimensions, it is unclear what condition guarantees the decomposition of the globally elementary weak KAM solutions. It appears closely related to the problem whether there are infinitely many types of minimal homoclinic orbits to the Aubry class.

\begin{pro}\label{globalpro1}
Let $u^{\pm}_i$ be the globally elementary weak KAM solution for $\mathcal{M}^i$. Then, $u^{\pm}_i$ remains bounded on the whole ray $\{x_0+tg:t\in\mathbb{R}_+\}$ for each $g\in H_1(\mathcal{A}^i, \mathbb{Z})$, where $\mathcal{A}^i\supset\mathcal{M}^i$ is an Abury class; for $g\in H_1(M,\mathcal{A},\mathbb{Z})/K$,  $u^{\pm}_i$ grows up linearly, or asymptotically linearly on the ray $\{x_0+tg:t\in\mathbb{R}_+\}$.
\end{pro}
\begin{proof}
For arbitrarily large $t$, there exists $x\in\mathcal{M}_{\infty}^i$ such that $\text{\rm dist}(x_0+tg,x)\le 2$ and $\mathcal{M}_{\infty}^i\cap D\neq\varnothing$ where $D$ is the unit cube containing the origin. Let $x^*=\pi_{\infty}x$, $\bar\xi$ be a curve connecting $x^*$ to $x$, $\xi=\pi_{\infty}\bar\xi$, then $[\xi]\in H_1(\mathcal{A},\mathbb{Z})$. By definition,
$$
\inf_{[\xi]\in H_1(\mathcal{A},\mathbb{Z})}\inf_{\stackrel{\xi(0)=\xi(k)}{\scriptscriptstyle k\in\mathbb{Z}_+}}\int_0^k L(d\xi(t),t)dt=0
$$
it proves the first conclusion.

For the second, one can see from Proposition \ref{homopro2} that it grows up at least linearly. Given $g\in H_1(M,\mathcal{A}^i,\mathbb{Z})/K$ finitely many elements $g_0,g_1,\cdots,g_r\in H_1(\mathbb{T}^n,\mathbb{Z})$
exists such that for each $g=\sum_{i+0}^rj_ig_i$ with $j_i\in\mathbb{Z}_+$. Thus,
$$
h^{\infty}_g(x,x)\le\sum_{i=0}^r j_ih^{\infty}_{g_i}(x,x).
$$
For each $\pi_{\infty}x\in\mathcal {M}^i$, $[x-x^*]=g$, we have
$$
u(x)-u(x^*)=h_g^{\infty}(\pi_{\infty}x,\pi_{\infty}x).
$$
This completes the proof.
\end{proof}

\section{\ui Dynamics around fixed point}
\setcounter{equation}{0}
Given a Tonelli Lagrangian $L$: $T\mathbb{T}^n\to\mathbb{R}$, let $c_0\in\arg\min\alpha$. Any minimal measure with zero-rotation vector must be $c_0$-minimal measure. In this section we study the dynamics around the Mather set for the class $c_0$. The motivation comes from following argument.

Let us consider the normal form of a nearly integrable Hamiltonian
$$
H(p,q)=h_0(p)+\epsilon P(p,q), \qquad (p,q)\in\mathbb{R}^{d}\times\mathbb{T}^{d}.
$$
around a complete resonant point. Let $\omega(y)=\nabla h_0(y)$ denote the frequency vector of the unperturbed system. A frequency $\omega$ is called complete resonant of (minimal) period $T$ if $T\omega\in\mathbb{Z}^d$ and $t\omega\notin\mathbb{Z}^d$ for each $t\in (0,T)$. By finitely many steps of KAM iteration and one step of linear coordinate transformation on torus, one obtains a normal form of nearly integrable Hamiltonian (see Appendix A)
$$
\tilde H(\tilde x,\tilde y)=\tilde h(\tilde y)+\epsilon \tilde Z(x,\tilde y)+\epsilon \tilde R(\tilde x,\tilde y)
$$
where $\tilde x=(x,x_d)$, $\tilde y=(y,y_d)$, $(x,y)\in\mathbb{T}^{d-1}\times\mathbb{R}^{d-1}$, $\tilde H$ is well-defined in $(\tilde x,\tilde y)\in\mathbb{T}^{d}\times B_d(\tilde y^*)$, $\partial\tilde h(\tilde y^*)=(0,\omega_{d})$ and $\epsilon\tilde R$ is a higher order term.

Since $\partial_{y_{d}}\tilde h(\tilde y^*)=\omega_{d}\ne 0$, there exists some function $Y(x,y,\tau)$ solving the equation $\tilde H(x,-\tau,y,Y(x,y,\tau))=E$ provided $E>\min\alpha_{\tilde H}$, which defines a time-periodic Hamiltonian system with $(d-1)$-degrees of freedom. Here $\tau=-x_{d}$ plays the role of time.  One can write
$$
Y(x,y)=h(y)+\epsilon Z(x,y)+\epsilon R(x,y,\tau)
$$
where $\epsilon R$ is a higher order term of $\epsilon$ and $\partial h(y^*)=0$, i.e. the complete resonance reduces to zero frequency. Omitting the higher order term, one obtains Hamiltonian with $d-1$ degrees of freedom
$$
\bar Y(x,y)=h(y)+\epsilon Z(x,y).
$$
It determines a Lagrangian we shall study in this section.

\subsection{Flat of the $\alpha$-function}
By definition, a subset is called a flat of certain $\alpha$-function if, restricted on this set, the $\alpha$-function is affine, and no longer affine on any set properly containing the flat. As $\alpha$-function is convex with super-linear growth, each flat is a convex and bounded set. Given an $n$-dimensional flat $\mathbb{F}$, a subset in $\partial \mathbb{F}$ is called an edge if it is contained in a $(n-1)$-dimensional hyperplane. Since each flat is convex, each edge is also convex.

\begin{theo}\label{flatthm1}
Given a class $c_0\in H^1(\mathbb{T}^n,\mathbb{R})$, if the minimal measure is uniquely ergodic, supported on a hyperbolic fixed point, then there exists an $n$-dimensional flat $\mathbb{F}_0\subset H^1(\mathbb{T}^n,\mathbb{R})$ such that this point supports a $c$-minimal measure for all $c\in\mathbb{F}_0$.
\end{theo}
\noindent{\bf Remark}: The condition of this theorem does not exclude topological non-triviality of the Aubry set. An example is the product of $n$ pendulums. The Aubry set covers the whole torus
$\mathbb{T}^n$ if the Lagrangian $L$ is replaced by $L-\langle c,\dot x\rangle$ with $c$ being on the boundary of the flat.
\begin{proof}
By translation one can assume that the fixed point is at $(x,\dot x)=(0,0)$, by adding a closed 1-form and a constant to the Lagrangian, one can assume $c_0=0$ and $L(0,0)=0$.

To each closed curve $\xi$: $[-T,T]\to\mathbb{T}^n$ with $\xi(-T)=\xi(T)$ a first homology class $[\xi]=g\in H(\mathbb{T}^n,\mathbb{Z})$ is associated. We consider the quantity
$$
A(g)=\liminf_{T\to\infty}\inf_{\stackrel {\xi(-T)=\xi(T)}{\scriptscriptstyle [\xi]=g}} \int_{-T}^TL(d\xi(t))dt.
$$
By the condition assumed on $L$, one has that $A(g)\ge0$ for any $g\ne 0$. There exist at least $n+1$ irreducible classes $g_i\in H(\mathbb{T}^n,\mathbb{Z})$ and $n+1$ minimal homoclinic orbits $d\gamma_i$ such that $A([g_i])=A(\gamma_i)$ \cite{Be1}. Clearly, $H_1(\mathbb{T}^n,\mathbb{Z})$ can be generated by
the homology classes of all minimal homoclinic curves over $\mathbb{Z}_+$.

We abuse the notation $g$ to denote homology class $g\in H_1(\mathbb{T}^n,\mathbb{Z})$ or to denote a point $g\in\mathbb{Z}^n$. For each curve $\bar\gamma_T$: $[-T,T]\to\mathbb{R}^n$ with $\bar\gamma_T(-T)=0$ and $\bar\gamma_T(T)=g$, one has
$$
A(g)=\liminf_{T\to\infty}\inf_{\stackrel {\bar\xi(-T)=0}{\scriptscriptstyle \bar\xi(T)=g}} \int_{-T}^TL(d\bar\xi(t))dt.
$$
Recall the definition of globally elementary weak-KAM and note that the point $x=0$ is the support of the minimal measure. Let $u_0^-$ ($u_0^+$) denote the backward (forward) globally elementary weak-KAM for $\mathcal{M}^0=\{x=0\}$, we have
\begin{equation}
A(g)=u_0^-(g)-u_0^-(0),\qquad A(-g)=u_0^+(0)-u_0^+(g).\notag
\end{equation}
By setting $u_0^-(0)=u_0^+(0)$, we claim
\begin{equation}\label{flateq1}
A(g)+A(-g)=u^-_0(g)-u^+_0(g)>0.
\end{equation}

The quantity $A(g)$ is achieved may not by a curve connecting the origin to $g\in\mathbb{Z}^n$, but may by the conjunction of several curves  $\bar\gamma_1\ast\bar\gamma_2\ast\cdots\ast\bar\gamma_m$. Let $\bar\gamma_i$: $\mathbb{R}\to\mathbb{R}^n$ denote a curve ($i=1,\cdots,m$), the conjunction implies that $\bar\gamma_i(-\infty)=\bar\gamma_{i-1}(\infty)$. Let $\gamma_i=\pi_{\infty}\bar\gamma_i$, where $\pi_{\infty}:\mathbb{R}^n\to\mathbb{T}^n$ denotes the standard projection. In this case, $\gamma_1,\cdots,\gamma_m$ are minimal homoclinic curves such that $g=\sum_{i=1}^m[\gamma_i]$. By the definition of elementary weak-KAM, each $\bar\gamma_i$ is a $(u_0^-,L)$-calibrated curve. Let $g_i=\sum_{j=1}^i[\gamma_j]$. Obviously, some large $t_0>0$ exists such that $(\bar\gamma_i(t),\dot{\bar\gamma}_i(t))$ stays in the local stable manifold of the point $(x,\dot x)=(g_i,0)$ whenever $t\ge t_0$. Therefore, some constant $C_i$ exists such that
\begin{equation}\label{flateq2}
u^-_0(\bar\gamma_i(t))=u^+_{g_i}(\bar\gamma_i(t))+C_i,\qquad \forall\ t\ge t_0,
\end{equation}
where we use $u^-_{g_i}$ and $u^+_{g_i}$ to denote the globally elementary-KAM based on $x=g_i$. Clearly, $u^+_{g_i}$ and $u^-_{g_i}$ generate the local stable and unstable manifold around the point $(x,\dot x)=(g_i,0)$ respectively.

Because that $u^{\pm}_{g_i}$ is $L$-dominate function, for $x\in B_{\delta}(g_i)$ with suitably small $\delta>0$ we have (see Theorem 5.1.2 in \cite{Fa2})
\begin{equation}\label{flateq2.1}
u^+_{g_i}(x)-u^{+}_{g_i}(g_i)\le u^{\pm}_{0}(x)-u^{\pm}_0(g_i)\le u^-_{g_i}(x)-u^{-}_{g_i}(g_i).
\end{equation}
Remember that $u_0^-\ge u_0^+$. If
$$
A(g)+A(-g)=u_0^-(g)-u_0^+(g)=0,
$$
substituting $u^+_{g_i}$ in (\ref{flateq2.1}) by the expression in (\ref{flateq2}) we see  that some $t_0$ exists so that
$$
u_0^+(\bar\gamma_m(t))= u^-_0(\bar\gamma_m(t)), \qquad \forall\ t\ge t_0.
$$
Here, $t_0$ is chosen so that $\bar\gamma_m(t))\in B_{\delta}(g_m)$ for $t\ge t_0$.
Since $u_0^+$ is an $L$-dominate function and $\bar\gamma_i$ is a $(u_0^-,L)$-calibrated curve for each $1\le i\le m$,
\begin{align*}
u_0^+(\bar\gamma_i(t_0))-u_0^+(\bar\gamma_i(t_1))&\le\int_{t_1}^{t_0}L(d\bar\gamma_i(s))ds,\\
u_0^-(\bar\gamma_i(t_0))-u_0^-(\bar\gamma_i(t_1))&=\int_{t_1}^{t_0}L(d\bar\gamma_i(s))ds
\end{align*}
hold for any $t_1\le t_0$. This induces that $u_0^-(\bar\gamma_m(t))=u_0^+(\bar\gamma_m(t))$ holds for all $t\in\mathbb{R}$, and induces in turn the inequality for $i=m-1,m-2,\cdots$, and finally we have
$$
u_0^-(\bar\gamma_1(t))=u_0^+(\bar\gamma_1(t)), \qquad \forall\ t\in\mathbb{R}.
$$
Since $u_0^-(x)\ge u_0^+(x)$ holds in a small neighborhood of $0$
\begin{equation*}
u_0^-(\bar\gamma_1(t))= u_0^+(\bar\gamma_1(t)), \qquad \forall\ t\in\mathbb{R}.
\end{equation*}

On the other hand, as the fixed point $\{x=0\}$ is hyperbolic, some $\delta>0$ exists such that
\begin{equation}
u^-_0(x)-u^+_0(x)>0, \qquad \forall\ x\in B_{\delta}(0)\backslash\{0\},\notag
\end{equation}
if we set $u^-_0(0)=u^+_0(0)$. This contradiction proves the formula (\ref{flateq1}).

Let
$$
\mathbb{G}_0=\{g\in H_1(\mathbb{T}^n,\mathbb{Z}):\exists\ \gamma: \mathbb{R}\to\mathbb{T}^n\ s.t. \ [\gamma]=g,\ A(\gamma)=0\}.
$$
$\mathbb{G}_0$ is said to generate a rational direction $g\in\mathbb{Z}^n$ over $\mathbb{Z}_+$ if there exist $k,k_i\in\mathbb{Z}_+$ and $g_i\in \mathbb{G}_0$ such that
$$
kg=\sum k_ig_i.
$$
It is an immediate consequence of the formula (\ref{flateq1}) that once $\mathbb{G}_0$ generates a rational direction $g\in\mathbb{Z}^n$ over $\mathbb{Z}_+$, then it can not generate the direction $-g$ over $\mathbb{Z}_+$. Therefore, the set
$$
\text{\rm span}_{\mathbb{R}_+}\mathbb{G}_0=\{\Sigma a_ig_i:\ g_i\in\mathbb{G}_0,\ a_i\ge 0\}
$$
is a cone properly restricted in half space. Thus, there exists an $n$-dimensional cone $\mathbb{C}_0$ such that
$$
\langle c,g\rangle >0,\qquad \forall\ c\in\mathbb{C}_0,\ g\in \text{\rm span}_{\mathbb{R}_+}\mathbb{G}_0.
$$

Since the minimal measure for zero cohomology class is supported on the fixed point, $\tilde{\mathcal{N}}(0)$ is composed of those minimal homoclinic orbits along which the action equals zero. According to the upper semi-continuity of Ma\~n\'e set in cohomology class, any minimal measure $\mu_c$ is supported by a set lying in a small neighborhood of these homoclinic orbits if $|c|$ is very small. Consequently. we have $\rho(\mu_c)\in\text{\rm span}_{\mathbb{R}_+}\mathbb{G}_0$, where $\rho(\mu_c)$ denotes the rotation vector of $\mu_c$.

Let us consider a cohomology class $c$ such that $-c\in\mathbb{C}_0$ and $|c|\ll 1$. We claim that the $c$-minimal measure is also supported on the fixed point. Indeed, if it is not true, we would have positive average action of $L$: $ A(\mu_c)>0$, since the minimal measure for zero class is assumed unique and supported on the fixed point. By the choice of $c$ one has that $\langle c,\rho(\mu_c)\rangle<0$. Thus, one obtains
$$
A_c(\mu_c)=A(\mu_c)-\langle c,\rho(\mu_c)\rangle>0=A_c(\mu),
$$
it deduces absurdity. For this class $c$, the action of the Lagrangian $L_c=L-\langle c,\dot x\rangle$ along any minimal homoclinic curve $\gamma$ is positive,
$$
A(\gamma)-\langle c,[\gamma]\rangle >0,
$$
namely, the Aubry set for this class is also a singleton. Consequently, $\mu_{c'}$ is also supported on this point if $c'$ is sufficiently close to $c$. This verifies the existence of $n$-dimensional flat.
\end{proof}

\noindent{\bf Eigenvalues of the fixed point}

Let us consider the eigenvalues of the fixed point by assuming the hyperbolicity, denoted by $\lambda_i$ $(i=1,2,\cdots, 2n)$. Under the hyperbolic assumption, half of these have positive real part, other half have negative real part. In general, these eigenvalues may have non-zero imaginary part. But in nearly integrable systems, all eigenvalues are real.

\begin{pro}\label{flatpro1}
If the Lagrangian is a small perturbation of integrable one $L=\ell(\dot x)+\epsilon P(x,\dot x)$ where $\ell$ is positive definite in $\dot x$, then for generic $P$ and for sufficiently small $\epsilon$, all eigenvalues at the fixed point are real and different.
\end{pro}
\begin{proof}
Let $A=\partial^2_{\dot x\dot x}L$, $B=\partial^2_{x\dot x}P$, $C=\partial^2_{xx}P$ evaluated at the fixed point. As the minimal measure is supported on a fixed point, $C$ is positive definite. We consider the linearized equation and assume the solution with the form of $x=\xi \exp{\sqrt{\epsilon}\lambda t}$,
then
\begin{equation}\label{flateq3}
\left |\lambda^2 A-\sqrt{\epsilon}\lambda(B-B^t)-C\right |_{n\times n}=0.
\end{equation}
Let $A_0=\partial^2_{\dot x\dot x}\ell$, evaluated at the fixed point. For generic $C$, all solutions of the equation
\begin{equation}\label{flateq4}
\left |\lambda^2 A_0-C\right |_{n\times n}=0.
\end{equation}
are real and different from each other: $\lambda=\pm\lambda_1,\pm\lambda_2,\cdots,\pm\lambda_n$,
$\lambda_i\neq\lambda_j$ if $i\neq j$. Since (\ref{flateq3}) is a small perturbation of (\ref{flateq4}), all solutions of (\ref{flateq3}) are different, and consequently, real. If there was a complex solution $\lambda=\sigma+i\omega$, $\pm\sigma\pm i\omega$ would be solution also, which is guaranteed by the Hamiltonian structure. It implies the existence of more than $2k$ solutions, but it is absurd.
\end{proof}

\noindent{\bf The shape of the flat}

Here, we are concerned about the flat $\mathbb{F}_0=\mathscr{L}_{\beta}(0)$. It is a $n$-dimensional flat if the $c$-minimal measure is supported on the hyperbolic fixed point for each $c\in\text{\rm int}\mathbb{F}_0$. By coordinate translation, we assume it is at the origin: $(\dot x,x)=(0,0)$. Correspondingly, in canonical coordinates the fixed point is also at the origin $(x,y)=(0,0)$. Let $(\xi_i^{\pm},\eta_i^{\pm})$ denote the eigenvector for $\pm\lambda_i$, where $\xi_i^{\pm}$ is for the $x$-coordinates, $\eta_i^{\pm}$ is for the $y$-coordinates. We assume

1, all eigenvalues are real number and different;

2, all minimal homoclinic curves approach to the fixed point in the direction $\xi_1^{\pm}$ as $t\to\mp\infty$.

The condition 1 is obviously generic. To see the genericity of the condition 2, let us remind reader that there is, generically, at most one minimal homoclinic curve for each homology class. By further perturbation, it approaches to the fixed point in the direction of $\xi_1^{\pm}$. Since there are countably many homology classes at most, the genericity is obtained.

For $\theta>0$ and $\xi\in\mathbb{R}^n\backslash\{0\}$, we define a cone
$$
C(\xi,\theta)=\{x\in\mathbb{R}^n:|\langle x,\xi\rangle|\ge\theta\|\xi\|\|x\|\},
$$
and let
$$
C(\xi,\theta,d)=\{x\in C(\xi,\theta):\|x\|=d\}.
$$

\begin{pro}\label{flatpro2}
Assume that $(x,y)=(0,0)\in\{H^{-1}(0)\}$ is a hyperbolic fixed point for $\Phi_H^t$, where all eigenvalues are real and different:
$$
Spec\{J\nabla H\}=\{\pm\lambda_1,\cdots,\pm\lambda_n; \ \ 0<\lambda_1<\cdots<\lambda_n\}.
$$
Let $(\xi_i^{\pm},\eta_i^{\pm})$ denote the eigenvector for $\pm\lambda_i$, where $\xi_i^{\pm}$ is for the $x$-coordinates, $\eta_i^{\pm}$ is for the $y$-coordinates. Let $(x(t),y(t))\subset\{H^{-1}(0)\}$ be an orbit such that $x(t)$ passes through a ball $B_{\delta}(0)\subset\mathbb{R}^n$, $x(-T)\in\partial B_{\delta}(0)$, $x(T)\in\partial B_{\delta}(0)$ and $x(t)\in\text{\rm int}B_{\delta}(0)$ for all $t\in (-T,T)$. Then, for suitably small $\delta>0$ and $\theta=\frac 12$,  there exist sufficiently large $T_0>0$ such that for $T\ge T_0$ one has
$$
(x(-T),x(T))\notin C(\xi_1^+,\theta,\delta)\times C(\xi_1^-,\theta,\delta).
$$
For $T\to\infty$, one has
$$
x(-T)\in C(\xi_i^+,1-o(\delta),\delta) \ \ \ \text{\rm or}\ \ \  x(T)\in C(\xi_j^-,1-o(\delta),\delta).
$$
for certain $i,j\ne1$.
\end{pro}
\begin{proof}
By certain symplectic coordinate transformation, the Hamiltonian is assumed to have the normal form
$$
H(x,y)=\sum_{i=1}^n\frac 12\Big(y_i^2-\lambda_i^2x_i^2\Big)+P_3(x,y)
$$
where $P_3=O(\|(x,y)\|^3)$ is a higher order term. By the method of variation of constants, we obtain the solution of the corresponding Hamilton equation
\begin{align}\label{flateq5}
x_i(t)=&e^{-\lambda_it}(b_{i}^{-}+F_i^-)+e^{\lambda_it}(b_{i}^{+}+F_i^+), \\
y_i(t)=&-\lambda_ie^{-\lambda_it}(b_{i}^{-}+F_i^-)+\lambda_ie^{\lambda_it}(b_{i}^{+}+F_i^+),\notag
\end{align}
where $b_i^{\pm}$ are constants determined by boundary condition and
\begin{align*}
F_i^-=&\frac{1}{2\lambda_i} \int_0^te^{\lambda_is}(\lambda_i\partial_{y_i}P_3+\partial_{x_i}P_3)(x(s),y(s))ds, \\
F_i^+=&\frac{1}{2\lambda_i} \int_0^te^{-\lambda_is}(\lambda_i\partial_{y_i}P_3-\partial_{x_i}P_3) (x(s),y(s))ds.
\end{align*}
Substituting $(x,y)$ with the formula (\ref{flateq5}) in the Hamiltonian we obtain a constraint for the constants $b_i^{\pm}$:
\begin{equation}\label{flateq6}
H(x(t),y(t))=-2\sum_{i=1}^n\lambda_i^2b_i^-b_i^++P_3((b^+_i+b^-_i),\lambda_i(b^+_i-b^-_i))
\end{equation}
Let us estimate the size of the constants $c_i^{\pm}$ by the boundary conditions $x(T)=(x^+_1,x^+_2,\cdots,x^+_k)\in\partial B_{\delta}(0)$, $x(-T)=(x^-_1,x^-_2,\cdots,x^-_k) \in
\partial B_{\delta}(0)$ and assuming
\begin{equation}\label{flateq7}
\min\{|x_1^-|,|x_1^+|\}\ge \frac{\delta}2.
\end{equation}
For $\theta=1/2$, $(x(-T),x(T))\in C(\xi_1^-,\theta,\delta)\times C(\xi_1^+,\theta,\delta)$ implies (\ref{flateq7}) holds. Since the curve $x|_{[-T,T]}$ stays inside of the ball $B_{\delta}(0)$ and $T$ is sufficiently large,  the orbit $(x,y)|_{[-T,T]}$ stays near the stable and unstable manifold of the fixed point. Note $P=O(\|(x,y)\|^3)$, we obtain from the theorem of Grobman-Hartman that
\begin{align}\label{flateq7.1}
x_i^-=&b_i^-e^{\lambda_iT}+b_i^+e^{-\lambda_iT}+o(\delta),\\
x_i^+=&b_i^-e^{-\lambda_iT}+b_i^+e^{\lambda_iT}+o(\delta).\notag
\end{align}
For sufficiently large $T>0$, it deduces from the assumption (\ref{flateq7}) that
$$
|b_1^{\pm}|\ge\frac {\delta}{3} e^{-\lambda_1T},
$$
and
$$
|b_i^{\pm}|\le 2\delta e^{-\lambda_iT},\qquad \forall\ i=2,\cdots,k.
$$
Since $\lambda_1<\lambda_i$ for each $i\ge 2$, $|b_i^{\pm}|\ll|b_1^{\pm}|$ if $T$ is sufficiently large. In this case, we obtain from (\ref{flateq6}) that
$$
|H(x(t),y(t))|>|\lambda_1^2b_1^+b_1^-|>0.
$$
It contradicts the assumption that $(x(t),y(t))\in\{H^{-1}(0)\}$. Let $T\to\infty$, one easily sees the last conclusion.
\end{proof}

This proposition tells us following fact. In the energy level $\{H^{-1}(0)\}$ there does not exist such an orbit passing through $B_{\delta}(0)$ in the way that it enters into the ball in a direction close to $\xi_1^+$ and leaves in a direction close to $\xi_1^-$.

\begin{theo}\label{flatthm2}
Let $\mathbb{F}_0=\mathscr{L}_{\beta}(0)$ be an $n$-dimensional flat of the $\alpha$-function. Each minimal homoclinic curve $\gamma$ is assumed approaching to the fixed point in the direction of the eigenvectors corresponding to the smallest eigenvalue $$
\lim_{t\to\pm\infty}\frac{\dot\gamma(t)}{\|\dot\gamma(t)\|}=\frac{\xi_1^{\mp}}{\|\xi_1^{\mp}\|},
$$
all eigenvalues are assumed real and different.  It is also assumed that, for each $c\in\mathbb{F}_0$  $($including the boundary$)$, the minimal measure is uniquely supported on the hyperbolic fixed point. Then, there exists a finite set
$$
H_{\mathbb{F}_0}=\{g_1,g_2,\cdots,g_m\}\subset H_1(\mathbb{T}^n,\mathbb{Z})
$$
such that $[\gamma]\in H_{\mathbb{F}_0}$ if $\gamma$ is a minimal homoclinic curve. Consequently, the flat $\mathbb{F}_0$ is a polygon with finitely many edges, denoted by $\mathbb{E}_1,\cdots,\mathbb{E}_m$. Each edge $\mathbb{E}_i$ is associated with a homological class $g_i$ such that $\mathcal{A}(c)$ is composed by minimal homoclinic curves with homological type $g_i$ if $c$ is in the interior of $\mathbb{E}_i$ and
$$
\langle c-c',g_i\rangle =0, \qquad \forall\ c,c'\in\mathbb{E}_i.
$$
\end{theo}
\begin{proof}
As the minimal measure is uniquely ergodic and supported on a point for each class in $\mathbb{F}_0$, the Ma\~n\'e set consists of homoclinic orbits and the point itself.

For each class $c\in\partial\mathbb{F}_0$, we claim that the Ma\~n\'e set contains at least one homoclinic orbit. Otherwise, for each class $c'\notin\mathbb{F}_0$ very close to $c$, the homology of the Ma\~n\'e set is trivial, the same as that for $c$. It is guaranteed by the upper semi-continuity of Ma\~n\'e set in cohomology class. It follows that $\langle c,\rho(\mu_c)\rangle =\langle c',\rho(\mu_c')\rangle=0$ and
$$
-\alpha(c')=A(\mu_{c'})-\langle c',\rho(\mu_c')\rangle\ge A(\mu_c)=-\alpha(c).
$$
However, as $c'\notin\mathbb{F}_0$, one has $\alpha(c')>\alpha(c)$. The contradiction verifies our claim.

Approached by minimal periodic curves, each minimal homoclinic curve $\gamma$ stays in certain Aubry set:
$$
\cup_{t\in\mathbb{R}}\gamma(t)\subset\mathcal{A}(c),\qquad \forall\ c\in\lim_{\delta\downarrow 0} \mathscr{L}_{\beta}(\delta[\gamma])\subset \mathbb{F}_0.
$$
where the limit is in the sense of Hausdorff.

If $H_{\mathbb{F}_0}$ contains infinitely many elements, there would be infinitely many minimal homoclinic curves $\gamma_1,\gamma_2\cdots\gamma_k\cdots$ such that $[\gamma_i]\neq
[\gamma_j]$ provided $i\ne j$. Thus we have two possibilities.

1, a neighborhood $B_d(0)$ of the fixed point exists such that each minimal homoclinic curve $\gamma$ hits the sphere $\partial B_d(0)$ exactly twice, i.e. $\exists$ $t^-<t^+$ such that $\gamma(t)\in B_d(0)$ for all $t\in (-\infty,t^-]\cup [t^+,\infty)$ and $\gamma(t)\notin B_d(0)$ for all $t\in (t^-,t^+)$;

2, for any small $d>0$, there are infinitely many minimal homoclinic curves $\gamma_{i_1}, \gamma_{i_2}\cdots $ passing through the sphere $\partial B_d(0)$ in finite time, i.e. $\exists$ $t^-<t_0^-<t_0^+<t^+$ such that $\gamma(t)\in B_d(0)$ for all $t\in (-\infty,t^-]\cup[t^-_0,t_0^+]\cup [t^+,\infty)$ and $\gamma(t)\notin B_d(0)$ holds for some $t\in (t^-,t^-_0)$ as well as for some $t\in (t^+_0,t^+)$.

Let us study the first possibility. Denote by $t^-_i<t^+_i$ the time when the minimal homoclinic curve $\gamma_i$ hits the sphere $\partial B_d(0)$. For each $\gamma_i$, there is a segment $\gamma_i|_{(t^-_i,t^+_i)}$ staying outside of $B_d(0)$. Each $d\gamma_i|_{(t_i^-,t_i^+)}$ generates a probability measure $\mu_i$ on $T\mathbb{T}^n$ such that
$$
\int fd\mu_i=\frac 1{|t_i^+-t_i^-|}\int_{t_i^-}^{t_i^+}f(d\gamma_i(s))ds
$$
holds for each continuous function $f$: $T\mathbb{T}^n\to\mathbb{R}$. As all these curves have different homology class, $\|[\gamma_i]\|\to\infty$ as $i\to\infty$. As the speed along these curves are uniformly bounded, we have
$$
|t_i^+-t_i^-|\to\infty, \qquad \text{\rm as }\ i\to\infty.
$$
Let $c_i\in\partial\mathbb{F}_0$ be the class such that $\gamma_i\subset\mathcal{A}(c_i)$ and let $c^*\in\partial\mathbb{F}_0$ be an accumulation point of $\{c_i\}$, some invariant probability measure
$\mu^*$ exists such that $\mu_i\rightharpoonup\mu^*$, it is $c^*$-minimal. Clearly, $\mu^*$ is not supported on the fixed point, it contradicts the assumption that the minimal measure is always uniquely ergodic for each $c\in\mathbb{F}_0$.

Let us study the second possibility. In this case, for suitably small $\delta>0$, there is an infinite sequence of homoclinic curves $\gamma_i$ and correspondingly the sequence of time $t^-_i<t^+_i$ such that $\gamma_i(t)\in B_{\delta}(0)$ for each $t\in [t^-_i,t^+_i]$, $\gamma_i(t_i^{\pm}\mp\epsilon) \notin B_{\delta}(0)$ and $|t_i^+-t_i^-|\to\infty$ as $i\to\infty$. Indeed, if  $|t_i^+-t_i^-|$ remains bounded, one can choose $\delta'<\delta$ such that these curves hit the sphere $\partial B_{\delta'}(0)$ twice only. It is the first case again. By using Proposition \ref{flatpro2}, we find that some $\theta>0$ exists such that one of the inequalities in the following holds for each $i$
\begin{equation}\label{flateq8}
\Big\|\frac{\dot\gamma_i(t_i^-)}{\|\dot\gamma_i(t_i^-)\|}-\frac{\xi_1^+}{\|\xi_1^+\|}\Big\|>\theta,\qquad
\Big\|\frac{\dot\gamma_i(t_i^+)}{\|\dot\gamma_i(t_i^+)\|}-\frac{\xi_1^-}{\|\xi_1^-\|}\Big\|>\theta
\end{equation}
provided $i$ is sufficiently large. It implies that there exists some minimal homoclinic curve $\gamma$ as well as some eigenvector $\xi^-_{k_1}$ or $\xi^+_{k_2}$ with $k_1\ne 1$ and $k_2\ne 1$ such that at least one of the following
holds
$$
\lim_{t\to-\infty}\frac{\dot\gamma(t)}{\|\dot\gamma(t)\|}=\frac{\xi_{k_2}^+}{\|\xi_{k_2}^+\|},\qquad
\lim_{t\to\infty}\frac{\dot\gamma(t)}{\|\dot\gamma(t)\|}=\frac{\xi_{k_1}^-}{\|\xi_{k_1}^-\|}.
$$
This leads to a contradiction to the assumption, then verifies the finiteness of $H_{\mathbb{F}_0}$.

Let $\gamma,\gamma'$ be two minimal homoclinic curves contained in the Aubry set $\mathcal{A}(c)$, $\mathcal{A}(c')$ respectively. Let $\Gamma=\{\xi c+(1-\xi)c': \xi\in [0,1]\}$. If $\Gamma$ intersects
the interior of $\mathbb{F}_0$, then $[\gamma]\neq [\gamma']$. Indeed, by definition we have
\begin{align*}
A(\gamma)-\langle c,[\gamma]\rangle&=0, \qquad A(\gamma')-\langle c,[\gamma']\rangle \ge 0;\\
A(\gamma')-\langle c',[\gamma']\rangle&=0, \qquad A(\gamma)-\langle c',[\gamma]\rangle \ge 0,
\end{align*}
it follows from $[\gamma]=[\gamma']$ that $A(\gamma)=A(\gamma')$. Consequently,
\begin{align*}
0&=\xi (A(\gamma)-\langle c,[\gamma]\rangle)+(1-\xi)(A(\gamma')-\langle c',[\gamma']\rangle) \\
&=A(\gamma)-\langle\xi c+(1-\xi)c',[\gamma]\rangle\\
&=A(\gamma')-\langle\xi c+(1-\xi)c',[\gamma']\rangle.
\end{align*}
It implies that both $\gamma$ and $\gamma'$ lie in the Aubry set for $\xi c+(1-\xi) c'$. On the other hand, the Aubry set for each class in the interior of $\mathbb{F}_0$ contains the fixed point only. The contradiction implies that $[\gamma]\neq [\gamma']$. Therefore, $\mathbb{F}_0$ is a polygon with exactly $m$ edges, each edge corresponds to one homology type of minimal homoclinic curve.

Let $\gamma$ be a minimal homoclinic curve lying in the Aubry set for $c\in\text{\rm int}\mathbb{E}_i$. Then, one has $A(\gamma)-\langle c,[\gamma]\rangle=0$. As the Aubry set remains the same for all classes in the interior of the edge, one has $A(\gamma)-\langle c',[\gamma]\rangle=0$ for each $c\in\text{\rm int}\mathbb{E}_i$. Consequently, one has $\langle c-c',[\gamma]\rangle =0$ for all $c,c'\in\text{\rm int}\mathbb{E}_i$. As it is $(n-1)$-dimensional, each edge $\mathbb{E}_i$ determines a unique homology class $g_i$ such that $[\gamma]=g_i$ if $\gamma$ is a minimal homoclinic curve lying in the Aubry set.
\end{proof}

\subsection{Modulus of continuity in terms of energy}

Let $\gamma_0$ be a minimal homoclinic curve approaching to the fixed point in the direction of $\xi^{\pm}$ corresponding to the smallest eigenvalue $\pm\lambda_1$, let $\mathbb{E}_0$ be a edge of $\mathbb{F}_0$. In this subsection we assume that

1, for each $c\in\mathbb{F}_0$, the Mather set contains exactly one fixed point $(x,\dot x)=(0,0)$;

2, for each $c\in int\mathbb{E}_0$, the Aubry set consists of the fixed point and one minimal homoclinic curve $\gamma_0$: $\mathcal{A}(c)=\cup_{t\in\mathbb{R}}\gamma_0(t)\cup\{0\}$;

3, there exist a sequence of positive numbers $\nu_i\downarrow0$ and sequence of ergodic minimal measure $\mu_{i}$ such that $\rho(\mu_{i})=\nu_i[\gamma_0]$, and $\mathcal{A}(c)=\text{\rm supp}\mu_i$ for each $c\in\text{\rm int} \mathscr{L}_{\beta}(\nu_i[\gamma_0])$.

By definition, there exists an elementary weak KAM for each $\mu_{i}$ corresponding to the energy $E_i=\alpha(\mathscr{L}_{\beta}(\nu_i[\gamma_0]))$. The main purpose of this section is to study the modulus of continuity of some functions in terms of energy at $E=0$.

\noindent{\bf Dependence of the average speed on energy}.

According to Birkhoff's ergodic theorem, there is an orbit $d\zeta_i$: $\mathbb{R}\to\mathbb{T}^n$ of $\phi_L^t$ in the Mather set such that
$$
\frac 1{2T}A(\zeta_i|_{[-T,T]})\to A(\mu_i)\ \ \text{\rm and}\ \ \frac 1{2T}(\bar\zeta_i(T)-\bar\zeta_i(-T))\to\rho(\mu_i) \ \ \ \text{\rm as}\ T\to\infty,
$$
where $\bar\zeta_i$ stands for a lift of $\zeta_i$ to the universal covering space. By the upper semi-continuity of Ma\~n\'e set, the curve $\zeta_i$ passes through the ball $B_{\delta}(0)$ infinitely many times if $\nu_i$ is small. Denoted by $t_{i,k}^+$ and $t_{i,k}^-$ the time when $\zeta_{i}$ enters and leaves the ball respectively, i.e $\zeta_{i}(t)\in B_{\delta}(0)$ for each
$t\in[t_{i,k}^+,t_{i,k}^-]$, $\zeta_{i}(t_{i,k}^{\pm}\mp\delta)\notin B_{\delta}(0)$. Clearly,
$$
|t_{i,k}^--t_{i,k}^+|\to\infty,\qquad \text{\rm as}\ E_i\to 0.
$$
If $\zeta_i$ is a periodic curve, some $t_i>0$ exists so that $t_{i,k+1}^+=t_{i,k}^-+t_i$ holds for all $k\in\mathbb{Z}$.

Let $t^+,t^-\in\mathbb{R}$ such that the minimal homoclinic curve $\gamma_0$ enters $B_{\delta}(0)$ at $t=t^+$ and leaves $B_{\delta}(0)$ at $t=t^-$. By the upper semi-continuity of Ma\~n\'e set one has
$$
(\zeta_{i}(t_{i,k}^{\pm}),\dot\zeta_{i}(t_{i,k}^{\pm}))\to(\gamma_0(t^{\pm}),\dot\gamma_0(t^{\pm})).
$$
As each minimal homoclinic curve approaches to the fixed point in the direction $\xi_1^{\pm}$,
\begin{equation}\label{regularenergyeq1}
\Big\|\frac{\dot\zeta_{i}(t_{i,k}^{\pm})}{\|\dot\zeta_{i}(t_{i,k}^{\pm})\|}-\frac{\xi_1^{\mp}}{\|\xi_1^{\mp}\|} \Big\|<\frac 14
\end{equation}
holds if ${\delta}>0$ is suitably small and $t_{i,k}^--t_{i,k}^+$ is suitably large.

Each segment $\zeta_{i}|_{[t_{i,k}^+,t_{i,k}^-]}$ solves the Hamilton equation. In the coordinates of normal form, it is given by Eq. (\ref{flateq7.1}) and the integral constants $b_i^{\pm}$ satisfy the constraint (\ref{flateq6}). The condition (\ref{regularenergyeq1}) induces the following
$$
\frac 12 {\delta}e^{-\lambda_1|t_{i,k}^--t_{i,k}^+|/2}\le |b_1^{\pm}|\le 2{\delta} e^{-\lambda_1|t_{i,k}^--t_{i,k}^+|/2}.
$$
For small ${\delta}>0$ and sufficiently large $|t_{i,k}^--t_{i,k}^+|$, we have
$$
|b_j^{\pm}|\le 2{\delta}e^{-\lambda_j|t_{i,k}^--t_{i,k}^+|/2},\qquad \forall\ j=2,\cdots,n,
$$
$$
|P_3((b^+_j+b^-_j),\lambda_j(b^+_j-b^-_j))|\le Ce^{-3\lambda_1|t_{i,k}^--t_{i,k}^+|/2}
$$
where the constant $C$ depends only on the function $P_3$. So, for suitably small ${\delta}>0$  and sufficiently large $|t_{i,k}^--t_{i,k}^+|$,  we obtain from (\ref{flateq6}) that
\begin{align*}
E_i=&\Big|-2\sum_{j=1}^n\lambda_j^2b_i^+b_i^-+P_3((b^+_j+b^-_j),\lambda_j(b^+_j-b^-_j))\Big|\\
\ge&\frac 12\lambda_1^2{\delta}^2e^{-\lambda_1|t_{i,k}^--t_{i,k}^+|}-8\sum_{j=2}^n \lambda_j^2{\delta}^2e^{-\lambda_j|t_{i,k}^--t_{i,k}^+|}-Ce^{-3\lambda_1|t_{i,k}^--t_{i,k}^+|/2}\notag\\
\ge&\frac 14\lambda_1^2{\delta}^2e^{-\lambda_1|t_{i,k}^--t_{i,k}^+|}
\end{align*}
Under the same condition, $E_i$ is obviously upper bounded by
\begin{align*}
E_i=&\Big|-2\sum_{j=1}^n\lambda_j^2b_i^+b_i^-+P_3((b^+_j+b^-_j),\lambda_j(b^+_j-b^-_j))\Big|\\
\le&8\lambda_1^2{\delta}^2e^{-\lambda_1|t_{i,k}^--t_{i,k}^+|}+8\sum_{j=2}^n \lambda_j^2{\delta}^2e^{-\lambda_j|t_{i,k}^--t_{i,k}^+|}+Ce^{-3\lambda_1|t_{i,k}^--t_{i,k}^+|/2}\notag\\
\le&9\lambda_1^2{\delta}^2e^{-\lambda_1|t_{i,k}^--t_{i,k}^+|}.
\end{align*}
Therefore, we find the dependence of speed on the energy
\begin{equation}\label{regularenergyeq2}
|t_{i,k}^--t_{i,k}^+|=\frac 1{\lambda_1}|\ln E_i|-\frac 2{\lambda_1}|\ln \delta|+\tau_{i,k}
\end{equation}
where $\tau_{i,k}$ is uniformly bounded for each $k\in\mathbb{Z}$:
$$
\frac 1{\lambda_1}(2\ln\lambda_1-2\ln 2)\le\tau_{i,k}\le \frac 1{\lambda_1}(2\ln\lambda_1+3\ln 3).
$$
Obviously, $t_{i,k+1}^+-t_{i,k}^-\to t^+-t^-$ as $i\to\infty$, some
$E_0>0$ exists such that
$$
\frac 12(t^+-t^-)\le t_{i,k+1}^+-t_{i,k}^-\le 2(t^+-t^-),\qquad\text{\rm if}\ E_i\le E_0.
$$
Set $\tau_0=\frac 2{\lambda_1}|\ln 3\lambda_1|+2(t^+-t^-)$, one has
\begin{equation}\label{regularenergyeq3}
\Big|\frac 1\nu_i-\frac 1{\lambda_1}\Big|\ln\frac{E_i}{\delta^{2}}\Big|\Big|\le\tau_0.
\end{equation}
Recall the meaning of $\nu_i$: $\nu_i[\gamma_0]$ is the rotation vector of the minimal measure$\mu_i$.

\subsection{Around the two-dimensional flat}

In this section we restrict ourselves to the special case that the system has two degrees of freedom: $n=2$. The task of this section is to study the structure of the Mather sets as well as of the Ma\~n\'e sets in a neighborhood of the resonant point. Under the coordinate transformation (\ref{normaleq10}), it corresponds to a fixed point.

Since each Aubry set is a Lipschitz graph over the configuration manifold which is two dimensional here, each orbit in an Aubry set has to be {\it parallel} to any other orbit in the same set in the sense that these curves do not intersect each other. On the other hand, for autonomous system, $\beta(\lambda\omega)$, regarded as the function of $\lambda\in\mathbb{R}$, is differentiable at each $\lambda\neq 0$ (see \cite{Ms}). Thus, we have

\begin{pro}\label{flatpro3}
Assume that $L$ is an autonomous Tonelli Lagrangian defined on $\mathbb{T}^2$. For each non-zero rational vector, the Mather set consists of periodic orbits with the same rotation vector.
\end{pro}

Each minimal measure with zero-rotation corresponds to the minimum of the $\alpha$-function. There are two {\it nondegenerate} cases for the set of minimal point $\mathbb{F}_0\subset H^1(\mathbb{T}^2,\mathbb{R})$. We call a case {\it nondegenerate} if it persists under small perturbation.

1, $\mathbb{F}_0$ is a two-dimensional flat. Typically, for each class in the interior of $\mathbb{F}_0$, the minimal measure is supported on a fixed point, or a shrinkable periodic orbit $(\gamma,
\dot\gamma)$, i.e. $[\gamma]=0$. The fixed point (periodic orbit) is of hyperbolic type.

2, $\mathbb{F}_0$ is one-dimensional. Typically, the minimal measure is supported on two periodic orbits $(\gamma_-, \dot\gamma_-)$ and $(\gamma_+, \dot\gamma_+)$ with the property:
$$
[\gamma_-]/\|[\gamma_-]\|=-[\gamma_+]/ \|[\gamma_+]\|.
$$

The set $\mathbb{F}_0$ is a singleton only when the $\alpha$-function is differentiable at this point. Otherwise, the $\beta$ function can not have a two-dimensional flat. In this case, the Mather set contains two circles with different rotation direction, but it violates the Lipschitz property.

Let us study the first case and assume that $0\in\text{\rm int}\mathbb{F}_0$, the point $(x,\dot x)=(0,0)$ supports the minimal measure. Then, $\mathcal{A}(c)=\{0\}$ for all $c\in int\mathbb{F}_0$. The study is similar if it is supported on a shrinkable closed orbit.

Given the $\alpha$ as well as the $\beta$-function, let us recall the Fenchel-Legendre transformation $\mathscr{L}_{\beta}$: $H_1(M,\mathbb{R})\to H^1(M,\mathbb{R})$ is defined as
$$
\mathscr{L}_{\beta}(\omega)=\{c: \alpha(c)+\beta(\omega) =\langle
c,\omega\rangle\}.
$$
Let
\begin{equation*}
\partial^*\mathbb{F}_0=\{c\in\partial\mathbb{F}_0: \ \mathcal{M}(c)\backslash\{x=0\}\neq\varnothing \},\qquad \Omega_{\mathbb{F}_0}=\mathscr{L}_{\beta}^{-1}(\partial^*\mathbb{F}_0),
\end{equation*}
it may be non-empty. Here is an example:
$$
L=\frac 12\dot x_1^2+\frac {\lambda^2}2\dot x_2^2+V(x)
$$
where $|\lambda|\ne 1$, the potential satisfies the following conditions: $x=0$ is the minimal point of $V$ only; there exist two numbers $d>d'>0$ such that for any closed curve $\gamma$: $[0,1]\to\mathbb{T}^2$ passing through the origin with $[\gamma]\ne 0$ one has
$$
\int_0^1V(\gamma(s))ds\ge d;
$$
$V=d'+(x_2-a)^2$ when it is restricted a neighborhood of circle $x_2=a$ with $a\ne 0$ mod 1. In this case, $\partial\mathbb{F}_0\cap\{c_2=0\}=\{c_1=\pm\sqrt{2d'}\}$. Indeed,
$$
L\pm c_1\dot x_1=\frac 12(\dot x_1\pm c_1)^2+\frac{\lambda^2}2\dot x_2^2+V(x)-\frac 12c_1^2,
$$
the Mather set for $c=(\pm\sqrt{2d'},0)$ consists of the point $x=0$ and the periodic curve $x(t)=(x_{1,0}\mp\sqrt{2d'}t,a)$.

Clearly, the set $\partial^*\mathbb{F}_0$ is closed with respect to $\mathbb{F}_0$. If it is non-empty, the existence of infinitely many $\bar M$-minimal homoclinic orbits has been proved in \cite{Zhe,Zho1}. These orbits are associated with different homological classes. If $\partial^*\mathbb{F}_0=\varnothing$, there are at least three minimal homoclinic orbits to the fixed point.

The existence of homoclinic orbit to some Aubry set is closely related to the existence of the flat of the $\alpha$-function.
\begin{lem}\label{flatlem1}
Given $c,c'\in\mathbb{F}$, let $c_{\lambda}=\lambda c+(1-\lambda)c'$. Then
$$
\tilde{\mathcal{A}}(c)\cap\tilde{\mathcal{A}}(c')=\tilde{\mathcal{A}}(c_{\lambda}), \qquad \forall\ \lambda\in (0,1).
$$
\end{lem}
\begin{proof}
Using argument in \cite{Ms}, for any curve $\gamma$: $\mathbb{R}\to M$, we have
$$
[A_{c_{\lambda}}(\gamma|_{I})]=\lambda [A_c(\gamma|_{I})]+(1-\lambda)[A_{c'}(\gamma|_{I})],\qquad \forall\
I\subset\mathbb{R}.
$$
As both $\lambda>0$ and $1-\lambda>0$, one has that $[A_c(\gamma)]=[A_{c'}(\gamma)]=0$ if $[A_{c_{\lambda}}(\gamma)]=0$.
\end{proof}
\begin{lem}\label{flatlem2}
Let $\mathbb{F}_0$ be a 2-dimensional flat, the Mather set is a singleton for each class in the interior of $\mathbb{F}_0$, let $\mathbb{E}_i$ be an edge of $\mathbb{F}_0$, then
$$
\mathcal{A}(c')\supsetneq\mathcal{A}(c)
$$
holds for $c'\in\partial\mathbb{F}_0$ $(\partial\mathbb{E}_i)$ and $c\in int \mathbb{F}$ $(int \mathbb{E}_i)$ respectively.
\end{lem}
\begin{proof}
As the Mather set is a singleton for each $c\in\text{\rm int}\mathbb{F}_0$, each orbit in the Aubry set is either the fixed point itself, or a homoclinic orbit to the point with zero first homology. Indeed, let $[\gamma]$ denote its first homology of the homoclinic curve $\gamma$ in the Aubry set, then
$$
\int_{-\infty}^{\infty}L(d\gamma(t))dt-\langle c,[\gamma]\rangle=0
$$
holds for each $c\in\text{\rm int}\mathbb{F}_0$. It follows that $\langle c-c',[\gamma]\rangle=0$ for $c,c'\in\text{\rm int}\mathbb{F}_0$. Since $\mathbb{F}_0$ shares the same dimension of the configuration space, $[\gamma]=0$. In fact, for classical mechanical system, the Aubry set consists of the fixed point only for $c\in\text{\rm int}\mathbb{F}_0$. If $c'\in\partial\mathbb{F}_0\backslash\partial^*\mathbb{F}_0$, as shown in the proof of Theorem \ref{flatthm2}, the Aubry set $\mathcal{A}(c')$ contains at least one minimal homoclinic curve with non-zero first homology. If $c'\in\partial^*\mathbb{F}_0$, the certain $c'$-minimal measure $\mu_{c'}$ exists with $\rho(\mu_{c'})\ne 0$. In both cases, $\mathcal{A}(c')\supsetneq\mathcal{A}(c)$ if $c\in\text{\rm int}\mathbb{F}_0$.

Let $\mathbb{E}_i$ be an edge. For $c\in\text{\rm int}\mathbb{E}_i$, the Aubry set contains one or more homoclinic curves, all of them share the same homology class, denoted by $g(\mathbb{E}_i)$ which is of course non-zero. If $\mathcal{M}(c)$ contains other curves, these curves also share the same rotation vector as $\langle c-c',g(\mathbb{E}_i)\rangle=0$ holds for $c,c'\in\text{\rm int}\mathbb{E}_i$.

Let $c'\in\partial\mathbb{E}_i$ and $c\in\text{\rm int}\mathbb{E}_i$, one chooses $c^*\in\partial\mathbb{F}_0 \backslash\mathbb{E}_i$ arbitrarily close to $c'$. As the straight line connecting $c$ to $c^*$ passes through the interior of $\mathcal{F}_0$, we obtain from Lemma \ref{flatlem1} that $\mathcal{A}(c)\cap \mathcal{A}(c^*)=\mathcal{A}(c_0)$ with $c_0\in int\mathbb{E}_i$. For any curve $\zeta$ contained in $\mathcal{A}(c^*)\backslash\mathcal{A}(c_0)$, it follows from the formulation
$$
0=\int (L(d\zeta(t))-\langle c^*,\dot\zeta\rangle)dt=\int (L(d\zeta(t))-\langle c,\dot\zeta\rangle)dt+\langle c-c^*,[\zeta]\rangle
$$
that $\langle c-c^*,[\zeta]\rangle\neq0$ holds. We claim $[\zeta]\neq g(\mathbb{E}_i)$. Let us assume the contrary and consider the case that $\zeta$ is a homoclinic curve and $\mathcal{A}(c)$ contains a homocilinic curve $\gamma$. In this case, by assuming that $\alpha(c)=0$ for $c\in\mathbb{F}_0$, we have
$$
\int_{-\infty}^{\infty} L(d\zeta)dt -\langle c^*,[\zeta]\rangle=0, \qquad \int_{-\infty}^{\infty} L(d\gamma)dt -\langle c,g(\mathbb{E}_i)\rangle=0.
$$
Since the class $c^*$ is not on the straight line containing $\mathbb{E}_i$, we have $\langle c^*-c,g(\mathbb{E}_i)\rangle\ne0$. If $\langle c^*-c,g(\mathbb{E}_i)\rangle>0$ we would have
$$
\int_{-\infty}^{\infty} L(d\gamma)dt-\langle c^*,[\gamma]\rangle=\int_{-\infty}^{\infty} L(d\gamma)dt-\langle c,[\gamma]\rangle- \langle c^*-c,g(\mathbb{E}_i)\rangle<0
$$
If $[\zeta]=g(\mathbb{E}_i)$ and $\langle c^*-c,g(\mathbb{E}_i)\rangle<0$ we would have
$$
\int_{-\infty}^{\infty} L(d\zeta)dt-\langle c,[\zeta]\rangle=\int_{-\infty}^{\infty} L(d\gamma)dt-\langle c^*,[\zeta]\rangle+ \langle c^*-c,g(\mathbb{E}_i)\rangle<0
$$
Both cases are absurd as $\alpha(c)=\alpha(c^*)=0$. Because $[\zeta]\neq g(\mathbb{E}_i)$, some $x^*\in\mathcal{A}(c^*)$ remains far away from $\mathcal{A}(c)$. Let $c^*\to c'$, the accumulation point of these points does not fall into $\mathcal{A}(c)$, it implies  $\mathcal{A}(c')\supsetneq\mathcal{A}(c)$. The proof is similar if $\xi$ as well as $\gamma$ is a curve lying in the Mather set.
\end{proof}

Recall the definition of $G_m$ in the section 2: a first homology class $g\in G_m$ if and only if there exists a minimal homoclinic orbit $d\gamma$ such that $[\gamma]=g$. Let $G_{m,c}\subset G_m$ be defined such that $g\in G_{m,c}$ if and only if there exists a minimal homoclinic orbit $d\gamma$ in $\tilde{\mathcal{A}}(c)$ such that $[\gamma]=g$. We say that there are $k$-types of minimal homoclinic orbits in $\tilde{\mathcal{A}}(c)$ if $G_{m,c}$ contains exactly $k$ elements. For an edge we define $G_{m,\mathbb{E}_i}=G_{m,c}$ for each $c\in int \mathbb{E}_i$, from the proof of Lemma \ref{flatlem2} one can see that it makes sense.
\begin{theo}\label{flatthm3}
Let $\mathbb{F}_0$ be a two dimensional flat, $\mathcal{M}(c_0)$ is a singleton for $c_0\in int\mathbb{F}_0$. Let $\mathbb{E}_i$ denote an edge of $\mathbb{F}_0$ (not a point), then

1, either $\mathbb{E}_i\cap\partial^*\mathbb{F}_0 =\varnothing$ or $\mathbb{E}_i\subset\partial^*\mathbb{F}_0$;

2, if $\mathbb{E}_i\cap\partial^*\mathbb{F}_0 =\varnothing$, then $G_{m,\mathbb{E}_i}$ contains exactly one element, if $\mathbb{E}_i\subset\partial^*\mathbb{F}_0$, all curves in $\mathcal{M}(\mathbb{E}_i)\backslash\{0\}$ have the same rotation vector;

3, if $c\in\partial\mathbb{E}_i$ and $c\notin\partial^*\mathbb{F}_0$ then $G_{m,c}$ contains exactly two elements;

4, if $\mathbb{E}_i, \mathbb{E}_j\subset\partial^*\mathbb{F}_0$, then either $\mathbb{E}_i$ and $\mathbb{E}_j$ are disjoint, or $\mathbb{E}_i=\mathbb{E}_j$;

5, if $\mathbb{E}_i\subset\partial^*\mathbb{F}_0$, $\mathcal{M}(c)=\mathcal{M}(c')$ holds for $c\in\partial \mathbb{E}_i$ and $c'\in int \mathbb{E}_i$.
\end{theo}
\begin{proof}
For the conclusion 1, as $\mathcal{A}(c)=\mathcal{A}(c')$ if $c,c'\in\text{\rm int}\mathbb{E}_i$ \cite{Ms}, we only need to consider $c\in\partial\mathbb{E}_i$. If it is not true, there would exist an invariant measure $\mu_c$, not supported on the singleton and minimizing the action
$$
\int Ld\mu_c-\langle \rho(\mu_c),c\rangle=-\alpha(c),
$$
but not minimizing the $c'$-action for $c'\in\text{\rm int}\mathbb{E}_i$. As the configuration space is $\mathbb{T}^2$, the Lipschitz graph property of Aubry set will be violated if the rotation vector of the measure $\rho(\mu_c)$ is not parallel to $g\in G_{m,\mathbb{E}_i}$. So, $\langle\rho(\mu_c), c-c'\rangle=0$ holds for  $c'\in\text{\rm int}\mathbb{E}_i$, thus $\mu_c$ also minimizes the action for $c'\in\text{\rm int}\mathbb{E}_i$. This leads to a contradiction. Since $\partial^*\mathbb{F}_0$ is closed, once $\text{\rm int}\mathbb{E}_i\subset\partial^*\mathbb{F}_0$, then whole edge is also contained in $\partial^*\mathbb{F}_0$.

The conclusion 2 follows from the fact that $\langle c-c',[\gamma]\rangle=0$ holds for any $c,c'\in\text{\rm int}\mathbb{E}_i$ and any $\gamma\in\mathcal{A}(c)$, the conclusion 3 follows from that $\mathcal{A}(c)\varsupsetneq\mathcal{A}(c')$ if $c'\in\text{\rm int}\mathbb{E}_i$.

If the conclusion 4 was not true, for the cohomology class in $\mathbb{E}_i\cap \mathbb{E}_j$ the Mather set would contain two closed circles with different homology, but it violates the Lipschitz graph property of Aubry set. With the same reason we have the conclusion 5.
\end{proof}

By this theorem, each edge $\mathbb{E}_i\subset\partial^*\mathbb{F}_0$ aslo uniquely determines a class $g(\mathbb{E}_i)$ so that for each $c\in int \mathbb{E}_i$, the rotation vector of each $c$-minimal measure has the form $\nu g(\mathbb{E}_i)$ ($\nu>0$). For brevity, we also use the notation $\mathcal{M}(\mathbb{E}_i)= \mathcal{M}(c)$ for $c\in \mathbb{E}_i$.

\begin{figure}[htp] 
  \centering
  \includegraphics[width=6.6cm,height=3.2cm]{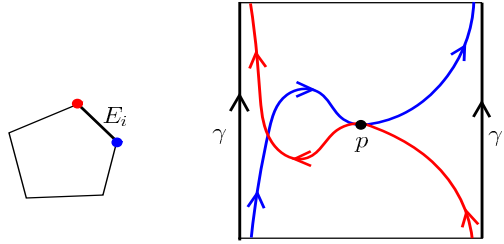}
  \caption{$\mathbb{E}_i\subset\partial^*\mathbb{F}_0$, $\mathcal{M}(\mathbb{E}_i)=\{0\}\cup
  \{\gamma\}$. The blue curve is in $\mathcal{A}(c)$ for $c$ at one end point
  of $\mathbb{E}_i$, the red curve is in $\mathcal{A}(c')$ for $c'$ at another end
  point of $\mathbb{E}_i$.}
  \label{Fig1}
\end{figure}
Given two homology classes $g,g'\in H_1(\mathbb{T}^2,\mathbb{Z})$, we call them adjacent if $g\in G_{m,\mathbb{E}}$, $g'\in G_{m,\mathbb{E}'}$, $\mathbb{E}\cap\partial^*\mathbb{F}_0=\varnothing$, $\mathbb{E}'\cap\partial^*\mathbb{F}_0=\varnothing$, $\mathbb{E}$ and $\mathbb{E}'$ are adjacent. The special topology of two-dimensional torus induces some restrictions on adjacent homologies.
\begin{lem}\label{flatlem3}
Let $\mathbb{E},\mathbb{E}'\subset\partial\mathbb{F}_0\backslash\partial^*\mathbb{F}_0$ be two adjacent edges and assume $c\in \mathbb{E}\cap \mathbb{E}'$. If $(m,n)=g\in G_{m,\mathbb{E}}$ and $(m',n')=g'\in G_{m,\mathbb{E}'}$, then one has that $m'n-mn'=\pm 1$.
\end{lem}
\begin{proof}
The Aubry set $\tilde{\mathcal{A}}(c)$ contains homoclinic orbits with two two classes $(m,n)$ and $(m',n')$, both are irreducible. Guaranteed by the Lipschitz graph property, these curves intersect each other only at the fixed point. In the universal covering space $\mathbb{R}^2$, each curve in the lift of the homoclinic curves are determined by the equation
$$
mx_1+nx_2=k,\qquad m'x_1+n'x_2=k'.
$$
The solution of the equations corresponds to the intersection point which are lattice points in $\mathbb{Z}^2$ for any $(k,k')\in\mathbb{Z}^2$. To guarantee this property, the necessary and sufficient condition is $mn'-m'n=\pm 1$.
\end{proof}

For each indivisible homological class $0\neq g\in H_1(\mathbb{T}^2,\mathbb{Z})$, either $\mathscr{L}_{\beta}(\lambda g)\notin\partial\mathbb{F}_0$ for any $\nu>0$, or some $\lambda_0>0$ exists such that $\mathscr{L}_{\beta}(\lambda_0 g)\in\partial^*\mathbb{F}_0$.

In the first case, $\mathscr{L}_{\beta}(\lambda g)\to\partial\mathbb{F}_0\backslash\partial ^*\mathbb{F}_0$ as $\lambda\downarrow 0$, at least one periodic curve $\gamma_{\lambda}\subset\mathcal{M}(c)$ exists for $c\in\mathscr{L}_{\beta}(\lambda g)$ with $\lambda>0$. It is impossible that $d(c,\mathscr{L}_{\beta}(\lambda g))\to 0$ holds for $c\in\partial^*\mathbb{F}_0$, as in that case certain $c$-minimal measure $\mu_c$ would exist so that $\rho(\mu_c)$ is not parallel to $[\gamma]$. It will violate the Lipschitz property. Generically, $(\gamma_{\lambda},\dot\gamma_{\lambda})$ is hyperbolic and $\mathscr{L}_{\beta} (\lambda g)$ is an interval if $\lambda>0$. If $g\in G_{m,\mathbb{E}_i}$, then $\mathscr{L}_{\beta}(\lambda g)$ approaches to certain edge $\mathbb{E}_i$. If $g=k_ig_i+k_{i+1}g_{i+1}$ with indivisible $(k_i,k_{i+1})\in\mathbb{Z}^2_+$, $g_i\in G_{m,\mathbb{E}_i}$, $g_{i+1}\in G_{m,\mathbb{E}_{i+1}}$, $\mathbb{E}_i$ and $\mathbb{E}_{i+1}$ are two adjacent edges. As $\lambda\downarrow 0$, the interval will shrink to a vertex where $\mathbb{E}_i$ is joined to $\mathbb{E}_{i+1}$, and we have a sequence of closed orbits $\{d\gamma_{\lambda}\} =\{\cup_t(\gamma_{\lambda}(t),\dot\gamma_{\lambda}(t))\}$. Its Kuratowski upper limit set is obviously in the Aubry set for certain $c\in\partial\mathbb{F}_0\backslash\partial^*\mathbb{F}_0 $, thus, consists of minimal homoclinic orbits to the fixed point. As $c$ approaches to the vertex, the Mather set approaches to a set of figure eight:
$$
\mathcal{M}(c)\to \gamma_i\ast\gamma_{i+1},
$$
where $\gamma_{\ell}\subset\mathcal{A}(E_{\ell})$ is a minimal homoclinic orbit such that $[\gamma_{\ell}]=g_{\ell}$ for ${\ell}=i,j$.
\begin{figure}[htp]
  \centering
  \includegraphics[width=7cm,height=3.0cm]{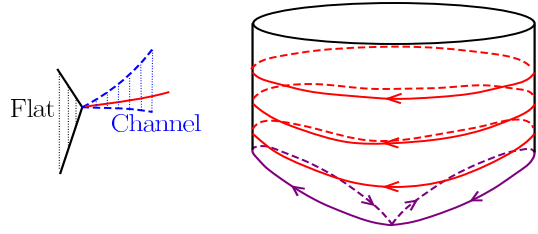}
  \caption{For each $c$ on the red line in the channel, the Aubry set is a closed orbit in the cylinder, it approaches to a curve of figure eight.}
  \label{fig2}
\end{figure}

To be more precise, let us study this phenomenon in the finite covering space $\bar M=\bar k_1\mathbb{T} \times\bar k_2\mathbb{T}$ where $\bar k_m=k_ig_{im}+k_{i+1}g_{i+1,m}$ for $m=1,2$ if we write $g_j=(g_{j1},g_{j2})$ for $j=i,i+1$. Let $\sigma$: $\{1,2,\cdots,k_i+k_{i+1}\}\to \{i,{i+1}\}$ be a permutation such that the cardinality $\#\sigma^{-1}(i)=k_i$ and $\#\sigma^{-1}({i+1})=k_{i+1}$. The lift of homoclinic curve $\gamma_i$ as well as $\gamma_{i+1}$ to $\bar M$ contains several curves, not closed. Pick up one curve $\bar\gamma_{\sigma(1)}$ in the lift of $\gamma_{\sigma(1)}$, it determines a unique curve $\bar\gamma_{\sigma(2)}$ such that the end point of $\bar\gamma_{\sigma(1)}$ is the starting point of $\bar\gamma_{\sigma(2)}$, and so on. See Figure \ref{fig3}.
\begin{figure}[htp]
  \centering
  \includegraphics[width=7.0cm,height=4.2cm]{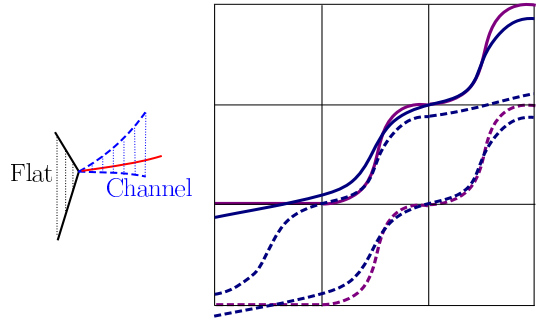}
  \caption{$[\gamma_1]=(1,0)$, $[\gamma_2]=(1,1)$, $k_1=1$, $k_2=2$. For each class $c$ on the red line, $\rho(\mu_c)=\lambda([\gamma_1]+2[\gamma_2])$. The solid blue line represent a periodic curve, the solid purple line represent the conjunction of the minimal homoclinic curves. The dashed lines represent the image of the Deck-transformation.}
  \label{fig3}
\end{figure}
We are going to show below that there exists a unique permutation $\sigma$ such that one Aubry class in $\mathcal{A}(c,\bar M)$
$$
\mathcal{A}_i(c,\bar M)\to \bar\gamma_{\sigma(1)}\ast
\bar\gamma_{\sigma(2)}\ast\cdots\ast\bar\gamma_{\sigma(k_i+k_{i+1})}
$$
as $c$ approaches to the vertex along the path in the channel.
As the minimal curve $\gamma_{\lambda}$ is periodic with the homological class $[\gamma_{\lambda}]=k_ig_i+k_{i+1}g_{i+1}$, the permutation $\sigma$: $\mathbb{Z}\to \{i,i+1\}$ is $(k_i+k_{i+1})$-periodic. Since $k_i$ is prime to $k_{i+1}$, we have $k_i=k_{i+1}=1$ if $k_i=k_{i+1}$.
\begin{lem}\label{flatlem4}
The permutation is uniquely determined by $k_i$ and $k_{i+1}$. If $k_{i}>k_{i+1}$, the following holds for $j=1,\cdots,k_i+k_{i+1}$
\begin{align*}
&\sigma(j+j_0)=i,  \hskip 1.0 true cm \text{\rm if}\ \ (a_j)\ne 0; \\
&\sigma(j+j_0)=i+1,\hskip 0.35 true cm \text{\rm if}\ \ (a_j)=0.
\end{align*}
\end{lem}
\begin{proof}
By the assumption, there exists only one minimal homoclinic curve $\gamma_j$ such that $[\gamma_j]=g_j$ for $j=i,i+1$. Because of the lemma \ref{flatlem3}, we can assume that $g_i=(1,0)$ and $g_{i+1}=(0,1)$ by choosing suitable coordinates on $\mathbb{T}^2$. We choose two sections $I^-$ and $I^+$ in a small neighborhood of the origin such that, emanating from the origin, these homoclinic curves pass through $I^-$ and $I^+$ successively before they return back to the origin as $t\to\infty$. In the section $I^{\pm}$ we choose disjoint subsections $I^{\pm}_i$ and $I^{\pm}_{i+1}$ such that the curve $\gamma_j$ passes through $I^{\pm}_{j}$ for $j=i,i+1$.

Let $\gamma_{\lambda}$ be the minimal periodic curve with rotation vector $\lambda g$. For small $\lambda>0$, $\gamma_{\lambda}$ falls into a small neighborhood of these two homoclinic curves. So it has to pass either through $I^{\pm}_i$ or through $I^{\pm}_{i+1}$. Let $t_{\ell}^{\pm}$ be the time for $\gamma_{\lambda}$ passing through $I^{\pm}$ with $\cdots<t_{{\ell}-1}^-<t_{\ell}^+<t_{\ell}^-<t_{{\ell}+1}^+<\cdots$, and it does not tough these sections whenever $t\ne t_k^{\pm}$. By definition, the period of the curve equals $t_{k_1+k_2}^{\pm}-t_0^{\pm}$.   If the curve intersects $I^+_i$ at $t^+_{\ell}$ and intersects $I^-_{i+1}$ at $t^-_{\ell}$, then the segment $\gamma_{\lambda}|_{[t_{{\ell}-1}^-,t_{\ell}^+]}$ keeps close to $\gamma_i$ and $\gamma_{\lambda}|_{[t_{\ell}^-,t_{\ell+1}^+]}$ keeps close to $\gamma_{i+1}$, so one has $\gamma_{\lambda}(t^-_{\ell-1})\in I_i^-$ and $\gamma_{\lambda}(t^+_{\ell+1})\in I^+_{i+1}$.
\begin{figure}[htp]
  \centering
  \includegraphics[width=7.8cm,height=3.5cm]{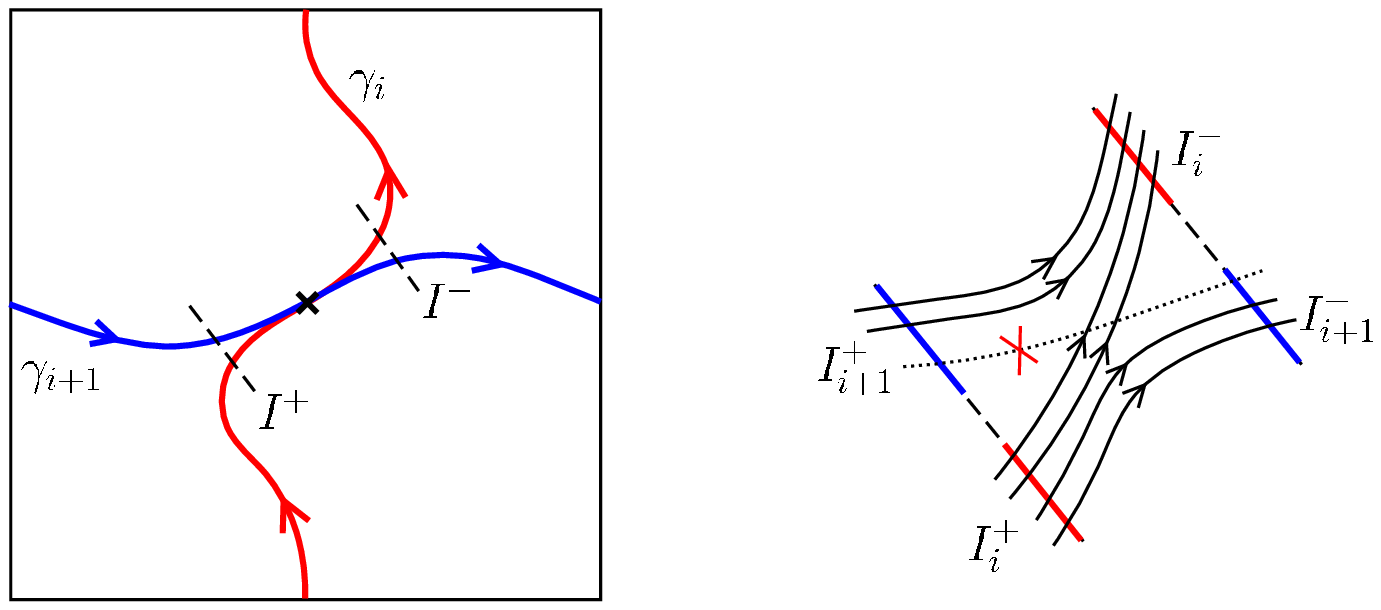}
  \label{fig4}
\end{figure}
As the curve $\gamma_{\lambda}$ is minimal, it does not have self-intersection. Thus, once there exists $t_j^{\pm}$ such that $\gamma_{\lambda}(t^+_j)\in I^+_i$ and $\gamma_{\lambda}(t^-_j)\in I^-_i$, then there does not exist $t_{j'}^{\pm}$ such that $\gamma_{\lambda}(t^+_{j'})\in I^+_{i+1}$ and $\gamma_{\lambda}(t^-_{j'})\in I^-_{i+1}$. Therefore, there is a set $J\subset\{1,2,\cdots,k_i+k_{i+1}\}$ with cardinality $\#(J)=k_i-k_{i+1}$ such that for $j\in J$ one has $\gamma_{\lambda}(t^{\pm}_{j})\in I^{\pm}_{i}$, for $j\notin J$ one either has $\gamma_{\lambda}(t^{+}_{j})\in I^{+}_{i}$ and $\gamma_{\lambda}(t^{-}_{j})\in I^{-}_{i+1}$ or has $\gamma_{\lambda}(t^{+}_{j})\in I^{+}_{i+1}$ and $\gamma_{\lambda}(t^{-}_{j})\in I^{-}_{i}$.

By introducing coordinate transformation on $T$: $\mathbb{T}^2\to\mathbb{T}^2$ such that  $T_*g=g$ $\forall\,g\in H_1(\mathbb{T}^2,\mathbb{Z})$, let us think the curve $T\gamma_{\lambda}$ as a straight line projected down to the unit square, a fundamental domain of $\mathbb{T}^2$.  Starting from a point $z^h_0=(x_0,0)$, the line successively reaches to the points $z^h_1=(x_1,0),\cdots,z^h_m=(x_m,0),\cdots,z^h_{k_i}=z^h_{0}$ where $x_m=(x_0+mk_{i+1}/k_i\mod 1,0)$ with small $x_0>0$. To connect the point $(x_{m-1},0)$ to the point $(x_m,1)$, the curve $T\gamma_{\lambda}$ does not touch the vertical boundary lines if
$$
\Big[(m-1)\frac{k_{i+1}}{k_i}\Big]=\Big[m\frac{k_{i+1}}{k_i}\Big],
$$
where $[a]$ denote the largest integer not bigger than the number $a$, and it has to pass through the vertical lines at some point $z^v_m=(0\mod 1,y_m)$ if
$$
\Big[(m-1)\frac{k_{i+1}}{k_i}\Big]+1=\Big[m\frac{k_{i+1}}{k_i}\Big].
$$
We define an order $\prec$ for these $k_i+k_{i+1}$ points such that $z^h_j\prec z^h_k$ iff $j<k$ and $z^h_j\prec z^v_{j+1}\prec z^h_{j+1}$ iff $[jk_i/k_{j+1}]+1=[(j+1)k_i/k_{j+1}]$.

Returning back to the original coordinates, the curve $\gamma_{\lambda}$ falls into a neighborhood of the curves $\gamma_{i}$ and $\gamma_{i+1}$, intersects the horizontal line $\Gamma_h=T^{-1}\{(x_1,x_2):x_1=\frac12\mod 1\}$ at $T^{-1}z_j^h$ and intersects the vertical line $\Gamma_v=T^{-1}\{(x_1,x_2):x_2=\frac12\mod 1\}$ at $T^{-1}z_j^v$, $[\Gamma_h]=g_{i+1}$ and $[\Gamma_v]=g_i$. Naturally, the map $T$ induces the order among these points: $T^{-1}z^{h,v}_j\prec T^{-1}z^{h,v}_{\ell}$ if and only if $z^{h,v}_j\prec z^{h,v}_{\ell}$. If the curve passes the point $T^{-1}z^h_j$ at $t\in(t^-_j,t^+_{j+1})$, the segment $\gamma_{\lambda}|_{[t^-_j,t^+_{j+1}]}$ falls into a neighborhood of $\gamma_i$, otherwise, it falls into a neighborhood of $\gamma_{i+1}$. In this way, we obtained a unique permutation $\sigma$ up to a translation.
\begin{figure}[htp]
  \centering
  \includegraphics[width=7cm,height=3.5cm]{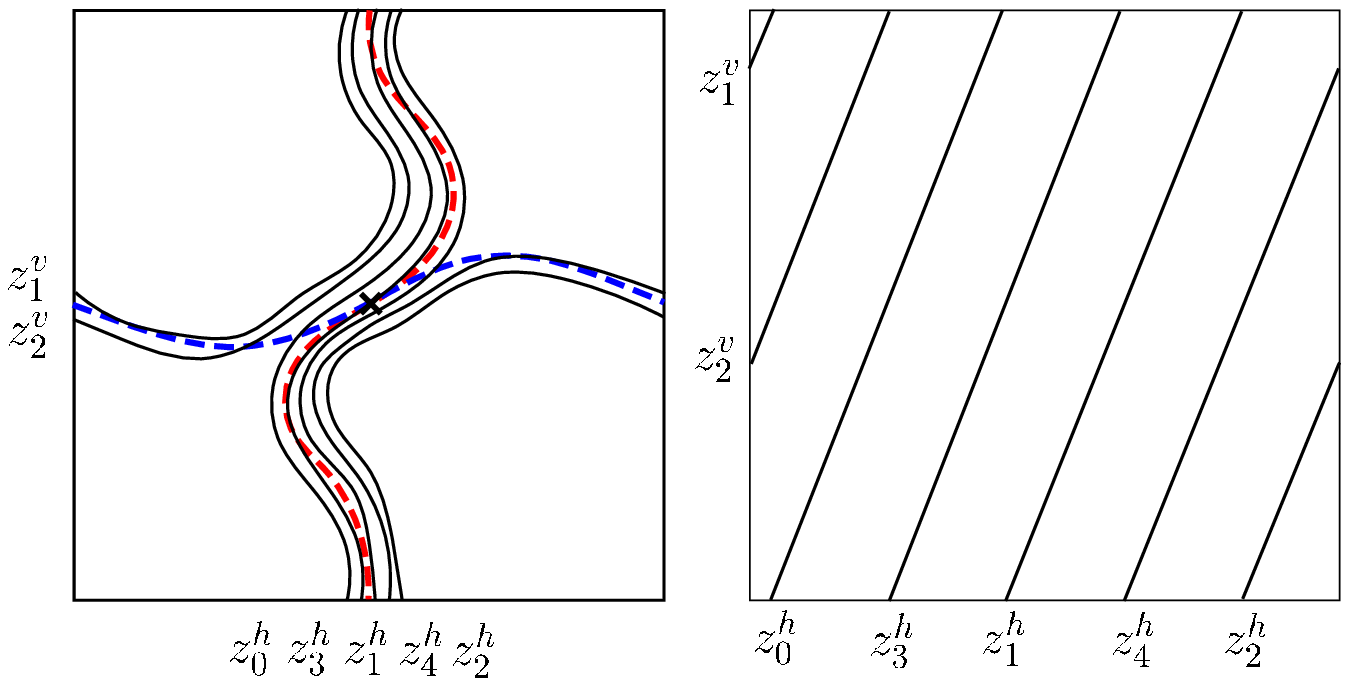}
  \label{fig5}
\end{figure}
\end{proof}

In the second case, some $\lambda_0>0$ exists such that $\mathscr{L}_{\beta}(\lambda g)\in\partial^* \mathbb{F}_0$. It is typical that certain edge $\mathbb{E}_i$ exists such that $\mathscr{L}_{\beta}(\lambda g)=\mathbb{E}_i\subset\partial^*\mathbb{F}_0$, the $c$-minimal measure is supported on a hyperbolic periodic orbit for each $c\in\mathbb{E}_i$. Thus, there is a channel $\mathbb{C}$ endding at $\mathbb{E}_i$ such that the family of periodic orbits $\cup_{c\in\mathbb{C}}\tilde{\mathcal{M}}(c)$ constitutes a normally hyperbolic cylinder.
\begin{figure}[htp]
  \centering
  \includegraphics[width=7cm,height=3.5cm]{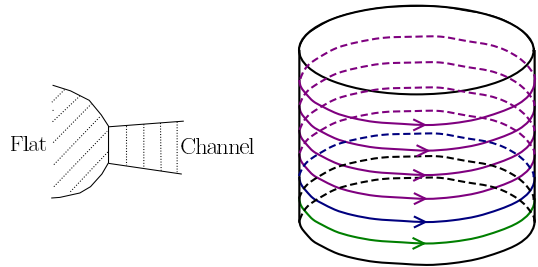}
    \caption{For each $c\in\mathbb{C}$, the Aubry set is a closed orbit in the cylinder (in purple),  the closed circle in blue is in the Aubry set for $c\in\mathbb{F}_0\cap\mathbb{C}$, the closed orbit in green is not global minimum.}
  \label{fig6}
\end{figure}

In the following, we shall study the dynamics on certain energy level of truncated normal form $H(x,y,y_3)=E$. Since one obtains the condition $\partial_{y_3}H\neq 0$ from the normal form, some function $y_3=Y(x,y)$ solves the equation $H=E$. Treat $Y(x,y)$ as the new Hamiltonian and let $\tau=-x_3$ be the time, one obtains a system with two degrees of freedom, which is equivalent to the dynamics on the energy level $\{H^{-1}(E)\}$.

\begin{theo}\label{flatthm5}
For the Hamiltonian $H(x,y,x_n,y_n)$ we assume that $\partial_{y_n}H\neq 0$ on $\{H^{-1}(E)\}\cap\{y_n\in [y_n^-,y_n^+]\}$. Let $y_n=Y(x,y,\tau)$ be the solution of $H=E$ $(\tau=-x_n)$. Let $\alpha_H$ and $\alpha_G$ be the $\alpha$-function for $L_H$ and $L_G$ respectively, where
$$
L_H(x,x_n,\dot x,\dot x_n)=\max_{y,y_n}\langle (\dot x,\dot x_n),(y,y_n)\rangle -H(x,y,x_n,y_n),
$$
$$
L_Y(x,y,\tau)=\max_{y}\langle \dot x,y\rangle -Y(x,y,\tau),
$$
Then for $\alpha_Y(c)\in [y_n^-,y_n^+]$ we have $(c,\alpha_Y(c))\in \alpha_H^{-1}(E)$.
\end{theo}
\begin{proof}
Let $\tilde c=(c,\alpha_Y(c))$, $\tilde\gamma=(\gamma,\gamma_n)$, $\tilde x=(x,x_n)$ and $\tilde y=(y,y_n)$. Let $\gamma$ be $c$-minimal curve for the Lagrange flow $\phi_{L_F}^t$,  $\tilde\gamma$ is then $\tilde c$-minimal curve for the Lagrange flow $\phi_{L_H}^t$ if $\gamma_n=x_n$ and $\tilde\gamma$ is re-parameterized $\tau\to t$. If $x=x(\tau)$ is a solution of $\phi_{L_F}^t$, one obtains $y=y(\tau)$ from the Hamiltonian equations. Since $H(\tilde x(t),\tilde y(t))\equiv E$, we find
\begin{align*}
[A_Y(\gamma)]&=\int\Big(\Big\langle\frac{dx}{d\tau},y-c\Big\rangle-y_n+\alpha_Y(c)\Big)d\tau\\
&=\int(\langle\dot{\tilde x},\tilde y-\tilde c\rangle -H+E)dt\\
&=[A_H(\tilde\gamma)].
\end{align*}
This completes the proof.
\end{proof}
Let $\pi_3: \mathbb{R}^3\to\mathbb{R}^{n-1}$ be the projection $\pi_3\tilde x=x$. By this theorem, $\pi_3^{-1}:H^1(\mathbb{T}^{2},\mathbb{R})\to \alpha_H(E)$ is a homeomorphism for $c\in\mathbb{F}_0+d$, the $d$-neighborhood of the flat $\mathbb{F}_0$. Thus, what we obtained in this subsection have their counterpart in the energy level set $\{H^{-1}(E)\}$ where the class $\tilde c\in\pi_3^{-1}(\mathbb{F}_0+d)\cap\alpha^{-1}_H(E)$.

\section{\ui Normally hyperbolic invariant cylinder}
\setcounter{equation}{0}

In this section, normally hyperbolic invariant cylinder is proved to exist in certain neighborhood of double resonant point. It uses the normal form which is obtained in the appendix where several steps of KAM iteration and one step of linear coordinate transformation were carried out. All these coordinate transformations are symplectic. Since Aubry set and Ma\~n\'e set are symplectic invariants \cite{Be2}, it is good enough to study these objects by considering the normal form.

\subsection{Homogenized Hamiltonian}
The normal form of the Hamiltonian takes the form
\begin{equation}\label{homogenizedeq1}
H=\tilde h(\tilde y)+\epsilon\tilde Z(x,\tilde y)+\epsilon\tilde R(\tilde x,\tilde y),
\end{equation}
where $\tilde x=(x,x_3)=(x_1,x_2,x_3)$ and $\tilde y=(y,y_3)=(y_1,y_2,y_3)$. Since $h$ is positive definite, a unique curve exists along which $\partial_{y}h=(0,0)$. This curve passes through the energy level $\{h^{-1}(\tilde E)\}$ transversally at a unique point $\tilde y_0$. As $\tilde E>\min\alpha$, one has $\partial_{y_3}h(\tilde y_0)=\omega_3\ne 0$. By coordinate translation, we assume $\tilde y_0=0$.

As the dynamics we are going to study is restricted on an energy level $\{H^{-1}(\tilde E)\}$, it can be reduced to a system with two and half degrees of freedom. Since $\omega_3\ne 0$, the equation $H(\tilde x,\tilde y)=\tilde E$ uniquely determines a smooth function $y_3=y_3(x,y,x_3)$ in certain neighborhood of $y_3=0$. Treating $-\omega_3 y_3=Y$ as a new Hamiltonian and $\omega_3^{-1}x_3=\tau$ as a new time variable, we obtain a time-periodic system with two degrees of freedom. Correspondingly, we have the normal form
\begin{equation}\label{homogenizedeq2}
Y(x,y,\tau)=h(y)+\epsilon Z(x,y)+\epsilon R(x,y,\tau)
\end{equation}
where $h(0)=0$, $\partial_yh(0)=0$ and $\epsilon R$ is as small as $\epsilon\tilde R$.

Each $\text{\bf k}\in\mathbb{Z}^2$ determines a resonant curve $\Gamma_{\text{\bf k}}=\{y\in\mathbb{R}^2:\langle\text{\bf k},\partial_yh(y)\rangle=0\}$. Normally hyperbolic cylinder is searched when $y$ varies along this curve. Recall that the normal form remains valid in the domain $\{\|y\|\le O(\epsilon^{\kappa})\}$ ($\frac 16<\kappa\le\frac 13$). The $\sqrt{\epsilon}$-neighborhood of the curve is covered by as many as $O(\epsilon^{-\kappa+\frac 12})$ small balls with radius $O(\sqrt{\epsilon})$. Given a ball centered at $y=y_j\in\Gamma_{\text{\bf k}}$, by rescaling variables $y-y_j=\sqrt{\epsilon}p$, $s=\sqrt{\epsilon}\tau$, one obtains an equivalent Hamiltonian equation
\begin{align*}
\frac{dx}{ds}&=\frac {\omega_j}{\sqrt{\epsilon}}+A_jp+\Big(\frac{\partial h(\sqrt{\epsilon}p)} {\epsilon\partial p}-\frac {\omega_j}{\sqrt{\epsilon}}-A_jp\Big) +\frac{\partial Z}{\partial p}
+\frac{\partial R}{\partial p},\\
\frac{dp}{ds}&=-\frac{\partial Z}{\partial x}-\frac{\partial R}{\partial x}
\end{align*}
which corresponds to the Hamiltonian
$$
G_{\epsilon}=\frac 1{\sqrt{\epsilon}}\langle\omega_j,p\rangle+\frac 12\langle A_jp,p\rangle+V_j(x)+ Z_{\epsilon}(x,\sqrt{\epsilon}p)+R_{\epsilon}(x,\sqrt{\epsilon}p,s/\sqrt{\epsilon})
$$
where $\omega_j=\partial h(y_j)$, $A_j=\partial^2 h(y_j)$, $V_j(x)=Z(x,y_j)$ and
$$
Z_{\epsilon}=\frac{1}{\epsilon}h(y_j+\sqrt{\epsilon}p)-\frac 1{\sqrt{\epsilon}}\langle\omega_j,p\rangle-\langle A_jp,p\rangle+ Z(x,\sqrt{\epsilon}p+y_j)-Z(x,y_j).
$$
According to Appendix A, one has $Z_{\epsilon}=O(\sqrt{\epsilon})$ and $\|R_{\epsilon}\|_{C^2}=O(\epsilon^{3\sigma-2\rho})$, where the $C^2$-norm is with respect to $(x,p)$ only. To guarantee the covering property shown in the appendix, one choose $\sigma=\frac 17$, so we have $3\sigma-2\rho=\frac 1{21}>0$. Obviously, one has
\begin{pro}\label{homogenizationpro}
Each orbit $(x(s),p(s))$ of the Hamiltonian flow $\Phi^s_{G_{\epsilon}}$ uniquely determines an orbit $(x(\tau),y(\tau))=(x(s/\sqrt{\epsilon}),y_j+\sqrt{\epsilon}p(s/\sqrt{\epsilon}))$ of $\Phi^{\tau}_{G}$. If $G_{\epsilon}(x(s),p(s))=E_{\epsilon}$ and $G(x(\tau),y(\tau))=E$, then $E=\epsilon E_{\epsilon}$.
\end{pro}

The Hamiltonian $G_{\epsilon}$ is a small perturbation of the homogenized Hamiltonian
$$
\bar G=\frac 1{\sqrt{\epsilon}}\langle\omega_j,p\rangle+\frac 12\langle A_jp,p\rangle+V_j(x),
$$
which determines the Lagrangian
$$
\bar L=\frac 12\Big\langle A^{-1}_j\Big(\dot x-\frac{\omega_j}{\sqrt{\epsilon}}\Big),\dot x-\frac{\omega_j}{\sqrt{\epsilon}}\Big\rangle-V_j(x).
$$
Let us first consider the case when $y_j=0$, it follows that $\omega_j=0$. The variable $y$ is restricted in the domain $\|y\|\le K\sqrt{\epsilon}$, where the higher order term is bounded by $\|R_{\epsilon}\|_{C^2}=O(\epsilon^{5\sigma-\frac 16})$ (see (\ref{normaleq1})). Correspondingly, let $A=A_j$ and $V=V_j$ for $y_j=0$:
$$
\bar G=\frac 12\langle Ap,p\rangle+V(x), \qquad \bar L=\frac 12\langle A^{-1}\dot x,\dot x\rangle-V(x).
$$
For this Hamiltonian system, the maximal point of $V$ determines a stationary solution which corresponds to a minimal measure of $\bar L$, where the matrix
$$
\left (\begin{matrix}0 & A\\
-\partial^2_xV & 0
\end{matrix}\right )
$$
has 4 real eigenvalue $\pm\lambda_1,\pm\lambda_2$.  By translation of coordinates, it is generic that

({\bf H1}): {\it $V$ attains its maximum at $x=0$ only, the Hessian matrix of $V$ at $x=0$ is negative definite. All eigenvalues are different: $-\lambda_2<-\lambda_1<0<\lambda_1<\lambda_2$}.

Such a hypothesis leads to certain hyperbolicity of minimal homoclinic orbits. Let us consider the case: for $c\in\partial\mathbb{F}_0$ the Aubry set $\mathcal{A}(c)=\cup_{t\in\mathbb{R}}\zeta(t)$, where $\zeta$: $\mathbb{R}\to M$ is a minimal homoclinic curve. By the assumption {\bf H1}, the fixed point $z=(x,y)=0$ has its locally stable manifold $W^+$ as well as the locally unstable manifold $W^-$. They
intersect each other transversally at the origin. As each homoclinic orbit entirely stays in the stable as well as in the unstable manifolds, along such orbit their intersection can not be transversal in the standard definition, but transversal module the curve:
$$
T_xW^-\oplus T_xW^+=T_xH^{-1}(E)
$$
holds for $x$ is on minimal homoclinic curves. Without danger of confusion, we also call the intersection transversal.

If we denote by $\Lambda^+_i=(\Lambda_{xi},\Lambda_{yi})$ the eigenvector corresponding to the eigenvalue $\lambda_i$, where $\Lambda_{xi}$ and $\Lambda_{yi}$ are for the $x$- and $y$-coordinate respectively, then the eigenvector for $-\lambda_i$ will be  $\Lambda^-_i=(\Lambda_{xi},-\Lambda_{yi})$. it is also a generic condition that

({\bf H2}): {\it with each $g\in H_1(\mathbb{T}^2,\mathbb{Z})$, there is at most one minimal orbit associated, the stable manifold intersects the unstable manifold transversally along each minimal homoclinic orbit. Each minimal homoclinic orbit approaches to the fixed point along the direction $\Lambda_1$: $\dot\gamma(t)/\|\dot\gamma(t)\| \to\Lambda_{x1}$ as $t\to\pm\infty$.}

Recall the set $\partial^*\mathbb{F}_0$. If $c\in\partial^*\mathbb{F}_0\subset\partial\mathbb{F}_0$, the $c$-minimal measure consists of two or more ergodic components. Because of Theorem \ref{flatthm3} and the countability of homology classes of all homoclinic curves, we have another generic condition

({\bf H3}): {\it For each $c\in\partial^*\mathbb{F}_0$, the Aubry set  does not contain minimal curve homoclinic to the origin $($fixed point$)$.}

\subsection{Cylinder for truncated Hamiltonian: near double resonance}

Let us start with a Hamiltonian with two and half degrees of freedom:
\begin{equation}\label{cylindereq3}
G_{\epsilon}=\frac 12\langle Ap,p\rangle+V(x)+Z_{\epsilon}(x,p)+R_{\epsilon}(x,\sqrt{\epsilon}p,s/\sqrt{\epsilon})
\end{equation}
where $(x,\sqrt{\epsilon}p,s/\sqrt{\epsilon})\in\mathbb{T}^2\times\mathbb{R}^2\times\mathbb{T}$, $Z_{\epsilon}=O(\sqrt{\epsilon})$, $\|R_{\epsilon}\|_{C^2}=O(\epsilon^{5\sigma-\frac 16})$ (see Theorem \ref{normalthm1}) where the $C^2$-norm is with respect to $(x,p)$. Recall the homogenized Hamiltonian as well as the homogenized Lagrangian
$$
\bar G(x,p)=\frac 12\langle Ap,p\rangle+V(x),\qquad \bar L(x,\dot x)=\frac 12\langle A^{-1}\dot x,\dot x\rangle-V(x)
$$
where $\dot x=\frac {dx}{ds}$.

By the assumption ({\bf H1}), the fixed point $(x,\dot x)=0=\tilde{\mathcal{M}}(c)$ each $c\in\mathbb{F}_0$ which is a 2-dimensional flat. As classified before, for each $c\in\partial\mathbb{F}_0\backslash \partial^*\mathbb{F}_0$ the Aubry set consists of minimal homoclinic orbits plus the fixed point.

Given an irreducible class $g\in H_1(\mathbb{T}^2,\mathbb{Z})$, it is not necessary that some minimal homoclinic curve $\gamma$ exists such that $[\gamma]=g$. Let us consider the $c$-minimal measure for $c\in\mathscr{L}_{\beta}(\nu g)$ with $\nu>0$. As we shall see later, it is generic that the minimal measure is supported on at most two periodic orbits, denoted by $d\gamma_{\nu}$, both are hyperbolic, namely, it has the stable and unstable manifold. These periodic orbits constitute a two-dimensional cylinder. However, it appears not reasonable to assume the normal hyperbolicity for $\Phi_{\bar G}^s$ in usual sense as the speed along the orbit may undergo large variation, especially, when it is very close to some homoclinic orbit. Therefore, one can not see the separation of the spectrum of $D\Phi_{\bar G}^s$  in normal and in tangent direction.  As the first step, let us study the case when $c\to\partial\mathbb{F}_0$.

There are two cases alternatively as $\nu$ decreases: $\mathscr{L}_{\beta}(\nu g)\to\partial \mathbb{F}_0$ as $\nu$ decrease to zero, or $\exists$ $\nu_0>0$ such that $\mathscr{L}_{\beta} (\nu_0g)\in\partial\mathbb{F}_0$.

In the first case, the cohomology class approaches to some edge $\mathbb{E}_i\subset\partial \mathbb{F}_0\backslash\partial^* \mathbb{F}_0$ or to some vertex where two adjacent edges $\mathbb{E}_i,\mathbb{E}_{i+1}\subset\partial\mathbb{F}_0\backslash \partial^*\mathbb{F}_0$ joint together. Under the hypothesis ({\bf H2}), for class $c$ in the interior of the edge, the Aubry set $\mathcal{A}(c)$ contains exactly one minimal homoclinic curve and the point of the origin. Let $\gamma_j$ be the minimal homoclinic curve related to $\mathbb{E}_j$ for $j=i,i+1$, it determines the minimal homoclinic orbit $(x_i(s),p_i(s))\subset\bar G^{-1}(0)$. Denote by $g_j=[\gamma_j]$ the homology class. If $g=g_i$, $\gamma_{\nu}\to\gamma_i$ as $\nu$ decreases to zero.  If there exist two positive integers $k_i,k_{i+1}$ such that $g=k_ig_i+k_{i+1}g_{i+1}$, the curve $x_{\nu}$ approaches to the set $\cup_{t\in\mathbb{R}}x_i(t)\cup x_{i+1}(t)$ (figure eight) as $\nu\to 0$, folding $k_i$ and $k_{i+1}$-times along $x_i$ and $x_{i+1}$ respectively.

The periodic curve $\gamma_{\nu}$ determines a periodic orbit $(\gamma_{\nu},y_{\nu})$ in the phase space. It stays in certain energy level set $H^{-1}(E)$. For $g=k_ig_i+k_{i+1}g_{i+1}$ and suitably small $E>0$, by the study in Section 3.2  (Eq.(\ref{regularenergyeq2})), the period $T$ is related to the energy by the formula
$$
T=T(E,g)=\tau_{E,g}-\frac 1{\lambda_1}(k_i+k_{i+1})\ln E
$$
where $\tau_{E,g}\to k_i\tau_{E,g_i}+k_{i+1}\tau_{E,g_{i+1}}$ as $E\to 0$, both $\tau_{E,g_i}$ and $\tau_{E,g_{i+1}}$ is bounded.

To study the dynamics around the minimal homoclinic orbits $z_{\ell}=(x_{\ell},p_{\ell})$  ($\ell=i,i+1$), we use a new canonical coordinates $(x,p)$ such that, restricted in a small neighborhood of $z=0$, one has the form
$$
\bar G=\frac 12(p_1^2-\lambda_1^2x_1^2)+\frac 12(p_2^2-\lambda_2^2x_2^2)+P_3(x)
$$
where $P_3(x)=O(\|x\|^3)$. Without losing generality, we assume $x_{\ell,1}(s)\downarrow 0$ as $s\to -\infty$, $x_{\ell,1}(s)\uparrow 0$ as $s\to \infty$ and $\dot x_{\ell}(s)/ \|\dot x_{\ell}(s)\|\to(1,0)$ as $s\to\pm\infty$. Here the notation is taken as granted: $x_{\ell}=(x_{\ell,1},x_{\ell,2})$. We choose 2-dimensional disk lying in $\bar G^{-1}(E)$
$$
\Sigma^{\mp}_{E,\delta}=\{(x,p)\in\mathbb{R}^4:\|(x,p)\|\le d,\bar G(x,p)=E, x_1=\pm\delta\}.
$$
Because of the special form of $\bar G$, one has
$$
\Sigma^{\mp}_{0,\delta}=\{x_1=\pm\delta, p_1^2+p_2^2-\lambda_2^2x_2^2 =\lambda_1^2\delta^2 -2P_3(\pm\delta,x_2),\|(x,p)\|\le d\}.
$$
Let $W^-$ ($W^+$) denote the unstable (stable) manifold of the fixed point which entirely stays in the energy level set $\bar G^{-1}(0)$. If $P_3=0$, the tangent vector of $W^-\cap\Sigma^-_{0,\delta}$ has the form $(0,\pm 1,0,\pm\lambda_2)$. So, the tangent vector of $W^-\cap\Sigma^-_{0,\delta}$ takes the form
\begin{equation*}\label{tangentvector}
v_{\delta}^-=(v_{x_1},v_{x_2},v_{p_1},v_{p_2})=(0,\pm 1,p_{1,\delta},\pm\lambda_2+p_{2,\delta})\in T_{z^-_{\delta}}(W^-\cap\Sigma^-_{0,\delta})
\end{equation*}
where both $p_{1,\delta}$ and $p_{2,\delta}$ are small.

Denote by $T^{\pm}_{\delta,\ell}$ the time when the homoclinic orbit $z_\ell(s)$ passes through $\Sigma^{\pm}_{0,\delta}$. As $\partial_{y_1}\bar G>0$ holds at the point $z_{\ell}\cap\{x_1=\pm\delta\}$, both homoclinic orbits $z_i(s)$ and $z_{i+1}(s)$ approach in the same direction to the fixed point, the section $\Sigma^+_{0,\delta}$ as well as $\Sigma^-_{0,\delta}$ intersects these two homoclinic orbits transversally. Let $z^{\pm}_{\delta,\ell}$ denote the intersection point of $z_\ell(s)$ with $\Sigma^{\pm}_{0,\delta}$. In a small neighborhood of that point $B_{\varepsilon}(z^-_{\delta,\ell})$, one obtains a map $\Psi_{0,\delta}$: $\Sigma^-_{0,\delta}\cap B_{\varepsilon}(z^-_{\delta,\ell})\to \Sigma^+_{0,\delta}$  in following way, starting from a point $z$ in this neighborhood, there is a unique orbit which moves along $z_{\ell}(s)$ and comes to a point $\Psi_{0,\delta}(z)\in \Sigma^+_{0,\delta}$ after a time approximately equal to $T_{\delta,\ell}^+-T_{\delta,\ell}^-$.

Let us fix small $D>0$. There exists $C_0>1$ (depending on $D$) such that
$$
C_0^{-1}\le \|D\Psi_{0,D}(z^-_{D,\ell})|_{T(W^-\cap\Sigma^-_{0,D})}\|, \|D\Psi_{0,D}^{-1}(z^+_{D,\ell})|_{T(W^+\cap\Sigma^+_{0,D})}\|\le C_0
$$
holds for both $\ell=i$ and $\ell=i+1$. Clearly, one has $C_0\to\infty$ as $D\to 0$.

As the homoclinic curves approach to the origin in the direction of $(1,0)$ in $x$-space, for small $\delta\ll D$, there exists a constant $\mu_1>0$ such that $\mu_1\downarrow 0$ as $D\to 0$ and
$$
\frac {1}{\lambda_1+\mu_1}\ln\Big(\frac{D}{\delta}\Big)\le T^{-}_{D,\ell}-T^{-}_{\delta,\ell}, T^+_{\delta, \ell}-T^+_{D,\ell}\le \frac {1}{\lambda_1-\mu_1}\ln\Big(\frac{D}{\delta}\Big).
$$
The Hamiltonian flow $\Phi^t_{\bar G}$ defines a map $\Psi^-_{0,\delta,D}$: $\Sigma^-_{0,\delta}\to\Sigma^-_{0,D}$ and a map $\Psi^+_{0,\delta,D}$: $\Sigma^+_{0,D}\to\Sigma^+_{0,\delta}$: emanating from a point in $\Sigma^-_{0,\delta}$ ($\Sigma^+_{0,D}$) there exists a unique orbit which arrives $\Sigma^-_{0,D}$ ($\Sigma^+_{0,\delta}$) after a time bounded by the last formula.

Restricted in the ball $B_D$, let us consider the variational equation of the flow $\Psi_{\bar G}^s$ along the homoclinic orbit $z_j(s)$. It follows from the normal form of the homogenized Hamiltonian $\bar G$ that the tangent vector $(\Delta x,\Delta p)=(\Delta x_1,\Delta x_2,\Delta p_1,\Delta p_2)$ satisfies the variational equation
\begin{equation}\label{cylindereq4}
\Delta\dot  x_i=\Delta p_i, \qquad \Delta\dot  p_i=\lambda_i^2\Delta x_i-\Psi_{1i}(x_{\ell}(s))\Delta x_1-\Psi_{2i}(x_{\ell}(s)) \Delta x_2, \ \ \ i=1,2
\end{equation}
where $\Psi_{ij}=\partial_{x_i}\partial_{x_j}P_3$. Clearly, $|\Psi_{ij}(x_{\ell}(s))|\le C_1\|x_{\ell}(s)\|$ with $C_1>0$ if $\|x_{\ell}(s)\|$ is small. Since the homoclinic orbit approaches to the fixed point in the direction of $(\dot x,\dot p)=(1,0,\lambda_1^2,0)$, one has
$$
De^{-(\lambda_1+\mu_1)(s-T^+_{D,\ell})}\le\|x(s)|_{[T^+_{D,\ell},\infty)}\|\le De^{-(\lambda_1-\mu_1)(s-T^+_{D,\ell})}.
$$
For the initial value $\Delta z(T^+_{D,\ell})=(\Delta x(T^+_{D,\ell}),\Delta p(T^+_{D,\ell}))$ satisfying the condition
$$
|\langle\Delta z(T^+_{D,\ell}),v^-_{\delta}\rangle|\ge 2/3\|\Delta z(T^+_{D,\ell})\|\|v^-_{\delta}\|
$$
$(v^-_{\delta}=(0,\pm 1,p_{1,\delta},\pm\lambda_2+p_{2,\delta}))$ one obtains from the hyperbolicity that
$$
C_2^{-1}\|\Delta z(T^+_{D,\ell})\|e^{(\lambda_2-\mu_1)(T^+_{\delta,\ell}-T^+_{D,\ell})}\le\|\Delta z(T^+_{\delta,\ell})\|\le C_2\|\Delta z(T^+_{D,\ell})\|e^{(\lambda_2+\mu_1)(T^+_{\delta,\ell}-T^+_{D,\ell})}
$$
holds for some constant $C_2>1$ depending on $\lambda_i$ as well as on $P$.
Therefore, for each vector $v\in T_{z_D^+}\Sigma^+_{0,D}$ which is nearly parallel to $T_{z_D^+}(W^-\cap\Sigma^+_{0,D})$: $|\langle v,v'\rangle|\ge\frac 23\|v\|\|v'\|$ holds for $v'\in T_{z_D^+}(W^-\cap\Sigma^+_{0,D})$ we obtain from the last two formulae that
$$
C_2^{-1}\Big(\frac D{\delta}\Big)^{\frac{\lambda_2}{\lambda_1}-\mu_2}\le
\lim_{\|v\|\to 0}\frac{\|D\Psi^+_{0,\delta,D}(z^+_{D,\ell})v\|}{\|v\|}\le C_2\Big(\frac D{\delta}\Big)^{\frac{\lambda_2}{\lambda_1}+\mu_2}.
$$
Similarly, one has
$$
C_3^{-1}\Big(\frac D{\delta}\Big)^{\frac{\lambda_2}{\lambda_1}-\mu_2}\le \|D\Psi^-_{0,\delta,D}(z^-_{\delta,\ell})|_{T_{z^-_{\delta}}(W^-\cap\Sigma^-_{0,\delta})}\|\le C_3\Big(\frac D{\delta}\Big)^{\frac{\lambda_2}{\lambda_1}+\mu_2},
$$
where $C_3>1$ also depends on $\lambda_i$ as well as on $P$, $\mu_2>0$ and $\mu_2\to 0$ as $D\to 0$.

By the construction, the 2-dimensional disk $\Sigma^{-}_{0,\delta}$ intersects the unstable manifold $W^-$ along a curve. Let $\Gamma^{-}_{\delta,\ell}\subset W^{-}\cap\Sigma^{-}_{0,\delta}$ be a very short segment of the curve, passing through the point $z^{-}_{\delta,\ell}$. Pick up a point $z^*_\ell$ on the homoclinic orbit $z_\ell$ far away from the fixed point and take a 2-dimensional disk $\Sigma^*_\ell\subset\bar G^{-1}(0)$ containing the point $z^*_\ell$ and transversal to the flow $\Phi^s_{\bar G}$ in the sense that $T_{z^*_\ell}\bar G^{-1}(0)=\text{\rm span}(T_{z^*_\ell} \Sigma_\ell,J\nabla\bar G(z^*_\ell))$. The Hamiltonian flow $\Phi^s_{\bar G}$ sends each point of $\Gamma^{-}_{\delta,\ell}$ to this disk provided it is close to $z^-_\ell$. In this way, one obtains a map $\Psi^{-,*}_{\delta,\ell}$: $\Sigma^-_{0,\delta}\to\Sigma^*_\ell$. Let $\Gamma^{-,*}_{\delta,\ell}= \Psi^{-,*} _{\delta,\ell} \Gamma^{-}_{\delta,\ell}$. According to the assumption ({\bf H2}), one has $T_{z^*_{\ell}}\bar G^{-1}(0)=\text{\rm span}(T_{z^*_{\ell}}W^+,T_{z^*_{\delta,\ell}}W^-)$.  Thus, one also has $T_{z^*_{\ell}}\bar G^{-1}(0)=\text{\rm span}(T_{z^*_{\ell}}W^+,T_{z^*_{\ell}} \Gamma^{-,*}_{\ell})$. It follows from the $\lambda$-lemma that $\Psi_{0,\delta}(\Gamma^-_{\delta,\ell})$ keeps $C^1$-close to $W^-\cap\Sigma^+_{0,\delta}$ at the point $z^+_{\delta,\ell}$ and $\Psi^{-1}_{0,\delta}(\Gamma^+_{\delta,\ell})$ keeps $C^1$-close to $W^+\cap\Sigma^-_{0,\delta}$ at the point $z^-_{\delta,\ell}$ provided $\delta>0$ is sufficiently small. As $\Psi_{0,\delta}=\Psi^-_{0,\delta,D}\circ\Psi_{0,D}\circ\Psi^+_{0,\delta,D}$, one obtains
$$
C_4^{-1}\left(\frac D{\delta}\right)^{2(\frac {\lambda_2}{\lambda_1}-\mu_2)}\le \|D\Psi_{0,\delta}(z^-_{\delta})|_{T_{z^-_{\delta}}(W^-\cap\Sigma^-_{0,\delta})}\|
\le C_4\left(\frac D{\delta}\right)^{2(\frac {\lambda_2}{\lambda_1}+\mu_2)},
$$
and
$$
C_4^{-1}\left(\frac D{\delta}\right)^{2(\frac {\lambda_2}{\lambda_1}-\mu_2)}\le \|D\Psi_{0,\delta}^{-1} (z^+_{\delta}) |_{T_{z^+_{\delta}}(W^+\cap\Sigma^+_{0,\delta})}\|
\le C_4\left(\frac D{\delta}\right) ^{2(\frac {\lambda_2}{\lambda_1}+\mu_2)},
$$
where $C_4=C_0C_2C_3>1$. See the figure below.

\begin{figure}[htp]
  \centering
  \includegraphics[width=6.5cm,height=6.7cm]{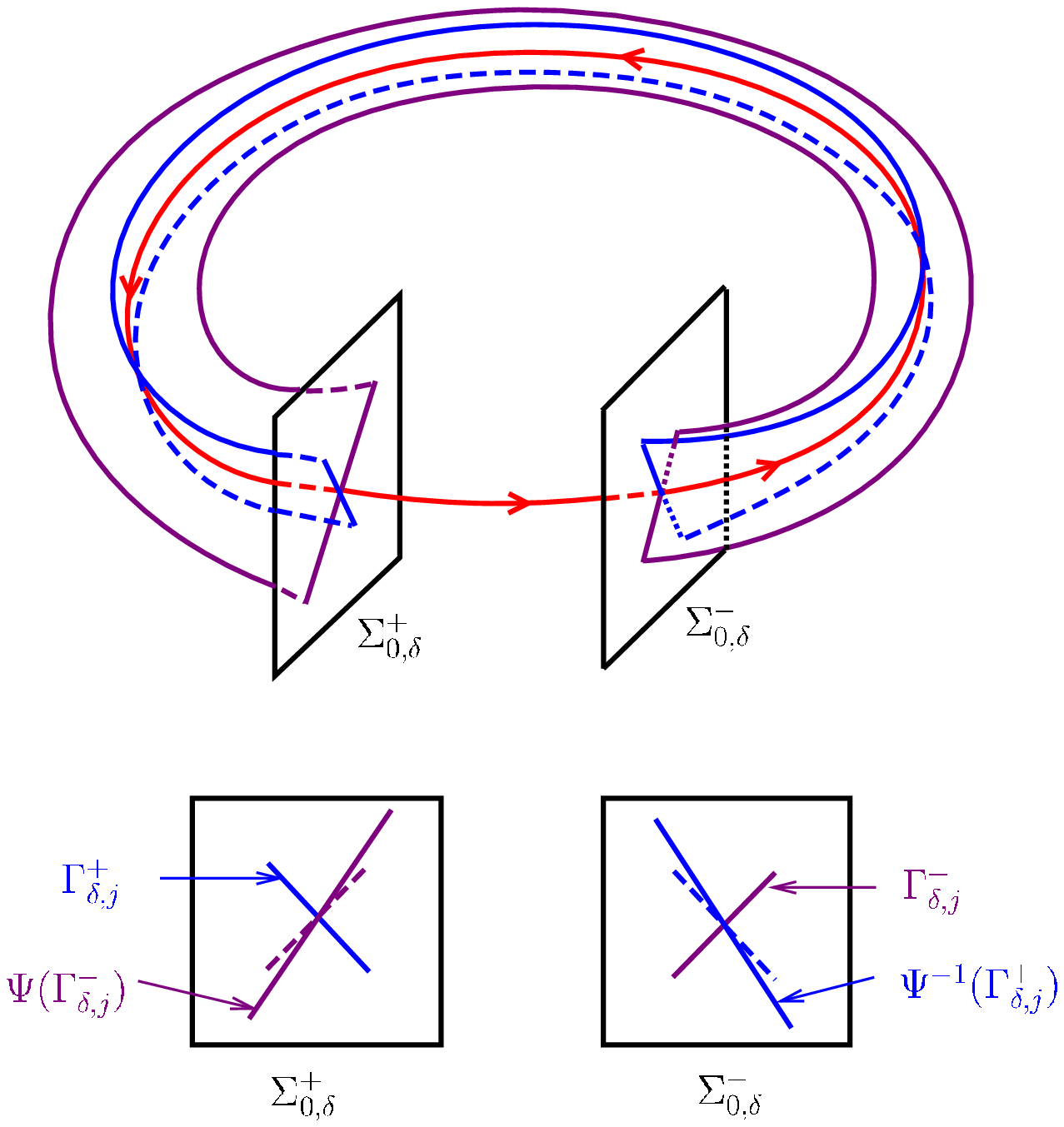}
  \caption{}
  \label{fig7}
\end{figure}

Recall the definition, $\Sigma^{\pm}_{E,\delta}$ is a two-dimensional disk lying in the energy level set $\bar G^{-1}(E)$. For $E>0$ sufficiently small, $\Sigma^{\pm}_{E,\delta}$ is $C^{r-1}$-close to $\Sigma^{\pm}_{0,\delta}$ respectively. Let $z_{E}(s)=(x_{E}(s),p_{E}(s))$ be the minimal periodic orbit staying in the energy level set $\bar G^{-1}(E)$, it approaches to the homoclinic orbit as $E$ decreases to zero. Thus, for sufficiently small $E>0$, it passes through the section $\Sigma^{-}_{E,\delta}$ as well as $\Sigma^{+}_{E,\delta}$ $k_1+k_2$ times for one period. We number these points as $z^{\pm}_{E,k}$ ($k=1,2,\cdots k_1+k_2$) by the role that emanating from a point $z^-_{E,k}$, the orbit reaches to the point $z^+_{E,k+1}$ after time $\Delta t^-_{E,k}$, then to the point $z^-_{E,k+1}$ and so on. Note that $\Delta t^-_{E,k}$ remains bounded uniformly for any $E>0$.  Restricted on small neighborhoods of these points, denoted by $B_d(\bar z^{\pm}_{E,k})$, the flow $\Phi_{\bar G}^t$ defines a local diffeomorphism $\Psi_{E,\delta}$: $\Sigma^{-}_{E,\delta}\supset B_d(\bar z^{-}_{E,k})\to \Sigma^{+}_{E,\delta}$. Because of the smooth dependence of ODE solutions on initial data, a small $\varepsilon>0$ exists such that, for the vector $v^{\pm}$ $\varepsilon$-parallel to $T_{z^{\pm}_{\delta}}(W^{\pm}\cap\Sigma^{\pm}_{0,\delta})$ in the sense that $|\langle v^{\pm},v^{\pm}_0\rangle|\ge (1-\varepsilon) \|v^{\pm}\|\|v^{\pm}_0\|$ holds for some $v^{\pm}_0\in T_{z^{\pm}_{\delta}}(W^{\pm}\cap\Sigma^{\pm}_{0,\delta})$, we obtain from the hyperbolicity of $\Psi_{0,\delta}$ (see the formulae above Figure \ref{fig7}) that
$$
C_5^{-1}\left(\frac D{\delta}\right)^{2(\frac {\lambda_2}{\lambda_1}-\mu_3)}\le \frac {\|D\Psi_{E,\delta}(z^-_{E,k})v^-\|}{\|v^-\|}
\le C_5\left(\frac D{\delta}\right)^{2(\frac {\lambda_2}{\lambda_1}+\mu_3)},
$$
and
$$
C_5^{-1}\left(\frac D{\delta}\right)^{2(\frac {\lambda_2}{\lambda_1}-\mu_3)}\le \frac {\|D\Psi_{E,\delta}^{-1}(z^+_{E,k})v^+\|}{\|v^+\|}
\le C_5\left(\frac D{\delta}\right) ^{2(\frac {\lambda_2}{\lambda_1}+\mu_3)}
$$
where $C_5\ge C_4>1$, $0<\mu_3\to 0$ as $D\to 0$. If the vector $v^-$ is chosen  $\varepsilon$-parallel to $T_{z^{-}_{\delta}}(W^{-}\cap\Sigma^{-}_{0,\delta})$ then the vector $D\Psi_{E,\delta}(z^-_{E,k})v^-$ is $\varepsilon$-parallel to $T_{z^{+}_{\delta}}(W^{+}\cap\Sigma^{+}_{0,\delta})$.

For $E>0$, the Hamiltonian flow $\Phi_{\bar G}^t$ defines local diffeomorphism $\Psi^+_{E,\delta,\delta}$: $\Sigma^{+}_{E,\delta}\supset B_d(\bar z^{+}_{E,k})\to\Sigma^{-}_{E,\delta}$. To make $\Psi^+_{E,\delta,\delta}(B_d(\bar z^{+}_{E,k}))\subset\Sigma^{-}_{E,\delta}$ one has $d\to 0$ as $E\to 0$. According to the study in Section 3.2 (cf. formula (\ref{regularenergyeq2})), starting from $\Sigma^{+}_{E,\delta}$, the periodic orbit comes to $\Sigma^{-}_{E,\delta}$ after a time approximately equal to
$$
T=\frac 1{\lambda_1}\Big|\ln\Big(\frac{\delta^2}{E}\Big)\Big|+\tau_{\delta}
$$
in which $\tau_{\delta}$ is uniformly bounded as $\delta\to 0$. Given a vector $v$, we use $v_i$ denote the $(x_i,p_i)$-component. For a vector $v^+$ $\varepsilon$-parallel to $T_{z^+_{0,\delta}}(W^-\cap\Sigma^+_{0,\delta})$, there is $C>0$ such that $\|v^+_2\|\ge C\|v^+_1\|$. From Eq.(\ref{cylindereq3}) one obtains
\begin{align}\label{cylindereq5}
\|v^+_2\|e^{(\lambda_2-\mu)T}\le&\|D\Psi^+_{E,\delta,\delta}(z^+_{E,k})v^+_2\|\le\|v^+_2\|e^{(\lambda_2+ \mu)T},\\
\|v^+_1\|e^{(\lambda_1-\mu)T}\le&\|D\Psi^+_{E,\delta,\delta}(z^+_{E,k})v^+_1\|\le\|v^+_1\|e^{(\lambda_1+ \mu)T}\notag
\end{align}
where $0<\mu\to 0$ as $\delta\to0$. It follows that the vector $D\Psi^+_{E,\delta,\delta}(z^+_{E,k})v^+$ is $\varepsilon$-parallel to $T_{z^-_{0,\delta}}(W^-\cap\Sigma^-_{\delta})$ and
$$
C_6^{-1}\Big(\frac {\delta^2}E\Big)^{\frac{\lambda_2}{\lambda_1}-\mu_4}
\le\frac{\|D\Psi^+_{E,\delta,\delta}(z^+_{E,k})v^+\|}{\|v^+\|}
\le C_6\Big(\frac {\delta^2}E\Big)^{\frac{\lambda_2}{\lambda_1}+\mu_4}
$$
where $C_6>1$ and $\mu_4\downarrow 0$ as $\delta\downarrow 0$. Similarly, for a vector $v^-$ $\varepsilon$-parallel to $T_{z^-_{0,\delta}}(W^+\cap\Sigma^-_{0,\delta})$, one sees that the vector $D{\Psi^+_{E,\delta,\delta}(z^-_{E,j})}^{-1}v^-$ is $\varepsilon$-parallel to $T_{z^-_{0,\delta}}(W^-\cap\Sigma^-_{0,\delta})$ and
$$
C_6^{-1}\Big(\frac {\delta^2}E\Big)^{\frac{\lambda_2}{\lambda_1}-\mu_4}
\le\frac{\|D{\Psi^+_{E,\delta,\delta}}^{-1}(z^-_{E,k})v^-\|}{\|v^-\|}
\le C_6\Big(\frac {\delta^2}E\Big)^{\frac{\lambda_2}{\lambda_1}+\mu_4}.
$$

The composition of the two maps constitutes a Poin\'care map $\Phi_{E,\delta}=\Psi^+_{E,\delta,\delta}\circ\Psi_{E,\delta}$, it maps a small neighborhood of the point $z^-_{E,k}$ in $\Sigma^{-}_{E,\delta}$ to a small neighborhood of the point $z^-_{E,k+1}$ in $\Sigma^{-}_{E,\delta}$. For a vector $v^-$ $\varepsilon$-parallel to $T_{z^-_{0,\delta}}(W^-\cap\Sigma^-_{0,\delta})$ the vector $D\Phi_{E,\delta}(z^-_{E,k})v^-$ is still $\varepsilon$-parallel to $T_{z^-_{0,\delta}}(W^-\cap\Sigma^-_{0,\delta})$
\begin{equation}\label{cylindereq6}
\Lambda^{-1}\left(\frac {D^2}{E}\right)^{\frac{\lambda_2}{\lambda_1}-\mu_5}
\le\frac {\|D\Phi_{E,\delta}(z^-_{E,k})v^-\|}{\|v^-\|}
\le \Lambda\left(\frac{D^2}{E}\right)^{\frac {\lambda_2}{\lambda_1}+\mu_5},
\end{equation}
and for a vector $v^+$ $\varepsilon$-parallel to $T_{z^-_{0,\delta}}(W^+\cap\Sigma^-_{0,\delta})$ the vector $D\Phi_{E,\delta}^{-1}(z^-_{E,k})v^+$ is still $\varepsilon$-parallel to $T_{z^-_{0,\delta}}(W^+\cap\Sigma^-_{0,\delta})$
\begin{equation}\label{cylindereq7}
\Lambda^{-1}\left(\frac {D^2}{E}\right)^{\frac{\lambda_2}{\lambda_1}-\mu_5}
\le\frac{\|D\Phi_{E,\delta}^{-1}(z^-_{E,k})v^+\|}{\|v^+\|}
\le \Lambda\left(\frac {D^2}{E}\right)^{\frac{\lambda_2}{\lambda_1}+\mu_5}
\end{equation}
holds for each $k$, where $\Lambda\ge C_5C_6>1$, $0<\mu_5\to 0$ as $D\to 0$. Therefore, each point $z^-_{E,k}$ is a hyperbolic fixed point for the map $\Phi^{k_i+k_{i+1}}_{E,\delta}$, $\{z^-_{E,k}:k=1,\cdots,k_i+k_{i+1}\}$ is a hyperbolic orbit of $\Phi_{E,\delta}$. By Lemma \ref{flatlem4}, these points are uniquely ordered, $k_i+k_{i+1}$ is the minimal period. As these points approach to the fixed point as $E\downarrow0$, the hyperbolicity guarantees the uniqueness. It also guarantees the smooth continuation of periodic orbits. Therefore, we have

\begin{lem}\label{cylinderlem1}
Assume the conditions \text{\rm ({\bf H1})} and \text{\rm ({\bf H2})} and let $g\in H_1(\mathbb{T}^2,\mathbb{Z})$ be a class. If $\mathscr{L}_{\beta}(\nu g)\to\partial \mathbb{F}_0$ as $\nu\downarrow 0$, then there exists $E'>0$ such that for each  $c\in\mathscr{L}_{\beta}(\nu g)$ with $\alpha(c)=E\in (0,E']$ the Mather set $\tilde{\mathcal{A}}(c)$ consists of exactly one periodic orbit.

Let $\Sigma_E\subset\bar G^{-1}(E)$ be a two-dimensional disk transversally intersecting the orbit at $x_1=\delta$ such that $T_z\bar G^{-1}(E)=\text{\rm span}(T_z\Sigma_E,J\nabla\bar G(z))$ for $z\in\Sigma_E$ and let $\Phi_{E}$: $\Sigma_E\to \Sigma_E$ be the return map naturally determined by the flow $\Phi_{\bar G}^t$, there exists some $\lambda>1,C>0$ independent of $E\le E'$ such that
$$
\|D\Phi_{E}(z_{E,0})v^-\|\ge C E^{-\lambda}\|v^-\|,\qquad \forall\ v^-\in T_{z_{E,0}}W^-_E;
$$
$$
\|D\Phi_{E}(z_{E,0})v^+\|\le C^{-1}E^{\lambda}\|v^-\|,\qquad \forall\ v^+\in T_{z_{E,0}}W^+_E,
$$
where $z_{E,0}$ is the point where the periodic orbit intersects $\Sigma_{E}$, $W^{\pm}_E$ denotes the stable $($unstable$)$ manifold of the periodic orbit. Therefore, the periodic orbits for $E\in (0,E']$ constitute a smooth cylinder.
\end{lem}

In the second case, some $\nu_0>0$ exists such that $\mathscr{L}_{\beta}(\nu_0g)\in \partial^*\mathbb{F}_0$. It is typical that $\mathscr{L}_{\beta}(\nu_0g)$ is an edge of $\mathbb{E}_i\subset\partial^*\mathbb{F}_0$, where each Mather set consists of a hyperbolic periodic orbit $z_0(s)\subset\bar G^{-1}(0)$ and the fixed point. The uniqueness of minimal periodic orbit for $\nu$ close to $\nu_0$ follows from the implicit function theorem.

Given $g\in H_1(\mathbb{T}^2,\mathbb{Z})$, the set $\cup_{\nu>0}\mathscr{L}_{\beta}(\nu g)$ is a channel in $H^1(\mathbb{T}^2,\mathbb{Z})$. There is a path $\Gamma_g\subset\cup_{\nu>0} \mathscr{L}_{\beta}(\nu g)$ reaching to the boundary $\partial\mathbb{F}_0$, along which the $\alpha$-function monotonely decreases to the minimum as the cohomology class approaches to $\partial\mathbb{F}_0$. With the argument as above and Theorem \ref{AppenHyperTh1} shown in the appendix, we find the following condition also generic

({\bf H4}): {\it Given a class $g\in H_1(\mathbb{T}^2,\mathbb{Z})$ and an energy $E^*>0$, there are at most finitely many $E_i\in (0,E^*]$ such that for $c\in\Gamma_g$ with $\alpha(c)=E_i$ the Mather set consists of two periodic orbits, for all other $c\in\Gamma_g$ with $\alpha(c)\ne E_i$, the Mather set consists of exactly one periodic orbit. All these periodic orbits are hyperbolic.}

We call these $\{E_i\}$ bifurcation points. Let $E_1>0$ be the smallest one. Each energy $E<E_1$ uniquely determines a hyperbolic periodic orbit $\{x_E(s),p_E(s):s\in\mathbb{R}\}$. Thus, we introduce a notation
$$
\Pi_{E_0,E_1,g}=\{(x_{E}(s),p_{E}(s)):[x_{E}]=g,E\in [E_0,E_1],s\in\mathbb{R}\},
$$
with small $E_0>0$. It is a two-dimensional cylinder composed of a family of periodic orbits, may approach to a curve of figure-of-eight as $E_0$ decreases to zero. Obviously, the cylinder is invariant for the Hamiltonian flow $\Phi_{\bar G}^t$ of $\bar G$. Let $T(E)$ denote the period of the periodic orbit in $\bar G^{-1}(E)$, one has
$$
\int_{\Pi_{E_0,E_1,g}}\omega=\int_{E_0}^{E_1}\int_{0}^{T(E)}dE\wedge dt>0.
$$

The cylinder may be slant and crumpled, not standard. To see how the symplectic area is related to the usual area of the cylinder, let us  study the dependence of the fixed point $(x_2(E),p_2(E))$ of the Poincar\' e return map $\Phi_{E,\delta}$ on $E$. By definition, the fixed point is a solution of the equation
\begin{equation}\label{cylindereq8}
\Phi_{E,\delta}(x_2(E),p_2(E))-(x_2(E),p_2(E))=0.
\end{equation}
Emanating from a point $(\delta,p_1,x_2,p_2)\in \bar G^{-1}(E)$ the orbit reach a point  $z\in\{x_1=-\delta\}$ after a time $\tau(E,\delta)$ which remains bounded as $E\downarrow 0$. Let $z'\in\{x_1=-\delta\}$ be the point corresponding to $(\delta,p'_1,x_2,p_2)\in \bar G^{-1}(E')$, obtained in the same way. The difference of the $(x_2,y_2)$-coordinate of these two points is bounded by $d_0|p_1-p'_1|$ where $d_0$ depends on $\delta$.  Let $(\Delta x, \Delta y)$ be the solution of the variational equation (\ref{cylindereq3}) along the periodic solution $(x_{E}(s),p_{E}(s))$ passing through a neighborhood of the origin, let $s_0<s_1$ be the time such that the first coordinate $x_{E,1}(s_0)=-\delta$ and $x_{E,1}(s_1)=\delta$, the quantity $s_1-s_0$ is bounded by (\ref{regularenergyeq2}). In virtue of the formula (\ref{cylindereq5}), it yields
$$
\|(\Delta x, \Delta y)(s_1)\|\le C_7E^{-\frac{\lambda_2}{\lambda_1}-\mu_6}\|(\Delta x, \Delta y)(s_0)\|
$$
where $0<\mu_6\to 0$ as $\delta\to 0$. Therefore, we find that
$$
\Big\|\frac{\partial\Phi_E}{\partial p_1}\Big\|\le C_8E^{-\frac{\lambda_2}{\lambda_1}-\mu_6}.
$$
As the quantity $\|\frac{\partial\Phi_E}{\partial (x_2,p_2)}\|$ is bounded by (\ref{cylindereq6}), the quantity for the inverse of $\Phi_E$ is bounded by (\ref{cylindereq7}), we obtain from the equation (\ref{cylindereq8}) that
\begin{equation}\label{cylindereq9}
\Big\|\frac{\partial x_2}{\partial p_1}\Big\|,\Big\|\frac{\partial p_2}{\partial p_1}\Big\|\le C_9E^{-2\mu_6}.
\end{equation}
It yields a relation between the symplectic area $\omega$ and the usual area $S$ of the cylinder
\begin{equation}\label{cylindereq10}
|\omega|\ge C_{10}E^{2\mu_6}|S|.
\end{equation}

\begin{theo}\label{cylinderthm1}
We assume the conditions \text{\rm ({\bf H1}, {\bf H2}, {\bf H4})}. For each $E_0\in (0,E_1]$, the cylinder $\Pi_{E_0,E_1,g}$ is normally hyperbolic for the map $\Phi_{\bar G}^{\Delta t_{E}}$, where $\Delta t_{E_0}=2\lambda_1^{-1}|\ln E_0|$.
\end{theo}
\begin{proof}
For any $0<E_0<E_1$, the cylinder $\Pi_{E_0,E_1,g}$ is a 2-dimensional symplectic sub-manifold, invariant for the Hamiltonian flow $\Phi^s_{\bar G}$.  However, it is not clear whether this cylinder is normally hyperbolic for the time-1-map $\Phi_{\bar G}=\Phi^s_{\bar G}|_{s=1}$, as it is possible that
\begin{align*}
m(D\Phi_{\bar G}|_{T\Pi_{E_0,E_1,g}})=&\inf\{|D\Phi_{\bar G}v|:v\in T\Pi_{E_0,E_1,g}, |v|=1\}<1,\\
&\|D\Phi_{\bar G}|_{T\Pi_{E_0,E_1,g}}\|>1,
\end{align*}
and we do not have the estimate on the norm of $D\Phi_{\bar G}$ acting on the normal bundle.

As the first step of the proof, let us search for the normal hyperbolicity of $\Phi^s_{\bar G}$ with large $s$ for the cylinder $\Pi_{E_0,E',g}$, where $E'$ denotes the largest value such that the formulae (\ref{cylindereq6}) and (\ref{cylindereq7}) hold for each $E\le E'$. From these formulae, one sees that the smaller the energy reaches, the stronger hyperbolicity the map $\Phi_{E,\delta}$ obtains. The strong hyperbolicity is obtained by passing through small neighborhood of the fixed point. However, on the other hand, the smaller the energy decreases, the longer the return time becomes.

Let $\Delta t_{E,k}$ denote the time interval  such that, starting from $z^-_{E,k}$, the periodic orbit comes to $z^-_{E,k+1}$ after time $\Delta t_{E,k}$. In virtue of the study in Section 3, Eq. (\ref{regularenergyeq2}),
$$
\Delta t_{E,k}\approx\tau_{E,g_{\ell}}-\lambda_1^{-1}\ln E, \qquad \ell=i, \ \text{\rm or}\ i+1,
$$
where $\tau_{E,g_{\ell}}$ is uniformly bounded and we take $\tau_{E,g_{\ell}}$  if the segment of the periodic orbit keeps close to the homoclinic orbit $\gamma_{\ell}$. Note $\Delta t_{E}=\frac 2{\lambda_1}|\ln E|$ is much larger than $\max_k\Delta t_{E,k}$. Thus, starting from any point $z$ on the minimal periodic orbit $z_{E}(s)$, $\Phi_{\bar G}^s(z)$ passes through the neighborhood of the fixed point after time $\Delta t_{E}$. It implies that the map $\Phi_{\bar G}^s|_{s=T_{E}}$ obtains strong hyperbolicity on normal bundle.

To measure how the map $D\Phi^s_{\bar G}$ acts on the tangent bundle, let us study how the map $\Phi^s_{\bar G}$ elongates or shortens small arc of the periodic orbit. As the orbit passes through the neighborhood of the origin $O_{\delta}(0)$ in a time approximately equal to $-\lambda_1^{-1}\ln \delta^{-2}E$ the variation of the length of short arc is between $O(E_0^{1+\mu_7})$ and $O(E_0^{-1-\mu_7})$, where $\mu_7>0$ is small. Because of the relation between the symplectic area $\omega$ and the usual area $S$ of the cylinder, given by the formula (\ref{cylindereq10}), the variation of $\|D\Phi^s_{\bar G}\|$, restricted on the tangent bundle of the cylinder, is between $O(E_0^{1+\mu_7+2\mu_6})$ and $O(E_0^{-1-\mu_7-2\mu_6})$. Because of periodicity, it is independent of $s$.

Thus, the normally hyperbolic property becomes clear: the tangent bundle of $M$ over $\Pi_{E_0,E_1,g}$ admits $D\Phi_{\bar G}^s|_{s=T_{E}}$-invariant splitting
$$
T_zM=T_zN^+\oplus T_z\Pi_{E_0,E_1,g}\oplus T_zN^-
$$
and some $\Lambda_1\ge 1$, $\Lambda_2\ge 1$ and small $\nu>0$ exist such that
\begin{equation}\label{cylindereq11}
\Lambda_1^{-1}E_0^{1+\nu}<\frac {\|D\Phi^s_{\bar G}(z)v\|}{\|v\|}<\Lambda_1E_0^{-1-\nu},\qquad \forall\ v\in T_z\Pi_{E_0,E_1,g},
\end{equation}
\begin{equation*}
\frac{\|D\Phi_{\bar G}^{s}(z)v\|}{\|v\|}\le \Lambda_2 E_0^{\frac{\lambda_2}{\lambda_1}-\nu},\qquad \forall\ v\in T_zN^+,
\end{equation*}
\begin{equation*}
\frac{\|D\Phi_{\bar G}^{s}(z)v\|}{\|v\|}\ge\Lambda_2^{-1} E_0^{-\frac{\lambda_2}{\lambda_1}+\nu}, \qquad \forall\ v\in T_zN^-,
\end{equation*}
hold for $s\ge\Delta t_{E}$ (cf. (\ref{cylindereq6}) and (\ref{cylindereq7})). Note that $\lambda_2/\lambda_1-\nu>1+\nu$ provided $\nu>0$ is suitably small. The formula (\ref{cylindereq11}) satisfies the definition of normal hyperbolicity.

For each $E\in[E' ,E_1]$, let $z_{E}(s)$ be the minimal periodic orbit, $\Sigma_{E}\subset\bar G^{-1}(E)$ be a 2-dimensional disk intersecting $z_{E}(s)$ transversally at the point $z_{E,0}$, $\Phi_{E}$: $\Sigma_{E}\to\Sigma_{E}$ be the Poincar\'e return map. By the generic property \text{\rm ({\bf H4})}, $z_{E,0}$ is the hyperbolic fixed point of $\Sigma_{E}$ and $\Lambda_{2}>1$ exists such that
\begin{equation*}
\|D\Phi_{E}(z_{E,0})v^-\|\ge\Lambda_{2}\|v^-\|, \qquad \forall\ v^-\in T_{z_{E,0}}(W^-_{E}\cap\Sigma_{E}),
\end{equation*}
\begin{equation*}
\|D\Phi_{E}(z_{E,0})v^+\|\le\Lambda_{2}^{-1}\|v^+\|, \qquad \forall\ v^+\in T_{z_{E,0}}(W^+_{E}\cap\Sigma_{E}).
\end{equation*}
As the cylinder is foliated into periodic orbits and $\Phi_{\bar G}$ preserves the symplectic form, some $\Lambda_{1}\ge 1$ exists such that
\begin{equation*}
\Lambda_{1}^{-1}\|v\|\le\|D\Phi_{\bar G}^{s}(z)v\|\le\Lambda_{1}\|v\|, \qquad \forall\ v^-\in T_{z}\Pi_{E',E_1,g}
\end{equation*}
holds for any $s>0$. Choosing $m\in\mathbb{N}$ such that $\Lambda_{2}^m\ge 2\Lambda_{1}$, one obtains the normal hyperbolicity for $\Phi_{\bar G}^{s}(z)$ with $s\ge mT(E')$, where $T(E')$ is the period of the periodic solution in $\bar G^{-1}(E')$,  $z\in\Pi_{E',E_1,g}\cap \bar G^{-1}(E)$ with $E\in [E',E_1]$. By choosing suitably small $E_0>0$, we have $\Delta t_{E_0}\ge mT(E')$.
\end{proof}

\subsection{Persistence of cylinder: near double resonance}

To apply the theorem of normally hyperbolic manifold \cite{HPS} to the Hamiltonian $G_{\epsilon}$ of (\ref{cylindereq3}), we note that the Hamiltonian $\bar G_{\epsilon}=\bar G+Z_{\epsilon}$ is autonomous with two degrees of freedom. As $Z_{\epsilon}=O(\sqrt{\epsilon})$, and because of the non-degeneracy assumption for $V$ (({\bf H1},{\bf H2})), we see that for each suitably small $\epsilon>0$, the map $\Phi_{\bar G_{\epsilon}}$ admits invariant cylinder also, denoted by  $\Pi_{E_0,E_1,g}$ still, with the normally hyperbolic properties (see Formulae (\ref{cylindereq11})), independent of the size of $\epsilon$.

Let us consider the persistence of $\Pi_{E'_0,E_1,g}$ with $E'_0=\epsilon^{2d}$ with $d>0$. As the perturbation depends on time $s$, we use $\Phi^{s,s_0}_{G_{\epsilon}}$ to denote the map from the time $s_0$-section to the time $s$-section, omit the symbol $s_0$ if $s_0=0$. Since these hyperbolic properties are posed for the map $\Phi^{s,s_0}_{\bar G_{\epsilon}}$ with large $s-s_0$, one has to measure how large the quantity $\|\Phi_{\bar G_{\epsilon}}^{s,s_0}-\Phi_{G_{\epsilon}}^{s,s_0}\|$ will be. As $\bar G_{\epsilon}$ is autonomous, $\Phi_{\bar G_{\epsilon}}^{s,s_0}=\Phi_{\bar G_{\epsilon}}^{s-s_0}$.

\begin{lem}\label{cylinderlem2}
Let the equation $\dot z=F_{\epsilon}(z,t)$ be a small perturbation of $\dot z=F_0(z,t)$, let $\Phi_{\epsilon}^t$ and $\Phi_0^t$ denote the flow determined by these two equations respectively. Then
$$
\|\Phi_{\epsilon}^t-\Phi_{0}^t\|_{C^1}\le \frac BA(1-e^{-At})e^{2At}
$$
where $A=\max_{t,\lambda=\epsilon,0}\|F_{\lambda}(\cdot,t)\|_{C^2}$ and $B=\max_t\|(F_{\epsilon}-F_0)(\cdot,t)\|_{C^1}$.
\end{lem}
\begin{proof} Let $z_{\lambda}(t)$ denote the solution of the equations $\dot z=F_{\lambda}(z,t)$ for $\lambda =\epsilon,0$ respectively, and $z_{\epsilon}(0)=z(0)$. Let $\Delta z(t)=z_{\epsilon}(t)-z(t)$, then $\Delta z(0)=0$ and
$$
\Delta\dot z=\partial_zF_{\epsilon}((\nu z+(1-\nu)z_{\epsilon})(t),t)\Delta z+(F_{\epsilon}-F_0)(z(t),t)
$$
where $\nu=\nu(t)\in [0,1]$. Therefore, one has
$$
\|\Delta\dot z\|\le\max\|\partial_zF_{\epsilon}\|\|\Delta z\|+\max\|F_{\epsilon}-F_0\|.
$$
Let $\Delta z=y-\frac BA$, we have $\dot y\le Ay$. It follows from Gronwell's inequality that
$$
\|\Delta z(t)\|\le\frac BA(e^{At}-1).
$$

Along the orbit $z_{\lambda}(t)$, the differential of the flow $\Phi^t_{\lambda}$  obviously satisfies the equation
$$
\frac d{dt}D\Phi^t_{\lambda}=\partial_zF_{\lambda}(z_{\lambda}(t),t)D\Phi^t_{\lambda},\qquad \lambda=\epsilon, 0.
$$
Therefore, for each tangent vector $v$ attached to $z_{\lambda}(0)$ one has
$$
\|D\Phi^t_{\lambda}v\|\le\|v\|e^{At}.
$$

To study the differential of $\Phi_{\epsilon}^t-\Phi_{0}^t$, let us consider the equation of secondary variation. Let $\delta z_{\lambda}$ be the solution of the variational equation $\delta\dot z_{\lambda}=\partial_zF_{\lambda}(z_{\lambda}(t),t)\delta z_{\lambda}$ for $\lambda=\epsilon,0$ respectively, where $z_{\lambda}(t)$ solves the equation $\dot z_{\lambda}= F_{\lambda}(z_{\lambda},t)$ and $z_{\epsilon}(0)=z(0)$.  To measure the size $\Delta\delta z=\delta z_{\epsilon}-\delta z$ with the condition $z_{\epsilon}(0)=z(0)$, we make use of the relations such as $v=\delta z_{\epsilon}(0)=\delta z(0)$, $\|\delta z(t)\|\le\|v\|e^{At}$ and find that
\begin{align*}
\Big\|\frac {d(\Delta\delta z)}{dt}\Big\|\le &\max\|\partial _zF_{\epsilon}\|\|\Delta\delta z\| +\max\|\partial^2_zF\|\|\Delta z(t)\|\|\delta z(t)\| \\
&+\max\|\partial_z(F_{\epsilon}-F)\|\|\delta z(t)\| \\
\le &A\Delta\delta z+B\|v\|e^{2At}.
\end{align*}
Let $\Delta\delta z=y+\frac BA\|v\|e^{2At}$, we have $\dot y\le Ay$. Using Gronwell's inequality again, one obtains an upper bound of the variation of the differential
$$
\|\Delta\delta z(t)\|\le\frac BA\|v\|(1-e^{-At})e^{2At}.
$$
Note that $v$ represents initial tangent vector, it completes the proof.
\end{proof}

Let us applying this lemma to the Hamiltonian $G_{\epsilon}$. Treating $R_{\epsilon}$ as the function of $(x,p)$ we find that there exist some constants $C_{11},C_{12}>0$ independent of $\epsilon$ such that
$$
\max_s\|J\nabla\bar G_{\epsilon}-J\nabla G_{\epsilon}\|_{C^1}\le\max\Big|\frac{\partial^2 R'_{\epsilon}} {\partial x\partial p}\Big|\le C_{11}\epsilon^{5\sigma-\frac 16}
$$
as $\|\epsilon R_{\epsilon}(\cdot,s)\|_2=O(\epsilon^{\frac 56+5\sigma})$ (see Theorem \ref{normalthm1}). Since the function $\bar G_{\epsilon}$ comes from $h+\epsilon Z$ which is $C^r$-smooth ($r\ge 8$), one has $\max_s\|J\nabla\bar G_{\epsilon}\|_{C^2}=\max_s\|\bar G_{\epsilon}\|_{C^3}<C_{12}$. For $s-s_0=\frac 2{\lambda_1}|\ln\epsilon^{2d}|$ one obtains from Lemma \ref{cylinderlem2} that
$$
\|\Phi^{s,s_0}_{\bar G_{\epsilon}}-\Phi^{s,s_0}_{G_{\epsilon}}\|_{C^1}\le \frac{C_{11}}{C_{12}}\epsilon^{5\sigma-\frac 16-\frac {8C_{12}d}{\lambda_1}}
$$
If the condition $0<d<\frac{\lambda_1}{8C_{12}}(5\sigma-\frac 16)$ is satisfied,
then $\|\Phi^{s,s_0}_{\bar G_{\epsilon}}-\Phi^{s,s_0}_{G_{\epsilon}}\|_{C^1}\to 0$ as $\epsilon\to 0$. It allows one to apply the theorem of normally hyperbolic manifold to obtain the existence of invariant cylinder for the flow $\Phi^{s,s_0}_{G_{\epsilon}}$ in the extended phase space $\mathbb{T}^2\times\mathbb{R}^2\times\sqrt{\epsilon}\mathbb{T}$, which is a small deformation of $\Pi_{E'_0,E_1,g}\times\sqrt{\epsilon}\mathbb{T}$.

Be aware of the fact that $\Pi_{E'_0,E_1,g}$ is a cylinder with boundary, normally hyperbolic and invariant for $\Phi^s_{\bar G_{\epsilon}}$, where $s=\frac 2{\lambda_1}|\ln \epsilon^{2d}|$, we do not expect that the whole cylinder survives small perturbation, it may lose some part close to the boundary. To measure to what range the cylinder survives, we see that the variation of the energy along each orbit of $\Phi^{s,s_0}_{ G_{\epsilon}}$ is bounded by
\begin{equation}\label{cylindereq12}
\Big |\frac d{ds}G_{\epsilon}(z(s),s)\Big |=|\partial _sG_{\epsilon}(z,s)|=\frac 1{\sqrt{\epsilon}}\Big|\frac{\partial R_{\epsilon}}{\partial\tau}\Big|\le C_{13}\epsilon^{5\sigma-\frac 16}
\end{equation}
here, the estimate $\|\epsilon R_{\epsilon}\|_{1}\le O(\epsilon^{\frac 43+5\sigma})$ is used (see Theorem \ref{normalthm1}). Assume that the number $d$ satisfies the condition
\begin{equation}\label{cylindereq13}
d<\min\Big\{\frac 12\Big(5\sigma-\frac 16\Big),\frac{\lambda_1}{8C_{12}}\Big(5\sigma-\frac 16\Big),\frac 12\Big\}
\end{equation}
one sees that, starting from the energy level $G_{\epsilon}^{-1}(\epsilon^{d})$, after a time of $s-s_0=\frac {2}{\lambda_1}|\ln\epsilon^{2d}|$, the orbit of $\Phi^{s,s_0}_{G_{\epsilon}}$ can not reach the energy level $G_{\epsilon}^{-1}(\epsilon^{2d})$ if $\epsilon$ is suitably small so that $2\epsilon^d(1+2C_{13}\lambda^{-1}|\ln\epsilon^{2d}|)<1$. Indeed, under such condition one has
\begin{align}\label{cylindereq14}
G_{\epsilon}(z(s),s)&\ge G_{\epsilon}(z(s_0),s_0)-\int_{s_0}^s\Big |\frac d{dt}G_{\epsilon}(z(t),t)\Big|dt\\
&\ge\epsilon^d-\frac {2C_{13}}{\lambda_1}\epsilon^{2d}|\ln\epsilon^{2d}|>\frac12\epsilon^d+\epsilon^{2d}.\notag
\end{align}

To use the theorem of normally hyperbolic invariant manifold, let us introduce a modified Hamiltonian. Let $u$: $\mathbb{R}\to\mathbb{R}_+$ be a smooth function so that $u=0$ for $t\le 1$ and $u=1$ for $t\ge 2$,
$$
G'_{\epsilon}=\bar G_{\epsilon}+u\Big(\frac 2{\epsilon^{d}}\Big(\bar G_{\epsilon}-\epsilon^{2d}\Big)+1\Big)R_{\epsilon}
$$
it coincides with $G_{\epsilon}$ for $(x,p)\in\bar G_{\epsilon}^{-1}(E)$ with $E\ge\frac 12 \epsilon^d+ \epsilon^{2d}$ and coincides with $\bar G_{\epsilon}$ for $(x,p)\in\bar G_{\epsilon}^{-1}(E)$ with $E\le\epsilon^{2d}$. For small $d$ satisfying the condition (\ref{cylindereq13}) and small $\epsilon$, the cylinder $\Pi_{E'_0,E_1,g}$ survives the perturbation $\Phi^{s,s_0}_{G'_{\epsilon}}\to\Phi^{s,s_0}_{\bar G_{\epsilon}}$ and the bottom remains invariant for $\Phi^{s,s_0}_{G'_{\epsilon}}$.

By the definition, one has $G_{\epsilon}=G'_{\epsilon}$ for $(x,p)\in\bar G_{\epsilon}^{-1}(E)$ with $E\ge\frac 12\epsilon^d+\epsilon^{2d}$ and it follows from (\ref{cylindereq14}) that $G_{\epsilon}(\Phi^{s,s_0}_{G'_{\epsilon}}(x,p),s)\ge\frac12\epsilon^d+\epsilon^{2d}$ provided $G_{\epsilon}(x,p,s_0)\ge\epsilon^d$ and $s\in [0,\frac 2{\lambda_1}|\ln\epsilon^{2d}|]$. So, for $E_0=\epsilon^d$ the invariant cylinder $\Pi_{E_0,E_1,g}\times\sqrt{\epsilon}\mathbb{T}$ persists under the perturbation $\Phi^{s,s_0}_{\bar G_{\epsilon}}\to\Phi^{s,s_0}_{G_{\epsilon}}$, denoted by $\tilde\Pi_{E_0,E_1,g}$. A point $(x,p,s)\in\tilde\Pi_{E_0,E_1,g}$ implies $G_{\epsilon}(x,p,s)\in[E_0,E_1]$. The invariance is in the sense that, emanating from any point in $\tilde\Pi_{E_0,E_1,g}$, the orbit has to pass through the bottom of the cylinder if it is going to leave the cylinder.

\noindent{\bf Location of Aubry set in the cylinder}

As the working space here is phase space, we say that an Aubry set $\tilde{\mathcal{A}}(c)$ is located in the cylinder $\tilde\Pi_{E_0,E_1,g}$ if for each $c$-static curve $\gamma$, the orbit in the phase space $(x(s)=\gamma(s), p(s)=\partial_{\dot x}L_{G_{\epsilon}}(\gamma(s), \dot\gamma(s),s),s)\in\tilde\Pi_{E_0,E_1,g}$.

Recall the Hamiltonian $G_{\epsilon}$ defined in (\ref{cylindereq3}) and note the Hamiltonian $\bar G_{\epsilon}=G_{\epsilon}-R_{\epsilon}$ is autonomous. Let $\alpha_{G_{\epsilon}}$ and $\alpha_{\bar G_{\epsilon}}$ denote the $\alpha$-function for the Lagrangians determined by $G_{\epsilon}$ and $\bar G_{\epsilon}$ respectively.

Since $\Pi_{0,E_1,g}$ is a hyperbolic cylinder, invariant for the Hamiltonian flow $\Phi^t_{\bar G_{\epsilon}}$, the channel $\mathbb{W}_g=\cup_{\nu\in (0,\nu_1]}\mathscr{L}_{\beta}(\nu g)$ has a foliation into a family of segments of line (one-dimensional flat), where $\nu_{1}>0$ is chosen so that $\alpha_{\bar G_{\epsilon}}(c)\le E_1$ for each $c\in\mathscr{L}_{\beta}(\nu g)$ if $\nu\le\nu_1$. We claim that the $\alpha$-function $\alpha_{\bar G_{\epsilon}}$ is smooth in $\mathbb{W}_g$. Indeed, restricted on each of these flats the function $\alpha_{\bar G_{\epsilon}}$ keeps constant, while restricted on a line $\Gamma_g$ orthogonal to these flats, the function is smooth because $\bar G_{\epsilon}$ can be treated as a Hamiltonian with one degree of freedom when it is restricted on the cylinder. If $g=k_ig_i+k_{i+1}g_{i+1}$ we consider the Hamiltonian in the finite covering space $\bar M=\bar k_1\mathbb{T} \times\bar k_2\mathbb{T}$ where $\bar k_m=k_ig_{im}+k_{i+1}g_{i+1,m}$ for $m=1,2$ if we write $g_j=(g_{j1},g_{j2})$ for $j=i,i+1$. In that space there are $k_i+k_{i+1}$ fixed points. Because of Hartman's theorem for two-dimensional system, it is $C^1$-conjugate to a linear equation $\dot x=\lambda_1y$, $\dot y=\lambda_1x$ around each fixed point. Therefore, for small $E>0$ some $C^1$-function $\tau_g(E)$ exists such that the period of the frequency $\nu g$ is $T_{\nu g}=\lambda_1^{-1}(k_i+k_{i+1})(-\ln E+\tau_g(E))$ (see (\ref{regularenergyeq2})). Since $\partial\alpha_{\bar G_{\epsilon}}=\nu g$
\begin{equation}\label{cylindereq15}
\frac{\lambda_1}{-\ln E+\tau_g(E)}=\|\partial\alpha_{\bar G_{\epsilon}}\|,\qquad \forall\ c\in \Gamma_g,
\end{equation}
As $\alpha_{\bar G_{\epsilon}}(c)=E$ remains constant when it is restricted on each of its flats which are orthogonal to the line $\Gamma_g$, we find that for $0<\alpha_{\bar G_{\epsilon}}\ll 1$
\begin{equation}\label{cylindereq16}
\langle\partial^2\alpha_{\bar G_{\epsilon}}v,v\rangle=\frac{\lambda_1^2(1-\alpha_{\bar G_{\epsilon}} \tau'_g(\alpha_{\bar G_{\epsilon}}))} {\alpha_{\bar G_{\epsilon}}(-\ln\alpha_{\bar G_{\epsilon}}+ \tau_g(\alpha_{\bar G_{\epsilon}}))^3}>0,
\end{equation}
where $v$ is the direction of $\Gamma_g$ and $\|v\|=1$. It follows that
\begin{equation}\label{cylindereq17}
\alpha_{\bar G_{\epsilon}}(c)-\alpha_{\bar G_{\epsilon}}(c_{\omega}) \ge\langle\omega,c-c_{\omega}\rangle+\frac12\langle\partial^2\alpha_{\bar G_{\epsilon}}(c)v,v\rangle|c-c_{\omega}|^2
\end{equation}
holds for $\omega=\nu g$, $c_{\omega}\in\mathscr{L}_{\beta}(\omega)\cap\Gamma_g$ $c\in\Gamma_g$, $\alpha_{\bar G_{\epsilon}}(c)>\alpha_{\bar G_{\epsilon}}(c_{\omega})$.

Let $c^*$ be the class so that $\alpha_{G_{\epsilon}}(c)=\alpha_{\bar G_{\epsilon}}(c^*)$. To measure the difference of $c^*-c$, we find $\frac 12|\langle c-c^*,\omega(c^*)\rangle|\le|\alpha_{\bar G_{\epsilon}}(c^*)-\alpha_{\bar G_{\epsilon}}(c)|=|\alpha_{G_{\epsilon}}(c)-\alpha_{\bar G_{\epsilon}}(c)|$. As the $\alpha$-function undergoes small variation: $|\alpha_L(c)-\alpha_{L'}(c)|\le\varepsilon$ for small perturbation $L'\to L$ with $\|L'-L\|_{C^1}\le\varepsilon$ \cite{Ch}, we find $|\alpha_{G_{\epsilon}}(c)-\alpha_{\bar G_{\epsilon}}(c)|\le \epsilon^{2d}$ when $c$ is restricted on the path $\Gamma_g$.  Therefore, we obtain that
\begin{equation}\label{cylindereq18}
|\langle c^*-c,\omega(c^*)\rangle|\le 2\epsilon^{2d}.
\end{equation}

\begin{lem}
In the Aubry set for $c\in\Gamma_g\cap\alpha^{-1}_{G_{\epsilon}}(E)$ with $E\ge 2\epsilon^{d}$, any orbit  does not hit the energy level set $G_{\epsilon}^{-1}(E)$ with $E\le\epsilon^d$.
\end{lem}
\begin{proof}
If the lemma does not hold, there would exist an orbit $(x(s),p(s))$ in the Aubry set for $c\in\Gamma_g\cap\alpha^{-1}_{G_{\epsilon}}(2\epsilon^{d})$, which hits the energy level $G_{\epsilon}^{-1}(\epsilon^d)$ at the time $s=s_0\mod\sqrt{\epsilon}$, i.e. $G_{\epsilon}(x(s_0),y(s_0),s_0)=\epsilon^d$. Since the orbit entirely stays in the invariant cylinder and the perturbation is of order $\epsilon^{2d}$, it returns back to the neighborhood of $(x(s_0),y(s_0))$ after a time $S=\lambda_1^{-1}(k_i+k_{i+1})\ln\epsilon^d+\tau_{\epsilon}$ where $\tau_{\epsilon}$ remains bounded as $\epsilon\to 0$. To see how close it could be, we obtain from (\ref{cylindereq12}) that
\begin{align}\label{cylindereq19}
&|G_{\epsilon}(x(S+s_0),p(S+s_0),S+s_0)-G_{\epsilon}(x(s_0),p(s_0),s_0)|\\
&\le \int_{s_0}^{S+s_0}\Big |\frac d{ds}G'_{\epsilon}(z(s),s)\Big|ds\le C_{15}\epsilon^{2d}|\ln\epsilon^{d}|. \notag
\end{align}
As $\bar G_{\epsilon}^{-1}(E)\cap\Pi_{0,E_1,g}$ is an invariant circle for $\Phi^t_{\bar G_{\epsilon}}$, the perturbed cylinder is $O(\epsilon^{2d})$-close to the original one \cite{BLZ} and the cylinder may be crumpled but at most up to the order $O(E^{-2\mu_6})$ (cf. (\ref{cylindereq9})), some large $k\in\mathbb{Z}$ exists such that
$S=k\sqrt{\epsilon}$  and
$$
\|(x(S+s_0),p(S+s_0))-(x(s_0),p(s_0))\|\le C_{14}\epsilon^{2d(1-\mu_6)}|\ln\epsilon^{d}|.
$$
Since the curve $x(s)$ is assumed $c$-static, it follows that
\begin{equation}\label{cylindereq20}
\Big|\int_{s_0}^{S+s_0}(L_{G_{\epsilon}}(x(s),\dot x(s),s)-\langle c,\dot x(s)\rangle+\alpha_{G_{\epsilon}}(c))ds \Big|\le C_{16}\epsilon^{2d(1-\mu_6)}|\ln\epsilon^{d}|.
\end{equation}

As the cylinder $\Pi_{E_0,E_1,g}\times\sqrt{\epsilon}\mathbb{T}$ is $\epsilon^{2d}$-close to $\tilde\Pi_{E_0,E_1,g}$, there is a $c'$-minimal orbit $(x'(s),y'(s))$ of the Hamiltonian flow $\Phi^s_{\bar G_{\epsilon}}$ on $\Pi_{E_0,E_1,g}$ such that $\alpha_{\bar G_{\epsilon}}(c')=\epsilon^d$ and $\|(x'(s_0),y'(s_0))-(x(s_0),y(s_0))\|\le O(\epsilon^{2d(1-\mu_6)})$. Let $\Gamma_x=\bigcup_{s=s_0}^{s_0+S}(x(s),y(s))$ and $\Gamma_{x'}=\bigcup_{s=s_0}^{s_0+S'}(x'(s),y'(s))$ where $S'$ is the period of $x'(s)$, we have an estimate on the Hausdorff distance
$d_H(\Gamma_x,\Gamma_{x'})\le O(\epsilon^{2d(1-\mu_6)}|\ln\epsilon^d|)$. Consequently, we have
$$
\int_{\Gamma_x}\langle y,dx\rangle-\int_{\Gamma_{x'}}\langle y,dx\rangle=O(\epsilon^{2d(1-\mu_6)}|\ln\epsilon^d|).
$$
As $\bar G_{\epsilon}(x'(s),y'(s))\equiv\alpha_{G_{\epsilon}}(c')$ we have
\begin{align}
0&=\int(L_{\bar G_{\epsilon}}(x'(t),\dot x'(t))-\langle c',\dot x'(t)\rangle+\alpha_{\bar G_{\epsilon}}(c'))dt \notag\\
&=\int \langle y'(s)-c',\dot x'(s)\rangle ds\notag
\end{align}
Let $\bar x(s)$ be the lift of $x(s)$ to the universal covering space, it follows that
\begin{align}
&\int_{s_0}^{S+s_0}\langle y(s)-c,\dot x(s)\rangle ds\notag\\
=&\int_{s_0}^{S+s_0}\langle y(s)-c',\dot x(s)\rangle ds-\int_{s_0}^{S'+s_0}\langle y'(s)-c',\dot x'(s)\rangle ds \notag\\
&-\langle c-c',\bar x(S+s_0)-\bar x(s_0)\rangle\notag\\
=&\int_{\Gamma_x}\langle y,dx\rangle-\int_{\Gamma_{x'}}\langle y',dx'\rangle+O(\epsilon^{2d(1-\mu_6)}|\ln\epsilon^d|)\notag\\
&-\langle c-c',\bar x(S+s_0)-\bar x(s_0)\rangle\notag\\
=&-\langle c-c',\bar x(S+s_0)-\bar x(s_0)\rangle+O(\epsilon^{2d(1-\mu_6)}|\ln\epsilon^{2d}|)\notag
\end{align}
Since it follows from (\ref{cylindereq19}) that
$$
\alpha_{G_{\epsilon}}(c)-G_{\epsilon}(x(s),y(s),s)\ge\alpha_{G_{\epsilon}}(c)-\alpha_{\bar G_{\epsilon}}(c')-O(\epsilon^{2d}|\ln\epsilon^{d}|)
$$
holds for all $s\in[s_0,S+s_0]$, we find
\begin{align}\label{cylindereq21}
&\int_{s_0}^{S+s_0}(L_{G_{\epsilon}}(x(s),\dot x(s),s)-\langle c,\dot x(s)\rangle+\alpha_{G_{\epsilon}} (c))ds\\
=&\int_{s_0}^{S+s_0}\Big(\langle y(s)-c,\dot x(s)\rangle+(\alpha_{G_{\epsilon}}(c)- G_{\epsilon}(x(s),y(s),s))\Big)ds\notag\\
\ge&\, (\alpha_{G_{\epsilon}}(c)-\alpha_{\bar G_{\epsilon}}(c'))S-\langle c-c',\bar x(S+s_0)-\bar x(s_0)\rangle-O(\epsilon^{2d}|\ln\epsilon^{2d}|)\notag\\
\ge&\, C_{17}\epsilon^d.\notag
\end{align}
To verify the second inequality, let $c^*$ be the class such that $\alpha_{\bar G_{\epsilon}}(c^*)=\alpha_{G_{\epsilon}}(c)$, then we obtain from the formula (\ref{cylindereq18}) that
$$
|c^*-c|\le 3\lambda_1^{-1}\epsilon^{2d}|\ln\epsilon^d|.
$$
For small $\epsilon$ such that $\epsilon^d\ge 4\epsilon^{2d}$ we find from (\ref{cylindereq15}) that
$$
|c'-c^*|\ge \frac 1{\|\partial\alpha_{\bar G_{\epsilon}}\|}\Big(\alpha_{\bar G_{\epsilon}}(c')-\alpha_{\bar G_{\epsilon}}(c^*)\Big)\ge C_{18}\epsilon^{d}|\ln\epsilon^{d}|
$$
holds for $c^*,c'\in\Gamma_g$ and $\alpha_{\bar G_{\epsilon}}(c')>\alpha_{\bar G_{\epsilon}}(c^*)$. Therefore, one obtains from (\ref{cylindereq16}) and (\ref{cylindereq17}) that
$$
\alpha_{\bar G_{\epsilon}}(c')-\alpha_{\bar G_{\epsilon}}(c^*)-\langle c'-c^*,\omega\rangle\ge C_{19}\frac{\epsilon^{d}}{|\ln\epsilon^{d}|},
$$
from which one obtains the second inequality of (\ref{cylindereq21}) from the first one.
As $\mu_6$ is very small, the formula (\ref{cylindereq21}) contradicts (\ref{cylindereq20}). It completes the proof.
\end{proof}

For $d>0$ satisfying the condition (\ref{cylindereq13}), going back to the original coordinates ($E\to\epsilon E$, $y=\sqrt{\epsilon}p$ and $s=\sqrt{\epsilon}\tau$), we obtain

\begin{theo}\label{cylinderthm2}
For an irreducible class $g\in H_1(\mathbb{T}^2,\mathbb{Z})$, there exists a 3-dimensional cylinder $\tilde \Pi_{E_0,E_1,g}\subset\mathbb{R}^2 \times\mathbb{T}^3$ associated with a channel $\mathbb{W}_g\subset H^1(\mathbb{T}^2,\mathbb{R})$ such that

1, the cylinder $\tilde\Pi_{E_0,E_1,g}$ is a small deformation of the cylinder
$$
\{(x_{E}(\tau+\tau^*),y_{E}(\tau+\tau^*),\tau^*): [x_{E}]=g,(\tau,\tau^*)\in\mathbb{T}^2,E\in[E_0,E_1]\},
$$
where $E_0=\epsilon^{\frac 1d}$, $x_{E}$ is the minimal periodic curve for $L_{\bar Y}-\langle c,\dot x\rangle+E$, the function $\bar Y$ solves the equation $(\tilde h+\epsilon\tilde Z)(x,y,\bar Y(x,y))=\tilde E$, $c\in\mathbb{W}_g$ and $y_E=\partial_{\dot x}L_{\bar Y}(x_E,\dot x_E)$;

2, the cylinder is invariant for the Hamiltonian flow $\Phi^{\tau,\tau_0}_{Y}$: $\forall$ $(x,y,\tau_0)\in \tilde\Pi_{E_0,E_1,g}$, if $\Phi^{\tau,\tau_0}_{Y}(x,y)\notin\tilde\Pi_{\epsilon,g}$ holds for $\tau>\tau_0$ then $\exists$ $\tau'\in(\tau_0,\tau)$ such that $\Phi^{\tau',\tau_0}_{Y}(x,y)$ is on the boundary of $\tilde\Pi_{E_0,E_1,g}$, where $Y$ solves the equation $H(x,y,-\tau,Y(x,y,\tau))=\tilde E$;

3, $\tilde\Pi_{E_0,E_1,g}$ is normally hyperbolic for $\Phi^{\tau,\tau_0}_{Y}$ with $\tau-\tau_0=\frac{2d}{\lambda_1}\frac {\ln\epsilon}{\sqrt{\epsilon}}$;

4, this channel reaches to a small neighborhood of the flat $\mathbb{F}_0$ in the sense
$$
\min_{c\in\mathbb{W}_g}\alpha_Y(c)-\min_{c\in H^1(\mathbb{T}^2,\mathbb{R})}\alpha_Y(c)=2\epsilon^{1+d}.
$$
For each $c\in\mathbb{W}_g$ with $\alpha(c)\ge 2\epsilon^{1+d}$, the Aubry set entirely stays in the cylinder.
\end{theo}

Let us consider the autonomous Hamiltonian $H$ with the form of (\ref{homogenizedeq1}). Let $\tilde E>\min\alpha_H$, $Y(x,y,\tau)$ be the function solves the equation $H(x,y,x_3,Y(x,y,-x_3))=\tilde E$. Then $Y$ has the form of (\ref{homogenizedeq2}). The energy $E$ of $G$ corresponds to the coordinate $y_3$ for the autonomous Hamiltonian $H$.  Applying Theorem \ref{flatthm5} one obtains
\begin{theo}\label{cylinderthm3}
Assume $\tilde E>\min\alpha_H$. Given an irreducible $g\in H_1(\mathbb{T}^2,\mathbb{Z})$, there is a $3$-dimensional cylinder $\tilde\Pi_{E_0,E_1,g}\subset H^{-1}(\tilde E)$ associated with a channel $\tilde{\mathbb{W}}\subset \alpha^{-1}(\tilde E)$ such that

1, the cylinder $\tilde\Pi_{E_0,E_1,g}$ is a small deformation of the cylinder
$$
\{(x_{E}(x_3),y_{E}(x_3),x^*_3,E): [x_{E}]=g,(x_3,x^*_3)\in\mathbb{T}^2,E\in[E_0,E_1]\};
$$

2, the cylinder $\tilde\Pi_{E_0,E_1,g}$ is invariant for the Hamiltonian flow $\Phi^t_{H}$: for each $(\tilde x,\tilde y)\in \tilde\Pi_{E_0,E_1,g}$, if $\Phi^t_{H}(\tilde x,\tilde y)\notin\tilde\Pi_{\epsilon,g}$ holds for certain $t>0$ then $\exists$ $t'\in(0,t)$ such that $\Phi^{t'}_{H}(\tilde x,\tilde y)$ is on the boundary of $\tilde\Pi_{E_0,E_1,g}$;

3, $\tilde\Pi_{E_0,E_1,g}$ is normally hyperbolic for $\Phi^t_H$  with $t=\frac{2d}{\lambda_1}\frac {\ln\epsilon}{\sqrt{\epsilon}}$;

4, for each $\tilde c=(c,c_3)\in\tilde{\mathbb{W}}$ with $c_3\ge 2\epsilon^{1+d}$, the Aubry set is contained in that cylinder $\tilde{\mathcal{A}}(c)\subset\tilde\Pi_{E_0,E_1,g}$.
\end{theo}

By {\bf H4}, it is generic that there are finitely many bifurcation points, denoted by $c_1,c_2,\cdots,c_m\in\Gamma_g$ corresponding to the energy $E_1,E_2,\cdots E_m$. For each point $c\neq c_i$, the $c$-minimal measure for the Hamiltonian $\bar G_{\epsilon}$ is uniquely supported on a hyperbolic periodic orbit. For the class $c_i$, the Mather set is composed of two hyperbolic periodic orbits. The periodic orbits for $c\in (c_{i},c_{i+1})$ constitute a piece of cylinder $\Pi_i$. As these periodic orbits are hyperbolic, the cylinder $\Pi_i$ can be extended a little bit consisting of periodic orbits which are hyperbolic also, but do not support minimal measure. They are local minimal. These cylinders are clearly invariant and normally hyperbolic for the Hamiltonian flow $\Phi_{\bar G}^t$ with suitably large $t$. Applying the theorem of normally hyperbolic manifold, one can see that there exists some $\tilde\Pi_{i,\epsilon}$ which is invariant for the Hamiltonian flow determined by $G$, and keeps close to $\Pi_{i}\times\sqrt{\epsilon}\mathbb{T}$. For each $c\in\Gamma_g$, the Mather set for $G$ stays in the cylinder. Except for finitely many $c_{i,\epsilon}$ very close to $c_{i}$, the time-$\sqrt{\epsilon}$-section of the Mather set for other $c\in\Gamma_{g}$ is an invariant circle, or periodic points or Aubry-Mather set in the cylinder, for bifurcation point $c=c_{i,\epsilon}$, the Mather set consists of two parts, one stays in $\tilde\Pi_{i,\epsilon}$ another one stays in $\tilde\Pi_{i+1,\epsilon}$.

\subsection{Transition from double to single resonance}
Some normally hyperbolic invariant cylinder has been shown to reach $O(\epsilon^{\frac12+d})$-neighborhood of the double resonant point. This cylinder extends to the place a bit far away from the double resonant point. To see how to transit from double resonance to single resonance, let us homogenize the Hamiltonian in a region $\|y-y_j\|\le O(\sqrt{\epsilon})$ and choose different $y_j$. Recall that the normal form remains valid in the domain $\{\|y\|\le O(\epsilon^{\kappa})\}$ ($\frac 16<\kappa\le\frac 13$). The $\sqrt{\epsilon}$-neighborhood of the curve is covered by as many as $O(\epsilon^{\kappa-\frac 12})$ small balls with radius $O(\sqrt{\epsilon})$. Such approach is based on the following:
\begin{pro}
For nearly integrable Lagrangian $L(x,\dot x,t)=\ell(\dot x)+\epsilon\ell_1(x,\dot x,t)$, each orbit in Mather set  $(\gamma,\dot\gamma)$
$$
\|\dot\gamma(t)-\dot\gamma(0)\|\le O(\sqrt{\epsilon}),\qquad \forall\ t\in\mathbb{R}.
$$
\end{pro}
The result is proved in \cite{BK} for time-1-map. It is also true for the Hamiltonian flow  i.e.  $\|y(t)-y(0)\|\le O(\sqrt{\epsilon})$. It makes sense for us to homogenize the Hamiltonian in the range $\|y-y_j\|\le K\sqrt{\epsilon}$ with suitably large $K>0$.

Using new variables $y-y_j=\sqrt{\epsilon}p$ and $s=\sqrt{\epsilon}\tau$, the homogenized Hamiltonian equation turns out to be the following form
$$
\frac{dx}{ds}=\frac{\omega}{\sqrt{\epsilon}}+Ap,\qquad
\frac{dp}{ds}=-\frac{\partial V}{\partial x}(x,y_j),
$$
where $A=\partial^2h(y_j)$, the frequency $\omega=\partial h(y_j)$ satisfies a resonant condition. The corresponding Lagrangian reads
$$
L(\dot x,x)=\frac 12\Big\langle A^{-1}\Big(\dot x-\frac{\omega}{\sqrt{\epsilon}}\Big),\Big(\dot x-\frac{\omega} {\sqrt{\epsilon}}\Big)\Big\rangle-V(x)
$$

With the potential $V(x)$ on the torus one associates its time average $[V]$ along the orbits of the linear flow defined by $\omega$:  $x\to x+\omega t$
$$
[V](x)=\frac 1T\int_0^TV(x+\omega t)dt,
$$
where $T$ is the period of the frequency $\omega$. The function $[V]$ is then defined on a circle. Let $x_0$ be the maximal point of $[V]$, it corresponds to a circle on $\mathbb{T}^2$. The averaged Hamiltonian is also associated with a Lagrangian
$$
[L](\dot x,x)=\frac 12\Big\langle A^{-1}\Big(\dot x-\frac{\omega}{\sqrt{\epsilon}}\Big),\Big(\dot x-\frac{\omega} {\sqrt{\epsilon}}\Big)\Big\rangle-[V](x).
$$
Let $T_{\omega,\epsilon}$ be the period of the frequency $\omega/\sqrt{\epsilon}$, $\xi_{\omega,\epsilon}$: $[0,T_{\omega,\epsilon}]\to\mathbb{T}^2$ be the minimizer of the action
$$
\inf_{[\xi]=g_{\omega}}\int_0^{T_{\omega,\epsilon}}[L](d\xi(s))ds,
$$
then it is a curve of maximal points of $[V]$ with constant speed $\dot\xi_{\omega,\epsilon}= \omega/\sqrt{\epsilon}$. Consider $[V]$ as a function defined on $\mathbb{T}^2/\xi_{\omega,\epsilon}$ and denote by $[V]''$ the second derivative for the variable of $\mathbb{T}^2/\xi_{\omega,\epsilon}$, where we use $\xi_{\omega,\epsilon}$ to denote the circle $\cup_{t\in[0,T_{\omega,\epsilon}]} \xi_{\omega,\epsilon}(t)$.

Let $\xi_{\omega,\epsilon}+\delta$ denote a translation of $\xi_{\omega,\epsilon}$ such that $d(\xi_{\omega,\epsilon}+\delta,\xi_{\omega,\epsilon})=\delta$ and let $\gamma_{\omega,\epsilon}$: $[0,T_{\omega,\epsilon}]\to\mathbb{T}^2$ be the minimizer of the action
$$
\inf_{[\xi]=g_{\omega}}\int_0^{T_{\omega,\epsilon}}L(d\xi(s))ds,
$$
then we have
\begin{pro}
Assume $-[V]$ is non-degenerate at its minimal point $-[V]''>\Lambda$, and assume some $\lambda>0$ exists so that $T_{\omega,\epsilon}=\epsilon^{\lambda}$. Then there exist some constants $D,D'>0$ such that the minimizer $\gamma_{\omega,\epsilon}$ entirely stays in $D\epsilon^{\lambda}$-neighborhood of the circle $\xi_{\omega,\epsilon}+\delta$, i.e. $d(\gamma(s), \xi_{\omega,\epsilon}+\delta)< D\epsilon^{\lambda}$ holds for each $s\in [0,\epsilon^{\lambda}]$ and $|\delta|\le D'\epsilon^{\lambda/2}$.
\end{pro}
\begin{proof}
Since the minimizer solves the Lagrange equation, its second derivative remains bounded. Thus, as the first step, we claim that the minimizer stays entirely in $D\epsilon^{\lambda}$-neighborhood of $\xi_{\omega,\epsilon}+\delta$, a translation of the circle $\xi_{\omega,\epsilon}$. If not, the oscillation of $\gamma_{\omega,\epsilon}$ in the direction perpendicular to $\omega$ would not be smaller than $2D\epsilon^{\lambda}$. As the average speed is $O(\epsilon^{-\lambda})$, there would be a point on the minimizer where $\|\dot{\gamma}_{\omega,\epsilon} -\omega/\sqrt{\epsilon}\|>D$. Since the potential is bounded $|V|<C_{20}$, one obtains that $A_L(\gamma_{\omega,\epsilon})\ge (C_{21}D^2-C_{20})\epsilon^{\lambda}>C_{20}\epsilon^{\lambda}$ if we choose $D^2>2C_{21}^{-1}C_{20}$. On the other hand, the action along the curve $\gamma(t)=x_0+\omega/\sqrt{\epsilon}t$ would be not bigger then $C_{20}\epsilon^{\lambda}$. The contradiction implies our claim.

Let us compare the action of $L$ along the curve $\gamma_{\omega,\epsilon}$ with that along the curve $\xi_{\omega,\epsilon}$. If $|\delta|>D'\epsilon^{\lambda/2}$, by what we have proved, some $x_1\in\mathbb{T}^2/\xi_{\omega,\epsilon}$ exists such that
$$
|\gamma_{\omega,\epsilon}(t)-(x_1+\omega/\sqrt{\epsilon}t)|\le D\epsilon^{\lambda}, \qquad |(x_1-x_0)/\xi_{\omega,\epsilon}|\ge D'\epsilon^{\lambda/2}.
$$
It follows that
\begin{align*}
A(\gamma_{\omega,\epsilon})-A(\xi_{\omega,\epsilon})=&\frac 12\int_0^{\epsilon^{\lambda}} \Big\langle A^{-1}\Big(\dot\gamma_{\omega,\epsilon}(t)-\frac{\omega}{\sqrt{\epsilon}}\Big), \Big(\dot\gamma_{\omega,\epsilon}(t)-\frac{\omega}{\sqrt{\epsilon}}\Big)\Big\rangle dt\\
&-\int_0^{\epsilon^{\lambda}}\Big(V(\gamma_{\omega,\epsilon}(t))-V\Big(x_0+\frac{\omega} {\sqrt{\epsilon}} t\Big) \Big)dt\\
>&-\int_0^{\epsilon^{\lambda}}\Big(V(\gamma_{\omega,\epsilon}(t))-V\Big(x_1+\frac{\omega} {\sqrt{\epsilon}} t\Big) \Big)dt\\
&+(-[V](x_1)+[V](x_0))\epsilon^{\lambda}\\
>&\Big(\frac 12{C_{22}}D'^2-|\max\partial V|\Big)\epsilon^{2\lambda}
\end{align*}
it contradicts the minimality of the curve $\gamma_{\omega,\epsilon}$ if $D'>0$ is chosen suitably large. \end{proof}

Recall the picture of minimal periodic orbit close to double resonance, one can see from this proposition how the shape of the periodic orbit changes when it moves away from double resonance to single resonance.

\section{\ui Annulus of incomplete intersection}
\setcounter{equation}{0}

Let us also start with the Hamiltonian $G_{\epsilon}$ defined by Formula (\ref{cylindereq3}), it has two and half degrees of freedom. Given any two homology class $g,g'\in H_1(\mathbb{T}^2,\mathbb{Z})$, The theorem \ref{cylinderthm2} confirms the existence of two wedge-shaped regions $\mathbb{W}$ and $\mathbb{W}'$ which reach to the boundary of the annulus
$$
\mathbb{A}_0=\Big\{c\in H^1(M,\mathbb{R}):0<\alpha_{G_{\epsilon}}(c)-\min\alpha_{G_{\epsilon}}<D\epsilon^{1+d}
\Big\},
$$
For each class in $\mathbb{W}$ and $\mathbb{W}'$, the Aubry set lies in the normally hyperbolic cylinder and, by the result for {\it a priori} unstable systems, can be connected to other Aubry set lying in the cylinder under certain generic conditions. However, it seems unclear whether these two wedges can reach to the flat $\mathbb{F}_0$. Thus, a notable difficulty rises as these cylinders are separated by an annulus $\mathbb{A}_0$ around the flat $\mathbb{F}_0$, it is the problem of crossing double resonance.

It is the goal of this section to find an annulus $\mathbb{A}\supsetneq\mathbb{A}_0$ where those two wedge-shaped regions are plugged into and for each class in that annulus, the stable set of the Aubry set ``intersects" the unstable set non-trivially, possibly incomplete. In other words, for each class in this region, the Ma\~n\'e set does not cover the whole configuration space.

\subsection{The Ma\~n\'e set for $c\in\partial^* \mathbb{F}_0$}

As the first step, let us consider the Hamiltonian $\bar G$ and study all cases when the Ma\~n\'e set covers the whole configuration space.

For each $c\in\partial^* \mathbb{F}_0$, except for the minimal measure $\mu$ supported on the fixed point $(x,\dot x)=0$, some minimal measure exists with non-zero rotation vector. In the covering space $\bar\pi$: $\mathbb{R}^2\to\mathbb{T}^2$, a disk $B_{\delta}(0)$ is contained in a strip bounded by two $c$-static curves $\xi_c$ and $\xi'_c$, both curves are in the Mather set: $\bar\pi\xi_c, \bar\pi\xi'_c\subset\mathcal{M}(c)$ no other $c$-static curve in the Mather set touches the interior of this strip. Let $U^{\pm}_c$ and $U'^{\pm}_c$ denote the elementary weak KAM solution determined by $\xi_c$ and $\xi'_c$ respectively, we investigate what happens when $U^-_c-U'^+_c=0$ holds in this strip. As the configuration space is two dimensional, for each $x$ in this region, $(x,y)=(x,\partial U^-_c(x))=(x,\partial U'^+_c(x))$ uniquely determines a $c$-semi static curve which lies entirely in this strip. The $c$-semi static curves considered here are all determined by $U_c^-=U_c'^+$. It is possible that some curve approaches to the origin as $t\to\infty(-\infty)$, in this case, because of {\bf H3}, it approaches to the curve $\xi_c$ ($\xi_c'$) as $t\to -\infty$($\infty$).

Supported on the fixed point, the measure $\mu$ is minimal for all $c\in\mathbb{F}_0$. Thus, there always exists some semi-static curve $\gamma_c^{\pm}$ connecting the fixed point to the support of other $c$-minimal measure $\mu_c$
$$
\lim_{t\to\pm\infty}\gamma_c^{\pm}(t)=0, \ \ \text{\rm and}\ \ \lim_{t\to\mp\infty}\gamma_c^{\pm}(t)\to\pi_x\text{\rm supp}\mu_c.
$$
As all eigenvalues are assumed different, generically, all minimal homoclinic curves approach to the fixed point in the direction $\Lambda_{1,x}$, associated to the smallest eigenvalue:
$$
\lim_{t\to\pm\infty}\frac {\dot\gamma_{c_i}^{\pm}(t)}{\|\dot\gamma_{c_i}^{\pm}(t)\|}
=\pm\Lambda_{1,x}.
$$
But this does not exclude the possibility that some $c$-semi static curves approach to the point in the direction of $\Lambda_{2,x}$. It provides us a criterion to classify the cases when the Ma\~n\'e set covers the whole configuration manifold.

\noindent{\bf Case 1}: no $c$-semi static curve approaches the origin in the direction of $\Lambda_{1,x}$. In this case, as $|\lambda_1|<|\lambda_2|$, there exist exactly two semi-static curves $\gamma^{\pm}_c$ such that $\gamma^{\pm}_c(t)\to 0$ as $t\to\pm\infty$. They approach the origin in the direction $\Lambda_{2,x}$ and
$$
\lim_{t\to\infty}\frac {\dot\gamma_{c}^{+}(t)}{\|\dot\gamma_{c}^{+}(t)\|}
=\lim_{t\to-\infty}\frac {\dot\gamma_{c}^{-}(t)}{\|\dot\gamma_{c}^{-}(t)\|}.
$$

Other cases are classified under the condition that there exist some $c$-static curves approaching the origin in the direction of $\Lambda_{1,x}$. Since the curves $\xi_c$ as well as $\xi'_{c}$ is disjoint with the origin, some number $\delta>0$ exists such that these two curves do not hit the ball $B_{\delta}(0)$. The number $\delta$ seems depending on $c$. Let $\gamma_c^+$ be a semi-static curve approaching the origin as $t\to\infty$, it intersects the circle $\partial B_{\delta}(0)$ at some point. Let $I^{\pm}\subset\partial B_{\delta}(0)$ be such a set that passing through each point $x\in I^{\pm}$ a $c$-semi static curve approaches to the origin, as $t\to\pm\infty$, in the direction of $\Lambda_{1,x}$. By assumption, the set $I^+$ is not empty. Obviously, $I^+$ does not occupy the whole circle and can be made closed by adding at most two points, through which some semi-static curves approach the origin in the direction of $\Lambda_{2,x}$.

Passing through a point $x\in \partial B_{\delta}(0)$ very close to $I^+$, there is a unique $c$-semi static curve, determined by $U^-_c=U_c'^+$. Because of Proposition \ref{flatpro2}, the curve $\gamma_c$ will get very close to the origin and leave in a direction far away from $\Lambda_{1,x}$. Let $I^+_i$ be a connected component of $I^+$, it may be a point or an interval. If it is a point, let $x_i,x'_i\in\partial B_{\delta}(0)$ be two sequences of points such that they approach $I^+_i$ from different sides. Let $\gamma_i$ ($\gamma'_i$) be the semi static curve passing through $x^+_i$ ($x'^+_i$), it shall intersect the circle $\partial B_{\delta}(0)$ at a point $x^-_i$ ($x'^-$) respectively. Some $x^-,x'^-\in\partial B_{\delta}(0)$ exist so that $x^-_i\to x^-$, $x'^-_i\to x'^-$ as $i\to\infty$.

If $x^-=x'^-$, it determines a $c$-semi static curve approaches the origin as $t\to-\infty$. Because of Proposition \ref{flatpro2}, it approaches in the direction of $\Lambda_{2,x}$. This leads to

\noindent{\bf Case 2}: there exists exactly one $c$-semi static curve approaching origin in the direction of $\Lambda_{2,x}$.

If $x^-\ne x'^-$, let $I^-_i$ denote the arc bounded by these two points, not containing $I^+_i$. One can see from the proof of Proposition \ref{flatpro2} that the angle of this arc is not smaller than $\pi/2$. Passing from each point in the interior of the arc, the $c$-semi static curve approaches to the origin as $t\to -\infty$ and these curves constitute a sector. Since the fixed point is hyperbolic, it has its stable and unstable manifolds $W_0^{\pm}$. Thus, some some generating function $U^{\pm}$ and $r>0$ exist such that $W_0^{\pm}|_{B_{r}(0)}=\text{\rm graph}dU^{\pm} |_{B_{r}(0)}$. As the orbits determined by the curves entirely lie in the unstable manifold of the fixed point, one has $U_c^-=U'^+_c=U^-$ in the sector. Therefore, the size of the sector-shaped region is independent of the size of $\delta$. This leads to

\noindent{\bf Case 3}: in the disk $B_{r}(0)$ there is a sector-shaped region with the field angle not smaller than $\pi/2$. In this sector, one has $U_c^-=U'^+_c=U^-$.

Let $\gamma^+_c$ ($\gamma^-_c$) be $c$-semi static curve passing through a point in $I_i^+$ ($I^-_i$) respectively, then they approach the origin in opposite direction as $t\to\pm\infty$ respectively, i.e.
$\lim_{t\to\infty}\dot\gamma_c^+(t)\|\dot\gamma_c^+(t)\|^{-1}=\lim_{t\to-\infty}\dot\gamma_c^-(t) \|\dot\gamma_c^-(t)\|^{-1}$. To verify this claim, let us assume the contrary. Thus, these two curves cut the ball $B_{\delta}(0)$ into two parts, one is sharp wedge-shaped, denoted by $W$.
\begin{figure}[htp]
  \centering
  \includegraphics[width=5.0cm,height=2.5cm]{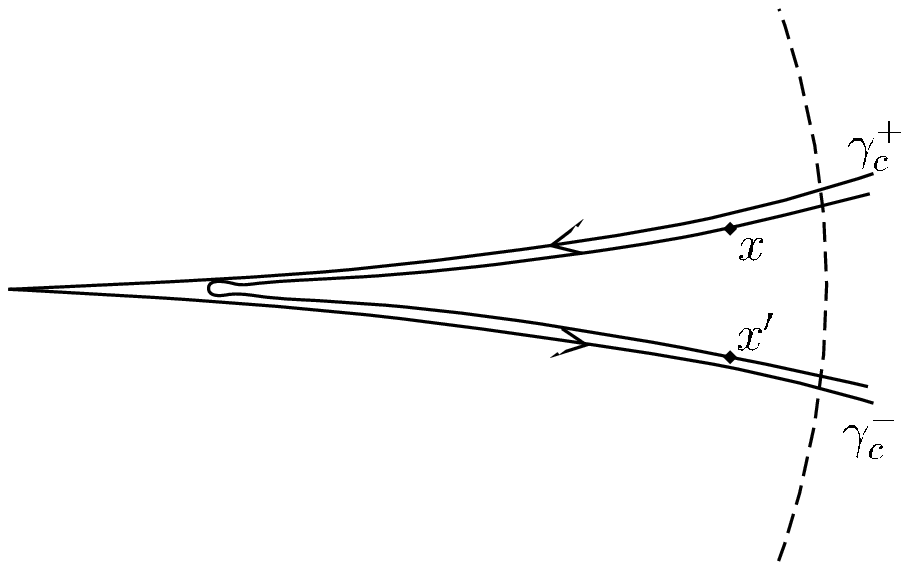}
  \label{fig8}
\end{figure}
We choose a $c$-semi static curve lying in $W$ and keeping very close to the curves $\gamma_c^-$ and $\gamma_c^+$. In canonical coordinates such that $\bar G=\frac 12(p_1^2-\lambda_1x_1^2)+\frac 12(p_2^2-\lambda_2x_2^2)+O(\|(x,p)\|^3)$, the set $W$ has a vertex at the origin. As both $\gamma_c^{+}$ and $\gamma_c^{-}$ approach the origin in the direction of $\Lambda_{1,x}$, there exists $\delta_1\le\delta$ such that $|x_2|\le |x_1|^3$ if $(x_1,x_2)\in W$ and $|x_1|\le\delta_1$. Since the fixed point is hyperbolic, it has local stable and unstable manifold, determined by the generating functions $U^+$ and $U^-$ respectively. Restricted in $W$, these functions satisfy the condition
$$
U^-(x)-U^-(0)\ge\frac{\lambda_1^2}{3}\|x\|^2,\qquad U^+(0)-U^+(x)\ge\frac{\lambda_1^2}{3}\|x\|^2, \ \ \ \forall\ \|x\|\le\delta.
$$
Pick up two points $x$ and $x'$ very close to $\gamma_c^{\pm}$ respectively, through which some $c$-semi static curve $\gamma_c$ passes, namely, some $t'>t$ exist such that $\gamma_c(t)=x$ and $\gamma_c(t')=x'$. Note the orbit determined by $\gamma_c^+$ ($\gamma_c^+$) lies in the stable (unstable) manifold, by definition ones has
$$
A[\gamma_c|_{[t,t']}]\ge\frac 34\Big(U_c^-(x')-U_c^+(x)\Big)\ge\frac {\lambda_1^2}4(\|x'\|^2+\|x\|^2).
$$
If we choose $x$ sharing the same first coordinate with $x'$ and connect them with a straight line $\zeta$: $[0,|x_2-x'_2|]\to\mathbb{T}^2$, then $|\dot\zeta|\le O(1)$ and the action along this curve one has $A[\zeta]\le O(\|x\|^3)$. It contradicts the minimality of $\gamma_c$, thus the claim is proved.

We claim that $I^+$ has only one connected component. Otherwise, there would be two connected component $I^+_k$ and $I^+_i$. By the definition, passing through a point $x\in I^+_j$ ($x'\in I^+_k$) there is a $c$-semi static curve $\gamma_x$ ($\gamma_{x'}$) which approaches the curve $\xi_c$ as $t\to\infty$ and approaches the origin as $t\to\infty$. These two curves divided the strip into two parts $S=S_1\cup S_2$, where $S_1$ is such a strip that passing through any point $x^*\in S_1$, the $c$-semi static curve will approach the origin as $t\to\infty$. However, there exists a point $x^*\in S_1\cap(\partial B_{\delta}\backslash I^+)$, it implies that passing thorough $x^*$, the $c$-semi-static curve will approach the curve $\xi'_c$, namely, it would intersect either $\gamma_x$ or $\gamma_{x'}$. It is absurd. Thus, we obtain the left picture in Figure \ref{fig9}.
\begin{figure}[htp]
  \centering
  \includegraphics[width=9.7cm,height=3.7cm]{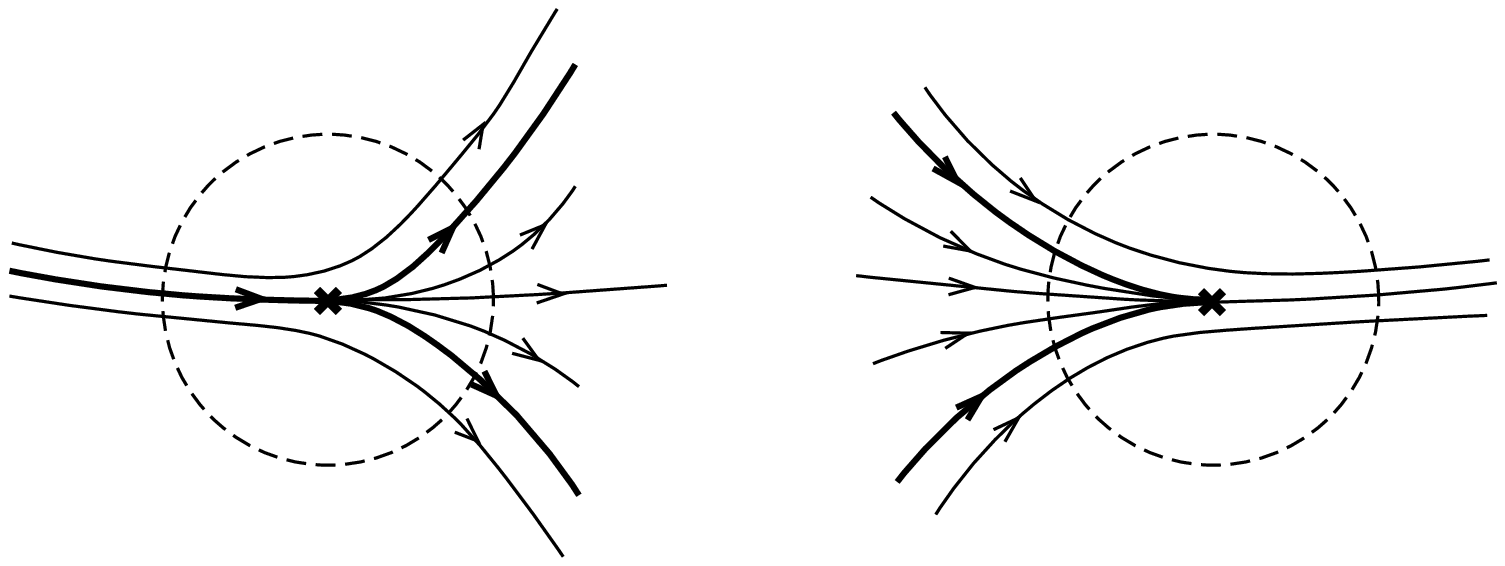}
  \caption{}
  \label{fig9}
\end{figure}

By similar argument applying to the set $I^-$, we have  either the case 2 again or

\noindent{\bf Case 4}: in the disk $B_{r}(0)$ there is a sector-shaped region with the field angle not smaller than $\pi/2$, where $U_c^-=U'^+_c=U^+$, see the right picture in Figure \ref{fig9}.

We claim that all of these cases do not occur for generic potential $V$. The first two cases takes place at most for four invariant measures, as there are only four curves which approaches the origin in the direction of $\Lambda_{2,x}$. Each of these curves approaches at most one Mather set. These Mather sets correspond to at most four edges of $\mathbb{F}_0$. Let $V_{\delta}-V$ be a non-negative function such that its support does not touch these four curves as well as the support of the minimal measure. By perturbing the potential $V\to V_{\delta}$, one can see that the Mather set remains unchanged, but the Ma\~n\'e set does not cover $\mathbb{T}^2$ for each of these four cohomology classes.

Both case 3 and 4 take place also for at most four Mather sets, as each sector-shaped region has the field angle not smaller than $\pi/2$, and the orbit determined by $(x,y)=(x,\partial_xU^{\pm})$  approaches one Mather set only. Let us destruct it one by one. If some sector-shaped region $S^+\subset B_r(0)$ exists where $U_c^-=U_c'^+=U^+$, $\pi_x\text{\rm supp}\mu_c\cap S^+=\varnothing$. We divide it into three sub-sectors $S^+=S^+_1\cup S^+_2\cup S^+_3$, each of which is composed of $c$-semi static curves approaching the origin as $t\to\infty$ and $S^+_1$ is disjoint with $S^+_3$.  We introduce another potential $V_{\delta}$ such that  the function $V_{\delta}-V$ is non-negative, $\text{\rm supp}(V_{\delta}-V)\subset S^+_2\backslash B_{r_1}$ ($r_1<r$).

For the perturbed Lagrangian determined by $\bar G_{\delta}=\frac 12\langle Ap,p\rangle+V_{\delta}(x)$, the minimal measure for the class $c$ is the same as that for unperturbed Hamiltonian. Let $U^-_{c,\delta}$, $U'^+_{c,\delta}$ be the elementary weak KAM solutions of the perturbed Hamiltonian, associated to the minimal measure $\mu_c$ and $\mu'_c$ respectively, one has
$$
\arg\min(U^-_{c,\delta}-U'^+_{c,\delta})\cap \text{\rm supp}(V_{\delta}-V)=\varnothing,\qquad
\arg\min(U^-_{c,\delta}-U'^+_{c,\delta})\supset S^+_1\cup S^+_3.
$$

Under such perturbation, there might be another cohomology class $c'$ such that $U_{c'}^--U'^+_{c'}=0$ holds on the whole torus and a sector $S^-$ exists where $U_{c'}^-=U'^+_{c'}=U^-$. Note $\pi_x\text{\rm supp}\mu_{c'}\cap S^-=\varnothing$, we split it into three sub-sectors $S^-=S^-_1\cup S^-_2\cup S^-_3$, each of which is composed by $c'$-semi static curves approaching to the origin as $t\to -\infty$ and $S^-_1$ is disjoint with $S^-_3$.  We introduce again a perturbed potential $V_{\delta}$ such that  the function $V'_{\delta}-V$ is non-negative, $\text{\rm supp}(V'_{\delta}-V)\subset S^-_2\backslash B_{r_1}$ ($r_1<r$).

For the new perturbed Lagrangian, the minimal measure for the class $c'$ is the same as that for unperturbed one. Let $U^-_{c',\delta}$, $U'^+_{c',\delta}$ be the elementary weak KAM solutions of the perturbed Hamiltonian, determined by $\xi_{c'}$ and $\xi'_{c'}$ respectively, one also has
$$
\arg\min(U^-_{c',\delta}-U'^+_{c',\delta})\cap \text{\rm supp}(V_{\delta}-V)=\varnothing,\qquad
\arg\min(U^-_{c',\delta}-U'^+_{c',\delta})\supset S^-_1\cup S^-_3.
$$

For suitably small $r>0$, the Hamiltonian flow determined by $\bar G$ is well approximated by its linearized flow when they are restricted in the ball $B_r(0)$. For the linearized flow,  if $(x(t),y(t))$ is an orbit in the unstable manifold, $(x(-t),-y(-t))$ is an orbit in the stable manifold. Therefore, some sectors $\check S^{\pm}_k$ ($k=1,3$) exist so that $\check S^{\pm}_k\subset S^{\pm}_k$, $\check S^{-}_k\cap S_2^+=\varnothing$ and $\check S^{+}_k\cap S_2^-=\varnothing$ hold for $k=1,3$, each $\check S^{+}_k$ consists of $c$-semi static curves which approach the origin as $t\to\infty$ with
$$
\arg\min(U^-_{c,\delta}-U'^+_{c,\delta})\supset\check S^{+}_1\cup\check S^{+}_3,\ \ \ \arg\min(U^-_{c,\delta}-U'^+_{c,\delta})\cap \text{\rm supp}(V_{\delta}-V)=\varnothing,
$$
each $\check S^{-}_k$ consists of $c'$-semi static curves which approach the origin as $t\to -\infty$ and
$$
\arg\min(U^-_{c',\delta}-U'^+_{c',\delta})\supset\check S^{-}_1\cup\check S^{-}_3,\ \ \ \arg\min(U^-_{c',\delta}-U'^+_{c',\delta})\cap \text{\rm supp}(V_{\delta}-V)=\varnothing.
$$
Since there are at most two sectors corresponding to unstable manifold and two sectors corresponding to stable manifold, there are at most four pairs of static curves $(\xi_{c_i},\xi'_{c_i})$ ($i=1,2,3,4$) for which the case 3 and 4 takes place.

These four pairs of static curves $(\xi_{c_i},\xi'_{c_i})$ ($i=1,2,3,4$) corresponds to four edges (points) contained in $\partial^*\mathbb{F}_0$. By construction, the Ma\~n\'e set does not cover the whole torus $\mathbb{T}^2$ for each cohomology class one these four edges.  For any other class $c\in\partial^*\mathbb{F}_0$, the Ma\~n\'e set can not cover the whole torus also. Otherwise, there would be a sector $S^{\pm}$ of $B_r(0)$ where $U_c^-=U_c'^+=U^{\pm}$, but it is absurd since some $\check S^{\pm}_i\subset S^{\pm}$ where $U_{c_i}^-=U_{c_i}'^+=U^{\pm}$ holds for some $i\in (1,2,3,4)$. For each $x\in\check S^{\pm}_i$, $(x,v=\partial_y\bar G(x,\partial U^{\pm}(x))$ determines an orbit of the Lagrangian flow which approaches both to the support of $\mu_{c_i}$ and to the support of $\mu_{c}$ as $t\to\pm\infty$, it is impossible. Therefore, we have

\begin{lem}\label{beltlem1}
It is an open and dense condition for the potential $V$ that for all class $c\in\partial^*\mathbb{F}_0$, the Ma\~n\'e set does not cover the torus: $\mathcal{N}(c)\subsetneq\mathbb{T}^2$.
\end{lem}

\subsection{The Ma\~n\'e set for $c\in\partial\mathbb{F}_0\backslash\partial^* \mathbb{F}_0$}
This set contains at most countably many vertexes. Indeed, if both $\partial\mathbb{F}_0\backslash\partial^*\mathbb{F}_0$ and $\partial^*\mathbb{F}_0$ are non-empty, there do exist countably many vertexes (cf. Theorem \ref{flatthm3}).

Let $E_i\subset\partial\mathbb{F}_0\backslash\partial^*\mathbb{F}_0$ be an edge joined to other two edges at the vertex $c_i$, $c_{i+1}$ respectively. By Theorem \ref{flatthm3}, the Aubry set for $c_j$ consists of two minimal homoclinic curves $\gamma_{j-1}$ and $\gamma_j$. Denote by $g_j\in\mathbb{Z}^2$ the homology class of $\gamma_j$, then the matrix $(g_{j-1},g_j)$ is uni-module. By introducing suitable coordinates on $\mathbb{T}^2$, we can assume $g_i=(1,0)$. In this coordinate system, $g_{i-1}=(k,1)$ and $g_{i+1}=(k',-1)$.

\begin{figure}[htp] 
  \centering
  \includegraphics[width=7.0cm,height=2.5cm]{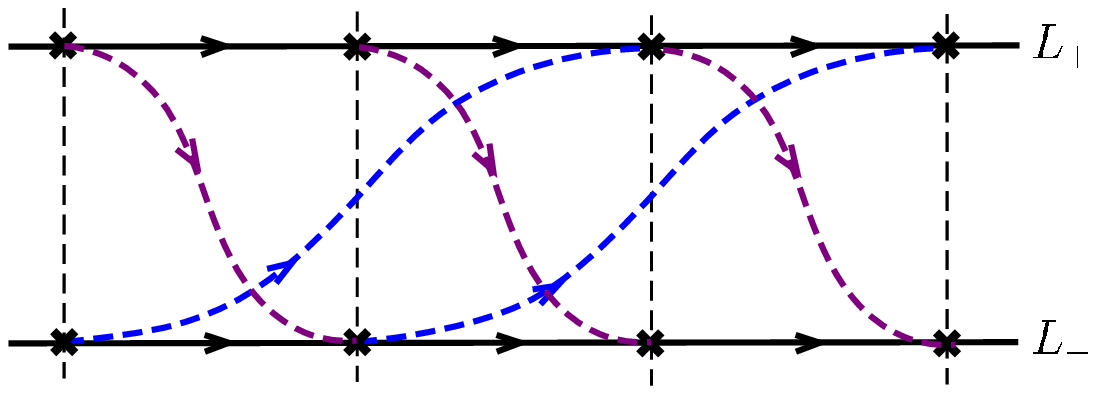}
  \caption{}
  \label{fig10}
\end{figure}
In this figure, each unit square represents a fundamental domain of $\mathbb{T}^2$ in the universal covering space, the horizontal line represents the lift of the homoclinic curve $\gamma_i$, which stays in the Aubry set for each $c\in E_i$. The blue dashed lines represent the lift of the $\gamma_{i-1}$ which stays in the Abury set for the class at one end-point of $E_i$. The purple dashed lines represents the lift of the $\gamma_{i+1}$ which stays in the Abury set for the class at another end-point of $E_i$.

Let us consider weak KAM solution $U_{i,\pm}^{\pm}$ in the strip bounded by the lines $L_+$ and $L_-$. According to Lemma \ref{cylinderlem1}, for $E-\min\alpha>0$ very small, there exists an interval $I_E\subset H^1(\mathbb{T}^2,\mathbb{R})$ such that for each $c\in I_E$, the Mather set consists of only one closed curve $\gamma_E$ such that $[\gamma_E]=g_i$. In the universal covering space, let $\bar\gamma_E$ be a component of the lift of $\gamma_E$ which approaches $L_+$ as $E\downarrow\min\alpha$. Let $U^{\pm}_{i,E}$ be the elementary weak KAM solution determined by $\bar\gamma_E$. As $E\downarrow\min\alpha$, $U^{\pm}_{i,E}\to U^{\pm}_{i,+}$ in $C^0$-topology. The function $U^{\pm}_{i,-}$ is obtained in the same way. The function
$U_{i,\pm}^{-}$ determines backward semi-static curves approaching to the line $L_{\pm}$ as the time approaches to minus infinity, $U_{i,\pm}^{+}$ determines forward semi-static curves approaching to the line $L_{\pm}$ as the time approaches to positive infinity respectively.

If we remove the coercive condition on these weak KAM solutions that $(x,\partial U^{\pm}(x))$ determines a backward (forward) semi-static curve approaching $L_+$ ($L_-$), then the weak KAM depends on $c\in E_i$ (no longer elementary). As $A_c(\gamma_{i\pm1})>0$ for each $c\in\text{\rm int}E_i$, starting from a point close to the line $L_+$ ($L_-$), the backward (forward) semi-static curve will approach to $L_+$ ($L_-$).

Let $c_{\lambda}=\lambda c_i+(1-\lambda)c_{i+1}$. For each $\lambda\in (0,1)$, by Proposition \ref{weakpro3}, the strip is divided into two connected parts $D_{\lambda}^+$ and $D_{\lambda}^-$ such that $U_{c_\lambda}^{+}|_{D_{\lambda}^+}=U^{+}_{i,+}$ and $U_{c_\lambda}^{+}|_{D_{\lambda}^-}=U^{+}_{i,-}$. Let $\gamma_{x,{\pm}}^+$ be the forward semi-static curve determined by $U_{i,\pm}^+$, starting from the point $x$. If $x\in D_{\lambda}^+$, the curve $\gamma^+_{x,+}$ is calibrated for $U^+_{i,+}$ which induces
\begin{align}\label{belteq-1}
A_L(\gamma_{x,{+}}^+)-\langle \gamma_{x,{+}}^+(\infty)-x,c_{\lambda}\rangle&= U^+_{c_{\lambda}}(\gamma^+_{x,+}(\infty))-U^+_{c_{\lambda}}(x)\\
&<A_L(\gamma_{x,{-}}^+)-\langle \gamma_{x,{-}}^+(\infty)-x,c_{\lambda}\rangle,\notag
\end{align}
where both $\gamma^+_{x,+}(\infty)$ and $\gamma^+_{x,-}(\infty)$ exist. Note  $\pi_{\infty}\gamma^+_{x,+}(\infty)= \pi_{\infty}\gamma^+_{x,-}(\infty)$ where $\pi_{\infty}:\mathbb{R}^2\to\mathbb{T}^2$ is the standard projection. One can see from Figure \ref{fig10} that
$$
\langle \gamma_{x,{+}}^+(\infty)-\gamma_{x,{-}}^+(\infty),c_{i}-c_{i+1}\rangle>0.
$$
It follows that the inequality (\ref{belteq-1}) also holds for $\lambda'>\lambda$, which implies that $x\in D_{\lambda'}^+$ also. Clearly, $D_{\lambda}^+$ expands and $D_{\lambda}^-$ shrinks as $\lambda$ increases. As the limit, we see that $D_{1}^-$ and $D_{0}^+$ occupies the whole strip.  Therefore, we have:

\begin{pro}\label{beltpro1}
Assume $E_i\subset\partial\mathbb{F}_0\backslash\partial^*\mathbb{F}_0$ be an edge joined to other two edges at the vertex $c_i$, $c_{i+1}$ respectively. Let  $U^{\pm}_j$ be the globally elementary weak KAM for $c_{\lambda}=c_j$ with $j=i,i+1$. Then, for each $c=\lambda c_i+(1-\lambda)c_{i+1}\in E_i$, the weak KAM solution is completely determined by $U^{\pm}_i$ and $U^{\pm}_{i+1}$ in the following sense: there is a partition $\mathbb{T}^2=D_{i,\lambda}^{\pm}\cup D_{i+1,\lambda}^{\pm}$ such that $D^{\pm}_{j,\lambda}$ is connected,
$$
U_{c_{\lambda}}^{\pm}|_{D_{j,\lambda}^{\pm}}=U_j^{\pm}|_{D_{j,\lambda}^{\pm}},
$$
where $D_{i,\lambda}^{\pm}\subset D_{i,\lambda'}^{\pm}$ and $D_{i+1,\lambda}^{\pm}\supset D_{i+1,\lambda'}^{\pm}$ if
$\lambda<\lambda'$, and $D_{i+1,0}^{\pm}=D_{i,1}^{\pm}=\mathbb{T}^2$.
\end{pro}

In virtue of this proposition, we see that, for all $c\in\partial\mathbb{F}_0\backslash\partial^*\mathbb{F}_0 $, the set of barrier functions $\{U_c-U_c^+\}$ are determined by countably many weak KAM solutions. For each edge $E_i$, there is  an open-dense set in $C^{r}$ space such that $\arg\min(U_c-U_c^+)\subsetneq \mathbb{T}^2$ holds for each $c\in E_i$ if the potential $V$ takes value in this set. Therefore, in virtue of the lemma \ref{beltlem1}, we have

\begin{theo}\label{beltthm1}
Let $\bar L=\frac 12\langle A^{-1}\dot x,\dot x\rangle-V(x)$.  A residual set  $\mathfrak{V}\subset C^r(\mathbb{T}^2,\mathbb{R})$ exists such that for each $V\in\mathfrak{V}$ and for each $c\in\partial\mathbb{F}_0$, the Ma\~n\'e set does not cover the whole configuration space: $\arg\min(U_c^--U_c^+)\subsetneq \mathbb{T}^2$.
\end{theo}

Because of this theorem, we have generic hypothesis

\noindent ({\bf H5}): The potential is chosen so that the Ma\~n\'e set does not cover the whole torus $\mathcal{N}(c)\subsetneq \mathbb{T}^2$ for each $c\in\partial\mathbb{F}_0$.

\subsection{Thickness of the annulus}

The Hamiltonian under consideration is given by the formula (\ref{homogenizedeq1}) which is autonomous. To start with, let us assume assume $\min\alpha_G=0$ and study elementary weak KAM solutions of the Hamilton-Jacobi equation
\begin{equation}\label{belteq0}
\partial_{\tau}u+G(x,\tau,\partial_xu+c)=\epsilon\Delta, \qquad c\in\partial\mathbb{F}_0
\end{equation}
where $G$ solves the equation $H(x,x_3,y,G)=\tilde E$ and $\tau=-x_3$. Expanding the Hamiltonian into Taylor series of $\epsilon$, replacing $u$ by $\sqrt{\epsilon}u$ and rescaling $\tau$ by $s=\sqrt{\epsilon}\tau$ we obtain the equation
\begin{equation}\label{belteq1}
\frac{\partial u}{\partial s}+\frac 12\Big\langle A\Big(\frac{\partial u}{\partial x}+c\Big),\frac{\partial u}{\partial x}+c\Big\rangle +V(x)+O(\sqrt{\epsilon})=\Delta
\end{equation}
where $A$ is the Hessian matrix of $h$ in $y$ at $y=0$, $V(x)=Z(x, 0)$ (cf. formula (\ref{cylindereq3})). In this equation, only higher order term depends on the time $s$. As the first step to study the weak KAM of this Hamilton-Jacobi equation, we omit the higher order term and let $\Delta=0$. Since the potential $V$ is independent of the time $s$, all elementary weak KAM solutions discussed in the last two subsections solve the equation
\begin{equation}\label{belteq2}
\frac 12\Big\langle A\frac{\partial u}{\partial x}+c, \frac{\partial u}{\partial x} +c\Big\rangle +V(x)=0,
\end{equation}
and for each $V\in\mathfrak{V}\subset C^r(\mathbb{T}^2,\mathbb{R})$ and each $c\in\partial\mathbb{F}_0$, the weak KAM solutions of this equation define a Ma\~n\'e set which does not cover the whole torus: $\arg\min(U_c^--U'^+_c)\subsetneq\mathbb{T}^2$. By the upper semi-continuity of the set of semi-static curves, some small $\Delta_0>0$ exists such that for each positive $\Delta\le\Delta_0$ and each $c\in\alpha^{-1}_{\bar G}(\Delta)$ (We use $\alpha_{\bar G}$ to denote the $\alpha$-function determined by the Hamiltonian $\bar G$) the Ma\~n\'e set does not cover the torus. It implies that $\arg\min(U_c^--U'^+_c) \subsetneq\mathbb{T}^2$ holds if $c\in\alpha_{\bar G}^{-1}(\Delta)$, both $U^-_c$ and $U'^+_c$ are the weak KAM solutions of the equation \ref{belteq2} with $\Delta\le\Delta_0$.

For each average action $\Delta>\min\alpha_{\bar G}$, the dynamics on the energy level $\bar G^{-1}(\Delta)$ is similar to twist and area-preserving maps. First of all, the rotation vector of each minimal measure is not zero. Thus, any minimal measure is not supported on fixed points. Secondly, for each class $c\in\alpha_{\bar G}^{-1}(\Delta)$, all $c$-minimal measures share the same rotation rotation direction, otherwise, the Lipschitz graph property will be violated. If the rotation direction is rational, each ergoidc minimal measure is supported on a periodic orbit.

From these properties one derives the following: for each $c\in\alpha^{-1}_{\bar G}(\Delta)$, there exists a circle $\Gamma_c\subset\mathbb{T}^2$ such that each semi-static curve passes through it transversally and $\arg\min(U_c^--U_c^+)\cap\Gamma_c\subsetneq\Gamma_c$. Since the set $\arg\min(U_c^--U_c^+)$ is closed, there exist finitely many intervals $I_{c,i}\subset\Gamma_c$ disjoint to each other such that $(\arg\min(U_c^--U_c^+)\cap\Gamma_c)\subset\cup I_{c,i}$.

As these functions are independent of $s=\sqrt{\epsilon}\tau$, all of these functions can be thought as the weak KAM solutions of the homogenized Hamilton-Jacobi equation
$$
\frac{\partial u}{\partial s}+\frac 12\Big\langle A\Big(\frac{\partial u}{\partial x}+c\Big),\frac{\partial u}
{\partial x}+c\Big\rangle +V(x)=\Delta
$$
if they are thought as the function of the variable $(x,s)$. Here, the cohomology class takes value on the circle:
$c\in\alpha^{-1}_{\bar G}(\Delta)$. It follows that, for each class $c\in\alpha^{-1}_{\bar G}(\Delta)$, there exists non-degenerate embedded two-torus $\Gamma_c\times\mathbb{T}\subset\mathbb{T}^3$ and finitely many intervals $I_{c,i}\subset\Gamma_c$ disjoint to each other such that each $c$-semi-static curve passes through the two-torus transversally and $\arg\min(U_c^--U_c^+)\cap(\Gamma_c\times\mathbb{T})\subset\cup I_{c,i}\times\mathbb{T}$. Here the circle $\mathbb{T}$ is for the time $s=\sqrt{\epsilon}\tau$.

Let us return back to the Hamilton-Jacobi equation (\ref{belteq1}). Recall the normal form of the Hamiltonian, we see that in the remainder $O(\sqrt{\epsilon})$, one contribution is from $R(x,\sqrt{\epsilon}p,\sqrt{\epsilon}^{-1}s)$ (see the formula (\ref{cylindereq3})), other contributions are independent of $s$. Again, by the upper semi-continuity of the set of semi-static curves, we have

\begin{theo}\label{beltthm2}
Under the hypotheses $(${\bf H1$\sim$H3, H5}$)$, some positive numbers $\Delta_0>0$ and $\epsilon_0>0$ exist, depending on the potential $V$, such that for each $\Delta<\Delta_0$, each $\epsilon\in (0,\epsilon_0)$ and each $c\in\alpha^{-1}(\Delta)$, all semi-static curves pass through  transversally the two-torus $\Gamma_c\times\{s\in\mathbb{T}\}$,
\begin{equation}\label{belteq3}
\arg\min(U_c^--U_c^+)\cap(\Gamma_c\times\{s=\text{\rm const}.\})\subset\bigcup I_{c,i}
\end{equation}
where $\Gamma_c$ is a circle located in a 2-torus $\{s=\rm constant\}\subset\mathbb{T}^3$,  $I_{c,i}\subset\Gamma_c$ are closed intervals, disjoint to each other and independent of the time $s$.
\end{theo}

Here, the semi-static curves are in the sense of extended configuration space, i.e. if $\gamma:\mathbb{R}\to M$ is a curve, we also call its graph a curve $\tilde{\gamma}(s)=(\gamma(s),s)\in M\times\mathbb{T}$.

Let us go back to the original scale. By Theorem \ref{beltthm2}, there exist a annulus-shaped region
$$
\mathbb{A}=\{c:0<\alpha_G(c)-\min\alpha_G\le\epsilon\Delta_0\}
$$
such that the condition (\ref{belteq3}) holds for each $c\in\mathbb{A}$. Recall this time-periodic system is deduced from the autonomous system restricted on certain energy level $H^{-1}(\tilde E)$. In virtue of Theorem \ref{flatthm5}, the counterpart of $\mathbb{A}$ in $H^1(\mathbb{T}^3,\mathbb{R})$ is
$$
\tilde{\mathbb{A}}=\{\tilde c=(c,c_3)\in\alpha^{-1}_H(\tilde E):0<c_3\le\epsilon\Delta_0\},
$$
where we notice that the sphere $\alpha^{-1}_H(\tilde E)$ is located in the upper half space of $\mathbb{R}^3$ and touches the plane $\{c_3=0\}$ where $c\in\mathbb{F}_0$.

In the original coordinates $(\tilde x,\tilde y)=(x,x_3,y,y_3)$, Theorem \ref{beltthm2} states such a fact: for each $\tilde c\in\tilde{\mathbb{A}}$, all $\tilde c$-semi static curves pass through the 2-torus $\Gamma_c\times\{x_3\in\mathbb{T}\}$ transversally and all intersection points are restricted in the strips $\cup I_{c,i}\times\{x_3\in\mathbb{T}\}$. However, the condition (\ref{belteq3}) dost not guarantee complete ``intersection" of the stable set with unstable set, in the sense that the set $\arg\min(U_c^--U_c^+)\cap\{\tau=0\}$ contains some disconnected points. Therefore, we call $\tilde{\mathbb{A}}$ the annulus of incomplete intersection. In this case, we do not expect to construct orbits connecting each Aubry set to any other Aubry set nearby, possible incompleteness may block some direction. However, once non-trivial intersection exists, it opens way to connect some Aubry set nearby. We shall show it in the subsection 7.2.

As $\epsilon\Delta_0\gg 2\epsilon^{1+d}$ provided $\epsilon>0$ is sufficiently small and $d>0$, we obtain the following:

{\bf Overlap Property}: {\it Given any two irreducible $g,g'\in H_1(\mathbb{T}^2,\mathbb{Z})$, there exists a positive number $\epsilon_0= \epsilon_0(V,g,g')>0$ such that the wedge-shaped regions intersects the annulus-shaped region: $\mathbb{W}_g\cap\mathbb{A} \ne\varnothing$ and $\mathbb{W}_{g'}\cap\mathbb{A}\ne\varnothing$} provided $0<\epsilon\le\epsilon_0$.

\section{\ui Local connecting orbits}
\setcounter{equation}{0}

To construct orbits connecting some Aubry set to another one nearby, we introduce two types of modified Tonelli Lagrangian, namely, the time-step and the space-step Lagrangian. They satisfy the conditions of positive definiteness, super-linear growth and completeness. The time-step Lagrangian $L:TM\times\mathbb{R}\to\mathbb{R}$ is not periodic in $t$ on the whole $\mathbb{R}$, instead, it is periodic when it is restricted either on $(-\infty,-\delta)$ or on $(\delta,\infty)$, i.e.
$L(\cdot,t)=L(\cdot,t+1)$ if $t,t+1\in (-\infty, -\delta)$ or $t,t+1\in (\delta,\infty)$. The second type of Lagrangian is defined on some covering space. Let $\pi :\bar M=\mathbb{R}\times\mathbb{T}^{n-1}\to M$. The space-step Lagrangian $L:T\bar M\times\mathbb{T}\to\mathbb{R}$ is not periodic in one
component of spaces coordinates $x_1$. It is periodic in $x_1$ when it is restricted either on $(-\infty,-\delta)$ or on $(\delta,\infty)$, i.e. $L(x_1,\cdot)=L(x_1+1,\cdot)$ if $x_1,x_1+1\in (-\infty, -\delta)$ or $x_1,x_1+1\in (\delta,\infty)$.

The existence of local connecting orbits is established based on some upper semi-continuity of minimal curves for the modified Lagrangian.

\subsection{Upper semi-continuity of minimal curves}
\noindent{\bf Time-step Lagrangian}: Let us consider time-step Lagrangian first. A curve $\gamma:\mathbb{R}\to M$ is called minimal if
$$
\int_{\tau}^{\tau'}L(\gamma(t),\dot\gamma(t),t)dt\le\int_{\tau}^{\tau'}L(\zeta(t),\dot\zeta(t),t)dt
$$
holds for any $\tau<\tau'$ and for any absolutely continuous curve $\zeta:[\tau,\tau']\to M$ with $\zeta(\tau)=\gamma(\tau)$ and $\zeta(\tau')=\gamma(\tau')$. Let $\mathscr{G}(L)$ denote the set of minimal curves for $L$. Let $\tilde{\mathcal{G}}(L)=\bigcup_{\gamma\in\mathscr{G}(L)}(\gamma(t),
\dot\gamma(t),t)$, $\mathcal{G}(L)=\pi\tilde{\mathcal{G}}(L)$ where $\pi:TM\times\mathbb{R}\to M\times\mathbb{R}$ is the standard projection.

\begin{theo}\label{semicontinuitythm1}
The set-valued map $L\to\mathscr{G}(L)$ is upper semi-continuous. Consequently, the map $L\to\tilde{\mathcal{G}}(L)$ is also upper semi-continuous.
\end{theo}
\begin{proof}
Let $K$ be the diameter of the closed manifold $M$, namely,
$$
K=\max_{x,x'\in M}\ell(x,x')
$$
where $\ell(x,x')$ denotes the length of the shortest geodesic connecting $x$ with $x'$. Let
$$
K_1=\sup_{\stackrel{(x,t)\in M\times\mathbb{R}}{\scriptscriptstyle \|v\|\le K}}L(x,v,t).
$$
As $L$ is assumed periodic for $t\le 0$ as well as for $t\ge 1$, $K_1$ is finite.

Let $\gamma$ be a shortest geodesic connecting the point $x$ to the point $x'$. Given time interval $[\tau,\tau']$ with $\tau'-\tau\ge 1$, we re-parameterize the geodesic $\gamma(s)$ by $\gamma'(\ell(x,x')(t-\tau)/(\tau'-\tau))$, then $\gamma'$: $[\tau,\tau']\to M$ is $C^1$-curve such that $\gamma'(\tau)=x$, $\gamma'(\tau')=x'$. Clearly, the action along this curve is not bigger than $K_1(\tau'-\tau)$. Obviously, there is an upper bound uniformly for all minimizing action of $L'$ if it is close to $L$ on $\{\|v\|\le K\}$, still denoted by
$$
h_{L'}((x,\tau),(x',\tau'))\le K_1(\tau'-\tau).
$$
If the Lagrangian has super-linear growth, some positive numbers $C,D>0$ exist such that $L'(x,\dot x,t)\ge C\|\dot x\|-D$ for all $(x,\dot x,t)\in TM\times\mathbb{R}$ and for all $L'$ close to $L$.
Therefore, if $\gamma$ is a minimizer, one obtains
\begin{equation}\label{semicontinuityeq1}
\frac{\text{\rm dist}(\gamma(\tau),\gamma(\tau'))}{\tau'-\tau}\le\frac 1{\tau'-\tau}\int_{\tau}^{\tau'}\|d\gamma\|\le\frac{K_1+D}C.
\end{equation}
As (\ref{semicontinuityeq1}) holds for any $\tau'-\tau\ge 1$, it implies that there must be some $t_i\in [\tau+i,\tau+i+1]$ for each $i\in\mathbb{Z}$ such that $\|\dot\gamma(t_i)\|\le C^{-1}(K_1+D)$. As it holds for any $x,x'\in M$, therefore, some positive number $K_2>0$ exists such that
$$
\phi^s\Big(\Big\{x,v,t_i:\|v\|\le\frac {K_1+D}C\Big\}\Big)\subset\Big\{x,v, t_i+s:\|v\|\le K_2\Big\}
$$
holds for all $s\in [0,2]$ and for all relevant $i$. It implies that $\|\dot\gamma(t)\|\le K_2$ holds for all $t\in[\tau,\tau']$.

Let $L_i\in C^r(TM\times\mathbb{R},\mathbb{R})$ be a sequence converging to $L$ in the following sense: there exists some $U_k\supset \{x,v,t:\| v\|\le K_2\}$, as well as a sequence of $\epsilon _i\to 0$ as $i\to\infty$ such that $\|L-L_i\|_{C^2(U_k,\mathbb{R})}\le\epsilon _i$. Let $\gamma _i$: $[\tau,\tau']\to M$ be the minimizer of $L_i$ with $\tau'- \tau\ge 1$, we then have $\|\dot\gamma_i(t)\|\le K_2$ for all $t\in [\tau,\tau']$. The set $\{\gamma_i\}$ is compact in the $C^1([\tau,\tau'],M)$-topology. Indeed, since $\partial ^2L/\partial\dot x^2$ is positive definite one can write the Lagrange equations in the form of $\ddot x=f(x,\dot x,t)$, which implies $\gamma _i$ is bounded in $C^2$-topology.

Let $\gamma $: $[\tau,\tau']\to M$ be one of the accumulation points of this set. Clearly, $\gamma$: $[\tau,\tau']\to M$ is the minimizer of $L$. Let $I_i=[\tau_i,\tau'_i]$ and let $\tau_i\to-\infty$ and $\tau'_i\to\infty$, we obtain a sequence of minimizers of $L_i$, $\gamma_i$: $I_i\to M$. By diagonal extraction argument some subsequence of $\gamma_i$ which converges $C^1$-uniformly on each compact set to a $C^1$-curve $\gamma$: $\mathbb{R}\to M$. Obviously, it is a minimal curve of $L$. This proves the upper semi-continuity.
\end{proof}

In application, the set $\mathscr{G}(L)$ seems too big for the construction of connecting orbits. For time-periodic Lagrangian, Ma\~n\'e set can be a proper subset of $\tilde{\mathcal{G}}(L)$,
$\tilde{\mathcal{N}}(L)\subsetneq\tilde{\mathcal{G}}(c)$. It is closely related to the problem whether the Lax-Oleinik semi-group converges or not (cf. \cite{FM}). For time-step Lagrangian, pseudo connecting curve is introduced to play roles similar to what semi-static curve does.

Each time-step Lagrangian $L$ uniquely determines two time-periodic Lagrangian $L^+$ and $L^-$ such that $L^+|_{(\delta,\infty)}=L|_{(\delta,\infty)}$ and $L^-|_{(-\infty,-\delta)}=L|_{(-\infty,-\delta)}$. Let
$-\alpha^{\pm}$ denote the minimal average action of $L^{\pm}$. For $m_0,m_1\in M$ and $T_0,T_1>0$, we define
\begin{equation}
h_{L}^{T_0,T_1}(m_0,m_1)=\inf_{\stackrel{\gamma(-T_0)=m_0}{\scriptscriptstyle\gamma(T_1)=m_1}} \int_{-T_0}^{T_1}L(d\gamma(t),t)dt+T_0\alpha^-+T_1\alpha^+.\notag
\end{equation}
Clearly the limit infimum is bounded
\begin{equation*}
|h_{L}^{\infty}(m_0,m_1)|=|\liminf_{T_0,T_1\to\infty}h_{L}^{T_0,T_1}(m_0,m_1)|<\infty.
\end{equation*}
Let $\{T_0^i\}_{i\in\mathbb{Z}_+}$ and $\{T_1^i\}_{i\in\mathbb{Z}_+}$ be the sequence of positive integers such that $T_j^i\to\infty$ ($j=0,1$) as $i\to\infty$ and the following limit exists
$$
\lim_{i\to\infty}h_{L}^{T_0^i,T_1^i}(m_0,m_1)=h_{L}^{\infty}(m_0,m_1).
$$
Let $\gamma_i(t,m_0,m_1)$: $[-T^i_0,T^i_1]\to M$ be a minimizer connecting $m_0$ and $m_1$
$$
h^{T_0^i,T_1^i}_{L}(m_0,m_1)=\int_{-T_0^i}^{T_1^i}L(d\gamma_i(t),t)dt +T^i_0\alpha^-+T^i_1\alpha^+.
$$
From the proof of Theorem \ref{semicontinuitythm1} one can see that for any compact interval $[a,b]$ there is some $I\in\mathbb{Z}_+$ such that the set $\{\gamma_i\}_{i\ge I}$ is pre-compact in $C^1([a,b],M)$.

\begin{lem}\label{semicontinuitylem1}
Let $\gamma$: $\mathbb{R}\to M$ be an accumulation point of $\{\gamma_i\}$. Then for $s,\tau\ge\delta$
\begin{align}\label{semicontinuityeq2}
A_{L}(\gamma|[-s,\tau])=&\inf_{\stackrel{s_1-s\in\mathbb{Z}, \tau_1-\tau\in\mathbb{Z}} {\stackrel{s_1,\tau_1\ge\delta}{\stackrel{\gamma^*(-s_1)=\gamma(-s)}{\scriptscriptstyle \gamma^*(\tau_1) =\gamma(\tau)}}}}\int_{-s_1}^{\tau_1}L(d\gamma^*(t),t)dt \\
&+(s_1-s)\alpha^-+(\tau_1-\tau)\alpha^+.\notag
\end{align}
\end{lem}
\begin{proof}: To prove the lemma let us suppose the contrary. Thus there would exist $\Delta>0$, $s_1,\tau_1\ge\delta$, $s_1-s\in\mathbb{Z}$, $\tau_1-\tau\in\mathbb{Z}$ and a curve $\gamma^*$: $[s_1,\tau_1]\to M$ with $\gamma^*(-s_1)=\gamma(-s)$, $\gamma^*(\tau)=\gamma(\tau_1)$ such that
$$
A_{L}(\gamma|[-s,\tau])\ge\int_{-s_1}^{\tau_1}L(d\gamma^*(t),t)dt+(s_1-s)\alpha^-+(\tau_1-\tau)\alpha^+ +\Delta.
$$
Let $\epsilon=\frac 13\Delta$. By the definition of limit infimum there exist $T^{i_0}_0>s$ and $T^{i_0}_1>\tau$ such that
\begin{equation*}
h_{L}^{T_0,T_1}(m_0,m_1)>h_{L}^{\infty}(m_0,m_1)-\epsilon, \qquad \forall \ \ T_0\ge T_0^{i_0}, \ T_1\ge T_1^{i_0},
\end{equation*}
and there exist subsequences $T_j^{i_k}$ ($j=0,1$; $k=0,1,2,\cdots$) such that $T_0^{i_k}-T_0^{i_0}\ge |s-s_1|$, $T_1^{i_k}-T_1^{i_0}\ge |\tau-\tau_1|$ and
\begin{equation*}
|h_{L}^{T_0^{i_k},T_1^{i_k}}(m_0,m_1)-h_{L}^{\infty} (m_0,m_1)|<\epsilon
\end{equation*}
holds for each $k>0$. Let $\gamma_{i_k}$ be the minimizer of $h_{L}^{T_0^{i_k},T_1^{i_k}}(m_0,m_1)$. By taking a subsequence further one can assume $\gamma_{i_k}\to\gamma$. In this case, for sufficiently large $k$, we are able to construct a curve $\gamma_{i_k}^*$: $[s_1,\tau_1]\to M$ which has the same endpoints as $\gamma_{i_k}$: $\gamma_{i_k}^*(-s_1)=\gamma_{i_k}(-s)$, $\gamma_{i_k}^*(\tau_1) =\gamma_{i_k}(\tau)$ and satisfies the following
\begin{equation*}
A_{L}(\gamma_{i_k}|[-s,\tau])\ge\int_{-s_1}^{\tau_1}L(d\gamma_{i_k}^*(t),t)dt +(s_1-s)\alpha^-+(\tau_1-\tau)\alpha^++\frac 23\Delta.
\end{equation*}
Extending $\gamma_{i_k}^*$ from $[s_1,\tau_1]$ to the $[-T_0^{i_k}-(s_1-s), T_1^{i_k}+(\tau_1-\tau)]$ by
$$
\gamma_{i_k}^*=\begin{cases}\gamma_{i_k}(t+s_1-s),\hskip 0.95 true cm t\le -s_1,\\
\gamma_{i_k}^*(t),\hskip 2.35 true cm -s_1\le t\le\tau_1,\\
\gamma_{i_k}(t-\tau_1+\tau),\hskip 0.9 true cm t\ge\tau_1,
\end{cases}
$$
and defining $T'_0=T_0^{i_k}+(s_1-s)$, $T'_1=T_1^{i_k}+(\tau_1-\tau)$ we find that
\begin{align*}
h_{L}^{T'_0,T'_1}(m_0,m_1)\le & A_{L}(\gamma_{i_k}^*|[-T'_0,T'_1])+T'_0\alpha^-+T'_1\alpha^+\\
\le &A_{L}(\gamma_{i_k}|[-T_0^{i_k},T_1^{i_k}])+T^{i_k}_0\alpha^-+T_1^{i_k}\alpha^+-\frac 23\Delta\\
\le &h_{L}^{\infty}(m_0,m_1)-\epsilon.
\end{align*}
But this contradicts the definition of the limit infimum as $T'_0\ge T_0$ and $T'_1\ge T_1$.
\end{proof}

We define so-called pseudo connecting curve set
$$
\mathscr{C}(L)=\{\gamma\in\mathscr{G}(L):\ (\ref{semicontinuityeq2})\ \text{\rm hold}\ \}.
$$
In application, we usually choose time-step Lagrangian so that the Aubry set of $L^-$ is different from that of $L^+$. Clearly, for $\gamma\in\mathscr{C}(L)$, the orbit $(\gamma(t),\dot\gamma(t))$ approaches the Aubry set $\tilde{\mathcal{A}}(L^-)$ as $t\to -\infty$ and approaches $\tilde{\mathcal{A}}(L^+)$ as $t\to \infty$. That is why we call it pseudo connecting curve. Let
$$
\tilde{\mathcal{C}}(L)=\bigcup_{\gamma\in\mathscr{C}(L)}(\gamma(t),\dot\gamma(t),t),\qquad \mathcal{C}(L)=\bigcup_{\gamma\in\mathscr{C}(L)}(\gamma(t),t).
$$
Clearly, if $L$ is periodic in $t$, then $\tilde{\mathcal{C}}(L)=\tilde{\mathcal{N}}(L)$ and $\mathcal{C}(L)=\mathcal{N}(L)$.

\begin{theo}\label{semicontinuitythm2}
The map $L\to\mathscr{C}(L)$ is upper semi-continuous. As the special case, the map $c\to\tilde{\mathcal{N}}(c)$ as well as the map $c\to\mathcal{N}(c)$ is upper semi-continuous.
\end{theo}
\begin{proof}: Let $L_i\to L$ be a sequence of time-step Lagrangian, let $\gamma_i\in\mathscr{C}(L_i)$ and let $\gamma$ be an accumulation point of the set $\{\gamma_i\in\mathscr{C}(L_i)\}_{i\in\mathbb{Z}^+}$. We claim that $\gamma\in\mathscr{C}(L)$. If $\gamma\notin\mathscr{C}(L)$, there would be two point $\gamma(s)$,$\gamma(\tau)\in M$ connected by another curve $\gamma^*$: $[-s-n_1,\tau+n_2]\to M$ and $\Delta>0$ such that
$$
A_{L}(\gamma^*)<A_{L}(\gamma|[-s,\tau])-n_1\alpha^--n_2\alpha^++\Delta
$$
where $s,s+n_1\ge\delta$, $\tau,\tau+n_2\ge\delta$. Since $\gamma$ is an accumulation point of $\gamma_i$, for any small $\epsilon>0$, there would be sufficiently large $i$ such that $\|\gamma-\gamma_i\|_{C^1[s,t]}<\epsilon$, and above inequality also holds for $A_{L_i}(\gamma_i|_{[s,\tau]})$. It follows that $\gamma_i\notin\mathscr{C}(L_i)$, contradicting to the assumption.
\end{proof}

\noindent{\bf Space-step Lagrangian}: Let $M=\mathbb{T}^n$ and $\pi :\bar M=\mathbb{R}\times\mathbb{T}^{n-1}\to M$, where $\mathbb{R}$ is for the coordinate $x_1$. The space-step Lagrangian $L$ is introduced to handle the problem of incomplete intersection. A space-step Lagrangian also uniquely determines two Lagrangian $L^-$ and $L^+$: $TM$ such that $L^-(x_1,\cdot)|_{(-\infty,-\delta)}=L(x_1,\cdot)|_{(-\infty,-\delta)}$ and $L^+(x_1,\cdot)|_{(\delta,\infty)}=L(x_1,\cdot)|_{(\delta,\infty)}$ if we treat $L^{\pm}$ as its natural lift to $T\bar M$. Let $\mu^{\pm}$ denote minimal measure of $L^{\pm}$ with $0$-cohomology class, $\omega(\mu^{\pm})=(\omega_1(\mu^{\pm}),\cdots,\omega_n(\mu^{\pm}))$ denote the rotation vector. We assume some conditions on the Lagrangian:

1, $\omega_1(\mu^{\pm})>0$ for each ergodic minimal measure $\mu^{\pm}$;

2, $\min\beta_{L^-}=\min\beta_{L^+}$, without losing of generality, it equals zero;

3, $|L^--L^+|\le\frac 12\min_{\omega_1=0}\{ \beta_{L^-}(\omega'),\beta_{L^+}(\omega')\}$.

\noindent It is shown in \cite{Lx} that some coordinates exists such that the first condition holds provided $\alpha(0)>\min\alpha$. As the minimal average action of $L^{\pm}$ is achieved on $\text{\rm supp}\mu^{\pm}$ with $\omega_1(\mu^{\pm})\neq 0$, one can see that $\min_{\omega_1(\nu)=0}\int L^{\pm}d\nu>\min\int L^{\pm}d\nu$, so the third condition makes sense. To introduce minimal curve for space-Lagrangian, we define
\begin{equation*}
h_{L}^{T}(\bar m_0,\bar m_1)=\inf_{\stackrel{\bar\gamma(-T)=\bar m_0} {\scriptscriptstyle \bar\gamma(T)=\bar m_1}}\int_{-T}^{T}L(\bar\gamma(t),\dot{\bar\gamma}(t))dt, \qquad \forall\ \bar m_0,\bar m_1\in\bar M.
\end{equation*}

\begin{lem}\label{semicontinuitylem2} If the rotation vector of each ergodic minimal measure has positive first component $\omega_1(\mu^{\pm})>0$, $\bar m_0\neq\bar m_1$, then
$$
\lim_{T\to 0}h_{L}^{T}(\bar m_0,\bar m_1)=\infty \ \ \ \
 and\ \ \ \ \lim_{T\to\infty}h_{L}^{T}(\bar m_0,\bar m_1)=\infty.
$$
\end{lem}
\begin{proof}
Let $\bar\gamma^{T}_{L}$: $[-T,T]\to\bar M$ be the minimizer of $h_{L}^{T}(\bar m_0,\bar m_1)$. Let $m_0=\pi \bar m_0$, $m_1=\pi \bar m_1$, $\zeta$: $[0,1]\to M$ be a smooth curve connecting $m_1$ to $m_0$, $\dot\zeta(0)=\dot{\bar\gamma}^{T}_{L}(T)$ and $\dot\zeta(1)=\dot{\bar\gamma}^{T}_{L}(-T)$. The action of $L^+$ along $\zeta$ is clearly bounded, thus for any $\epsilon>0$, one has $A_{L^+}(\zeta)\le 2T\epsilon$ provided $T$ is sufficiently large. The curve $\xi=\zeta\ast\pi \bar\gamma^{T}_{L}$ determines a holonomic probability measure $\nu^{T}_{L}\in\mathfrak{H}$ such that
$$
\int fd\nu^{T}_{L}=\frac 1{2T+1}\int_{-T}^{T+1}f(\xi(t),\dot\xi(t))dt\qquad \forall\ f\in C(TM,\mathbb{R}).
$$
Since $|\bar\gamma_L^T(T)-\bar\gamma_L^T(-T)|$ is bounded for any $T>0$, one has $\omega_1(\nu^{T}_{L})\to 0$ as $T\to\infty$. By using the third condition, we obtain
\begin{align*}
\frac 1{2T}h_{L}^{T}(\bar m_0,\bar m_1)=&\frac{2T+1}{2T}\int L^+d\nu^{T}_{L}-\frac 1{2T}
\int_0^1 L^+(\zeta(t),\dot\zeta(t))dt \\
&+\frac 1{2T}\int_{-T}^{T}(L-L^+)(\bar\gamma^{T}_{L}(t), \dot{\bar\gamma}^{T}_{L}(t))dt\\
\ge&\int L^+d\nu^{T}_{L}-\frac 12\min_{\omega_1=0}\beta_{L^+}(\omega)-\epsilon>0.
\end{align*}
It implies that $\lim_{T\to\infty}h_{L}^{T}(\bar m_0,\bar m_1)=\infty$. The case for $T\to 0$ is a consequence of the super-linear growth of $L$ in $\dot x$.
\end{proof}

As an intermediate step in introducing pseudo-connecting curve, we define a set of minimal curve $\mathscr{G}(L)$.

\begin{defi}A curve $\bar\gamma:\mathbb{R}\to\bar M$ is in $\mathscr{G}(L)$ if
$$
A_L(\bar\gamma|_{[-T,T]})=\inf_{T'\in\mathbb{R}_+} h_{L}^{T'}(\bar\gamma(-T),\bar\gamma(T)).
$$
\end{defi}
We claim that $\mathscr{G}(L)\neq\varnothing$. Indeed, denote by $\bar\gamma_{L}(\cdot,\bar m_0,\bar m_1):[-T,T]\to M$ the minimizer such that $\bar\gamma_{L}(-T)=\bar m_0$, $\bar\gamma_{L}(T)=\bar m_1$ and
$$
A(\bar\gamma_{L})=\int_{-T}^{T}L(\bar\gamma_{L}(t),\dot{\bar\gamma}_{L}(t))dt
=\inf_{T'\in\mathbb{R}_+} h_{L}^{T'}(\bar m,\bar m').
$$
Because of Lemma \ref{semicontinuitylem2}, this infimum is attained for finite $T>0$ if $\bar m_0$ and $\bar m_1$ are two different points in $\bar M$. The super-linear growth of $L$ in $\dot x$ guarantees that $T\to\infty$ as $-\bar m_{01},\bar m_{11}\to\infty$, where $\bar m_{i1}$ denotes the first coordinate of $\bar m_i$. Given an interval $[-T,T]$, for sufficiently large $-\bar m_{01},\bar m_{11}$, the set $\{\bar\gamma_{L}(\cdot,\bar m_0,\bar m_1)|_{[-T,T]}\}$ is
pre-compact in $C^1([-T,T],\bar M)$. Let $T\to\infty$. By diagonal extraction argument, there is a subsequence of $\{\bar\gamma_{L}(\cdot,\bar m_0,\bar m_1)\}$ which converges
$C^1$-uniformly on any compact set to a $C^1$-curve $\bar\gamma$: $\mathbb{R}\to\bar M$. Obviously, $\bar\gamma\in\mathscr{G}(L)$, and
\begin{pro}\label{semicontinuitypro1}
 Some number $K>0$ exists so that $|h_{L}^{T}(\bar\gamma(-T), \bar\gamma(T))|\le K$ holds for any curve $\bar\gamma\in \mathscr{G}(L)$ and any $T>0$.
\end{pro}
\begin{proof}
By the assumption, one has $\alpha_{L^{\pm}}(0)=\min\beta_{L^{\pm}}=0$. So, some $K'>0$ exists such that $|A(\gamma|_I)|\le K'$ holds for any interval $I\subset\mathbb{R}_+(\mathbb{R}_-)$ provided it is a forward (backward) semi-static curves for $L^{+}$ ($L^-$). Also, some $K''>0$ exists such that
$$
-K''\le\max_{\bar x,\bar x'\in \{x\in\bar M:|x_1|\le 1\}}\inf_{T\ge 0}h^T_L(\bar x,\bar x')\le K''.
$$
We claim that $K\le 2K'+K''$.

If there exists some $\bar\gamma\in\mathscr{G}(L)$ and some $T>0$ such that  $h_{L}^{T}(\bar\gamma(-T), \bar\gamma(T))>2K'+K''$, we join $\bar\gamma(-T)$ to $\bar\gamma(T)$ by another curve $\xi=\bar\gamma_-\ast\zeta\ast\bar\gamma_+$ where $\bar\gamma_-$ is a lift of backward semi-static curve $\gamma_-$ for $L_-$ such that $\bar\gamma(-T)=\bar\gamma_-(0)$, denote by $\bar x_-$ the intersection point of this curve with the section $\{\bar x\in\bar M:\bar x_1=-1\}$,  $\bar\gamma_+$ is a lift of forward semi-static curve $\gamma_+$ for $L_+$ such that $\bar\gamma(T)=\bar\gamma_+(0)$, denote by $\bar x_+$ the intersection point of this curve with the section $\{\bar x\in\bar M:\bar x_1=1\}$, $\zeta$ is a minimal curve of $L$ that connects the point $\bar x_-$ to $\bar x_+$. Obviously, one has $A_L(\xi)\le 2K'+K''<h_{L}^{T}(\bar\gamma(-T), \bar\gamma(T))$, but it contradicts the definition of $\mathscr{G}(L)$.
\end{proof}

Each $k\in\mathbb{Z}$ defines a Deck transformation ${\bf k}:\bar M\to\bar M$: ${\bf k}x=(x_1+k,x_2,\cdots,x_n)$. Let $\bar M^-_{\delta}=\{x\in\bar M:x_1<-\delta\}$, $\bar M^+_{\delta}=\{x\in\bar M:x_1>\delta\}$.

\begin{defi}\label{semicontinuitydef2}
A curve $\bar\gamma\in\mathscr{G}(L)$ is called pseudo connecting curve if the following holds
$$
A_L(\bar\gamma|_{[-T,T]})=\inf_{\stackrel{\stackrel{T'\in\mathbb{R}_+} {\scriptscriptstyle {\bf k}^-\bar\gamma(-T)\in \bar M^-_{\delta}}}{\scriptscriptstyle {\bf k}^+ \bar\gamma(T)\in \bar M^+_{\delta}}} h_{L}^{T'}({\bf k}^- \bar\gamma(-T),{\bf k}^+\bar\gamma (T))
$$
for each $\bar\gamma(T)\in \bar M^-_{\delta}$ and $\bar\gamma(T)\in\bar M^+_{\delta}$. Denote by $\mathscr{C}(L)$ the set of pseudo connecting curves.
\end{defi}

\begin{lem}\label{semicontinuitylem3}
The set $\mathscr{C}(L)$ is non-empty.
\end{lem}
\begin{proof} Let us start with a curve $\bar\gamma\in\mathscr{G}(L)$. Given $\Delta>0$, if some interval $[t^-_i,t^+_i]$  exists such that ${\bf k}^-_i \bar\gamma(t^-_i)$ can be connected to ${\bf k}^+_i\bar\gamma(t_i^+)$ by another curve $\zeta_i$ with smaller action
$$
A_L(\gamma|_{[t_i^-,t_i^+]})-A_L(\zeta_i)\ge\Delta>0,
$$
then one obtain a curve $\bar\gamma_i={\bf k}^-_i \bar\gamma|_{(-\infty,t^-_i]}\ast\zeta\ast{\bf k}^-_i \bar\gamma|_{[t^+_i,\infty)}$ by one step of such surgery.

Given any $\Delta>0$, we claim that there are finitely many intervals $[t_i^-,t_i^+]$ with $t_i^+\le t_{i+1}^-$ such that ${\bf k}^-_i \bar\gamma(t^-_i)$ can be connected to ${\bf k}^+_i\bar\gamma(t_i^+)$ by another curve $\zeta_i$ with the action $\Delta$ smaller than the original one.
Let us assume the contrary. Then, for any positive integer $m$, some large $T>0$ exists such that $[-T,T]\supset \cup_{i=1}^m [t_i^-,t_i^+]$. We can choose arbitrarily many of such intervals such that  either $t_1^->\delta$ or $t^+_m<-\delta$. In the first case, let $\bar x^-=\bar\gamma(-T)$ and $\bar x^+=\Pi_{\ell=1}^m{\bf k}^-_{\ell}{\bf k}^+_{\ell}\bar\gamma(T)$. By assumption, these two points can be connected by a curve $\zeta$ along which the action $A_L(\zeta)\le K-m\Delta$ as it follows from Proposition \ref{semicontinuitypro1} that $A_L(\bar\gamma|_{[-T,T]})\le K$. Since $m$ can be arbitrarily large, it implies the existence of a curve along which the action of $L$ approaches to minus infinity, it also contradicts Proposition \ref{semicontinuitypro1}.

Given a curve $\bar\gamma\in\mathscr{G}(L)$ and any small $\epsilon_i>0$, by finitely many steps of such surgery, we obtain a curve $\bar\gamma_i:\mathbb{R}\to\bar M$ with following properties:

1, for each small $\epsilon_i>0$, some large $T_i$ exists such that $\bar\gamma(-T_i)\in\bar M^-_{\delta}$, $\bar\gamma(T)\in\bar M^+_{\delta}$ and
$$
A_L(\bar\gamma_i|_{[-T_i,T_i]})\le\inf_{\stackrel{\stackrel{T'\in\mathbb{R}_+} {\scriptscriptstyle {\bf k}^-\bar\gamma(-T)\in \bar M^-_{\delta}}}{\scriptscriptstyle {\bf k}^+ \bar\gamma(T)\in \bar M^+_{\delta}}}h_{L}^{T'}({\bf k}^-\bar\gamma_i(-T),{\bf k}^+\bar\gamma_i(T)) +\epsilon_i.
$$

2, $\bar\gamma_i$ is smooth everywhere except for two points which fall beyond the region $\{x\in\bar M: |x_1|\le\Theta_i\}$, and $\Theta_i\to\infty$ as $\epsilon_i\to 0$.

Let $T'_i>0$ such that $\bar\gamma_{i1}(\pm T'_i)=\pm\Theta_i$. Because of Lemma
\ref{semicontinuitylem2}, we see that $T'_i\to\infty$ as $\Theta_i\to\infty$. In virtue of the argument before, for any large $T$ $\exists$ $i_0>0$ such that the set $\{\bar\gamma_i|_{[-T,T]}:i\ge i_0\}$ is pre-compact in $C^1([-T,T],\bar M)$. Let $T\to\infty$, by diagonal extraction argument, there is a subsequence of $\{\bar\gamma_i\}$ which converges $C^1$-uniformly on each compact set to a $C^1$-curve $\bar\gamma$: $\mathbb{R}\to\bar M$. Obviously, $\bar\gamma\in\mathscr{C}(L)$.
\end{proof}

\begin{theo}\label{semicontinuitythm3}
The map $L\to\mathscr{C}(L)$ is upper semi-continuous.
\end{theo}
\begin{proof}
Let $\bar\gamma_i\in\mathscr{C}(L_i)$, $L_i\to L$. If $\{\bar\gamma_i\}$ converges $C^1$-uniformly on each compact set to a $C^1$-curve $\bar\gamma$, it is obvious that $\bar\gamma\in\mathscr{C}(L)$.
\end{proof}
It is an immediate consequence of Definition \ref{semicontinuitydef2} that \begin{pro}\label{semicontinuitypro} If the space-step Lagrangian $L$ is periodic in $x_1$, then a curve $\bar\gamma\in\mathscr{C}(L)$ if and only if its projection $\gamma=\pi \bar\gamma$: $\mathbb{R}\to M$ is semi-static.
\end{pro}
Similar to the definition for time-step Lagrangian, we define
$$
\tilde{\mathcal{C}}(L)=\bigcup_{\bar\gamma\in\mathscr{C}(L)}(\bar\gamma(t),\dot{\bar\gamma}(t)), \qquad \mathcal{C}(L)=\bigcup_{\bar\gamma\in\mathscr{C}(L)}\bar\gamma(t).
$$
If $L$ is periodic in $x_1$, then $\pi \tilde{\mathcal{C}}(L)=\tilde{\mathcal{N}}(L)$ and $\pi \mathcal{C}(L)=\mathcal{N}(L)$.

\subsection{Local connecting orbits of type-$c$}
An orbit $d\gamma$ (A curve $\gamma$) is said connecting one Aubry set $\tilde{\mathcal{A}}(c)$ to another one $\tilde{\mathcal{A}}(c')$ if the $\alpha$-limit set of the orbit $d\gamma$ is contained in $\tilde{\mathcal{A}}(c)$ and the $\omega$-limit set  is contained in $\tilde{\mathcal{A}}(c')$. It is called local connecting orbit if these two classes are close to each other. It is called global when the two  classes  are far away from each other. In this subsection, we show how to construct local connecting orbits of type-$c$ by using so-called $c$-equivalence. This type of connecting orbits are found in the annulus of incomplete intersection and plays key role in establishing transition chain crossing strong double resonance.

For this purpose, we use the new version of $c$-equivalence introduced in \cite{LC}. The concept of $c$-equivalence was introduced in \cite{Ma2} for the first time, but it does not apply in interesting problems of autonomous system. The new version is defined not on the whole $M$, but on a non-degenerate embedded $(n-1)$-dimensional torus. We call $\Sigma_c$ non-degenerately embedded  ($n-1$)-dimensional torus by assuming a smooth injection $\varphi$: $\mathbb{T}^{n-1}\to\mathbb{T}^n$ such that $\Sigma_c$ is the image of $\varphi$, and the induced map $\varphi_*$: $H_1(\mathbb{T}^{n-1}, \mathbb{Z})\to H_1(\mathbb{T}^{n},\mathbb{Z})$ is an injection.

Let $\mathfrak{C}\subset H^1(\mathbb{T}^n,\mathbb{R})$ be a connected set where we are going to define $c$-equivalence. For each class $c\in \mathfrak{C} $, we assume that there exists a non-degenerate embedded $(n-1)$-dimensional torus $\Sigma_c\subset\mathbb{T}^n$ such that each $c$-semi static curve $\gamma$ transversally intersects $\Sigma_c$. Let
$$
\mathbb{V}_{c}=\bigcap_U\{i_{U*}H_1(U,\mathbb{R}): U\, \text{\rm is a neighborhood of}\, \mathcal {N}(c) \cap\Sigma_c\},
$$
here $i_U$: $U\to M$ denotes inclusion map. $\mathbb{V}_{c}^{\bot}$ is defined to be the annihilator of $\mathbb{V}_{c}$, i.e. if $c'\in H^1(\mathbb{T}^n,\mathbb{R})$, then $c'\in \mathbb{V}_{c}^{\bot}$ if and only if $\langle c',h \rangle =0$ for all $h\in \mathbb{V}_c$. Clearly,
$$
\mathbb{V}_{c}^{\bot}=\bigcup_U\{\text{\rm ker}\, i_{U}^*: U\, \text{\rm is a neighborhood of}\, \mathcal {N}(c) \cap\Sigma_c\}.
$$
Note that there exists a neighborhood $U$ of $\mathcal {N}(c)\cap\Sigma_c$ such that $\mathbb{V}_c=i_{U*}H_1(U,\mathbb{R})$ and $\mathbb{V}_{c}^{\bot}=\text{\rm ker}i^*_U$ (see \cite{Ma2}).

We say that $c,c'\in H^1(M,\mathbb{R})$ are $c$-equivalent if there exists a continuous curve $\Gamma$: $[0,1]\to \mathfrak{C}$ such that $\Gamma(0)=c$, $\Gamma(1)=c'$, $\alpha(\Gamma(s))$ keeps constant for all $s\in [0,1]$, and for each $s_0\in [0,1]$ there exists $\delta>0$ such that $\Gamma(s)-\Gamma(s_0)\in \mathbb{V}_{{\Gamma}(s_0)}^{\bot}$ whenever $s\in [0,1]$ and $|s-s_0|<\delta$.

Let $\{e_i\}_{1\le i\le n-1}$ be the standard basis of $H_1(\mathbb{T}^{n-1}, \mathbb{Z})$, one obtains $n$-dimensional vectors $\{g_{i+1}=\varphi_*(e_i)\in H_1(\mathbb{T}^{n}, \mathbb{Z})\}_{1\le i\le n-1}$. Because $\varphi$ is injection, there is a vector $g_1\in\mathbb{Z}^n$ such that the $n\times n$ matrix $G=(g_1,g_2,\cdots, g_n)$ is uni-module, i.e. $\text{\rm det}G=\pm 1$. In new coordinates system $x\to G^{-1}x$, the Lagrangian $\tilde L(\dot x,x)=L(G\dot x,Gx)$ is also $2\pi$-periodic in $x$. In new coordinates, let $\bar M=\mathbb{R}\times\mathbb{T}^{n-1}= \{x_1\in\mathbb{R},(x_2,\cdots,x_n)\in \mathbb{T}^{n-1}\}$ be a covering space of $\mathbb{T}^n$, $\pi:\bar M\to M=\mathbb{T}^n$. The lift of $\Sigma_c$, $\pi^{-1}(\Sigma_c)$ has infinitely many compact components $\{\Sigma_c^i\}_{i\in\mathbb{Z}}$. If $\varphi$ is linear, $\Sigma_c^i=\{x_1=2i\pi\}$. For the section $\Sigma=\{x_1=0\mod 1\}$ we have $\pi^{-1}(\Sigma)=\cup_{k\in\mathbb{Z}}\{x_1=k\}$ while for $\Sigma=\{x_1=x_2\}$, the lift $\pi^{-1}(\Sigma)$ consists of only one connected component.

\begin{theo}\label{typecthm1}
Assume the cohomology class $c^*$ is $c$-equivalent to the class $c'$ through the path $\Gamma$: $[0,1]\to H^1(\mathbb{T}^n,\mathbb{R})$. For each $s\in [0,1]$, the following are assumed:

1, there exists a coordinate systems $G_s^{-1}x$ where the first component of rotation vector is positive, $\omega_1(\mu_{\Gamma(s)})>0$ for each ergodic $\Gamma(s)$-minimal measure $\mu_{\Gamma(s)}$;

2, for the covering space $\bar M_s=\mathbb{R}\times\mathbb{T}^{n-1}$ in the coordinate system the lift of non-degenerately embedded codimension-one torus $\Sigma_{\Gamma(s)}$ has infinitely many connected and compact components, each of which is also a codimension-one torus.

Then there exist some classes $c^*=c_0, c_1,\cdots,c_k=c'$ on this path,  closed 1-forms $\eta_i$ and $\bar\mu_i$ on $M$ with $[\eta_i]=c_i$ and $[\bar\mu_i]=c_{i+1}-c_i$,  and smooth functions $\varrho_i$ on $\bar M$ for $i=0,1,\cdots,k-1$, such that the pseudo connecting curve set $\mathscr{C}(L_i)$ for the space-step Lagrangian
$$
L_i=L-\eta_i-\varrho\bar\mu_i
$$
possesses the properties:

\text{\rm (i)}, each curve $\bar\gamma\in\mathscr{C}(L_i)$ determines an orbit $(\gamma,\dot{\gamma})$ of $\phi_L^t$;

\text{\rm (ii)}, the orbit $(\gamma,\dot{\gamma})$ connects $\tilde{\mathcal{A}}(c_{i})$ to $\tilde{\mathcal{A}}(c_{i+1})$, i.e., the $\alpha$-limit set $\alpha(d\gamma)\subseteq\tilde{\mathcal{A}} (c_{i})$ and $\omega$-limit set $\omega(d\gamma)\subseteq\tilde{\mathcal{A}}(c_{i+1})$.
\end{theo}
\begin{proof}
By the definition of $c$-equivalence, there exists a path $\Gamma$: $[0,1]\to H^1(M,\mathbb{R})$ with $\Gamma(0)=c^*$, $\Gamma(1)=c'$ such that for each $c=\Gamma(s)$ ($s\in [0,1]$) on the path, there exists $\epsilon>0$ such that $\Gamma(s')-c\in \mathbb{V}_{{\Gamma}(s)}^{\bot}$ whenever $s'\in [0,1]$ and
$|s-s'|<\epsilon$. Thus, there exist a non-degenerately embedded  ($n-1$)-dimensional torus $\Sigma_c$,  a closed form $\bar\mu_{c}$ and a neighborhood $U$ of $\mathcal {N}(c)\cap \Sigma_{c}$ such that $[\bar\mu_{c}]=\Gamma(s')-c$ and $\text{\rm supp}\bar\mu_{c}\cap
U=\varnothing$.

In the new coordinates $x\to G^{-1}_{c}x$ on the torus as above, the codimension one hypersurface $\Sigma_c^0$ separates $\bar M$ into two parts, the upper part $\bar M^+$ and the lower part $\bar
M^-$. $\bar M^{\pm}$ extends to where the first coordinate $x_1\to\pm\infty$. Let $\Sigma_c^0+\delta$ denotes the $\delta$-neighborhood of $\Sigma_c^0$ in $\bar M$, we introduce a smooth function $\varrho\in C^r(\bar M,[0,1])$ such that $\varrho=0$ if $x\in\bar M^- \backslash (\Sigma_c^0+\delta)$, $\varrho=1$ if $x\in\bar M^+\backslash(\Sigma_c^0+\delta)$. Let $\eta$ and $\bar\mu$ are closed 1-forms on $M$ such that $[\eta]=c$ and $[\eta+\bar\mu]=c'$. These forms have natural lift on $\bar M$, with the same notation.

A sufficiently small $\delta>0$ can be chosen so that
$$
(\Sigma_{c}^0+\delta)\cap(\mathcal {C}(L+\eta)+2\delta)\subset U,
$$
It follows from the upper semi-continuity of $\mathcal {C}(L)$ w.r.t. $L$, we find
\begin{equation}\label{typeceq1}
(\Sigma_{c}^0+\delta)\cap(\mathcal{C}(L+\eta+\varrho\bar\mu)+\delta)\subset U,
\end{equation}
if $\varrho\bar\mu$ is $C^0$-sufficiently small. As $\bar\mu$ is carefully chosen so that its support is disjoint from $U$, each curve $\bar\gamma\in\mathscr{C}(L+\eta+\varrho\bar\mu)$ is clearly a solution of the Euler-Lagrange equation determined by $L$, the term $\varrho\bar\mu$ has no contribution to the equation along $\bar\gamma$. In other words, each curve in
$\mathscr{C}(L+\eta+\varrho\bar\mu)$ generates an orbit $d\gamma$ of $\phi_L^t$: $\mathbb{R}\to TM$.

The definition of $\mathscr{C}$ tells us that for each curve $\bar\gamma\in\mathscr{C}$, $\gamma|_{(-\infty,t_0]}$ is backward $\Gamma(s)$-semi static once $\bar\gamma|_{(-\infty,t_0]}$ falls entirely into
$\bar M^-\backslash (\Sigma_c^0+\delta)$, $\gamma|_{[t_1,\infty)}$ is forward $\Gamma(s')$-semi static once $\bar\gamma|_{[t_1,\infty)}$ falls entirely into $\bar M^+\backslash (\Sigma_c^0+\delta)$. Therefore,
$(\gamma(t),\dot\gamma(t))\to\tilde{\mathcal {A}}(\Gamma(s))$ as $t\to -\infty$ and $(\gamma(t),\dot\gamma(t))\to\tilde{\mathcal {A}}(\Gamma(s'))$ as $t\to\infty$.

Because of the compactness of $[0,1]$, there are finitely many numbers $s_0,\cdots,s_k\in [0,1]$ such that above argument applies if $s$ and $s'$ are replaced respectively by $s_i$ and $s_{i+1}$. Set $c_i=\Gamma(s_i)$.
\end{proof}

\begin{cor}\label{typeccor1} Let $c_i$, $\eta_i$, $\bar\mu_i$ and $\varrho_i$ be evaluated as in  Theorem \ref{typecthm1}. Let $U_i$ be an open neighborhood of $\mathcal {N}(c_i)\cap \Sigma_{c_i}^0$ such that $U_i\cap\text{\rm supp}\bar\mu_i=\varnothing$. Then, there exist large $K_i>0$, $T_i>0$ and small $\delta>0$ such that for each $\bar m,\bar m'\in\bar M$, with $-K_i\le\bar m_1\le -K_i+2\pi$, $K_i-2\pi \le\bar m'_1\le K_i$, the quantity $h_{\eta_i,\mu_i}^{T}(\bar m,\bar m')$ reaches its minimum at some $T<T_i$ and the corresponding minimizer $\bar\gamma_{i}(t,\bar m,\bar m')$ satisfies the condition
\begin{equation}\label{typeceq2}
\text{\rm Image}(\bar\gamma_{i})\cap(\Sigma_{c_i}^0+\delta)\subset U_i.
\end{equation}
\end{cor}

There is some flexibility to choose the coordinate system and the non-degenerately embedded codimension one torus. Let $\pi_s$: $\bar M_s\to M=\mathbb{T}^n$ be a covering space such that $\bar M_s=\mathbb{R} \times\mathbb{T}^{n-1}$ in the coordinate system $G_s^{-1}x$.

\begin{defi}\label{typecdef1}
For $s\in [0,1]$, the non-degenerately embedded codimension one torus $\Sigma_s$ is called admissible for the coordinate system $G_s^{-1}x$ if the lift of $\Sigma_s$ to the covering space $\bar M_s$ consists of infinitely many connected and compact components, the first component of the rotation vector is positive $\omega_1(\mu_{\Gamma(s)})$ for each ergodic $\Gamma(s)$-minimal measure.
\end{defi}

Let us describe how the equivalence relation is established between two classes near strong double resonance. Let $\Gamma\subset\mathbb{A}\subset\alpha^{-1}(E)$ be a curve skirting around the flat $\mathbb{F}_0$, along which the $\alpha$-function keeps constant and the third coordinate $c_3$ keeps constant as well. For each $c\in\Gamma$, there exists certain coordinate system and finitely many intervals $I_{c,i}$ for $x_2$-coordinates such that each $c$-semi static curve passes through the section $\Sigma_c=\{x_1=0\}$ and
$$
\mathcal{N}(c)\cap\Sigma_c\subset\{(x_1,x_2,x_3):x_1=0,x_2\in\cup I_{c,i},x_3\in\mathbb{T}\}.
$$
Clearly, some open set $U\supset\mathcal{N}(c)\cap\Sigma_c$ such that $V_c=i_{U*}H_1(U,\mathbb{R})=\text{\rm span}\{(0,0,1)\}$, from which one obtains that $V_c^{\perp}=\text{\rm span}\{(1,0,0),(0,1,0)\}$. For each class $c'\in\Gamma$ very close to $c$, one has $c'-c=(\Delta c_1,\Delta c_2,0)\in V_c^{\perp}$, thus, there exists a closed 1-form $\bar\mu$ such that $[\bar\mu]=c'-c$ and
$$
\text{\rm supp}\bar\mu\cap\mathcal{N}(c)\cap\Sigma_c=\varnothing.
$$
Therefore, all classes along the curve $\Gamma$ are equivalent in this case.

\subsection{Local connecting orbits of type-$h$}
Another type of local connecting orbits look like heteroclinic orbits. Therefore, we call them local connecting orbits of type-$h$.

It is used to handle a typical case when an Aubry set falls in a neighborhood $N$ of some lower dimensional torus such that $H_1(M,N,\mathbb{Z})\neq 0$. Equivalence relation seems not exist among those classes if the Aubry sets is located in $N$. However, each of these Aubry sets has homoclinic orbit, it may lead to the existence of heteroclinic orbits. Towards this goal, let us work in suitable finite covering manifold $\check{\pi}$: $\check{M}\to M$. In this covering space, these homoclinic orbits turn out to be semi-static orbits.  We assume that the Aubry set $\mathcal{A}(c,\check{M})$ consists of finitely many classes $\mathcal{A} (c)=\mathcal{A}_1\cup\cdots\cup\mathcal{A}_k$ $(k>1)$,  $\check{M}$ is chosen so that the lift of $N$, $\check{N}=N_1\cup\cdots\cup N_k$ with $k>1$, $\check{\pi}N_i=N$ and $\text{\rm dist}(N_i,N_j)>0$ provided $i\neq j$. In the following, we denote by $N_i$ the open neighborhood such that each $N_i$ contains one Aubry class $N_i\supset\mathcal{A}_i$.

If an Aubry set contains finitely many static classes only, denoted by $\tilde{\mathcal{A}}_i$ ($i=1,2,\cdots,k$), then these classes are transitive in the following sense: by rearranging the
subscripts, there exist $k$ semi-static curves $\gamma_{i,i+1}$ $(\text{\rm mod}\ k)$ such that $\omega(d\gamma_{i,i+1})\subseteq\tilde{\mathcal{A}}_{i+1}$ and $\alpha(d\gamma_{i,i+1})\subseteq
\tilde{\mathcal{A}}_i$ \cite{CP}. It does not exclude the case that some semi-static curve $\gamma_{i,j}$ exists such that $j\neq i+1$ $(\text{\rm mod}\ k)$, $\alpha(d\gamma_{i,j})\subseteq \tilde{\mathcal{A}}_{i}$ and $\omega(d\gamma_{i,j})\subseteq\tilde{\mathcal{A}}_j$. We say that $\tilde{\mathcal{A}}_i$ is connected to $\tilde{\mathcal{A}}_j$ through $\tilde{\mathcal{A}}_{i'}$ with $i'=i+1,i+2,\cdots,j-1$ if there exist semi-static curves $\gamma_{i',i'+1}$ such that $\omega(d\gamma_{i',i'+1})\subseteq \tilde{\mathcal{A}}_{i'+1}$ and $\alpha(d\gamma_{i',i'+1})\subseteq \tilde{\mathcal{A}}_{i'}$.

The Aubry set $\mathcal{A}_i$ is said to be directly connected to the Aubry set $\mathcal{A}_j$ if a semi-static curve $\gamma$: $\mathbb{R}\to M$ exists such that $\omega(d\gamma)\subseteq \tilde{\mathcal{A}}_{j}$ and
$\alpha(d\gamma)\subseteq \tilde{\mathcal{A}}_{i}$. That $\mathcal{A}_i$ is directly connected to $\mathcal{A}_j$ does not imply that $\mathcal{A}_j$ is directly connected to $\mathcal{A}_i$.

Pick up two points $x_i\in\mathcal{A}_i$, $x_j\in\mathcal{A}_j$, we consider the quantity
$$
h_c^{T}(x_i,x_j)=\inf_{\stackrel{\gamma(-T) =x_i}{\scriptscriptstyle \gamma(T)=x_j}} \int_{-T}^{T}
L_{c}(d\gamma(t))dt +2T\alpha(c).
$$
By standard notation,
$$
h_c^{\infty}(x_i,x_j)=\liminf_{T\to\infty}h_c^{T}(x_i,x_j).
$$
Let $\gamma^T$: $[-T,T]\to M$ be the minimal curve realizing the quantity $h_c^{T}(x_i,x_j)$. Let $[t_{i,T},t_{j,T}]$ be the sub-interval of $[-T,T]$ such that $\gamma^{T}(t)\notin N_i\cup N_j$ for $t\in(t_{i,T},t_{j,T})$ but $\gamma^{T}(t_{i,T})\in\bar N_i$ and $\gamma^{T}(t_{j,T})\in\bar N_j$. In the case that $\mathcal{A}_i$ is directly connected only to $\mathcal{A}_j$, $t_{j,T}-t_{i,T}$ is upper bounded uniformly for $T>0$. Some sequence of time $t_{T}$ and a positive number $\Delta>0$ such that $[t_{T}-\Delta, t_{T}+\Delta]\subset (t_{i,T},t_{j,T})$ for sufficiently large $T$. The set of curves $\{\gamma^{T}(t-t_{T}) |_{[-\Delta,\Delta]}\}$ is compact in $C^1$-topology. Let $\gamma|_{[-\Delta, \Delta]}$
be the accumulation point which can be uniquely extended to whole line $\gamma$: $\mathbb{R}\to M$. Clearly, $\alpha(d\gamma)\subset\tilde{\mathcal{A}}_i$ and $\omega(d\gamma)\subset \tilde{\mathcal{A}}_j$. If $\mathcal{A}_i$ is directly connected also to other $\mathcal{A}_k$, one can also obtain such a sequence of curves by introducing small perturbation so that $\mathcal{A}_i$ is directly connected only to $\mathcal{A}_j$ and the support of the perturbation does not touch the semi-static curves connecting $\mathcal{A}_i$ to $\mathcal{A}_j$.

Given a semi-static curve one can choose an $(n-1)$-dimensional disk $\Sigma$ intersecting the curve transversally.  This disk also intersects semi-static curves nearby. A semi-static curve is said {\it disconnected} to other semi-static curves if the intersection point is disconnected to the intersection points of all other semi-static curves.

\begin{theo}\label{typehthm1} {\rm (Connecting Lemma)}
Assume that the Aubry set contains finitely many classes $\mathcal{A}(c)=\mathcal{A}_1\cup\cdots \cup\mathcal{A}_k$, there exist open domains $N_1\cdots N_k$ such that $\mathcal{A}_i\subset N_i$ for each $1\le i\le k$ and $\text{\rm dist}(N_i,N_j)>0$ provided $i\neq j$. If each semi-static curves connecting different Aubry sets is disconnected to all other semi-static curve, then there exists some orbit $d\gamma'$ of $\phi_L^t$ connecting $\tilde{\mathcal{A}}(c)$ to $\tilde{\mathcal {A}}(c')$ provided $\alpha(c)=\alpha(c')$, the class $c'$ is close to the class $c$, $\mathcal{A}(c')\subset \cup_{i=1}^k N_i$  and two sets $N_i$, $N_j$ exist such that $\mathcal{A}(c')\cap N_i\neq\varnothing$ and $\mathcal{A}(c')\cap N_j\neq\varnothing$.
\end{theo}

\begin{proof}
In autonomous case, $\tilde{\mathcal{A}}(c)$ can be connected to $\tilde{\mathcal{A}}(c')$ only if $\alpha(c)=\alpha(c')$. If both $c$ and $c'$ are the minimal points of the $\alpha$-function, then
$\tilde{\mathcal {A}}(c)\cap\tilde{ \mathcal {A}}(c')\neq\varnothing$ (see \cite{Ms}), it is trivial to connect an Aubry set to itself. Thus we only need to work on the energy level set $H^{-1}(E)$ with $E>\min\alpha$, the minimum of the $\alpha$-function. In this case, we obtain from \cite{Lx} that
\begin{pro}\label{typehpro1}
Let $L:\mathbb{T}^n\to\mathbb{R}$ be an autonomous Lagrangian of Tonelli type, the class $c$ not be the minimal point of the $\alpha$-function, and $\Omega_c$ be the flat of the $\beta$-function such that
$$
\omega\in\Omega_c\ \ \Rightarrow\ \ \alpha(c)+\beta(\omega)=\langle c,\omega\rangle.
$$
Then, there exists a coordinate system such that each rotation vector in this flat has positive first component $\omega_1>0$.
\end{pro}

The existence of such connecting orbits is derived from the upper-semi continuity of pseudo-connecting orbit set (see Definition \ref{semicontinuitydef2}). For the definition of this set in autonomous case, we need to work in certain covering space $\pi:\bar M=\mathbb{R}\times\mathbb{T}^{n-1}$ where $\omega_1(\mu_c)>0$ holds for each ergodic minimal measure $\mu_c$. By Proposition \ref{typehpro1}, it is possible if we choose suitable coordinate system. Let $\bar\gamma$ denote the lift of the curve $\gamma:\mathbb{R}\to M$, $\bar\gamma_1$ denote the first coordinate.

Let $\Sigma_0=\{x:x_1=0\}$ be a codimension one hyperplane separating $\bar M$ into two parts, the upper part $\bar M^+$ connected to $\{x_1=\infty\}$ and the lower part $\bar M^-$ connected to $\{x_1=-\infty\}$. Let $\Sigma_0+\delta$ denote the $\delta$-neighborhood of $\Sigma_0$ in $\bar M$, we introduce a smooth function $\rho\in C^r(\bar M,[0,1])$ such that $\rho=0$ if $x\in\bar M^- \backslash (\Sigma_0+\delta)$, $\rho=1$ if $x\in\bar M^+\backslash (\Sigma_0+\delta)$. Let $\eta$ and $\bar\mu$ be closed 1-forms on $M$ such that $[\eta]=c$ and $[\eta+\bar\mu]=c'$. They have natural lift on $\bar M$. Let $\mu=\rho\bar\mu$. We carefully choose smooth function $\psi=\psi(x,\dot x)$ such that $\psi=0$ as $x_1\in (-\infty,-1)\cup (1,\infty)$ (the construction will be demonstrated later) and let
$$
L_{\eta,\mu,\psi}=L-\eta-\mu-\psi.
$$
Let $\bar m,\bar m'$ be two points in $\bar M$, we define
$$
h_{\eta,\mu,\psi}^{T}(\bar m,\bar m')=\inf_{\stackrel{\gamma(-T)=\bar m}{\scriptscriptstyle \gamma(T)=\bar m'}}\int_{-T}^T(L_{\eta,\mu,\psi}(d\gamma(t))+\alpha(c))dt.
$$
For small $\mu$ and $\psi$, the Lagrangian $L_{\eta,\mu,\psi}$ satisfies the conditions required for space-step Lagrangian. In the following we shall use the notation $\mathscr{C}_{\eta,\mu,\psi}= \mathscr{C}(L_{\eta,\mu,\psi})$ to denote the relevant set of the pseudo-connecting curves.

Let us recall a graph property. Given two Aubry classes $\tilde{\mathcal{A}}_i$ and $\tilde{\mathcal{A}}_j$, let $\tilde{\mathcal{N}}_{ij}$ be the set of all semi-static orbits whose $\alpha$-limit set is in $\tilde{\mathcal{A}}_i$ and the $\omega$-limit set is in $\tilde{\mathcal{A}}_j$. Let $\mathcal{N}_{ij}=\pi_x\tilde{\mathcal{N}}_{ij}$, where $\pi_x$ denotes the standard projection $TM\to M$. Then, the inverse of $\pi_x$, restricted on $\mathcal{N}_{ij}$, is Lipschitz. The proof
is the same as that for the graph property of Aubry set.

If $\tilde{\mathcal{A}}_i$ is directly connected to $\tilde{\mathcal{A}}_j$, there exists a semi-static orbit $d\zeta_{ij}$ connecting $\tilde{\mathcal{A}}_i$ to $\tilde{\mathcal{A}}_{j}$. Pick up a curve $\bar\zeta_{ij}$ in the lift of $\zeta_{ij}$ to $\bar M$ such that its intersection point $x_0$ with the section $\{x:x_1=0\}$ is not close to ${\cup\bar N_i}$, the lift of ${\cup N_i}$ to $\bar M$. Denote by $v_0=\dot{\zeta}_{ij}$ the velocity of $\zeta_{ij}$ at $\pi x_0$, obviously, $v_0\neq 0$. Let $\varrho'$ be a smooth function in $s$ such that $\varrho'=0$ for $s\le 0$, $\varrho'=1$ for $s>\delta$ and $\dot\varrho'>0$ for $s\in (0,\delta)$, where $\delta>0$ is suitably small. Let $\varrho_{ij}(x)=\varrho'(\langle x-x_0,v_0\rangle)$, then $\langle\partial\varrho_{ij}(x),v \rangle=\langle v_0,v\rangle\dot\varrho'(s)$ where $s=\langle x-x_0,v_0\rangle$.

We choose an $(n-1)$-dimensional plane $\Sigma_{ij,s}=\{x:\langle x-x_0,v_0\rangle=s\}$. Since the set of semi-static curves is totally disconnected, we can choose, for each $s\in [0,\delta]$, two suitably small $(n-1)$-dimensional topological disks $D'_{ij,s},D_{ij,s}$ located in $\Sigma_{ij,s}$ and small $\delta_1>0$ such that $D'_{ij,s} \cap(\cup N_j)=\varnothing$, $D'_{ij,s}\supset D_{ij,s}+\delta_1$, certain semi-static curve $\zeta_{ij}$ passes through the disk $D_{ij,s}$ and no semi-static curve in $\mathcal{N}_{ij}$ passes through $D'_{ij,s}\backslash D_{ij,s}$. These disks can be chosen so that the Hausdorff distance $d_H(D_{ij,s},D_{ij,s'})\to 0$ and $d_H(D'_{ij,s},D'_{ij,s'})\to 0$ as $s'\to s$. Let $D'_{ij}=\cup_{s\in [0,\delta]}D'_{ij,s}$, $D_{ij}=\cup_{s\in [0,\delta]}D_{ij,s}$. We choose a smooth non-negative function $w_{ij}$: $\bar M\to\mathbb{R}$ such that $\text{\rm supp}w_{ij}\cap\{x:0\le\langle x-x_0,v_0\rangle \le\delta\}=D'_{ij}$ and $w_{ij}\equiv\lambda$ if $x\in D_{ij}$.

For different $(i,j)\neq(i',j')$, it is possible that $D_{ij}\cap\mathcal{N}_{i',j'}\neq\varnothing$. But it does not make trouble, as $\tilde{\mathcal{N}}_{ij}\cap\tilde{\mathcal{N}}_{i'j'}=\varnothing$. Let $S_{ij}$ be the graph of a Lipschitz map $x\to\dot x$ containing $\tilde{\mathcal{N}}_{ij}$. Therefore, we can
choose a smooth function $\upsilon_{ij}$: $TM\to [0,1]$ such that $\upsilon_{ij}\equiv 1$ when $(x,\dot x)\in (S_{ij}+\delta_2)\cap TD'_{ij}$ and $\upsilon_{ij}\equiv 0$ when $(x,\dot x)\notin S_{ij}+\delta_3$, where $\delta_3>\delta_2>0$ are small numbers. As there are finitely many Aubry classes, we have $\text{\rm supp}\upsilon_{ij}\cap \text{\rm supp}\upsilon_{i'j'}=\varnothing$ if $(i,j)\neq(i',j')$.

Let us consider what curves contained in the set $\mathscr{C}_{\eta,0,\psi}$ by assuming
$$
\psi=\sum\upsilon_{ij}w_{ij}\langle\partial\varrho_{ij},\dot x\rangle.
$$
Since the term $\langle\partial\varrho_{ij},\dot x\rangle=0$ for $\{\langle x-x_0,v_0\rangle\le 0\}\cup \{ \langle x-x_0,v_0\rangle\ge\delta\}$, we do not care about how $w_{ij}$ is defined on those $(n-1)$-dimensional plane $\Sigma_{ij,s}$ with $s\notin [0,\delta]$. The set  $\{\langle x-x_0,v_0\rangle=s\}\subset\bar M$ may extend to infinity, but it does not make trouble since the support of $w_{ij}$ is contained in $D_{ij,s}$ when $\langle x-x_0,v_0\rangle=s$.

Let $\psi_0$ be the function defined on $TM$ such that $\pi\psi|_{\{x_1\in [-\pi,\pi]\}} =\psi_0$. By the construction of $\psi$ we see that $\psi_0$ is well-defined and smooth. The Aubry set for the Lagrangian $L-{\eta}-\psi_0$ is the same as for $L-{\eta}$. As there is no semi-static curve of $L-{\eta}$ touches the tube $D'_{ij}\backslash D_{ij}$, each semi-static curve of $L-\eta$ also solves the Euler-Lagrange equation determined by $L-{\eta}-\psi_0$. Because of the upper semi-continuity of $L\to\mathcal{N}(L)$, each semi-static curve for $L-\eta-\psi_0$ stays in a small neighborhood of $\mathcal{N}(L)$. Since $\psi<0$ if $x\in\cup D_{ij}$ and $\psi=0$ if $x\notin \cup D'_{ij}$,  a curve is still semi-static for $L-{\eta}-\psi_0$ if it is semi-static for $L-{\eta}$ and passes through the solid cylinder $D_{ij}$. It is based on following argument. Since the 1-form $w_{ij}\langle\partial \varrho_{ij}, dx\rangle$ is closed in $D_{ij}$, $\langle\partial\varrho_{ij},\dot x\rangle=\dot\varrho'(s)\langle\dot x,v_0\rangle=0$ on each $\Sigma_{ij,s}$ with $s\notin [0,\delta]$, this term has no contribution to the Euler-Lagrange equation along this semi-static curve, i.e. this curve solves the Euler-Lagrange equation determined by $L-\eta-\psi_0$ also. Any other semi-static curve for $L-\eta$ is no longer minimal for $L-\eta-\psi_0$ if it connects $\mathcal{A}_i$ to $\mathcal{A}_{j}$ but does not pass through $D_{ij}$, for there exists some semi-static curve $\zeta_{ij}$ of $L-\eta$ passing through $D_{ij}$, along which the action is smaller than the action along $\gamma_{ij}$.

Let us go back to the covering space $\bar M$. For small $\psi$, realized by choosing small $w_{ij}$, each curve $\bar\gamma\in\mathscr{C}_{\eta,0,\psi}$ stays in a small neighborhood of certain curve belong to the lift of the semi-static curve. It is due to the upper semi-continuity of $L\to\mathscr{C}(L)$.  By the discussion above, a curve does not belong to $\mathscr{C}_{\eta,0,\psi}$ if its projection does not belong to the Aubry set for $L-\eta$, or dose not pass through $D_{ij}$ although it is semi-static and connects $\mathcal{A}^i$ to $\mathcal{A}^j$.

Let $\bar\zeta_{ij}$ be a curve in $\mathscr{C}_{\eta,0,\psi}$ passing through $D_{ij}$. Its projection $\pi\bar\zeta$ connects $\mathcal{A}^i$ to $\mathcal{A}^j$. Let $k^*\bar\zeta_{ij}=\bar\zeta_{ij} +(k,0,\cdots,0)$ with $k\in\mathbb{Z}$ denote its shift. Each of these curves solves the Euler-Lagrange equation determined by $L_{\eta,0,\psi}$. However, except for $\bar\zeta_{ij}$, any other curve $k^*\bar\zeta_{ij}$ with $k\neq 0$ does not belong to $\mathscr{C}_{\eta,0,\psi}$ because they do not pass through $D_{ij}$, the action along these curves is bigger than the action along $\bar\zeta_{ij}$. It can be easily seen from the definition \ref{semicontinuitydef2}: minimal property persists under translation.

In the cylinder $\bar M$ we choose two sections $\Sigma^+$ and $\Sigma^-$ such that:

1, both are the deformation of the section $\{x:x_1=\text{\rm constant}\}$, they divide $\bar M$ into three parts, $\bar M^+$, $\bar M^-$ and $\bar M_0$. $\bar M^+$ is homeomorphic $(0,\infty)\times\mathbb{T}^{n-1}$, $\bar M^-$ is homeomorphic $(-\infty,0)\times\mathbb{T}^{n-1}$ and $\bar M_0$ is homeomorphic to $(0,1)\times\mathbb{T}^{n-1}$. Let $\Sigma_{\pm}$ denote the boundary of $\bar M^{\pm}$ shared with $\bar M_0$;

2, there exists $\delta_4>0$ such that $\cup D'_{ij}+\delta_4\subset\bar M_0$;

3, for each $\bar\zeta_{ij}\in\mathscr{C}_{\eta,0,\psi}$, both $\text{\rm Image}\bar\zeta\cap\bar M^+$ and $\text{\rm Image}\bar\zeta\cap\bar M^-$ are connected, i.e. if one moves into $\bar M_{\pm}$ along the curve as $t\to\pm\infty$ then it stays in $\bar M_{\pm}$ forever.

Let $U^+_{ij}$ be the tube connecting $D_{ij}$ to $\bar M^+$, $U^+_{ij}\cap D_{ij}=D_{ij,\delta}$, each $\bar\zeta_{ij}\in\mathscr{C}_{\eta,0,\psi}$ passes through $U^+_{ij}$, does not touch the boundary of $U^+_{ij}$ before it moves forward into $\bar M^+$. Similarly, we define the tube $U^-_{ij}$ connecting $D_{ij}$ to $\bar M^-$ such that $U^-_{ij}\cap D_{ij}=D_{ji,0}$, each of those curve passes through $U^-_{ij}$, does not touch the boundary of $U^-_{ij}$ before it is going to retreat back into $\bar M^-$.

Since there are finitely many Aubry classes only, by choosing suitably small $D'_{ij}$ we can assume $\text{\rm dist}(D'_{ij},D'_{i'j'})>0$ if $(i,j)\neq (i',j')$. A closed 1-form $\bar\mu$ clearly exists such that $[\bar\mu]=c'-c$ and $\text{\rm supp}\bar\mu \cap(\cup D_{ij})=\varnothing$. Let $\rho':\mathbb{R}\to [0,1]$ be a smooth function such that $\rho'=0$ for $s\le 0$, $\rho'=1$ for $s>\delta$ and let $U'_{ij}$ be an open set containing the closure of $U^+_{ij} \cup D'_{ij}\cup U^-_{ij}$ and $\text{\rm dist}(U'_{ij},U'_{i'j'})>0$ if $(i,j)\neq (i',j')$. We define a smooth
function $\rho$: $\bar M\to [0,1]$ such that $\rho(x)=\rho'(\langle x-x_0,v_0\rangle)$ if $x\in D_{ij}$ where $x_0=\bar\zeta_{ij} (t_0)$ and $v_0=\dot{\bar\zeta}_{ij}(t_0)$, $\rho=1$ if $x\in \bar M^+\cup (\cup U^+_{ij})$ and $\rho(x)=0$ if $x\in \bar M^-\cup(\cup U^-_{ij})$. By the construction of $\bar M^{\pm}$, $U^{\pm}_{ij}$ and $D_{i,j}$, we see the existence of such function.

Let us now study the Lagrangian $L_{\eta,\mu,\psi}$ with $\mu=\rho\bar\mu$. By condition, $\mathcal{A}(c')\cap N_i\neq\varnothing$, $\mathcal{A}(c')\cap N_j\neq\varnothing$ and $i\neq j$. Thus, there exist $x_i\in\mathcal{M}(c)\cap N_i$ and $x_j\in\mathcal{M}(c')\cap N_j$. Let $\bar x_i$ and $\bar x_j$ be two points in $\bar M$ such that $\pi\bar x_i=x_i$ and $\pi\bar x_j=x_j$ and let $\bar x_{ik}=\bar x_i-ke_1$ and $\bar x_{jk}=\bar x_j+ke_1$ where $e_1=(1,0,\cdots,0)$. Let $\bar\gamma_k$: $[-T,T] \to\bar M$ be the minimizer of
$$
\inf_{T'>0}h_{\eta,\mu,\psi}^{T'}(\bar x_{ik},\bar x_{jk})=
\int_{-T}^{T} L_{\eta,\mu,\psi}(d\bar\gamma_k(t))dt+2T\alpha(c),
$$
and let $k\to\infty$, we obtain a sequence of $\{\bar\gamma_k\}$. Let $\bar\gamma$: $\mathbb{R}\to \bar M$ be the accumulation point of the sequence. Due to the upper semi-continuity of $\mathscr
{C}_{\eta,\mu,\psi}$ with respect to $(\eta,\mu,\psi)$, the curve $\bar\gamma$ must pass through $\cup D_{ij}$ if $|c'-c|$ is suitably small. Thus, along the curve $\bar\gamma$ the term $\rho\bar\mu$ does not contribute the Lagrange equation, namely, the curve determines an orbit of $\phi_L^t$. Since this curve is in the set $\mathscr{C}_{\eta,\mu,\psi}$, therefore, it connects $\tilde{\mathcal{A}}(c)$ to $\tilde{\mathcal{A}}(c')$. This completes the proof.
\end{proof}

\subsection{Locally minimal property}

The orbit $d\gamma$ obtained in Theorem \ref{typehthm1} is locally minimal in the sense we define in the following. It is crucial for the variational construction of global connecting orbits. The set of local minimal curve will not be empty if the Aubry set $\mathcal{A}(c)$ has some totally disconnected minimal homoclinic orbit, the 1-form $\mu$ as well as the function $\psi$ is carefully chosen for the modified Lagrangian.

Here is the definition for autonomous systems:

\begin{defi}\label{localdef2} Let $N_1,\cdots,N_k\subset M$ be open domains such that $\text{\rm dist}(N_i,N_j)>0$ $(k>1)$. We assume that $\mathcal{A}(c), \mathcal{A}(c')\subset\cup N_i$, $[\eta]=c$, $[\eta+ \bar\mu]=c'$, $\alpha(c)=\alpha(c')$ and the first component of both $c$- and $c'$-minimal measures is positive $\omega_1(\mu_c)>0$, $\omega_1(\mu_{c'})>0$. Let $\pi:\bar
M=\mathbb{R}\times\mathbb{T}^{n-1}\to M$ be the covering space, denote by $\bar\gamma$ the lift of a curve $\gamma:\mathbb{R}\to M$. Then, $d\gamma$: $TM\to\mathbb{R}$ is called local minimal orbit of type-$h$ that connects $\tilde{\mathcal{A}}(c)$ to $\tilde{\mathcal{A}}(c')$ if

1, $d\gamma$ is an orbit of $\phi_L^t$, $\alpha(d\gamma)\subset\tilde{\mathcal{A}}(c)$ and $\omega(d\gamma)\subset\tilde{\mathcal{A}}(c')$. There exist $1\le i\neq j\le k$ such that $\alpha(d\gamma)\subset TN_i$ and $\omega(d\gamma)\subset TN_j$;

2, there exist two $(n-1)$dimensional disks $V_i^-$, $V_j^+\subset\bar M$ and positive numbers $T,d>0$ such that $\pi V_i^-\subset N_i\backslash\mathcal{A}(c)$, $\pi V_j^+\subset N_j\backslash\mathcal{A}(c')$, $\gamma$ transversally  passes $\pi V_i^-$ and $\pi V_j^+$ at the time $-T$ and $T$  respectively, and
\begin{align}\label{localmineq1}
&h_c^{\infty}(x^-,\pi\bar m_0)+h_{\eta,\mu,\psi}^{T'}(\bar m_0,\bar m_1)+h_{c'}^{\infty}(\pi\bar m_1,x^+) \\
&-\lim_{\stackrel{t^-_i\to\infty}{\scriptscriptstyle t^+_i\to\infty}} \int_{-t^-_i}^{t^+_i}L_{\eta,\mu,\psi} (d\gamma(t))dt-(t_i^-+t_i^+)\alpha(c)>0\notag
\end{align}
holds for each $(\bar m_0,\bar m_1,T')\in\partial(V_i^-\times V_j^+\times [T-d,T+d])$, $x^-\in N_i\cap\pi_x(\alpha(d\gamma))$ and $x^+\in N_j\cap\pi_x(\omega(d\gamma))$. Where $t^-_i\to\infty$ and $t^+_i\to\infty$ are the sequences such that $\gamma(-t^-_i)\to x^-$ and $\gamma(t^+_i)\to x^+$.
\end{defi}

In this definition, the term $h_c^{\infty}(x^-,\pi\bar m_0)+ h_{\eta,\mu,\psi}^{T'}(\bar m_0,\bar m_1)+h_{c'}^{\infty} (\pi\bar m_1,x^+)$ measures the smallest action of $L_{\eta,\mu,\psi}$ along those curves which join $m_0$ to $m_1$ with time $2T'$ such that $x^-$ is an accumulation point of these curves as $t\to-\infty$ and $x^+$ is an accumulation point of the curves as $t\to\infty$.

\noindent{\bf Remark}. In the space of curves, a neighborhood of the curve $\gamma$ consists of those curves that start from $V^-$ and reach $V^+$ within a time between $2(T-d)$ and $2(T+d)$. Different time scale determine orbits in different energy levels, that is why we consider the time scale $T'\in [T-d,T+d]$ as variable while we search for the local minimum.

\noindent{\bf Remark}. This definition applies also to the case that there exists only one Aubry class staying in the small neighborhood of lower-dimensional torus. In that case, we can consider a suitable finite covering of the configuration manifold. In the finite covering configuration space, there are more than one Aubry class.

The following is the version for time-periodic systems

\begin{defi}\label{localdef3}
Let $N_1,\cdots,N_k\subset M$ $(k>1)$ be open domains with$\text{\rm dist}(N_i,N_j)>0$. We assume that $\mathcal{A}_0(c), \mathcal{A}_0(c')\subset\cup N_i$, $[\eta]=c$, $[\eta+ \bar\mu]=c'$. Then, $d\gamma$: $TM\to\mathbb{R}$ is called  local minimal orbit of type-$h$ that connects
$\tilde{\mathcal{A}}(c)$ to $\tilde{\mathcal{A}}(c')$ if

1, $d\gamma$ is an orbit of $\phi_L^t$, the $\alpha$-limit and the $\omega$-limit sets of $d\gamma$ are contained in $\tilde{\mathcal{A}}(c)$ and $\tilde{\mathcal{A}}(c')$ respectively, $\alpha(d\gamma)|_{t=0}\subset TN_i$ and $\omega(d\gamma)|_{t=0}\subset TN_j$ with $i\neq j$;

2, there exist two open balls $V_i^-$, $V_j^+$ and two positive integers $t^-,t^+$ such that $\bar V_j^-\subset N_i\backslash\mathcal {A}_0(c)$, $\bar V_j^+\subset N_j\backslash\mathcal {A}_0(c')$, $\gamma(-k^-)\in V_i^-$, $\gamma(k^+)\in V_j^+$ and
\begin{align}\label{localeq1}
&h_c^{\infty}(x^-,m_0)+h_{\eta,\mu,\psi}^{k^-,k^+}(m_0,m_1)+h_{c'}^{\infty}(m_1,x^+)\notag\\
&-\liminf_{\stackrel {k^-_i\to\infty}{\scriptscriptstyle k_i^+\to\infty}}\int_{-k^-_i}^{k^+_i}
L_{\eta,\mu,\psi}(d\gamma(t),t)dt-k^-_i\alpha(c)-k^+_i\alpha(c')>0\notag
\end{align}
holds $\forall$ $(m_0,m_1)\in\partial(V_i^-\times V_j^+)$, $x^-\in N_i\cap\pi_x(\alpha(d\gamma))_{t=0}$, $x^+\in N_j\cap\pi_x(\omega(d\gamma))|_{t=0}$, where $k^-_i, k^+_i\in\mathbb{Z}^+$ are the sequences such that $\gamma(-k^-_i)\to x^-$ and $\gamma(k^+_i)\to x^+$.
\end{defi}

The set of curves starting from $V^-$ and reaching $V^+$ with time $k^-+k^+$ constitutes a neighborhood of the curve $\gamma$ in the space of curves. Once a curve $\tilde\gamma$ touches the boundary of this neighborhood, the action of $L_{\eta,\mu,\psi}$ along $\tilde\gamma$ will be larger than the action along $\gamma$. As $V^-$, $V^+$ and therefore $d>0$ can be chosen arbitrarily small, it is reasonable to call it locally minimal.

\section{\ui Variational construction of global connecting orbits}
\setcounter{equation}{0}

In this section we show how to construct global connecting orbits by variational method, provided a generalized transition chain exists. In the next section, the main result (Theorem \ref{mainthm}) is proved by showing the genericity of such transition chain.

\subsection{Generalized Transition chain}

The concept of transition chain was proposed by Arnold in his celebrated paper \cite{Ar1} where it is formulated in geometric language. The generalized transition chain formulated in our previous work \cite{CY1,CY2} is in variational version which need less information about the geometric structure.

\begin{defi}\label{chaindef1} {\rm (Autonomous Case)}
Let $c$, $c'$ be two cohomolgy classes in $H^1(M,\mathbb{R})$. We say that $c$ is joined with $c'$ by a generalized transition chain if a continuous curve $\Gamma$: $[0,1]\to H^1(M,\mathbb{R})$ exists such that $\Gamma(0)=c$, $\Gamma(1)=c'$, $\alpha(\Gamma(s))\equiv E>\min\alpha$ and for each $s\in [0,1]$  at least one of the following cases takes place:

$\text{\rm (H1)}$, the Aubry set is composed of finitely many classes only. There exist certain finite covering: $\check{\pi}:\check{M}\to M$, two open domains $N_1,N_2\subset\check M$ with $d(N_1,N_2)>0$, an $(n-1)$ dimensional disk $D_{s}$ and small numbers $\delta_s,\delta'_s>0$ such that

{\rm i}, the Aubry set $\mathcal{A}(\Gamma(s))\cap N_1\neq\varnothing$, $\mathcal{A}(\Gamma(s))\cap N_2\neq\varnothing$ and $\mathcal{A}(\Gamma(s'))\cap (N_1\cup N_2)\neq\varnothing$ for each $|s'-s|<\delta_s$,

{\rm ii}, $\check\pi\mathcal{N} (\Gamma(s),\check M)|_{D_{s}}\backslash (\mathcal{A}(\Gamma(s))+\delta'_s)$ is non-empty and totally disconnected;

$\text{\rm (H2)}$, For each $s'\in (s-\delta_s,s+\delta_s)$, $\Gamma(s')$ is equivalent to $\Gamma(s)$. Some section $\Sigma_s$ and some neighborhood $U$ of $\mathcal{N}(\Gamma(s))\cap \Sigma_{s}$ exist such that  $\Gamma(s')-\Gamma(s)\in\text {\rm ker}\,i^*_U$. Each class $\Gamma(s')$ is associated with an admissible section $\Sigma_{s'}$ and an admissible coordinate system $G_{s'}^{-1}x$.
\end{defi}

\noindent{\bf Remark}. Because of upper semi-continuity of Ma\~n\'e set, it is possible that there exist some classes for which both cases take place.

In the case (H1), if the Aubry set contains only one Aubry class, one can take some finite covering $\check\pi:\check M\to M$ non trivial if $H_1(M,\mathcal{A},\mathbb{Z})\neq 0$. A typical case is that $\mathcal{A}(\Gamma(s))$ is contained in a small neighborhood of lower dimensional torus. One takes suitable finite covering space so that $\mathcal{A}(\Gamma(s),\check M)$ contains exactly two connected components. If $\mathcal{A}(\Gamma(s))$ contains more than one class, we choose $\check M=M$.

The existence of generalized transition chain implies that there exists a sequence of local connecting orbits. More precisely, there exists a sequence of locally minimal curve $\gamma_i$, a sequence of numbers $s_i$ $(s=0,1,\cdots,m)$ such that $\alpha(d\gamma_i)\subset \mathcal{A}(\Gamma(s_i))$ and $\omega(d\gamma_i))\subset\mathcal {A}(\Gamma(s_{i+1}))$. Global connecting orbits are constructed shadowing these local connecting orbits.

One can also define generalized transition chain for time-periodic systems.
\begin{defi}{\rm (Time-periodic Case)}
Let $c$, $c'$ be two classes in $H^1(M,\mathbb{R})$. We say that $c$ is joined with $c'$ by a generalized transition chain if a continuous curve $\Gamma$: $[0,1]\to H^1(M,\mathbb{R})$ exists such that $\Gamma(0)=c$, $\Gamma(1)=c'$ and for each $s\in [0,1]$  at least one of the following cases takes place:

$\text{\rm (H1)}$, the Aubry set is composed of finitely many classes only. There exist certain finite covering: $\check{\pi}:\check{M}\to M$, two open domains $N_1, N_2$ with $d(N_1,N_2)>0$ and small number $\delta_s,\delta'_s>0$ such that

{\rm i}, the Aubry set $\mathcal{A}_0(\Gamma(s))\cap N_1\neq\varnothing$, $\mathcal{A}_0(\Gamma(s))\cap N_2\neq\varnothing$ and $\mathcal{A}_0(\Gamma(s'))\cap(N_1\cup N_2)\neq\varnothing$ for each $|s'-s|<\delta_s$,

{\rm ii}, $\check\pi\mathcal{N}_0 (\Gamma(s),\check M)\backslash(\mathcal{A}_0(\Gamma(s))+\delta'_s)$ is non-empty and totally disconnected;

$\text{\rm (H2)}$, For each $s'\in (s-\delta_s,s+\delta_s)$, $\Gamma(s')$ is equivalent to $\Gamma(s)$, namely, there exists a neighborhood of $\mathcal{N}_0(\Gamma(s))$, denoted by $U$,  such that  $\Gamma(s')-\Gamma(s)\in\text {\rm ker}\,i^*_U$.
\end{defi}

\subsection{Variational construction}

Given $x\in M$ and $c\in H^1(M,\mathbb{R})$, there exists at least a forward (backward) $c$-semi static curve $\gamma^+_c$: $[0,\infty)\to M$ ($\gamma^-_c$: $(-\infty,0]\to M$) such that $\gamma_c ^{\pm}(0)=x$. It determines certain velocity $v_{x,c}^{\pm}=\dot\gamma_c^{\pm}(0)$, for almost all points, the velocity is uniquely determined. Before proving the main theorem of this subsection, let us formulate and prove a proposition.
\begin{pro}\label{constructionpro1}
Given an Aubry set, the Aubry distance from any Aubry class $\mathcal{A}^i$ to all other Aubry classes is assumed have positive lower bound, namely, some $d>0$ exists such that $d_c(\mathcal{A}^i,\mathcal{A}^j)\ge d>0$ for all $j\neq i$. Let
$$
N_i=\{m\in M: h^{\infty}(m,x)+h^{\infty}(x,m)\le\frac d6,\ \forall\ x\in\mathcal{A}^i\},
$$
then for all $m_0,m_1\in N_i$ and for any $x\in\mathcal{A}^i$ one has
\begin{equation}\label{constructioneq1}
h^{\infty}(m_0,x)+h^{\infty}(x,m_1)=h^{\infty}(m_0,m_1);
\end{equation}
for any $m_0,m_1\in N_i$ and any $x\in\mathcal{A}\backslash\mathcal{A}^i$ one has
\begin{equation}\label{constructioneq2}
h^{\infty}(m_0,x)+h^{\infty}(x,m_1)\ge h^{\infty}(m_0,m_1)+\frac d2.
\end{equation}
\end{pro}
\begin{proof}: For each pair of points $(m_0,m_1)\in M\times M$, we claim that there exists some Aubry class $\mathcal{A}^j$ such that
$$
h^{\infty}(m_0,m_1)=h^{\infty}(m_0,x)+h^{\infty}(x,m_1)
$$
holds for each $x\in \mathcal{A}^j$. Indeed, let $k_i\to\infty$ be a subsequence of integers such that
$$
\lim_{i\to\infty}h^{k_i}(m_0,m_1)=h^{\infty}(m_0,m_1),
$$
let $\gamma^{k_i}$: $[-k_i,k_i]\to M$ be the minimizer for $h^{k_i}(m_0,m_1)$. There exists at least one point $x\in\mathcal{A}$ which is the accumulation point of $\{\gamma^{k_i}(t_i)\}_{i\in\mathbb{Z}}$. Otherwise, the quantity $h^{k_i}(m_0,m_1)\to\infty$ as $k_i\to\infty$.

Given $m\in N_i$, we claim that (\ref{constructioneq1}) and (\ref{constructioneq2}) hold if $m_0=m_1=m$. Let $k_{\ell}\to\infty$ be a sequence such that $\lim_{k_{\ell}\to\infty}h^{k_{\ell}}(m,m)
=h^{\infty}(m,m)$ and let $\gamma^{k_{\ell}}_m(t)$: $[-k_{\ell}, k_{\ell}]\to M$ be the minimizer of $h^{k_{\ell}}(m,m)$. There is a positive number $d_1>0$ such that the ordinary distance $d(\gamma^{k_{\ell}}_m(t), \mathcal{A}^j)\ge d_1>0$ for any $t\in [-k_{\ell},k_{\ell}]$ and $j\neq i$. Otherwise along the curve $\gamma_m^{k_{\ell}}(t)$ there exists a point getting closer and closer to a point $x_j\in\mathcal{A}^j$. Consequently, one would obtain from the property that $d_c(\mathcal{A}^i,\mathcal{A}^j)\ge d>0$ for each $j\neq i$ that
\begin{align*}
h^{\infty}(m,m)=& h^{\infty}(m,x_j)+h^{\infty}(x_j,m)\\
\ge &h^{\infty}(x_i,x_j)-h^{\infty}(x_i,m)+h^{\infty}(x_j,x_i)-h^{\infty}(m,x_i) \\
\ge &\frac 56d
\end{align*}
where $x_i\in\mathcal{A}^i$. On the other hand, we have
$$
h^{\infty}(m,m)\le h^{\infty}(m,x_i)+h^{\infty}(x_i,m)\le\frac 16d.
$$
It is a contradiction. Therefore, some $x_i\in\mathcal{A}^i$ and $t_{\ell}\in [0,k_{\ell}]$ exist such that $t_{\ell}\to\infty$ as $k_{\ell}\to\infty$ and $\gamma^{k_{\ell}}_m(t_{\ell})\to x_i$. This proves (\ref{constructioneq1}) in case $m_1=m_2$.

For different points $m_0,m_1\in N_i$ and $x\in\mathcal{A}^j$ with $j\neq i$, let $\zeta^{k}_s(t,m_0,x)$: $[-k,k]\to M$ be the curve which minimizes the quantity $h^{k}(m_0,x)$, let $k_{j}$ be the subsequence of $k$ such that $\lim_{k_{j}\to\infty}h^{k_{j}}(m_0,x) =h^{\infty}(m_0,x)$. In
autonomous case, it converges as $k\to\infty$. Similarly, we let $\zeta^{k}_u(t,x,m_1)$: $[-k,k]\to M$ be the curve which minimizes the quantity $h^{k}(x,m_1)$, let $k'_j$ be the sequence of $k$ such that $\lim_{k'_{j}\to\infty}h^{k'_{j}}(x,m_1) =h^{\infty}(x,m_1)$. Let $\ell=0,1$, $\gamma_{\ell}^{k}$: $[-k,k]\to M$ be the minimizer of $h^{k}(m_{\ell},m_{\ell})$ and let $k_{\ell}$ be the subsequence of $k$ such that $h^{k_{\ell}}(m_{\ell},m_{\ell})\to h^{\infty}(m_{\ell},m_{\ell})$. By the proof we just finished, there exists $x_{\ell}\in\mathcal{A}^i$ and integer $t_{\ell}^i\in [-k_{\ell},k_{\ell}]$ such that $\gamma_{\ell}^{k_{\ell}} (t^i_{\ell})\to x_{\ell}$ and $t^i_{\ell}\to\infty$ as $k_{\ell}\to\infty$. Let $\xi^k_{01}$: $[-k,k]\to M$ be the minimizer of $h^{k}(x_0,x_1)$, $k_{01}^i$ be the subsequence of $k$ such that $h^{k^i_{01}}(x_0,x_1)\to h^{\infty} (x_0,x_1)$, let $\xi^i_{10}$: $[0,k]\to M$ be the minimizer of $h^{k}(x_1,x_0)$, $k_{10}^i$ be the subsequence of $k$ such that $h^{k^i_{1}}(x_1,x_0)\to h^{\infty} (x_1,x_0)$. Given arbitrarily small $\delta>0$, we have sufficiently large $k_j$, $k'_j$, $k_0^i$, $k_1^i$, $k^i_{01}$ and $k^i_{10}$ such that
\begin{align*}
&|h^{\infty}(m_0,x)-h^{k_j}(m_0,x)|<\delta,\\
&|h^{\infty}(x,m_1)-h^{k'_j}(x,m_1)|<\delta,\\
&|h^{\infty}(m_{\ell},m_{\ell})-h^{k^i_{\ell}}(m_{\ell},m_{\ell})|
<\delta,\qquad \ell=0,1 \\
&|h^{\infty}(x_0,x_1)-h^{k^i_{01}}(x_0,x_1)|<\delta,\\
&|h^{\infty}(x_1,x_0)-h^{k^i_{10}}(x_1,x_0)|<\delta.
\end{align*}
Since $x_0,x_1\in \mathcal{A}^i$, we have $d_c(x_1,x_0)=0$. Consequently,
\begin{align}\label{constructioneq3}
&h^{t_0^i}(m_0,x_0)+h^{k^i_{01}}(x_0,x_1)+h^{k_1^i-t_1^i}(x_1,m_1)\\
+&h^{t_1^i}(m_1,x_1)+h^{k^i_{10}}(x_1,x_0)+h^{k_0^i-t_0^i}(x_0,m_0)\notag\\
\le&\frac 13d+6\delta.\notag
\end{align}
Since $x$ is in Aubry class $\mathcal{A}^j$, while $x_0,x_1\in\mathcal{A}^i$, one has
\begin{align}\label{constructioneq4}
&h^{k'_j}(x,m_1)+h^{t_1^i}(m_1,x_1)+h^{k^i_{10}}(x_1,x_0)\\
&+h^{k_0^i-t_0^i}(x_0,m_0)+h^{k_j}(m_0,x)\notag\\
\ge&d-5\delta.\notag
\end{align}
Because $\delta$ can be arbitrarily small, by subtracting (\ref{constructioneq3}) from (\ref{constructioneq4}) we obtain
\begin{align*}
&h^{\infty}(m_0,x)+h^{\infty}(x,m_1)-\frac 23d\\
\ge &h^{\infty}(m_0,x_0)+h^{\infty}(x_0,x_1)+h^{\infty}(x_1,m_1)\\
\ge &h^{\infty}(m_0,m_1)
\end{align*}
it verifies (\ref{constructioneq2}). Since (\ref{constructioneq2}) holds for each $x\in\mathcal{A}^j$ with $j\neq i$ and for any $m_0,m_1\in N_i$, (\ref{constructioneq1})  hold for each $x\in\mathcal{A}^i$  and for any $m_0,m_1\in N_i$. This completes the proof of the proposition.
\end{proof}
\begin{theo}\label{constructionthm1}
If $c$ is connected to $c'$ by a generalized transition chain, then

1, there exists an orbit of the Lagrange flow $\phi_L^t$, $d\gamma$: $\mathbb{R}\to TM$ which connects the Aubry set $\tilde{\mathcal{A}}(c)$ to $\tilde{\mathcal{A}}(c')$, namely, $\alpha(d\gamma)\subseteq\tilde{\mathcal{A}}(c)$ and $\omega(d\gamma)\subseteq\tilde{\mathcal{A}}(c')$;

2, given $x,x'\in M$ and arbitrarily small $\delta>0$, there exists an orbit  $(\gamma,\dot\gamma)$ of $\phi_L^t$ passing through $\delta$-neighborhood of the points $(x,v_{x,c}^{+})$  and  $(x',v_{x,c'}^{-})$ successively, namely, $t<t'$ such that $(\gamma(t),\dot\gamma(t))\in B_{\delta}(x,v_{x,c}^{+})$ and $(\gamma(t'),\dot\gamma(t'))\in B_{\delta}(x',v_{x,c'}^{-})$.
\end{theo}
\begin{proof}
We only need to study the autonomous case. Time periodic case can be treated in the same way. Therefore, one has that $\alpha(\Gamma(s))\equiv E> \min\alpha$. By adding suitable constant on the Lagrangian, we assume $E=0$ to simplify notation.

First of all, as a generalized transition chain $\Gamma(s)$ is assumed, the Aubry set $\tilde{\mathcal{A}} (\Gamma(s))$ can be connected to some $\tilde{\mathcal{A}} (\Gamma(s'))$ by locally minimal orbits of  either type-$h$, or type-$c$ if $s'$ close to $s$. So, there is a sequence $0=s_0<s_1<\cdots<s_k=1$ such that for each $0\le j<k$, $\tilde{\mathcal{A}} (\Gamma(s_j))$ is connected to $\tilde{\mathcal{A}}(\Gamma(s_{j+1}))$ by some local minimal orbits. The global connecting orbits are constructed shadowing such a sequence of orbits.

Let $c_j=\Gamma(s_j)$. We divide the set $\{0,1,\cdots,k\}$ into $m$ parts
$$
\{0,1,\cdots
k\}=\{0,1,\cdots,i_1\}\cup\{i_1+1,\cdots,i_2\}\cup\cdots\cup\{i_{m-1}+1,\cdots, i_m=k\}.
$$
The rule to make such a partition is that for all $i=i_{j},i_{j}+1,\cdots, i_{j+1}-1$, $\tilde{\mathcal{A}} (c_i)$ is connected to
$\tilde{\mathcal{A}}(c_{i+1})$ by a local minimal orbit of the same type. More precisely, let $\Lambda_c$ and $\Lambda_h$ be the subset of $\{i_1,i_2,\cdots, i_m\}$, $\Lambda_c\cup\Lambda_h= \{i_1,i_2,\cdots,i_m\}$, $\Lambda_c\cap\Lambda_h=\varnothing$. If $i_j\in\Lambda_{\imath}$, then for all $i=i_j,i_j+1,\cdots,i_{j+1}-1$, $\tilde{\mathcal{A}} (c_i)$ is connected to $\tilde{\mathcal{A}}(c_{i+1})$ by a local minimal orbit of type-$\imath$ ($\imath$ =$c$, or $h$).

Since the map $c\to\tilde{\mathcal{N}}(c,M)$ is upper semi-continuous, once the Ma\~n\'e set $\tilde{\mathcal{N}}(\Gamma(s))$ is in the case (H1) (or H2), then for $s'$ sufficiently close to $s$, the set $\tilde{\mathcal{N}}(\Gamma(s))$ is also in the case (H1) (or H2). Thus, for each $i_j\in\Lambda_h$, by choosing $c_{i_j-1}$ and $c_{i_{j+1}}$ sufficiently close to $c_{i_j}$ and $c_{i_{j+1}-1}$ respectively, we can assume that both $c_{i_j-1}$ and $c_{i_{j+1}}$ satisfy the condition (H1) also.

With the class $c_i$ we associate an admissible coordinate system $x\to G_i^{-1}x$ and let $G_i^{-1}=[g_{i,1}^{-1}, g_{i,2}^{-1},\cdots, g_{i,n}^{-1}]^t$ denote the inverse of $G_i$. Because we consider the problem on $H^{-1}(E)$ with $E>\min\alpha$, we can choose $G_i$ for each $i\in\Lambda_c$ (see Proposition \ref{typehpro1}) such that, in the new coordinate system, the first component of $\omega(\mu_{c_i})$ is positive for each ergodic $c_i$-minimal measure. In virtue of the upper semi-continuity of Ma\~n\'e set on cohomology class, one can assume that
$$
\langle g_{j,1}^{-1},\omega(\mu_{c_i})\rangle >0, \qquad \forall \ j=i-1,i
$$
holds for each ergodic component $\mu_{c_i}$ as $c_{i-1}$ is chosen suitably close to $c_i$. It means that the $\omega_1(\mu_{c_i})>0$ holds in the coordinates not only determined by $G_i$, but also determined by $G_{i-1}$ as well as by $G_{i+1}$. Therefore, $\exists$ $x_{i,1}>0$ such that
\begin{align}\label{constructioneq5}
&\langle g_{i,1}^{-1}, \Delta\tilde\gamma_i\rangle\ge 2\pi,\ \ \ \text{\rm whenever}\ \ \langle g_{i-1,1}^{-1}, \Delta\tilde\gamma_i\rangle\ge x_{i,1}, \notag\\
&\langle g_{i-1,1}^{-1}, \Delta\tilde\gamma_i\rangle\ge 2\pi,\ \ \ \text{\rm whenever}\ \ \langle g_{i,1}^{-1}, \Delta\tilde\gamma_i\rangle\ge x_{i,1}
\end{align}
holds for each $c_i$-semi-static curve $\gamma_i$, where $\tilde\gamma_i$ denotes a curve in the lift of $\gamma_i$ to universal covering space and
$\Delta\tilde\gamma_i=\tilde\gamma_i(t')-\tilde\gamma_i(t)$ with $t'>t$.

As the second step, let us describe the minimal properties of local connecting orbits of type-$h$ as well as of type-$c$.

{\bf The case of type-$h$}. For each integer $i\in\bigcup_{i_j\in\Lambda_h}\{i_j,i_j+1,\cdots,i_{j+1}-1\}$, the condition (H1) holds for generalized transition chain. Namely, in certain finite covering space $\check{\pi}:\check{M}\to M$, the Aubry set for $i$ and $i+1$ consists of more than one but finitely many classes $\mathcal{A}(c_{\ell}, \check{M})=\cup \mathcal{A}_{\ell}^j$ for $\ell=i,i+1$. By the assumption of (H1), some open domains $N^-_{i},N^+_{i+1}\subset\check M$ exist such that $d(N^-_i,N^+_{i+1})>0$, $\mathcal{A}(c_i,\check{M})\cap N^-_{i}\neq\varnothing$, $\mathcal{A}(c_i,\check{M})\cap N^+_{i+1}\neq\varnothing$, $\mathcal{A}(c_{i+1},\check{M})\cap N^+_{i+1}\neq \varnothing$ and $\mathcal {N}(c_i,\check M)\backslash (\mathcal{A}(c_i)+\delta'_i)\neq\varnothing$ is totally disconnected, with small $\delta'_i>0$.

The new coordinate system $x\to G_i^{-1}x$ of $\mathbb{T}^n$ is introduced such that  (\ref{constructioneq5}) holds. If writing $\check M=\{(x_1,x_2,\cdots,x_n):x_i\in\mathbb{T}\}$, we introduce the covering space $\bar M=\{(x_1,x_2,\cdots,x_n):x_1\in\mathbb{R}, x_i\in\mathbb{T}\ \text{\rm for}\ i\ge 2\}$. We shall work with this covering space $\pi:\bar M\to\check M$. For closed 1-forms $\eta_i$ and $\bar\mu_i$ on $\check{M}$ we use the same symbol to denote their natural lift to $\bar M$.

Recall the proof of Theorem \ref{typehthm1}. Some decomposition of $\bar M$ exists such that
$\bar M=\bar M_i^+\cup\bar M_{i,0}\cup\bar M_i^-$ such that $\bar M_i^+$ is diffeomorphic to $[0,\infty)\times \mathbb{T}^{n-1}$, $\bar M_i^-$ is diffeomorphic to $(-\infty,0]\times\mathbb{T}^{n-1}$ and $\bar M_{i,0}$ is diffeomorphic to $(0,1)\times \mathbb{T}^{n-1}$. Some open and connected disks $U^+_i,U^-_i,D_i,D'_i\subset \bar M_{i,0}$ and $\delta_i>0$ exist such that $D_i+\delta_i \subset D'_i$, $(\pi D'_i+\delta_i)\cap N^-_i=\varnothing$ and $(\pi D'_i+\delta_i)\cap N^+_{i+1}=\varnothing$, the intersection of any two of these sets is empty and the closure of $\bar M_i^+\cup U^+_i\cup D_i\cup U^-_i\cup\bar M_i^-$, denoted by $\bar M^c_i$, is connected, see Figure \ref{fig11}.
\begin{figure}[htp] 
  \centering
  \includegraphics[width=10cm,height=2.2cm]{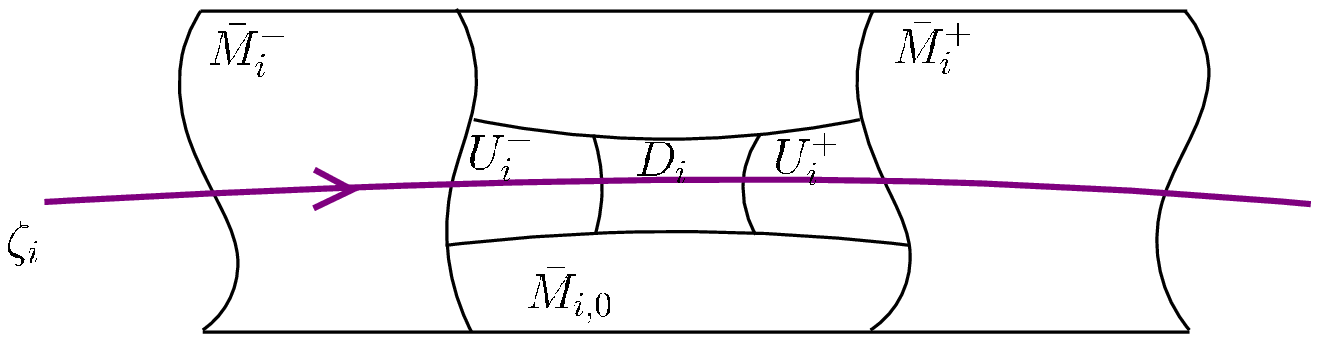}
  \caption{}
  \label{fig11}
\end{figure}

As it was studied in the subsection of 6.3 (local connecting orbit of type-$h$), some function $w_i$, $\rho_i$: $\bar M\to [0,1]$, some closed 1-form $\eta_i$, $\bar\mu_i$, $\varrho_i$ and some small constant $\delta_i>0$ exist such that $[\eta_i]=c_i$, $[\bar\mu_i]=c_{i+1}-c_i$, $\text{\rm supp}\bar\mu_i\cap\bar D_i=\varnothing$, $\text{\rm supp}w_i\subset D'_i$, $w_i|_{D_i}=\text{\rm constant}$, $\rho_i(x)=1$ if $x\in\bar M_i^+\cup U_i^+$ and $\rho_i(x)=0$ if $x\in U_i^-\cup\bar M_i^- $.

Let $\mu_i=\rho_i\bar\mu_i$, $\psi_i=w_i\varrho_i$ we introduce a space-step Lagrangian for the coordinate system $G_i^{-1}x$
$$
L_{\eta_i,\mu_i,\psi_i}=L-\eta_i-\mu_i-\psi_i.
$$
In virtue of Theorem \ref{typehthm1}, some curve $\bar\zeta_i\in\mathscr{C}_{\eta_i,\mu_i,\psi_i}$ (pseudo connecting orbit set) such that its projection down to $\check{M}$, $\zeta_i=\pi\bar\zeta_i$, determines an orbit $d\zeta_i=(\zeta_i,\dot\zeta_i)$ connecting certain Aubry class $\tilde{\mathcal{A}}^j_i$ to another Aubry class $\tilde{\mathcal{A}}^{j'}_{i+1}$. Here, the subscript $i$ indicates the Aubry set is for the cohomology class $c_i$, the superscript $j$ indicates which Aubry class it belongs to. Such curve stays entirely in the interior of $\bar M_i^c$. Therefore, along such curve both $\mu_i$ and $\psi_i$ do not contribute to the Euler-Lagrange equation, i.e. $d\zeta_i$ is an orbit of $\phi_L^t$.

As pointed out in Definition \ref{localdef2}, such local connecting orbit of type-$h$ is endowed with certain kind of local minimality. There exist two $(n-1)$-dimensional disks $V_i^-$ and $V_{i+1}^+$ with $\pi V^-_i \subset N^-_i\backslash\mathcal{A}^j_i$, $\pi V_{i+1}^+\subset N^+_{i+1}\backslash\mathcal{A}^{j'}_{i+1}$, large number $T_i^+>0$, suitably small $d_i>0$ and quite small $\epsilon_i^*>0$ such that $\bar\zeta_i(-T_i^+)\in V_i^-$, $\bar\zeta_i(T_i^+)\in V_{i+1}^+$ and
\begin{align}\label{constructioneq6}
\min&\Big\{h_{c_i}^{\infty}(x^-,m_0)+h_{\eta_i,\mu_i,\psi_i}^{T}(\bar m_0,\bar m_1)+h_{c_{i+1}}^{\infty} (m_1,x^+):\\
&(\bar m_0,\bar m_1,T)\in\partial (V_i^-\times V_{i+1}^+\times[T_i^+-d_i,T_i^++d_i])\Big\}\notag\\
\ge\min&\Big\{h_{c_i}^{\infty}(x^-,m_0)+h_{\eta_i,\mu_i,\psi_i}^{T}(\bar m_0,\bar m_1) +h_{c_{i+1}}^{\infty}(m_1,x^+):\notag\\
&(\bar m_0,\bar m_1,T)\in V_i^-\times V_{i+1}^+\times[T_i^+-d_i,T_i^++d_i] \Big\} +5\epsilon_i^*,\notag
\end{align}
where $x^-\in\pi_x\alpha(d\zeta_i)\subseteq\mathcal{A}^j_i$, $x^+\in\omega(d\zeta_i)\subseteq\pi_x \mathcal{A}^{j'}_{i+1}$. The disks $V^-_i$ and $V^+_{i+1}$ are chosen so that $\bar\zeta_i$ intersects them transversally, for each $(\bar m_0,\bar m_1,T')$ $\in V_i^-\times V_{i+1}^+\times[T_i^+-d_i,T_i^++d_i]$, the minimizer of $h_{\eta_i,\mu_i,\psi_i} ^{T}(\bar m_0,\bar m_1)$, $\bar\gamma_i(t,\bar m_0,$ $\bar m_1,T)$ has the property
\begin{equation}\label{constructioneq7}
\bar\gamma_i(t)\in \bar M_i^c\qquad\forall \ t\in [-T,T].
\end{equation}

Let $\zeta_{i-1}$ be a locally minimal curve such that the orbit $d\zeta_{i-1}$ connecting $\tilde{\mathcal{A}}_{i-1}$ to $\tilde{\mathcal{A}}_{i}$. Denote by $\tilde{\mathcal{A}}^{j'}_i$ the Aubry class which contains the $\omega$-limit set of $d\zeta_{i-1}$. It is possible that $\mathcal{A}^{j'}_i$ is different from $\mathcal{A}^j_i$ which contains the $\alpha$-limit set of $d\zeta_i$. Remember that we have assumed that each Aubry set consists of finitely many classes, let's say, $k_i$ classes. By the result in \cite{CP}, the subscript of Aubry classes can be rearranged such that some $c_i$-semi-static curve $\gamma_{i,j}$ exists such that $\alpha(d\gamma_{i,j})\subset\tilde{\mathcal{A}}^j_i$ and $\omega(d\gamma_{i,j}) \subset\tilde{\mathcal{A}}^{j+1}_i$ for $j=1,\cdots,k_i$ $(\text{\rm mod}\ k_i)$. So some positive integer $k\le k_i$ exists such that $j_i-j'_i=k\le k_i$ $(\text{\rm mod}\ k_i)$. Let $d_i=\min d_c(\mathcal{A}^j_i,\mathcal{A}^{j'}_i)$.

We choose some $(n-1)$-dimensional small disks $V^{\pm}_{i,j}$ with $j=j'_i,\cdots,j_i$ such that $V^+_{i,j'_i}=V^+_i$, $V^-_{i,j_i}=V^-_i$, $V_{i,j}^{\pm}$ is located within $N_{i,j}$, a small neighborhood of $\mathcal{A}^j_i$ such that $d_c(m,x)\le\frac{d_i}6$ holds for each $m\in N_{i,j}$ and each $x\in\mathcal{A}^j_i$ (the definition of $N_{i,j}$ is the same as $N_i$ in Proposition \ref{constructionpro1}), $\gamma_{i,j}$ intersects $V^-_{i,j}$ as well as $V^+_{i,j+1}$ transversally. These curves also have locally minimal property similar to the form of (\ref{constructioneq6}):
\begin{align}\label{constructioneq8}
\min&\Big\{h_{c_i}^{\infty}(x^-,m_0)+h_{c_i}^{T}(m_0,m_1)+h_{c_{i}}^{\infty}(m_1,x^+):\notag\\
&(\bar m_0,\bar m_1,T)\in\partial (V_{i,j}^-\times V_{i,j+1}^+\times[T^+_{i,j}-\tau_i,T^+_{i,j}+\tau_i]) \Big\}\\
\ge\min&\Big\{h_{c_i}^{\infty}(x^-,m_0)+h_{c_i}^{T}(m_0,m_1)+h_{c_{i}}^{\infty}(m_1,x^+):\notag\\
&(\bar m_0,\bar m_1,T)\in V_{i,j}^-\times V_{i,j+1}^+\times[T^+_{i,j}-\tau_i,T^+_{i,j}+\tau_i] \Big\}
+5\epsilon_i^*,\notag
\end{align}
where $x^-\in\mathcal{A}^j_i$, $x^+\in\mathcal{A}^{j+1}_i$, $T^+_{i,j}$ is the time such that
$\gamma_{i,j}(2T^+_{i,j})\in V_{i,j+1}^+$ if $\gamma_{i,j}(0)\in V^-_{i,j}$. As semi-static curves are totally disconnected, $V^-_{i,j}$ and $V^+_{i,j+1}$ are chosen so that any curve in $\mathcal{N}(c_i)$ does not touch the boundary of $V^-_{i,j}$ and of $V^+_{i,j+1}$.

Note that $h_c^{\infty}=\lim_{T\to\infty}h_c^T$ in the autonomous case \cite{Fa1}, we find from Proposition \ref{constructionpro1} that, for any $\epsilon^*_i>0$, there exists $T^-_{i,j}=T^-_{i,j}(\epsilon^*)>0$ such that
\begin{equation}
|h_{c_i}^T(m,m')-h^\infty_{c_i}(m,x)-h^\infty_{c_i}(x,m')|\le\epsilon^*_i\notag
\end{equation}
holds for each $T\ge T^-_{i,j}$, each $m,m'\in N_{i,j}$ and $x\in\mathcal{A}^j_i$.

Let $t^+_i=t^+_{i,j'}<t^-_{i,j'}<t^+_{i,j'+1}\cdots<t^+_{i,j}<t^-_{i,j}=t^-_i$, $2\Delta t^+_{i,j}=t^+_{i,j+1}-t^-_{i,j}$, $2\Delta t^-_{i,j}=t^-_{i,j}-t^+_{i,j}$. A curve $\gamma$: $[t^+_i,t^-_i]\to\check{M}$ is called {\it admissible} for $V^{\pm}_{i,j}$ if
\begin{equation}\label{constructioneq9}
\gamma(t^{\pm}_{i,j})=x^{\pm}_{i,j}\in V^{\pm}_{i,j}, \qquad \forall\ j'_i\le j\le
j_i,
\end{equation}
where $\Delta t^+_{i,j}$ and $\Delta t^-_{i,j}$ are  chosen to satisfy the condition
\begin{equation}\label{constructioneq10}
T^+_{i,j}-\tau_i\le\Delta t^+_{i,j}\le T^+_{i,j}+\tau_i
\end{equation}
where $\tau_i$ is chosen in the inequality (\ref{constructioneq8}), the condition
\begin{equation}\label{constructioneq11}
T^-_{i,j}+T^+_{i,j}+\tau_{i}\le \Delta t_{i,j}^-+\Delta t_{i,j}^+\le T^-_{i,j}+T^+_{i,j}+\tau^*_i
\end{equation}
which is set so that the minimal curve does not touch the boundary of $V^-_{i,j}$ provided it passes through $V^+_{i,j}$ at $t=t^+_{i,j}$ and through $V^+_{i,j+1}$ at $t=t^+_{i,j+1}$ and the condition
\begin{equation}\label{constructioneq12}
T^-_{i,j}+T^+_{i,j-1}+\tau_{i}\le \Delta t_{i,j}^-+\Delta t_{i,j-1}^+ \le T^-_{i,j}+T^+_{i,j-1}+\tau^*_i.
\end{equation}
which is set so that the minimal curve does not touch the boundary of $V^+_{i,j}$ provided it passes through $V^-_{i,j-1}$ at $t=t^-_{i,j-1}$ and through $V^-_{i,j}$ at $t=t^-_{i,j}$. These conditions define non-empty set for $(\Delta t^+_{i,j},\Delta t^-_{i,j})$ if we choose suitably large $\tau^*_i>0$.

We consider the minimum of the following action among all admissible curves:
$$
h_{c_i}^{t^+_i,t^-_i}(x^+_i,x^-_i)=\inf_{\stackrel{\scriptscriptstyle \gamma(t^+_i)=x^+_i\in V^+_{i,j'_i} } {\scriptscriptstyle \gamma(t^-_i)=x^-_i\in V^-_{i,j_i}}} \int_{t^+_i}^{t^-_i}(L-\eta_i)(d\gamma)dt.
$$
Let $\gamma(t,t^{\pm}_i,x^{\pm}_i)$: $[t^+_i,t^-_i]\to\check{M}$ be the minimizer of the action. If $t^-_{i,j}-t^+_{i,j}$ is sufficiently large, the minimizer is smooth at each $t^-_{i,j'_i}<t^+_{i,j'_i+1}<\cdots<t^+_{i,j_i}$. First of all, we claim that $(x_{i,j}^-,x^+_{i,j+1}, \Delta t^+_{i,j})\in\text{\rm int}(V^-_{i,j}\times V^+_{i,j+1}\times [T^+_{i,j}-\tau_i, T^+_{i,j}+\tau_i])$ holds for each $j'_i\le j<j_i$. If it does not hold for certain $j'_i\le j<j_i$, one obtains from (\ref{constructioneq8}) that
\begin{align}
&h_{c_{i}}^{\Delta t_{i}^-}(x^+_{i,j},x^-_{i,j})+h^{\Delta t_i^+}_{c_i}(x^-_{i,j},x^+_{i,j+1})+h^{\Delta t_{i+1}^-}_{c_{i}}(x^+_{i,j+1}, x^-_{i,j+1})\notag \\
\ge &h_{c_{i}}^{\infty}(\xi,x^-_{i,j})+ h^{\Delta t_i^+}_{c_i}(x^-_{i,j},x^+_{i,j+1})+h^{\infty}_{c_{i}} (x^+_{i,j+1},\zeta)\notag \\
&+h_{c_{i}}^{\infty}(x^+_{i,j},\xi) +h^{\infty}_{c_{i+1}}(\zeta,x^-_{i,j+1})-2\epsilon_i^*\notag \\
\ge &h_{c_{i}}^{\infty}(\xi,\hat x^-_{i,j})+ h^{\Delta t_i^+}_{c_i}(\hat x^-_{i,j},\hat x^+_{i,j+1})
+h^{\infty}_{c_{i}}(\hat x^+_{i,j+1},\zeta)\notag \\
&+h_{c_{i}}^{\infty}(x^+_{i,j},\xi) +h^{\infty}_{c_{i+1}}(\zeta,x^-_{i,j+1})+3\epsilon_i^* \notag\\
\ge &h_{c_{i}}^{\Delta t_{i}^-}(x^+_{i,j},\hat x^-_{i,j})+h^{\Delta t_i^+}_{c_i}(\hat x^-_{i,j},\hat x^+_{i,j+1}) +h^{\Delta t_{i+1}^-}_{c_{i}} (\hat x^+_{i,j+1}, x^-_{i,j+1})+\epsilon_i^*\notag
\end{align}
where $\xi\in\mathcal{A}^j_i$, $\zeta\in\mathcal{A}^{j+1}_i$, $\hat x_{i}^-$ as well as $\hat x_{i+1}^+$ is the intersection point of a semi-static curve $\gamma_{j,j+1}$ with $V^-_{i,j}$ and with $V^+_{i,j+1}$ respectively. The orbit $d\gamma_{j,j+1}$ connects $\mathcal{A}^j_i$ to $\mathcal{A}^{j+1}_i$. This contradicts the minimality of $\gamma$. The smoothness follows from the property that $(x_{i,j}^-,x^+_{i,j+1},\Delta t^+_{i,j})$ is in the interior of the domain. If the minimizer $\gamma(t,t^{\pm}_i,x^{\pm}_i)$ is not smooth at $x^-_{i,j}$, we join the points $\gamma(t^-_{i,j}-\delta,t^{\pm}_i,x^{\pm}_i)$ and $\gamma(t^-_{i,j}+\delta,t^{\pm}_i,x^{\pm}_i)$ by the minimizer of
$$
h_{c_i}^{\delta}(\gamma(t^-_{i,j}-\delta,t^{\pm}_i,x^{\pm}_i),\gamma(t^-_{i,j}+\delta,t^{\pm}_i,x^{\pm}_i)).
$$
As the minimizer approaches $x^-_{i,j}$ from both sides of $V^-_{i,j}$ as $t$ approaches $t^-_{i,j}$ from opposite direction \cite{BCV}, this minimizer also passes through $V^-_{i,j}$. Thus, one obtains a curve $\gamma'$ by replacing the segment of $\gamma(t,t^{\pm}_i,x^{\pm}_i)|_{t^-_{i,j}-\delta,t^-_{i,j} +\delta}$ with this minimizer. Let $t'$ be the time of this curve passing through $V^-_{i,j}$, clearly, $t^+_{i,j}-t'\in(T^+_{i,j}-\tau_i, T^+_{i,j}+\tau_i)$, the action along this curve is clearly smaller than the original one. But this is absurd.

{\bf The case of type-$c$}. For each integer $i\in\bigcup_{i_j\in\Lambda_c}\{i_j,i_j+1,\cdots,i_{j+1}-1\}$, there exist an admissible section $\Sigma_{c_i}$, a neighborhood $U_i$ of $\mathcal{N}(c_i)\cap \Sigma_{c_i}$, two closed 1-forms $\eta_i$ and $\bar\mu_i$ on $M$ with $[\eta_i]=c_i$, $[\bar\mu_i]=c_{i+1}-c_i$ and $\text{\rm supp}\bar\mu_i\cap U_i=\varnothing$. Correspondingly, an admissible coordinate system $q=G_i^{-1}x$ on $M$ is chosen such that in the new coordinates, one has a covering space $\pi_i$: $\bar M_i=\mathbb{R}\times\mathbb{T}^{n-1}\to M$,  the set $\pi^{-1}_i\Sigma_{c_i}$ consists of  infinitely many compact components $\Sigma_{c_i}^j= \{q=(q_1+j,q_2,\cdots,q_n):q\in\Sigma_{c_i}^0\}$ $(j\in\mathbb{Z})$. $\bar M_i$ is separated by $\Sigma_{c_i}^0$ into upper part $\bar M^+_i$ and lower part $\bar M^-_i$. A smooth function $\rho_i$: $\bar M\to [0,1]$ is constructed such that $\rho_i=0$ for $q\in\bar M^-_i \backslash(\Sigma_{c_i}^0 +\delta_i)$ and $\rho_i=1$ for $q\in\bar M^+_i\backslash (\Sigma_{c_i}^0+\delta_i)$. The number $\delta_i>0$ is  chosen so small such that $(\Sigma_{c_i}^0 +\delta_i)\cap(\mathcal{N}([\eta_i+\mu_i])+\delta_i) \subset U_i$, (cf. Formula (\ref{typeceq1})). Let $\mu_i=\rho_i\bar\mu_i$.

To make notation simpler, for each integer $i\in\bigcup_{i_j\in\Lambda_c} \{i_j,i_j+1,\cdots, i_{j+1}-1\}$, we let $\psi_i=0$, $V_i^{-}=\{q_1=-K_i\}$ and $V_{i+1}^{+}=\{q_1=K_i\}$ in the coordinate system $q=G_i^{-1}x$ (see the corollary \ref{typeccor1} for the definition of $K_i$). Again, let
$$
L_{\eta_i,\mu_i,\psi_i}=L-\eta_i-\mu_i-\psi_i.
$$
Since the class $c_i$ is equivalent to the class $c_{i+1}$ and they are close to each other, one sees from Theorem \ref{typecthm1} that each curve  $\bar\gamma \in\mathscr{C}_{\eta_i,\mu_i,\psi_i}$ determines a locally minimal orbit type-$c$ $d\gamma$ which is an orbit of $\phi_L^t$ and connects $\tilde{\mathcal{A}}(c_i)$ to $\tilde{\mathcal {A}}(c_{i+1})$.

Let $\bar m\in V_i^-$, $\bar m'\in V^+_{i+1}$ and let $\bar\gamma_i(t,\bar m,\bar m'):[-T,T]\to\bar M$ be the minimizer of
$$
h_{\eta_i,\mu_i,\psi_i}^T(\bar m,\bar m')=\inf_{T'>0}h_{\eta_i,\mu_i,\psi_i}^{T'}(\bar m,\bar m').
$$
According to Lemma \ref{semicontinuitylem2} and Corollary \ref{typeccor1}, $\exists$ $K_i>0$, $T_i^+=T_i^+(K_i)>0$, there exists $T< T_i^+$ such that $h_{\eta_i,\mu_i,\psi_i}^T(\bar m,\bar m')= \inf_{T'>0}h_{\eta_i,\mu_i,\psi_i}^{T'}(\bar m,\bar m')$ provided the first coordinate of $\bar m$ as well as of $\bar m'$ satisfies the condition that $\bar m_1\le -K_i$ and $\bar m'_1\ge K_i$. The  minimizer
$\bar\gamma_i(t,\bar m,\bar m'):[-T,T]\to\bar M$ satisfies
\begin{equation}\label{constructioneq13}
\bar\gamma_i(t,\bar m,\bar m')\in U_i, \qquad \text{\rm whenever}\ \ \bar\gamma_{i}(t,\bar m,\bar m')\in \Sigma_{c_i}^0+\delta_i.
\end{equation}

By the definition, the disks $V^{+}_i$ and $V^-_i$ are codimension one torus in different coordinate systems, their relative position needs to be fixed in the universal covering space. For this purpose, we define the following covering spaces:
$$
\mathbb{R}^n\xrightarrow{\bar\pi_i}\bar M_i\xrightarrow{\pi_i}\check{M}_i,
$$
where $\check M_i=\{(q_1,\cdots,q_n): q_i\ \text{\rm mod}\ 2i_j\pi \}$ and $\bar M_i=\mathbb{R}\times\{(q_2,\cdots,q_n): q_i\ \text{\rm mod}\ 2i_j\pi \}$ in the coordinate system $q=G_i^{-1}x$.  For simplicity of notation and without danger of confusion, we use the same symbol for a fundamental domain of $V^{\pm}_i$ in $\mathbb{R}^n$, i.e. restricted on $V^{\pm}_i$ the projection is a homeomorphism, $\bar\pi_iV^{-}_i=V^{-}_i$ and $\bar\pi_iV^-_{i+1}=V^-_{i+1}$. Both $V^-_i$ and $V^+_{i+1}$ are some translation of unit $(n-1)$-dimensional disk $\{q_1=0,q_i\in [0,2\pi)\,\text{\rm for}\, i=2,\cdots,n\}$. Clearly, $\bar\pi^{-1}_iV^-_i$ is parallel to $\bar\pi^{-1}_iV^+_{i+1}$ and there exist $(n-1)$ irreducible integer vectors $(v^i_2,v^i_3,\cdots v^i_n)$ tangent to $V^{-}_i$ ($V^+_{i+1}$) such that
$$
\bar\pi^{-1}_iV^-_i=\bigcup_{k_{\ell}\in\mathbb{Z},\ell=2,\cdots n}V^-_i+k_{\ell}v^i_{\ell},\qquad
\bar\pi^{-1}_iV^+_{i+1}=\bigcup_{k_{\ell}\in\mathbb{Z},\ell=2,\cdots n}V^+_{i+1}+k_{\ell}v^i_{\ell}.
$$

Given a fundamental domain $V^+_i$ and a $c_i$-semi static curve $\gamma_i$, there exists a curve in the lift $\gamma_i$ to the universal covering space, denoted by $\tilde\gamma_i$, which intersects $V^+_i\subset\mathbb{R}^n$. Clearly, one can choose a fundamental domain $V^-_i\subset\mathbb{R}^n$ (up to a translation) so that it intersects the curve $\tilde\gamma_i$ and $\bar\pi_{i-1}V^-_i$ is ``above" the $V^+_i\subset\bar M_{i-1}$ in the following sense: some suitably large $K'_i>0$ exists such that   $\min\{q^-_1-q^+_1:q^-\in\bar\pi_{i-1}V^-_i,q^+\in V^+_i\}=K'_i$. Note that $V^-_i$ is usually not parallel to $V^+_i$. We say that the two fundamental domains $V^+_i$ and $V^-_i$ are $(c_i,K'_i)$-related if they satisfy this condition.

Let $V^+_i$ and $V^-_i$ be $(c_i,K'_i)$-related fundamental domains. Given positive integers $k^+_i,k^-_i$, we define
$$
{\bf k}^{\pm}_iV^{\pm}_i=\bigcup_{|k_{\ell}|\le k^{\pm}_i,\ell=2,\cdots n}V^{\pm}_i+k_{\ell}v^i_{\ell},
$$
then $\bar\pi_{i-1}{\bf k}^{+}_iV^{+}_i=V^+_i\subset\bar M_{i-1}$. Let $\tilde x^+_i\in{\bf k}^{+}_iV^{+}_i$, $\tilde x^-_i\in{\bf k}^{-}_iV^{-}_i$ one defines the minimal action of $L_{c_i}$ connecting these two points
$$
h_{c_i}(\tilde x^+_i, \tilde x^-_i)=\inf_{T>0}\inf_{\stackrel{\scriptscriptstyle\tilde\zeta(-T) = \tilde x^+_i}{\scriptscriptstyle \tilde\zeta(T)=\tilde x^-_{i}}} \int_{-T}^TL_{c_i}(d\tilde\zeta(s))ds.
$$
Let
$$
h_{c_i}({\bf k}^{+}_iV^{+}_i, {\bf k}^{-}_iV^{-}_i)=\min_{\stackrel{\scriptscriptstyle\tilde x^+_i\in {\bf k}^{+}_iV^{+}_i}{\scriptscriptstyle \tilde x^-_i\in {\bf k}^{-}_iV^{-}_i}} h_{c_i}(\tilde x^+_i, \tilde x^-_i).
$$
Clearly, for fixed $k^+_i$, some positive number $\epsilon_i>0$ and suitably large integer $k^-_i$ exist such that ${\bf k}^{+}_iV^{+}_i$ does not touch ${\bf k}^{-}_iV^{-}_i$,
\begin{equation}\label{constructioneq14}
h_{c_i}(\tilde x^+_i, \tilde x^-_i)>h_{c_i}({\bf k}^{+}_iV^{+}_i, {\bf k}^{-}_iV^{-}_i)+\epsilon'_i, \qquad \text{\rm if}\  d(\tilde x^-_i,\partial{\bf k}^{-}_iV^{-}_i)\le 1.
\end{equation}
To understand this property let us consider those curves in the lift of $c_i$-semi static curves which pass through $V_i^+$. There exists $k_i>0$ such that the intersection points of these curves with $\bar\pi_i^{-1}V^-_i$ locate in the disk $\cup_{|k_{\ell}|\le k_i}V^-_i+k_{\ell}v^i_{\ell}$. The ``rotation vector" of this segment of the orbit can not be too far away from $\rho(\mu_i)$.

One can also define related fundamental domains $V^-_i$ and $V^+_{i+1}$. For a fundamental domain $V^-_i$ and a curve $\gamma\in\mathscr{C}_{\eta_i,\mu_i}$, we pick up a curve in the lift of this curve to the universal covering space, denoted by $\tilde\gamma$, which intersects the section $V^-_i$. Some fundamental domain $V^+_{i+1}$ exists where this curve intersects. Recall the projection of the two fundamental domains takes the form $V^+_{i+1}=\{q_1=K_i\}$ and $V^-_{i}=\{q_1=-K_i\}$ in the configuration space $\bar M_i$. We say that the two fundamental domains $V^-_i$ and $V^+_{i+1}$ are $(\eta_i,\mu_i,K_i)$-related.

As the Lagrangian $L_{\eta_i,\mu_i}$ is well-defined in the universal covering space, let us consider its action in the universal covering space:
$$
h_{\eta_i,\mu_i}(\tilde x^-_i,k^*\tilde x^+_{i+1})=\inf_{T>0}\inf_{\stackrel{\scriptscriptstyle\tilde\zeta(-T) = \tilde x^-_i}{\scriptscriptstyle \tilde\zeta(T)=k^*\tilde x^+_{i+1}}} \int_{-T}^TL_{\eta_i,\mu_i}(d\tilde\zeta(s))ds,
$$
where $k^*\tilde x^+_{i+1}=\tilde x^+_{i+1}+\sum_{\ell=2,\cdots n}k_{\ell}v^i_{\ell}$ stands for a translation of $\tilde x^+_{i+1}$ and $k=(k_2,\cdots k_n)$. Obviously, one has
$$
\inf_{k\in\mathbb{Z}^{n-1}}h_{\eta_i,\mu_i}(\tilde x^-_i,k^*\tilde x^+_{i+1})=h_{\eta_i,\mu_i}(\tilde x^-_i,\tilde x^+_{i+1})= \inf_{T>0}h_{\eta_i,\mu_i}^T(q^-_i,q^+_{i+1})
$$
where the term $\inf_{T>0}h_{\eta_i,\mu_i}^T(q^-_i,q^+_{i+1})$ was defined before by considering the action in the configuration space $\bar M_i$. As above, one defines
$$
h_{\eta_i,\mu_i}({\bf k}^{-}_iV^{-}_i,{\bf k}^+_{i+1}V^+_{i+1})=\min_{\stackrel{\scriptscriptstyle\tilde x^-_i\in {\bf k}^{-}_iV^{-}_i}{\scriptscriptstyle \tilde x^+_{i+1}\in {\bf k}^{+}_{i+1}V^{+}_{i+1}}} h_{\eta_i,\mu_i}(\tilde x^-_i, \tilde x^+_{i+1}).
$$
Again, for fixed $k^-_i$, some positive number $\epsilon_i>0$ and suitably large $k^+_{i+1}$ exist such that ${\bf k}^{-}_iV^{-}_i$ does not touch ${\bf k}^{+}_{i+1}V^{+}_{i+1}$ and
\begin{equation}\label{constructioneq15}
h_{\eta_i,\mu_i}(\tilde x^-_i, \tilde x^+_{i+1})>h_{\eta_i,\mu_i}({\bf k}^{-}_iV^{-}_i, {\bf k}^{+}_{i+1}V^{+}_{i+1})+\epsilon_i, \qquad \text{\rm if}\  d(\tilde x^+_{i+1},\partial{\bf k}^{+}_iV^{+}_{i+1})\le 1.
\end{equation}

Let $\tilde V^{\pm}_i={\bf k}^{\pm}_iV^{\pm}_i$. By induction, these sections $\tilde V^{\pm}_i$ are well defined such that $V^+_i$ and $V^-_i$ are $(c_i,K'_i)$-related, $V^-_i$ and $V^+_{i+1}$ are $(\eta_i,\mu_i,K_i)$-related, the formulae (\ref{constructioneq14}) and (\ref{constructioneq15}) are satisfied.

As the third step of the construction, let us clarify what conditions the candidates of minimal curve are required to satisfy.

Let $\gamma$: $[-K,K']\to M$ be an absolutely continuous curve joining $m$ to $m'$, i.e. $\gamma(-K)=m$ and $\gamma(K')=m'$.  We split the interval $[-K,K']$ into $2i_m+1$ subintervals
$$
[-K,K']=[t_0^+,t^-_0]\cup [t^-_0,t^+_1]\cup\cdots\cup [t^+_{i_m},t_{i_m}^-],
$$
where $t_0^+=-K$, $t_{i_m}^-=K'$. Correspondingly, we divide the curve into $2i_m+1$ segments $\gamma_i^-=\gamma|_{[t_{i}^+,t_i^-]}$, $\gamma_i^+=\gamma|_{[t_{i}^-,t_{i+1}^+]}$ for $i=0,1,2,\cdots, i_m-1$, and $\gamma_{i_m}^-=\gamma|_{[t_{i_m}^+,t_{i_m}^-]}$.

We fix a curve $\tilde\gamma$ in the lift of $\gamma$ to the universal covering space $\mathbb{R}^n$ by choosing $\bar\pi G_0^{-1}\tilde\gamma (t_0^-)\in V_0^-$. Correspondingly, each $\gamma_i^{\pm}$ has its lift $\tilde\gamma_i^{\pm}$ to $\mathbb{R}^n$.

The curve $\gamma$ is required to satisfy the conditions:

1, for each $i=0,1,2,\cdots i_m-1$,  there is some $k_i\in\mathbb{Z}$ such that
\begin{align}\label{constructioneq16}
&\bar\pi_iG_i^{-1}\tilde\gamma_i^+(t_i^-)-(2k_i\pi,0,\cdots,0)\in V_i^-,\notag \\
&\bar\pi_iG_i^{-1}\tilde\gamma_i^+(t_{i+1}^+)-(2k_i\pi,0,\cdots,0)\in V_{i+1}^+;
\end{align}

2, for $i\in\bigcup_{i_j\in\Lambda_c} \{i_j+1,\cdots,i_{j+1}-1\}$, $\tilde\gamma(t^{\pm}_i)\in\tilde V^{\pm}_i$. Let $\Delta t_i^+=\frac 12(t_{i+1}^+-t_{i}^-)$ and $\Delta t_i^-=\frac 12(t_i^--t_i^+)$. To formulate the conditions for $\Delta t_i^{\pm}$, let us consider the quantity
$$
h_{c_i}^{\Delta t}(\tilde x_i^+,\tilde x_i^-)=\inf_{\stackrel{\scriptscriptstyle\tilde\xi(-\Delta t)=\tilde x_i^+\in \tilde V^+_{i}} {\scriptscriptstyle \tilde\xi(\Delta t)=\tilde x_i^-\in \tilde V_i^-}}\int_{-\Delta t}^{\Delta t}(L-\eta_{i})(d\tilde\xi(t))dt.
$$
One obtains from the proof of Lemma \ref{semicontinuitylem2} that  $h_{c_i}^{\Delta t}(\tilde x_i^+, \tilde x_i^-)\to\infty$ as $\Delta t\to 0$ or $\to\infty$. Thus, if $T_i^-=T_i^-(\tilde x_i^+, \tilde x_i^-)$ is defined as the quantity such that $h_{c_i}^{T_i^-}(\tilde x_i^+, \tilde x_i^-)=\min_{\Delta t} h_{c_i}^{\Delta t}(\tilde x_i^+, \tilde x_i^-)$, then we find $0<T_i^-(\tilde x_i^+, \tilde x_i^-)<\infty$. Since both $V_{i,-}^+$ and $V_{i,+}^-$ are compact, there exist $0<\hat T_i^-<\breve{T}_i^-<\infty$ such that $\hat T_i^-<T_i^-(\tilde x_i^+, \tilde x_i^-)<\breve{T}_i^-$ holds for each $\tilde x_i^+\in\tilde V_{i,-}^+$ and $\tilde x_i^-\in\tilde V_{i,+}^-$. Let
\begin{equation}\label{constructioneq17}
\Delta T_i^-=[\hat T_i^-,\breve{T}_i^-].
\end{equation}

The range of $\Delta t_i^{\pm}$ is somehow implicitly defined. Let
\begin{equation}\label{constructioneq18}
\Delta T_i^+=[T_i^+-d_i,T_i^++d_i],\qquad \forall i\in\bigcup_{i_j\in\Lambda_h}\{i_j,i_j+1, \cdots,i_{j+1}-1\}
\end{equation}
\begin{equation*}
\Delta T_i^+=(0,T_i^+], \qquad \Delta T_i^-=[\hat T_i^-,\breve{T}_i^-],\qquad \text{\rm for other}\ i\le i_m.
\end{equation*}
See (\ref{constructioneq6}), (\ref{constructioneq13}) for the definition of $T_i^+$ and (\ref{constructioneq17}) for the definition of $\Delta T_i^-$ respectively.

The conditions for $\Delta t_i^{\pm}$ are the following:

1, $\Delta t^+_i\in\Delta T_i^+$ for all $0\le i<i_m$;

2, $\Delta t_i^-\in\Delta T_i^-$ for $i\in\bigcup_{i_j\in\Lambda_c}\{i_j+1, i_j+2,\cdots,i_{j+1}-1\}$;

3, for $i\in\bigcup_{i_j\in\Lambda_h}\{i_j, i_j+1,\cdots,i_{j+1}-1\}$, as it is assumed that the Aubry set $\mathcal{A}(c_i)$ contains finitely many classes, an orbit connects $\mathcal{A}(c_{i-1})$ to $\mathcal{A}(c_i)$ by approaching the Aubry class $\mathcal{A}^{j'}_i$ as $t\to\infty$, another orbit connects $\mathcal{A}(c_i)$ to $\mathcal{A}(c_{i+1})$ by approaching the Aubry class $\mathcal{A}^j_i$ as the time retreat back to $-\infty$. For the time interval $[t_i^+,t_i^-]$, one has the partition
$$
[t_i^+,t_i^-]= [t^+_i,t^-_{i,j'}]\cup[t^-_{i,j'},t^+_{i,j'+1}]\cup\cdots\cup[t^+_{i,j}, t^-_i],
$$
and has restrictions for these quantities, formulae (\ref{constructioneq10}) ,(\ref{constructioneq11}), (\ref{constructioneq12}) and
\begin{equation}\label{constructioneq19}
T^-_{i,j}+T_{i}^++d_{i}\le \Delta t_{i,j}^-+\Delta t_{i}^+\le T^-_{i,j}+T_{i}^++d^*_{i},
\end{equation}
\begin{equation}\label{constructioneq20}
T^-_{i+1,j'}+T_{i}^++d_{i}\le \Delta t_{i+1,j'}^-+\Delta t_{i}^+\le T^-_{i+1,j'}+T_{i}^++d^*_{i};
\end{equation}
with suitably large $d^*_i>0$.

4, for $i=i_j$ with $i_j\in\Lambda_h$, by definition, $c_i$ is equivalent to $c_{i-1}$, one has
\begin{equation}\label{constructioneq21}
\hat T_i^-+T_i^+\le\Delta t^+_i+\Delta t^-_i\le \breve{T}_i^-+T_i^++d^*_i;
\end{equation}

5, for $i=i_j$ with $i_j\in\Lambda_c$, by choosing $c_{i-1}$ suitable close to $c_i$ one can also assume that $c_i$ is equivalent to $c_{i-1}$. Thus, one has
\begin{equation}\label{constructioneq22}
\hat T_i^-+T_{i-1}^+\le\Delta t^+_{i-1}+\Delta t^-_i\le \breve{T}_i^-+T_{i-1}^++d^*_{i-1}.
\end{equation}

As the system is autonomous, by choosing sufficiently large $T_i^-$, these conditions defines non-empty set for $(\Delta t_i^+,\Delta t_i^-)$.

Finally, let us introduce a modified Lagrangian and verify the smoothness of the minimizer of the action. Recall $\mu_i$ and $\psi_i$ are defined on $\mathbb{R}\times\mathbb{T}^{n-1}$ in the coordinate system $q=G_i^{-1}x$, $G_i^*(\mu_i+\psi_i)
(d\tilde\gamma)=(\mu_i+\psi_i)(\bar\pi G_i^{-1}d\tilde\gamma)$ is well defined. We introduce a modified Lagrangian
$$
L_{\eta_i,\mu_i,\psi_i}\to L-\eta_i-(k_iG_i)^*(\mu_i+\psi_i)
$$
where $k_i^*$ is a translation of $q_1$: $(k_i)^*\phi(q,\dot q)=\phi(q_1-2\pi k_i,\hat q,\dot q)$ on $T\bar M_i$ and the integer $k_i$ is chosen so that (\ref{constructioneq16}) holds.

Let $\tilde\pi$: $\mathbb{R}^n\to M$ be the universal covering space. For a curve $\tilde\gamma$: $[-K,K']\to\mathbb{R}^n$, let $\gamma=\tilde\pi\tilde\gamma$: $[-K,K']\to M$. Let $\vec{t}=(t_0^-,t_1^{\pm},\cdots, t_{i_m-1}^{\pm},t_{i_m}^+)$, $\vec{x}=(\tilde x_0^-,\tilde x_1^{\pm},\cdots, \tilde x_{i_m-1}^{\pm},\tilde x_{i_m}^+)$, we consider the minimal action
\begin{align}\label{constructioneq23}
h_{L}^{K,K'}(m,m',\vec{x},\vec{t})&=\inf\sum_{i=0}^{i_m}\int^{t_{i}^-}_{t_i^+}(L-\eta_i) (d\tilde\gamma_i^-(t))dt \notag\\
&+\sum_{i=0}^{i_m-1}\int^{t_{i+1}^+}_{t_{i}^-}(L-\eta_i-(k_iG_i)^*(\mu_i+\psi_i))(d\tilde\gamma_i^+(t))dt
\end{align}
where the infimum is taken over all absolutely continuous curves $\tilde\gamma$: $[-K,K']\to\mathbb{R}^n$ with the boundary conditions $\tilde\gamma_i^+(t_i^-)=\tilde x_i^-$, $\tilde\gamma_i^+(t_{i+1}^+)=\tilde x_{i+1}^+$ for $i=0,1,\cdots, i_{m}-1$, $\gamma(-K)=m$, $\gamma(K')=m'$ and satisfying the condition (\ref{constructioneq16}). Moreover, restricted on $[t^+_i,t^-_i]$, $\gamma$ is admissible for the condition (\ref{constructioneq9}).

As the system is autonomous, the quantity $h_{L}^{K,K'}(m,m',\vec{z},\vec{t})$ remains constant if
$(\vec{t},K,K')$ is subject to a translation. Thus, it is a function of $K'-t_{i_m}^+$, $t_0^-+K$ and $\Delta\vec{t}=\{\Delta t^+_0,\Delta t^{\pm}_1,\cdots,\Delta t^{\pm}_{i_m-1}\}$. Denote by $\Delta\vec{T}$ the domain where $\Delta\vec{t}$ takes its value.  Let $\vec{V}=(\tilde V_0^-,\tilde V_1^{\pm},\cdots, \tilde V_{i_m-1}^{\pm}, \tilde V_{i_m}^+)$, where all entries have been well defined in the previous proof.

Denote by $\gamma(t;K,K',m,m',\vec{x},\Delta\vec{t})$ the curve along which the quantity of (\ref{construction 24}) is realized, it obviously depends on the value $K,K',m,m',\vec{x},\Delta\vec{t}$ and it may not be smooth at $\vec{t}$. Let $\vec{x}$ and $\Delta\vec{t}$ range over the set $\vec{V}$ and $\Delta\vec{T}$ respectively, one obtains a minimizer. The purpose of the following steps is to show that the minimizer is a solution of the Euler-Lagrange equation determined by $L$.

Let $h_{L}^{K,K'}(m,m')$ be the minimum of $h_{L}^{K,K'}(m,m',\vec{z},\vec{t})$ over $\vec{V}$ in $\vec{x}$ and over $\Delta\vec{T}$ in $\Delta\vec{t}$ respectively:
$$
h_{L}^{K,K'}(m,m')=\min_{\Delta\vec{t}\in\Delta\vec{T},\vec{x}\in\vec{V}} h_{L}^{K,K'} (m,m',\vec{x},\vec{t}),
$$
denote the minimal curve by $\gamma(t;K,K',m,m')$, we claim that $d\gamma(t;K,K',m,m')$ is a solution of the Euler-Lagrange equation of $L$ if $K$ and $K'$ are sufficiently large. To verify this claim, we need to show that

1, $d\gamma_i^+=d\gamma|_{\Delta t_i^+}$ solves the Euler-Lagrange equation determined by $L$. Restricted on $\Delta t_i^-$, it obviously solves the Euler-Lagrange equation.

2, $\gamma(t;K,K',m,m')$ has no corner at $\tilde x_i^-$ and $\tilde x_i^+$ for each $i=0,1,\cdots i_m-1$, i.e. it is smooth for the whole $t\in [-K,K']$. For each $i\in\Lambda_h$, as $\gamma_i^-=\gamma|_{\Delta t_i^-}$ is the minimizer for the curves admissible for the condition (\ref{constructioneq9}), it is smooth at each $t^+_{i,j'+1}<\cdots<t^+_{i,j}$.

Indeed, if $i\in\bigcup_{i_j\in\Lambda_h}\{i_j,i_j+1,\cdots,i_{j+1}-1\}$, we obtained from (\ref{constructioneq7}) that
$$
\bar\gamma_i^+(t)\in U_i\qquad \text{\rm when }\ \bar\gamma^+_{i,1}(t)-2k_i\pi\in\Sigma^0_{c_i}+\delta_i,
$$
where $\bar\gamma_i^+=\bar\pi_iG_i\tilde\gamma_i^+$. Since the support of $\bar\mu_i$ has no intersection with $U_i$ and $\psi_i$ is closed in $U_i$, while $\mu_i$ is closed and $\psi_i=0$ in the
region $\{\bar\gamma^+_{i,1}(t)-2k_i\pi\not\in [-\Delta_i,\Delta_i]\}$,  the term $\mu_i$ and $\psi_i$ have no contribution to the Euler-Lagrange equation along $\bar\gamma_i^+$. For other $i$, the conclusion is obtained from (\ref{constructioneq13}) by similar argument. This proves the first conclusion.

Recall the disks $\tilde V^-_i$ and $\tilde V^+_{i+1}$ are defined in the covering space $\mathbb{R}^n$. We claim that $\tilde\gamma$ does not touch the boundary of $\tilde V_i^-\times \tilde V_{i+1}^+\times [T_i^+-d_i,T_i^++d_i]$ for $i\in\bigcup_{i_j\in\Lambda_h}\{i_j, i_j+1,\cdots,i_{j+1}-1\}$. Let us assume the contrary, i.e. $(\tilde x_{i}^-,\tilde x_{i+1}^+, \Delta t_{i}^+)\in\partial (\tilde V_i^-\times \tilde V_{i+1}^+\times [T_i^+-d_i,T_i^++d_i])$ holds for some $i\in\bigcup_{i_j\in\Lambda_h}\{i_j, i_j+1,\cdots,i_{j+1}-1\}$. Let $\hat x_i^-=\bar\pi_i G_i^{-1}\tilde x_i^--(2k_i\pi,0,\cdots,0)$ and $\hat x_{i+1}^+=\bar\pi_i G_i^{-1}\tilde x_{i+1}^+-(2k_i\pi,0,\cdots,0)$. By the condition (\ref{constructioneq16}) we see that $\hat x_i^-\in V_i^-$ and $\hat x_{i+1}^+\in V^+_{i+1}$. Then, in $x\to G_i^{-1}x$-coordinates, we obtain from (\ref{constructioneq6}) and (\ref{constructioneq18}) that
\begin{align*}
&h_{c_{i}}^{\Delta t_{i}^-}(\pi_i\hat x_{i}^+,\pi_i\hat x_{i}^-)+h^{\Delta t_i^+} _{\eta_i,\mu_i,\psi_i}(\hat x_{i}^-,\hat x_{i+1}^+)+h^{\Delta t_{i+1}^-}_{c_{i+1}}(\pi_i\hat x_{i+1}^+,\pi_i\hat x_{i+1}^-)\\
\ge &h_{c_{i}}^{\infty}(\xi,\pi_i\hat x_{i}^-)+ h^{\Delta t_i^+}_{\eta_i,\mu_i,\psi_i} (\hat x_{i}^-,\hat x_{i+1}^+)
+h^{\infty}_{c_{i+1}}(\pi_i\hat x_{i+1}^+,\zeta)+h_{c_{i}}^{\infty}(\pi_i\hat x_{i}^+,\xi)\\
&+h^{\infty}_{c_{i+1}}(\zeta,\pi_i\hat x_{i+1}^-)-2\epsilon_i^*\\
\ge &h_{c_{i}}^{\infty}(\xi,\pi_i\bar x_{i}^-)+h^{T_i^+}_{\eta_i,\mu_i,\psi_i}(\bar x_{i}^-,\bar x_{i+1}^+) +h^{\infty}_{c_{i+1}}(\pi_i\bar x_{i+1}^+,\zeta)+h_{c_{i}}^{\infty}(\pi_i\hat x_{i}^+,\xi)\\
&+h^{\infty}_{c_{i+1}}(\zeta,\pi_i\hat x_{i+1}^-)+3\epsilon_i^* \\
\ge &h_{c_{i}}^{\Delta t_{i}^-}(\pi_i\hat x_{i}^+,\pi_i\bar x_{i}^-)+ h^{T_i^+}_{\eta_i,\mu_i,\psi_i} (\bar x_{i}^-,\bar x_{i+1}^+)
+h^{\Delta t_{i+1}^-}_{c_{i+1}}(\pi_i\bar x_{i+1}^+,\pi_i\hat x_{i+1}^-)+\epsilon_i^*
\end{align*}
where $\bar x_{i}^-$ and $\bar x_{i+1}^+$ are the intersection points of a curve in $\mathscr{C}_{\eta_i,\mu_i,\psi_i}$ with $V_{i}^-$ and with $V_{i+1}^+$ respectively, $\xi\in\mathcal{M}(c_{i-1})$ and $\zeta\in\mathcal{M}(c_i)$. This contradicts the minimality of $\gamma$, thus it verifies our claim.

To see that the curve $\tilde\gamma$ is smooth at $x_{i}^-$, let us assume the contrary again. Let $x'=\tilde\gamma(t_{i}^--\delta)$ and $x^*=\tilde\gamma(t_{i}^-+\delta)$, here $\delta$ is
chosen so small that $\Delta t_{i}^+\pm\delta\in [T_i^+-d_i,T_i^++d_i]$. This is possible since $(\hat x_{i}^-,\hat x_{i+1}^+,\Delta t_{i}^+)\not\in\partial (V_i^-\times V_{i+1}^+\times [T_i^+-d_i,T_i^++d_i])$ implies that $T_i^+-d_i<\Delta t_{i}^+<T_i^++d_i$. We join these two points by a minimizer
$\xi:[-\delta,\delta]\to M$ with $\xi(-\delta)=x'$ and $\xi(\delta)=x^*$
$$
[A_{c_{i}}(\xi|_{[-\delta,\delta]})]=\inf_{\stackrel{\zeta(-\delta)=x'}{\scriptscriptstyle \zeta(\delta)=x^*}} \int_{-\delta}^{\delta}(L-\eta_{i})(d\xi(s))ds.
$$
If $\xi$ passes through $V^-_{i}$, we obtain a curve $\gamma'$ by replacing the segment of the minimizer $\gamma|_{[t_{i}^--\delta,t_{i}^-+\delta]}$ with $\xi:[-\delta,\delta]\to M$. Let $t'^-_i$ be the time for $\gamma'$ passing through $V^-_i$, then $\frac 12(t_{i+1}^+-t'^-_i)\in\Delta T_i^+$ and $\frac 12(t'^-_{i}-t^+_i)\in\Delta T_i^-$. Thus, we obtain an absolutely continuous curve which is admissible for each required condition (see (\ref{constructioneq18}) and (\ref{constructioneq20})). Along this curve we obtain smaller action $h_{L}^{K,K'}(m,m')$, but this is absurd. So, we only need to show that $\xi$ passes through $V^-_{i}$. Indeed, as $V_i^-$ is chosen small and transversal to the local connecting curve in $\mathscr{C}_{\eta_i,\mu_i,\psi_i}$, $\gamma(t)$ approaches $V_i^-$ from different sides as $t\downarrow t_i^-$ and $t\uparrow t_i^-$ respectively. Otherwise, the minimality of $\gamma$ would be violated. One refers to \cite{BCV} for the details. The smoothness at $x_{i+1}^+$ can be proved similarly.

The smoothness of $\gamma$ at $t=t_i^{\pm}$ for $i\in\bigcup_{i_j\in\Lambda_c}\{i_j, i_j+1,\cdots,i_{j+1}-1\}$ is obvious.  Because of the formulae (\ref{constructioneq14}) and (\ref{constructioneq15}), $\tilde\gamma$ does not touch the boundary of $\tilde V^{\pm}_i$ at the time of $t^{\pm}_i$ respectively. Indeed, $\tilde\gamma$ approaches $\tilde V_i^{\pm}$ from different sides as $t\downarrow t_i^{\pm}$ and $t\uparrow t_i^{\pm}$ respectively. If $\tilde\gamma$ has a corner at $t=t^{\pm}_i$, let $\zeta$: $[t^{\pm}_i-\delta, t^{\pm}_i+\delta]\to\mathbb{R}^n$ be the minimizer of the action
$$
A(\zeta|_{[t^{\pm}_i-\delta, t^{\pm}_i+\delta]})=
\inf_{\stackrel{\scriptscriptstyle\xi(t^{\pm}_i-\delta) = \tilde\gamma(t^{\pm}_i-\delta)}{\scriptscriptstyle \xi(t^{\pm}_i+\delta) = \tilde\gamma(t^{\pm}_i+\delta)}} \int_{t^{\pm}_i-\delta}^{t^{\pm}_i+\delta}L_{c_i}(d\xi(s))ds,
$$
then the curve $\zeta$ passes through the disk $\tilde V^{\pm}_i$. Replacing $\tilde\gamma|_{[t^{\pm}_i-\delta, t^{\pm}_i+\delta]}$ by this minimizer $\zeta$ one obtains a curve with smaller action. The contradiction verifies the smoothness.

If $c_{i-1}$ is connected to $c_i$ by a type-$h$ orbit and $i\in\bigcup_{i_j\in\Lambda_c}\{i_j, i_j+1,\cdots,i_{j+1}-1\}$, then $V^+_i$ is a small disk. By the same argument as above, one obtains the smoothness of $\tilde\gamma$ at $t=t^-_i$ and the smoothness at $t=t^+_i$ from the arguments for type-$h$.

As the system is autonomous, the following limit exists
$$
h_{L}^{\infty}(m,m')=\lim_{K,K'\to\infty}h_{L}^{K,K'}(m,m').
$$
We pick out a sequence of $\gamma(t;K,K',m,m')$ for large $K$ and $K'$. Obviously, the set $\{\gamma (t;K,K',m,m')\}$ has at least one accumulation point $\gamma_{\infty}$: $\mathbb{R}\to M$ with the property $\alpha(d\gamma_{\infty})\subseteq\tilde{\mathcal{A}}(c)$ and $\omega(d\gamma_{\infty}) \subseteq\tilde{\mathcal{A}}(c')$. As we have shown, it is an orbit of $\phi_{L}^t$. This proves the first conclusion of the theorem.

For any two points $x,x'\in M$, the sequence $\{\gamma(t;K,K',x,x')|_{[0,K]}\}$ approaches to a forward $c$-semi static curve as $K\to\infty$, which starts from the point $x$, and $\{\gamma(t;K,K',x,x')|_{[t^-_{i_m},K']}\}$ approaches to a backward $c'$-semi static curve as $K'\to\infty$, which approach to the point $x'$. Therefore, for sufficiently large $K,K'$, the initial value $(\gamma,\dot\gamma)|_{t=0}$ falls into any prescribed $\delta$-neighborhood of the points $(x,v_{x,c}^{+})$  and the orbit reaches the $\delta$-neighborhood of $(x',v_{x,c'}^{-})$ at the time $t=K'$. This completes the proof.
\end{proof}

The proof for time-periodic system is similar, and a bit easier from technical point of view, since one can treat the time variable $t$ as the first angle variable and take $\{t=0\}$ the section for all classes. One does not need to introduce various coordinate systems $\{G_i^{-1}\}$ for different cohomology class. We omit the details here.

\section{\ui Proof of the main theorem}
\setcounter{equation}{0}

Once one obtains the existence of a generalized transition chain in the system (\ref{introeq1}),  Theorem \ref{mainthm} is proved by applying Theorem \ref{constructionthm1}. Therefore, the main purpose of this section is to show the genericity of such transition chains.

\subsection{Candidate of transition chain}
Let us consider the Hamiltonian (\ref{introeq1}). Any integer vector $\tilde k\in\mathbb{Z}^3$ determines a plane $\Sigma_{\tilde k}=\{\tilde\omega\in\mathbb{R}^3:\langle\tilde\omega,\tilde k\rangle=0\}$, which passes through the origin. Let $\Omega_E=\{\tilde\omega=\nabla h(\tilde y): \tilde y\in h^{-1}(E)\}\subset\mathbb{R}^3$, it is diffeomorphic to a $2$-sphere with the origin inside if $E>\min h$, because the Hamiltonian $h$ is assumed convex. Thus, the set $\Sigma_{\tilde k}\cap\Omega_E$ is a closed curve, denoted by $\Gamma_{\tilde\omega,\tilde k}$. Given any positive number $\delta>0$, some positive integer $K_{\delta}$ exists such  that $\cup_{\|\tilde k\|\le K_{\delta}}\Gamma_{\tilde\omega,\tilde k}$ constitutes a $\delta$-grid on $\Omega_E$ in the sense that the $M_h^{-1}\delta$-neighborhood of $\cup_{\|\tilde k\|\le K_{\delta}}\Gamma_{\tilde\omega,\tilde k}$ cover the whole sphere $\Omega_E$, where $M_h=\max_{\tilde y\in h^{-1}(E)}\|\partial^2 h(\tilde y)\|$. Therefore, there exists a resonant path
$$
\Gamma_{\tilde\omega}=\Gamma_{\tilde\omega,0}\ast\Gamma_{\tilde\omega,1}\ast\cdots\ast \Gamma_{\tilde\omega,m}.
$$
such that each rotation vector falls into its $M_h^{-1}\delta$-neighborhood, where $\Gamma_{\tilde\omega,\ell}$ represents a resonant path determined by one resonant relation (one integer vector). It is possible that $\Gamma_{\tilde\omega,\ell}$ and $\Gamma_{\tilde\omega,\ell'}$ are determined by the same resonant relation $\tilde k_{\ell}=\tilde k_{\ell'}$.
\begin{figure}[htp] 
  \centering
  \includegraphics[width=4cm,height=4cm]{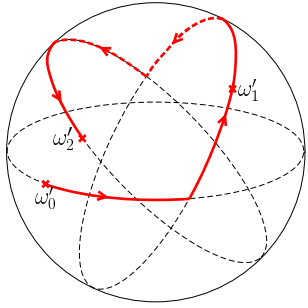}
  \caption{The resonant path in the surface of $h^{-1}(E)$.}
  \label{fig12}
\end{figure}

Obviously, the map $\partial h$ is a global diffeomorphism, which maps each closed curve $\Gamma_{\tilde\omega,\ell}$ onto the $2$-sphere $h^{-1}(E)$, $\Gamma_{\ell}=\partial h^{-1}\Gamma_{\tilde\omega,\ell}$. The $\delta$-grid on $\Omega_E$ induces a $\delta$-grid on $h^{-1}(E)$: the $\delta$-neighborhood of  $\cup_{\|\tilde k\|\le K_{\delta}}\Gamma_{\tilde k}$ covers the whole sphere.
Under the inverse of the frequency map $\tilde\omega\to\tilde y=(\nabla h)^{-1}(\tilde\omega)$ we obtain a path
$$
\Gamma=\Gamma_{0}\ast\Gamma_{1}\ast\cdots\ast\Gamma_{m}
$$
in action variable space, where $\Gamma_{\ell}=(\nabla h)^{-1}\Gamma_{\tilde\omega,\ell}$.

Let $\ell(\dot{\tilde x})=\max_{\tilde y}(\langle\dot{\tilde x},\tilde y\rangle-h(\tilde y))$ be the Lagrangian determined by the Hamiltonian $h$, $\phi_{\ell}^t$ be the Lagrange flow. As the system is integrable, the action variable $\tilde y$ keeps constant along each orbit of $\phi_{\ell}^t$ which obviously lies in the support of certain $c$-minimal measure with $\tilde c=\tilde y$. In this sense, one obtains a path $\Gamma_c\subset H^1(\mathbb{T}^3, \mathbb{R})$ and $\Gamma_c=\Gamma$ if we identify $H^1(\mathbb{T}^3,\mathbb{R})=\mathbb{R}^3$.

By the study of normal form, finitely many points $\tilde y_0,\tilde y_1,\cdots \tilde y_N\in\Gamma$ exist such that each $\tilde\omega_i=\nabla h(\tilde y_i)$ is rational frequency vector with period $T_i\le K_0\epsilon^{-\varrho}$ and
$$
\bigcup_{0\le i\le N}\{\tilde y:\|\tilde y-\tilde y_i\|\le\mu T_i^{-1}\epsilon^{\sigma}\}\supset \Gamma+\frac {\mu}2 \epsilon^{\frac 13},
$$
where $\varrho=(1-3\sigma)/3$, $\sigma<1/6$, see (\ref{normaleq9}). Obviously, $N$ depends on $\epsilon$, the size of perturbation. Under five steps of KAM iteration, we obtain the normal form
$$
H_i(\tilde x,\tilde y)=\tilde h(\tilde y)+\epsilon\tilde Z_{\epsilon,i}(\tilde x,\tilde y)+\epsilon\tilde R_{\epsilon,i}(\tilde x,\tilde y),
$$
which is valid in the domain $\{\tilde y:\|\tilde y-\tilde y_i\|\le\mu T_i^{-1} \epsilon^{\sigma}\}\times\mathbb{T}^3$. In which $\|\tilde R_{\epsilon,i}\|_{C^2}= O(\epsilon^{\frac 1{21}})$ and $\tilde Z_{\epsilon,i}$ is resonant with respect to $\omega_i$
$$
\tilde Z_{\epsilon,i}(\tilde x,\tilde y)=\sum_{\langle\tilde k,\omega_i\rangle=0}\tilde Z_{\epsilon,i,\tilde k}(\tilde y)e^{i\langle\tilde k, \tilde x \rangle},
$$
where the summation is made over all those $\tilde k$ spanned by $(\tilde k_i,\tilde k'_i)$: $\tilde k=j_1\tilde k_i+j_2\tilde k'_i$ in which $\|\tilde k_i\|,\|\tilde k'_i\|\le K_{\delta}$, $(\tilde k_i,\tilde k_i)$ is irreducible and the integer vector $\tilde k_i$ is used to determine the resonant path $\Gamma_i$, i.e. $\langle\nabla\tilde h(\tilde y),\tilde k_i\rangle=0$ holds for each $\tilde y\in\Gamma_i$.

For these two integer vectors $(\tilde k_i,\tilde k_i)$ there is another $\tilde k^*_i\in\mathbb{Z}^3$ such that the matrix $I_i=(\tilde k_i,\tilde k'_i,\tilde k^*_i)$ is uni-module. There are infinitely many $\tilde k^*_i$ satisfying the condition, we choose one so that its norm is as small as it can be. The coordinate transformation:
\begin{equation}\label{chaineq1}
\tilde q=I_i^t\tilde x,\qquad \tilde p=I_i^{-1}\tilde y.
\end{equation}
is obviously symplectic and $H_i(\tilde p,\tilde q)=H(I_i^{-t}\tilde x,I_i\tilde y)$ is also a function of $\tilde q$ defined in $\mathbb{T}^n$. Let $\tilde y$ be the point where $\nabla h_i(\tilde y)=\tilde\omega$, then the gradient of $h_i(\tilde p)=h(I_i\tilde y)$ satisfies
$$
\tilde\omega_i=\nabla h_i(p)=(0,0,\omega_{i3}),
$$
and $\tilde Z_{\epsilon,i}(\tilde p,\tilde q)=\tilde Z_{\epsilon,i}(I_i\tilde y,I_i^{-t}\tilde x)$ is independent of $q_3$, namely, $\tilde Z_{\epsilon,i}=\tilde Z_{\epsilon,i}(p,p_3,q)$ if we write $p=(p_1,p_2)$ and $q=(q_1,q_2)$.

Let us still use $(\tilde x,\tilde y)$ to denote the new coordinate system. Therefore, around a strong resonance point $\tilde\omega_i$ the normal form takes the form
\begin{equation}\label{chaineq2}
H_i(\tilde x,\tilde y)=h_i(\tilde y)+\tilde Z_{\epsilon,i}(x,y,y_3)+\tilde R_{\epsilon,i}(\tilde x,\tilde y)
\end{equation}
in the new coordinate system (\ref{chaineq1}). This form remains valid in $\mathbb{T}^3\times\{\|I_i(\tilde y-\tilde y_i)\|<T_i^{-1}\epsilon^{\sigma}\}$
and at $\tilde y=\tilde y_i$ one has $\nabla\tilde h_i=(0,0,\omega_3)$ with $\omega_3\neq 0$.

For our purpose, it is not necessary consider the Hamiltonian $H_i$ on the whole disk $\{\tilde y:\|\tilde y-\tilde y_i\|\le\mu T_i^{-1}\epsilon^{\sigma}\}$. Instead, we choose finitely many $\tilde y_{ij}\in\Gamma_i$ with $\tilde y_{i0}=\tilde y_i$ such that
$$
\cup_j\{\|\tilde y-\tilde y_{ij}\|<2K\sqrt{\epsilon}\}\supseteq\Gamma_i+ K\sqrt{\epsilon},
$$
and
$$
\text{\rm dist}(\tilde y_{ij'},\tilde y_{ij})\ge K\sqrt{\epsilon}\ \ \ \ \ \forall \ j'\neq j.
$$
where $K>0$ is a suitably large number. The results obtained in Section 4 and 5 can be applied to the Hamiltonian when it is restricted on each domain $\mathbb{T}^3\times\{\|y-y_{ij}\|<K\sqrt{\epsilon}\}$, especially on the domain $\mathbb{T}^3\times\{\|y-y_{i0}\|<K\sqrt{\epsilon}\}$.

Let $Y_{i}(x,y,\tau)$ be the solution of the equation $H_i(x,-\tau,y,Y_{i})=E$, where $H_i$ is given by (\ref{chaineq2}). It can be written in the form of
$$
Y_{i}=h_{i}(y)+\epsilon Z_{i}(x,y)+\epsilon R_{i}(x,y,\tau).
$$
The truncated form of $Y_{i}$
$$
Y_{i,T}=h_{i}(y)+\epsilon Z_{i}(x,y)
$$
is determined by $H_{i,T}=\tilde h(y)+\tilde Z_{\epsilon,i}(x,\tilde y)$, the truncated form of $H_i$. Denote $\tilde y_{ij}=(y_{ij},y_{ij,3})$ and let $y-y_{ij}=\sqrt{\epsilon}p$, $s=\sqrt{\epsilon}\tau$, we obtain from $Y_i$ the Hamiltonian
$$
G_{ij,\epsilon}=\frac 1{\sqrt{\epsilon}}\langle\omega_{ij},p\rangle+\frac 12\langle A_{ij}p,p\rangle+V_{ij}(x)+ Z_{ij,\epsilon}(x,\sqrt{\epsilon}p)+R_{ij,\epsilon}(x,\sqrt{\epsilon}p, s/\sqrt{\epsilon}),
$$
for $j=0$ we have
$$
G_{i0,\epsilon}=G_{i,\epsilon}=\frac 12\langle A_{i}p,p\rangle+V_{i}(x)+ Z_{i,\epsilon}(x,\sqrt{\epsilon}p) +R_{i,\epsilon}(x,\sqrt{\epsilon}p,s/\sqrt{\epsilon}),
$$
where $\omega_{ij}=\partial h_{i}(y_{ij})$, $A_{ij}=\partial^2 h_{i}(y_{ij})$, $A_{i}=\partial^2 h_{i}(y_i)$, $V_{ij}(x)=Z_{j}(x,y_{ij})$, $V_{i}(x)=Z_{j}(x,y_{i})$, $\|R_{ij,\epsilon}\|_{C^2},\|R_{i,\epsilon}\|_{C^2} =O(\epsilon^{\frac 1{21}})$ and $\|Z_{ij,\epsilon}\|_{C^2},\|Z_{i,\epsilon}\|_{C^2}=O(\sqrt{\epsilon})$ where the $C^2$-norm is with respect to $(x,p)$ only. By the choice of $y_{ij}$ and the convexity of $h$ we can see that
$$
\|\omega_{ij}\|\ge m_hK\sqrt{\epsilon}
$$
where $m_h$ is the lower bound of the eigenvalues of $\partial^2h$. Therefore, the $\alpha$-function $\alpha_{G_{ij,\epsilon}}$ for the Lagrangian determined by $G_{ij,\epsilon}$ does not reach its minimum when the action variable is restricted on the disk $\|p\|\le K$. As the frequency $\mathbb{R}^2\ni\omega_{ij}\neq 0$ satisfies certain resonant condition, the existence of normally hyperbolic cylinder is guaranteed by Theorem \ref{AppenHyperTh1} (see Appendix B) for generic $V_{ij}$. Therefore, all functions $G_{ij,\epsilon}$ ($j\neq 0$) are treated as {\it a priori} unstable Hamiltonian and the Hamiltonian $G_{i,\epsilon}$ is considered as the problem of double resonance.

Recall the Fenchel-Legendre transformation $\mathscr{L}_{\beta}$: $H_1(M,\mathbb{R})\to H^1(M,\mathbb{R})$, determined by the $\beta$-function. Let $\beta_h$, $\beta_{H_{i,T}}$ and $\beta_{H_{i}}$ be the $\beta$-function for $h$, $H_{i,T}$ and $H_{i}$ respectively. Obviously, $\mathscr{L}_{\beta_h}(\Gamma_{\omega,i})$ is still a curve. As it was studied in Subsection 4.3, $\mathscr{L}_{\beta_{H_{i}^T}}(\Gamma_{\omega,i})$ is composed of a flat $\mathbb{F}_0$ joined with two channels. See Figure \ref{fig13} below. These channels are joined to the flat either at a point or along an edge. The former case was thought difficult to handle.
\begin{figure}[htp] 
  \centering
  \includegraphics[width=9.5cm,height=4cm]{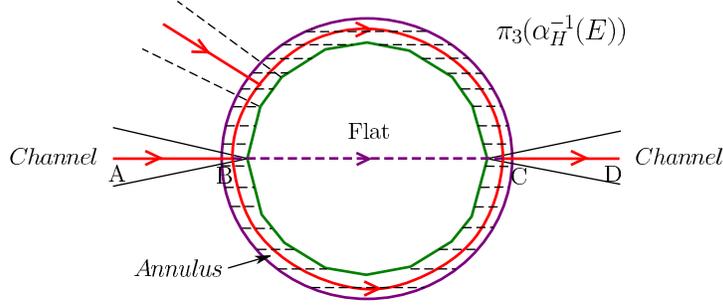}
  \caption{The transition chain under $\pi_3:\alpha^{-1}(E)\to\mathbb{R}^2$, represented by the thick solid red curve. Along the segment from $B$ to $C$, $c_3$ keeps constant. The purple dashed curve represents the curve $\mathscr{L}_{\beta_h}$.}
 \label{fig13}
\end{figure}

\subsection{Transition chain of incomplete intersection}

Let $\alpha_{H_i}, \alpha_{H_{i,T}}$ be $\alpha$-function determined by the Hamiltonian $H_i,H_{i,T}$  respectively. As it has been studied in the subsection 5.3, the double resonance corresponds to a flat $\mathbb{F}_0\subset\alpha_{H_{i}}^{-1}(E)$, around which there exists a annulus of incomplete intersection
$$
\tilde{\mathbb{A}}_T=\{(c_1,c_2,c_3)\in\alpha^{-1}_{H_{i,T}}(E):0<c_3\le\epsilon\Delta_0\}.
$$
The following has been proved generic in Theorem \ref{beltthm2}. For each $\tilde c\in \tilde{\mathbb{A}}_T$, the Ma\~n\'e set does not cover the whole $3$-torus. Thus, some $d_i>0$ exists such that for each $\tilde c\in \tilde{\mathbb{A}}_T$ the set
$$
N_{\tilde c,d_i}=\{x\in\mathbb{T}^3:U^-_{\tilde c}(x)-U'^+_{\tilde c}(x)<d_i\epsilon\}
$$
does not cover the whole $3$-torus, where $U^-_{\tilde c}$ and $U'^+_{\tilde c}$ are the elementary weak KAM solutions. Such results are obtained under the hypothesis ({\bf H1$\sim$4}) proposed in the section 5.

As the truncated system is independent of $x_3$, the Ma\~n\'e set for $H_{i,T}$ is independent of $x_3$
$$
\mathcal{N}_{H_{i,T}}(\tilde c)|_{\Sigma_s}=\mathcal{N}_{H_{i,T}}(\tilde c)|_{\Sigma_{s'}}, \qquad N_{\tilde c,d_i}|_{\Sigma_s}=N_{\tilde c,d_i}|_{\Sigma_{s'}}
$$
where $\Sigma_s$ is a co-dimension one section on which $x_3=s$. Let $\pi_3$: $\mathbb{R}^3\to\mathbb{R}^2$ be the standard projection: $\pi_3(x_1,x_2,x_3)=(x_1,x_2)$, let $\alpha_{Y_{i,T}}$ and $\alpha_{Y_{i}}$ be the Lagrangian determined by $Y_{i,T}$  and $Y_i$ respectively,  $\mathcal{N}_{Y_{i,T}}$ and $\mathcal{N}_{Y_i}$ denote the Ma\~n\'e set for the Lagrangian determined by $Y_{i,T}$ and $Y_i$ respectively. As $Y_{i,T}(x,y)$ solves the equation $H_{i,T}(x,y,Y_{i,T})=\alpha_{H_i}(c)$,  one has $\pi_3\mathcal{N}_{H_{i,T}}(\tilde c)= \mathcal{N}_{Y_{i,T}}(c)$. If $\mathcal{N}_{H_{i,T}}(\tilde c)$ does not cover the $3$-torus, $\mathcal{N}_{Y_{i,T}}(c)$ does not cover the $2$-torus. Because of $\alpha_{Y_{i,T}}(c)>\min\alpha_{Y_{i,T}}$, each $c$-minimal measure possesses non-zero rotation vector. Therefore, there exists some circle $\Sigma^1_c$ non-degenerately embedded into the $2$-torus such that each $c$-minimal curve passes through $\Sigma^1_c$ transversally and $\mathcal{N}_{Y_{i,T}}(c)|_{\Sigma^1_c}$ is topologically trivial, i.e. some open intervals $I_{\ell}\subset\Sigma^1_c$ exist such that
$$
\bigcup I_{\ell}\supset\mathcal{N}_{Y_{i,T}}(c)|_{\Sigma^1_c},\qquad I_{\ell}\cap I_{\ell'}=\varnothing,\ \ \forall\ \ell\neq\ell'.
$$
One can suitably choose $I_{\ell}$ so that
\begin{equation}\label{chaineq3}
\bigcup I_{\ell}\times\{x_3\in\mathbb{R}:\mod 2\pi\}\supset N_{\tilde c,d_i}.
\end{equation}

As $\|H_i- H_{i,T}\|_{C^2}\le O(\epsilon^{1+\frac 1{21}})$, some $\epsilon_i>0$ exists such that the Ma\~n\'e set for the Hamiltonian $H_i$
\begin{equation}\label{chaineq4}
\mathcal{N}_{H_i}(\tilde c)\subset N_{\tilde c,d_i}, \qquad \forall \ \epsilon<\epsilon_i.
\end{equation}

Let $\Gamma_{Y_i}=\tilde{\mathbb{A}}\cap\{c_3=Y_i\}$ where
$$
\tilde{\mathbb{A}}=\{(c_1,c_2,c_3)\in\alpha^{-1}_{H_i}(E):0<c_3\le\epsilon\Delta_0\}.
$$
It is a closed curve. By the preliminary works as above, some $c$-equivalence along the curve is established. Indeed, for each $\tilde c\in\Gamma_{Y_i}$, let
$$
\Sigma_{\tilde c}=\Sigma_c^1\times\{x_3\in\mathbb{R}\mod 2\pi\}.
$$
By the construction, each $\tilde c$-semi static curve passes through the section $\Sigma_{c}$ transversally. Recall
$$
V_{\tilde c}=\bigcap_U\{i_{U*}H_1(U,\mathbb{R}): U\, \text{\rm is a neighborhood of}\, \mathcal {N}(\tilde c)\cap\Sigma_{\tilde c}\},
$$
one sees that $\tilde c'-\tilde c\in V^{\perp}_{\tilde c}$ provided $\tilde c'$ is close to $\tilde c$, $c'_3=c_3$ and $\alpha_{H_i}(\tilde c')=\alpha_{H_i}(\tilde c)$, i.e. $c'\in\Gamma_{Y_i}$. In this case, some open set $U\supset\mathcal{N}_{H_i}(c)\cap\Sigma_{\tilde c}$ such that $V_{\tilde c}= i_{U*}H_1(U,\mathbb{R})=\text{\rm span}\{(0,0,1)\}$, from which one obtains that $V_{\tilde c}^{\perp} =\text{\rm span}\{(1,0,0),(0,1,0)\}$. For each class $\tilde c'\in\Gamma_{Y_i}$ close to $\tilde c$, one has $\tilde c'-\tilde c=(\Delta c_1,\Delta c_2,0)\in V_{\tilde c}^{\perp}$, thus, there exists a closed 1-form $\bar\mu$ such that $[\bar\mu]=c'-c$ and
$$
\text{\rm supp}\bar\mu\cap\mathcal{N}_{H_i}(\tilde c)\cap\Sigma_{\tilde c}=\varnothing.
$$
Thus, any two classes along the curve $\Gamma_{Y_i}$ is equivalent. Therefore, a transition chain for incomplete intersection is established, see Figure \ref{fig13}, the thick solid red curve from the point $B$ to the point $C$.

\subsection{Transition chain for complete intersection}

By the study in the subsection 4.2 (see Theorem \ref{cylinderthm2}), there are two wedge-shaped channels $\tilde{\mathbb{W}}_g =\cup_{\lambda\ge\lambda_0>0}\mathscr{L}_{\beta}(\lambda g)$ and $\tilde{\mathbb{W}}_{g'} =\cup_{\lambda\ge\lambda'_0>0} \mathscr{L}_{\beta}(\lambda g')$ which extend into the annulus $\tilde{\mathbb{A}}$. Corresponding to these two channels there exist two normally hyperbolic cylinder $\tilde\Pi_{E_0,E_1,g}$ and $\tilde\Pi_{E'_0,E'_1,g'}$ respectively, which are three-dimensional and invariant for the Hamiltonian flow: for each $\tilde c=(c_1,c_2,c_3)\in \tilde{\mathbb{W}}_g$, the Ma\~n\'e set $\tilde{\mathcal{N}}_{H_i}(\tilde c)\subset\tilde\Pi_{E_0,E_1,g}$ if $c_3\ge 2\epsilon^{1+d}$.

We are now in the situation that there is a normally hyperbolic cylinder $\tilde\Pi$ homeomorphic to $I\times\mathbb{T}^2$, the Aubry set is located on this cylinder for each cohomology class under consideration. If the Aubry set is a two-dimensional torus, it has its own stable and unstable manifold. It implies that the forward (backward) weak KAM solution is differentiable when it is restricted in a neighborhood of this 2-torus. Because weak KAM is a viscosity solution, any $C^1$ viscosity solution for Tonelli Hamiltonian must be $C^{1,1}$ \cite{CS,FS,Ri}. Therefore, in a small neighborhood of the Aubry set, the stable and unstable manifold are Lipschitz graphs. As the cylinder is smooth, the Aubry set is also a Lipschitz graph over two-torus.

Let $\Sigma\subset H^{-1}(E)$ be a four-dimensional section intersecting each orbit in the Aubry sets transversally. The set $\Pi=\Sigma\cap\tilde\Pi$ is a two-dimensional cylinder. In a neighborhood of $\Pi$ the Hamiltonian flow defines a return map on the section $\Sigma$. Restricted on the cylinder $\Pi$, each Aubry set is either periodic orbit, or Aubry-Mather set or invariant circle. Each circle is a Lipschitz curve.  A piece of the cylinder $\Pi$, bounded by two invariant circles, is invariant for the return map which preserves some ``area" element. Let $\psi$: $\Pi_0=[0,1]\times\mathbb{T}\to\Pi$ be the map, it pulls back the standard closed 2-form $\omega=dx\wedge dy$ to a 2-form on $\Pi$. Since the second de Rham cohomology of a cylinder is trivial, by Moser's theorem on the isotopy of symplectic forms, there exists a diffeomorphism $\psi_1$ which transforms this form to the standard 2-form, namely
$$
(\psi\circ\psi_1)^*\omega=d\theta\wedge dI.
$$
Since the return map $\Phi_H$ preserves the form $\omega$, one has
$$
((\psi\circ\psi_1)^{-1}\circ\Phi_H\circ(\psi\circ\psi_1))^*d\theta\wedge dI=d\theta\wedge dI.
$$
Let us consider those Aubry sets which are invariant two-torus, denoted by $\Upsilon_c$. We use the same notation for their intersection with $\Pi$, which are circles. Fix one circle $\Upsilon_{c_0}$, other circles are parameterized by the ``area" $\sigma$. Given any other circle $\Upsilon_c$, we obtain the algebraic area $\sigma$ of the region bounded by these two circles. If each circle is regarded as the graph of a function, then there is a regularity result \cite{CY1}
$$
\|\Upsilon_{c(\sigma)}-\Upsilon_{c(\sigma')}\|_{C^0}\le C_1\sqrt{|\sigma-\sigma'|}.
$$
Because the cylinder is normally hyperbolic, there is an segment of a line $I_{\sigma}\subset \alpha^{-1}(E)$ such that all cohomology classes located in this segment share the same Aubry set, an invariant 2-torus, so we have a map $\sigma\to I_{\sigma}$.

For a small segment of cylinder, some neighborhood $N\subset\mathbb{T}^3$ of a two-torus exists so that all Aubry sets on this cylinder fall into this neighborhood: $\mathcal{A}(c)\subset N$. In a suitably coordinate system we take a finite covering space $\check{M}$ so that the lift of $N$ consists of two connected components $N_l$ and $N_r$. The Ma\~n\'e set satisfies the condition
$$
\mathcal{N}(c,\check M)\backslash(N_{l}\cup N_{r})\neq\varnothing.
$$
To construct transition chain in this situation, one need to show it consists of totally disconnected semi-static curves when the Aubry set is a two-torus.

Let us consider the covering space $\pi_1: \bar M=\mathbb{R}\times\mathbb{T}^2$ such that the lift of $N$ contains infinitely many connected components, each of which is still a neighborhood of two-torus. We consider two adjacent components $N_l$ and $N_r$ in the lift of $N$, i.e. $\pi_{1}N_l=\pi_{1}N_r=N$ and no other component in the lift is located between them. The subscript $r$ means ``right" and $l$ means ``left". Correspondingly, denote by $\Upsilon_{l,\sigma}$ and $\Upsilon_{r,\sigma}$ the connected component in the lift of $\Upsilon_{\sigma}$ respectively, $\Upsilon_{l,\sigma}\subset N_l$ and $\Upsilon_{r,\sigma}\subset N_r$. The barrier function takes the form
$$
u^-_{l,\sigma}-u^+_{r,\sigma}\ \ \ \ \text{\rm or}\ \ \ \ u^-_{r,\sigma}-u^+_{l,\sigma}
$$
where $u^{\pm}_{l,\sigma}$ and $u^{\pm}_{r,\sigma}$ are the elementary weak KAM solution determined by $\Upsilon_{l,\sigma}$ and $\Upsilon_{r,\sigma}$ respectively. The elementary weak-KAM solution $u^{\pm}_{l,\sigma}$ is uniquely determined by $I_{\sigma}$, all classes in $I_{\sigma}$ share the same elementary weak-KAM solution. It is why we use the subscript $\sigma$. A point $\pi_1x\in\mathcal{N}(c)$ if and only if
$$
x\in\arg\min(u^-_{l,\sigma}-u^+_{r,\sigma}), \ \ \ \text{\rm or} \ \ \ x\in\arg\min(u^-_{r,\sigma}-u^+_{l,\sigma}).
$$
Let $M_0$ be a segment of $\mathbb{R}\times\mathbb{T}^2$ bounded by $\Upsilon_{l,c}$ and $\Upsilon_{r,c}$. The problem turns out to be the version: whether does the set $\arg\min(u^{-}_{l,\sigma}-u^{+}_{r,\sigma})|_{M_0\backslash(N_l\cup N_r)}$ consist of totally disconnected semi-static curves?

We only need to follow the argument in \cite{CY1,CY2,LC} if we are satisfied with the generic property in the category of Lagrangian, where the perturbations are functions also defined on $TM$: $L(x,\dot x)\to L(x,\dot x)-L_{\delta}(x,\dot x)$. In this paper, we are also going to prove the generic property in the sense of Ma\~n\'e, i.e. the perturbations are imposed on the potential $L(x,\dot x)\to L(x,\dot x)-V(x)$.

Let us construct the potential perturbations. Choose a 2-dimensional disk $D$ which transversally intersects the backward semi-static curves $\gamma_{x,\sigma_0}^-:(-\infty,0]\to\bar M$ with $\gamma_{x,\sigma_0}^-(0)=x\in D$. These curves approach $\Upsilon_{l,\sigma_0}$ as $t\to -\infty$. In suitable coordinate system we can assume that $D$ is located in the section
$$
D+d_1=\{(x_1,x_2,x_3):x_1=x_{10},|x_2-x_{20}|\le d+d_1, |x_3-x_{30}|\le d+d_1\}
$$
where $(x_{10},x_{20},x_{30})=x_0$. Let $D=(D+d_1)|_{d_1=0}$. We write the curve $\gamma^-_{x_0,\sigma_0}$ in the coordinate form
\begin{equation*}
\gamma^-_{x_0,\sigma_0}(t)=(x_{10}(t),x_{20}(t),x_{30}(t))
\end{equation*}
where $x_{10}$ is monotonely increases for $t\in [-T,0]$. Since continuous function can be approximated by smooth function, for any small $\delta>0$, a tubular neighborhood of the semi-static curve $\gamma_{x_0,\sigma_0}^-|_{[-T,0]}$ admits smooth foliation of curves $\zeta_{x}$: $(x,t)\in (D+d_1)\times[-T,0]\to\mathbb{T}^3$ such that each semi-static curve $\gamma^-_{x,\sigma_0}|_{[-T,0]}$ remains $\delta$-close to $\zeta_{x}$ in the sense that $d(\zeta_x(t),\gamma_{x,\sigma_0}^-(t))<\delta$ for all $t\in[-T,0]$. The tubular neighborhood is defined by the form
$$
\text{\uj C}=\cup_{-T\le t\le 0}\{\zeta_{x}(t):x\in D+d_1\}.
$$

Let $\rho$: $(D+d_1)\times\mathbb{R}\to\mathbb{R}$ be a smooth function such that $\rho(x,t)=\rho(x',t)$, $\rho(x,t)=0$ if $t\notin[-T+t_0,-t_0]$ with small $t_0>0$ and $\rho(x,t)>0$ if $x\in (-T+t_0,-t_0)$. As $\zeta_x$ is a smooth foliation of the tubular domain, it can be thought as a differeomorphism $\Psi$: $(D+d_1)\times [-T,0]\to\text{\uj C}$, namely, for $x'\in\text{\uj C}$ there exists unique $(x,t)\in (D+d_1)\times [-T,0]$ such that $\Psi(x,t)=\zeta_x(t)=x'$. With a smooth function $V$: $D+d_1\to\mathbb{R}$ we obtain a smooth function $\bar V$ defined on $\text{\uj C}$
\begin{equation}\label{completeeq5}
\bar V(x')=\rho(\Psi^{-1}(x'))V(\zeta_{x}(0)),
\end{equation}
Since unique $(x,t)\in (D+d_1)\times [-T,0]$ is determined by certain $x'\in\text{\uj C}$, some constant $C_2>0$ exists such that
\begin{equation}\label{completeeq6}
\int_{-T+t_0}^{-t_0}\bar V(\zeta_{x}(t))dt=C_2V(x), \qquad \forall x\in D+d_1.
\end{equation}

We construct the potential perturbation in the form of (\ref{completeeq5}) where $V$ ranges over the function space spanned by
\begin{align*}
\mathfrak{V}_{2}=&\varepsilon\Big(\sum_{\ell=1,2}a_{\ell}\cos2\ell\pi(x_2-x_{20}) +b_{\ell}\sin2\ell\pi(x_2-x_{20})\Big),\\
\mathfrak{V}_{3}=&\varepsilon\Big(\sum_{\ell=1,2}c_{\ell}\cos2\ell\pi(x_3-x_{30}) +d_{\ell}\sin2\ell\pi(x_3-x_{30})\Big),
\end{align*}
where each parameter of $(a_{\ell},b_{\ell},c_{\ell},d_{\ell})$ ranges over an unit interval $[1,2]$. If we construct a grid for the parameters $(a_{\ell},b_{\ell},c_{\ell},d_{\ell})$ by splitting the domain equally into a family of cubes and setting the size length by
$$
\Delta a_{\ell}=\Delta b_{\ell}=\Delta c_{\ell}=\Delta d_{\ell}=\varepsilon,
$$
the grid consists of as many as $[\varepsilon^{-8}]$ cubes.

Let us choose a neighborhood $\mathbb{I}_{\sigma_0}$ of the point $\sigma_0$ which satisfies the conditions:

1, for each $(x,\sigma)$ with $x\in D$ and $\sigma\in\mathbb{I}_{\sigma_0}$, there is a unique backward semi-static curve $\gamma^-_{x,\sigma}$ such that $\gamma^-_{x,\sigma}(0)=x$ and $\gamma^-_{x,\sigma}(t)\to\Upsilon_{l,\sigma}$ as $t\to -\infty$. It is guaranteed by the existence of unstable manifold and if $D$ is chosen close to $\Upsilon_{l,\sigma_0}$. By the definition, $\gamma_{x,\sigma}^-(t)\in\text{\uj C}$ for $t\in [-T,0]$ and $x\in D$, so each $\sigma\in\mathbb{I}_{\sigma_0}$ defines a linear operator
\begin{equation}\label{completeeq7}
\mathscr{K}_{\sigma}\bar V=\int_{-T}^0\bar V(\gamma_{x,\sigma}^-(t))dt;
\end{equation}

2, as each curve $\gamma_{x,\sigma_0}^-(t)$ stays in $\delta$-neighborhood of the fiber $\zeta_{x}$ for $t\in [-T,0]$ with small $\delta>0$, by choosing suitably small neighborhood $\mathbb{I}_{\sigma_0}$ (depending on the size of $D$) some constant $C_3>0$ exists such that
\begin{align}\label{completeeq8}
\text{\rm Osc}_{x\in D}(\mathscr{K}_{\sigma}\bar V-\mathscr{K}_{\sigma}\bar V')&=\max_{x,x'\in D}|\mathscr{K}_{\sigma}\bar V(x)-\mathscr{K}_{\sigma}\bar V'(x')|\notag\\
&>2^{-1}C_2\text{\rm Osc}_{x\in D}(V-V')\\
&>C_3\varepsilon\Delta\notag
\end{align}
with $\Delta=\max\{|a_{\ell}-a'_{\ell}|,|b_{\ell}-b'_{\ell}|,|c_{\ell}-c'_{\ell}|, |d_{\ell}-d'_{\ell}|\}$. Indeed, as $V$ is a linear combination of the functions $\{\sin\ell x_j,\cos\ell x_j:\ell =1,2,j=2,3\}$, there exists some number $d=d(D)>0$ depending on the size of $D$ only such that the Hausdorff distance
$$
d_H(V_{D}^{-1}(\min_DV+\frac 14\Delta),V_{D}^{-1}(\max_DV-\frac 14\Delta))\ge d(D)
$$
where $V_{D}^{-1}(\min_DV+\frac 14\Delta)=\{x\in D:V(x)\le\min _DV+\frac 14|(\max_DV-\min_DV)\}$ and $ V_{D}^{-1}(\max_DV-\frac 14\Delta)=\{x\in D:V(x)\ge\max _DV-\frac 14|(\max_DV-\min_DV)\}$. By requiring $\sigma$ suitably close to $\sigma_0$ and using the notation $\pi_x(x,t)=x$, we have
$$
\pi_x\Psi^{-1}\gamma_{x,\sigma}(t)\in V_{D}^{-1}(\min_DV+\frac 14\Delta)\qquad \text{\rm if}\ V(x)=\min_D V;
$$
and
$$
\pi_x\Psi^{-1}\gamma_{x,\sigma}(t)\in V_{D}^{-1}(\max_DV-\frac 14\Delta)\qquad \text{\rm if}\ V(x)=\max_D V.
$$
Therefore, one obtains (\ref{completeeq8}) from (\ref{completeeq5}), (\ref{completeeq6}) and (\ref{completeeq7});

3, for each $\sigma\in \mathbb{I}_{\sigma_0}$ and each $x\in D$, the forward semi-static curve $\gamma_{x,\sigma}^+$, determined by $u^+_{r,\sigma}$ with $\gamma_{x,\sigma}(0)=x\in D$, does not touch the support of $\rho\subset\text{\uj C}$ and approaches $\Upsilon_{r,\sigma}$ as $t$ increases to infinity.

For the perturbed system $L(\dot x,x)-\bar V(x)$, we use $u^+_{r,\sigma,\bar V}$ and $u^-_{l,\sigma,\bar V}$ to denote the weak KAM solution. By the construction of perturbation, the invariant cylinder remains unchanged. Restricted on the disk $D$, the forward weak-KAM solution $u^+_{r,\sigma,\bar V}$ is also unchanged $(u^+_{r,\sigma,V}-u^+_{r,\sigma})|_{x\in D}=0$, but the backward weak KAM solution undergoes small perturbation $u^-_{l,\sigma,\bar V}\neq u^-_{l,\sigma}$. To see how it is related to the potential, let us recall the following relations
$$
u^-_{l,\sigma}(\gamma_{x,\sigma}(0))-u^-_{l,\sigma}(\gamma_{x,\sigma}(-t))= \int_{-t}^0(L-\eta_{c})(d\gamma_{x,\sigma}(t))dt+Et
$$
if $\gamma_{x,\sigma}$ is a semi-static curve determined by $u^-_{l,\sigma}$ with $\gamma_{x,\sigma}(0)=x$ and $c\in I_{\sigma}$. We also have
$$
u^-_{l,\sigma,\bar V}(\gamma_{x,\sigma}(0))-u^-_{l,\sigma,\bar V}(\gamma_{x,\sigma}(-t))\le \int_{-t}^0(L-\bar V-\eta_{c})(d\gamma_{x,\sigma}(t))dt+Et.
$$
Clearly, for suitably large $t$ the backward weak-KAM solution $\gamma_{x,\sigma}(-t)$ shall retreat into a small neighborhood of $\Upsilon_{l,\sigma}$ where the weak KAM solution $u^-_{l,\sigma}$ also remains unchanged. Therefore we deduce from the last two formulae that
$$
u^-_{l,\sigma,\bar V}(x)-u^-_{l,\sigma}(x)\ge\int^0_{-T}\bar V(\gamma_{x,\sigma}(t))dt.
$$
In a similar way, we find
$$
u^-_{l,\sigma,\bar V}(x)-u^-_{l,\sigma}(x)\le\int^0_{-T}\bar V(\gamma_{x,\sigma,\bar V}(t))dt
$$
where $\gamma_{x,\sigma,\bar V}$ stands for the backward semi-static curve determined by the elementary weak-KAM solution $u^-_{l,\sigma,\bar V}$ with $\gamma_{x,\sigma,\bar V}(0)=x$. As $x$ is located in the region where the weak KAM solution is differentiable, we have $|\gamma_{x,\sigma,\bar V}(t)-\gamma_{x,\sigma}(t)| \to 0$ as $\bar V\to 0$, guaranteed by the upper-semi continuity of semi-static curves. Therefore, it follows that for $x\in D$
\begin{align}\label{completeeq9}
u^-_{l,\sigma,V}(x)-u^-_{l,\sigma,V'}(x)=&\int_{-T}^0(\bar V-\bar V')(\gamma_{x,\sigma,\bar V}^-(t))dt+o(\|\bar V-\bar V'\|),\\
=&(\mathscr{K}_{\sigma}+ \mathscr{R}_{\sigma})(\bar V-\bar V')\notag
\end{align}
where the linear operator $\mathscr{K}_{\sigma}$ is defined in (\ref{completeeq7}) and $\mathscr{R}_{\sigma}(\bar V-\bar V')=o(\|V-V'\|)$.

Next, let us consider all backward weak-KAM solutions for $\sigma\in\mathbb{I}_{\sigma}$.  Each parameter $\sigma\in\mathbb{I}_{\sigma}$ determines an interval $I_{c(\sigma)}$ for cohomology class. We restricted ourselves on a curve of first cohomology classes contained in the set $\cup I_{c(\sigma)}$ and intersecting each $I_{c(\sigma)}$ transversally. In this sense, we think the class defined on the interval $\mathbb{I}_c\ni c$ and the map $\sigma\to c(\sigma)$ is continuous. As  $h^{\infty}_{c(\sigma)}(x,x')=u^-_{l,\sigma}(x')-u^-_{l,\sigma}(x)$ if $x\in\Upsilon_{l,\sigma}$ and $h^{\infty}_{c(\sigma)}(x,x')=u^+_{r,\sigma}(x')-u^+_{r,\sigma}(x)$ if $x'\in\Upsilon_{r,\sigma}$, we obtain from Lemma 6.4 in \cite{CY2}
\begin{align}\label{completeeq10}
&|u^-_{l,\sigma}(x)-u^-_{l,\sigma'}(x)|\le
C_4(\sqrt{|\sigma-\sigma'|}+|c(\sigma)-c(\sigma')|),\\
&|u^+_{r,\sigma}(x)-u^+_{r,\sigma'}(x)|\le
C_4(\sqrt{|\sigma-\sigma'|}+|c(\sigma)-c(\sigma')|).\notag
\end{align}

We split the interval $\mathbb{I}_{\sigma}$ equally into $K_{\sigma}[\varepsilon^{-2}]$ parts and split the interval $\mathbb{I}_c$ equally into $K_c[\varepsilon^{-1}]$, where
$$
K_{\sigma}=\Big[L_{\sigma}\Big(\frac{12C_4}{C_3}\Big)^2\Big],\qquad
K_c=\Big[L_c\frac{12C_4}{C_3}\Big],
$$
$L_{\sigma}$ and $L_c$ are the length of $\mathbb{I}_{\sigma}$ and of $\mathbb{I}_c$ respectively. The grid over $\mathbb{I}_c\times\mathbb{I}_{\sigma}$ consists of as many as $K_{\sigma}K_c[\varepsilon^{-3}]$ cuboids in which $K_{\sigma},K_c$ are independent of $\varepsilon$. We pick up all cuboids which contain the points $(c,\sigma(c))$ and denote them by $\text{\uj c}_j$ with $j\in\mathbb{J}$, then the cardinality of the set $\mathbb{J}$ is not bigger than $K_{\sigma}K_c[\varepsilon^{-3}]$.

According to the definition, a point $(c_j,\sigma(c_j))\in\text{\uj c}_j$ corresponds to a barrier function $u^-_{l,\sigma_j}-u^+_{r,\sigma_j}$. Let us assume that some parameters $(a_{\ell,j},b_{\ell,j})$ exist such that
$$
\text{\rm Osc}_{x\in D}\min_{x_3}\Big(u^-_{l,\sigma_j}-u^+_{r,\sigma_j}-(\mathscr{K}_{\sigma_j}+ \mathscr{R}_{\sigma_j})\bar V_j\Big)=0
$$
where $\bar V_j=\rho\Psi^{-1} V_j$ is defined as in (\ref{completeeq5}) with $V_j\in\mathfrak{V}_2$ determined by the parameters. We consider another perturbation determined by the parameters $(a'_{\ell},b'_{\ell})$
$$
V'=\varepsilon\Big(\sum_{\ell=1,2}a'_{\ell}\cos2\ell \pi(x_2-x_{20})+b'_{\ell}\sin2\ell\pi(x_2-x_{20})\Big)
$$
and set $\bar V'=\rho\Psi^{-1} V'$. By using the formula (\ref{completeeq9}) we write the identity
\begin{align*}
u^-_{l,\sigma,\bar V'}-u^+_{r,\sigma,\bar V'}&=(u^-_{l,\sigma,\bar V'}-u^-_{l,\sigma_j,\bar V'})- (u^+_{r,\sigma,\bar V'}- u^+_{r,\sigma_j,\bar V'})\\
&+(u^-_{l,\sigma_j}-u^+_{r,\sigma_j})-(\mathscr{K}_{\sigma_j}+\mathscr{R}_{\sigma_j})\bar V_j\\
&+(\mathscr{K}_{\sigma_j}+\mathscr{R}_{\sigma_j})(\bar V_j-\bar V').
\end{align*}
For any point $(c,\sigma(c))\in\text{\uj c}_j$, in virtue of the formulae in (\ref{completeeq10}) the first term on the right-hand-side of the identity is not bigger than $C_3\varepsilon^2/3$. For
small $\|\bar V_j-\bar V'\|$ we have $\|(\mathscr{K}_{\sigma_j}+\mathscr{R}_{\sigma_j})(\bar V_j-\bar V')\|<\frac 13\|\mathscr{K}_{\sigma_j}(\bar V_j-\bar V')\|$. As both $V'$ and $V_j$ are independent of $x_3$, if the parameters $(a'_{\ell},b'_{\ell})$ satisfy
$$
\max\{|a_{\ell,j}-a'_{\ell}|,|b_{\ell,j}-b'_{\ell}|\}\ge\varepsilon
$$
we find from above identities and the estimate (\ref{completeeq8}) that
\begin{equation}\label{completeeq11}
\text{\rm Osc}_{x\in D}\min_{x_3}\Big(u^-_{l,\sigma}-u^+_{u,\sigma}-(\mathscr{K}_{\sigma}+ \mathscr{R}_{\sigma})\bar V'\Big)\ge \frac 13C_3\varepsilon^2>0.
\end{equation}
It implies that, for each small rectangle $\text{\uj c}_j$ we only need to cancel out at most $2^4$ $\varepsilon$-cubes from the grid for $\{\Delta a_{\ell},\Delta b_{\ell}:\ell=1,2\}$ so that the formula (\ref{completeeq11}) holds for the all other cubes. Let $j$ ranges over the set $\mathbb{J}$, we obtain a set $\text{\uj S}^c_2\subset\{a_{\ell}\in [1,2],b_{\ell}\in[1,2]:\ell=1,2\}$ with Lebesgue measure
$$
\text{\rm meas}\text{\uj S}^c_2\ge 1-2^4K_{\sigma}K_c\varepsilon,
$$
such that the formula (\ref{completeeq11}) holds for each $(a'_{\ell},b'_{\ell})\in\text{\uj S}^c_2$ and for each $\sigma\in\mathbb{I}_{\sigma_0}$.

By taking $V'\in\mathfrak{V}_3$, in the same way we can see that some set $\text{\uj S}^c_3\subset\{c_{\ell}\in [1,2],d_{\ell}\in[1,2]:\ell=1,2\}$ with Lebesgue measure
$$
\text{\rm meas}\text{\uj S}^c_3\ge 1-2^4K_{\sigma}K_c\varepsilon,
$$
such that the formula
\begin{equation}\label{completeeq12}
\text{\rm Osc}_{x\in D}\min_{x_2}\Big(u^-_{l,\sigma}-u^+_{u,\sigma}-(\mathscr{K}_{\sigma}+ \mathscr{R}_{\sigma})\bar V'\Big)>0
\end{equation}
for each $(c'_{\ell},c'_{\ell})\in\text{\uj S}^c_3$ and each $\sigma\in\mathbb{I}_{\sigma_0}$.

Therefore, for each $(a_{\ell},b_{\ell},c_{\ell},d_{\ell},)\in\text{\uj S}^c_2\times\text{\uj S}^c_3$, the formulae (\ref{completeeq11}) and (\ref{completeeq12}) implies that for all $\sigma\in\mathbb{I}_{\sigma_0}$ the diameter of each connected component of the set
$$
\arg\min(u^-_{l,\sigma,\bar V}-u^+_{r,\sigma,\bar V})|_D
$$
is smaller than $D$. As $\varepsilon>0$ can be arbitrarily small, for each disk $D$, an open-dense set $\mathfrak{V}_D$ exists such that this disconnect property holds for the system $L-\bar V$ with $\bar V\in\mathfrak{V}_D$. Since $\sigma$ is restricted on a closed set in the line which can be covered by finitely many $\mathbb{I}_{\sigma_i}$, this property is also open-sense for all $\sigma$ under our consideration.

Each section $D$ admits a hierachy of partition of small disks $D=\cup_jD_{kj}$ such that the size $D_{jk}$ approaches zero as $k\to\infty$, the intersection $\cap_k\mathfrak{V}_{D_{kj}}$ is a residual set. Therefore, we have proved
\begin{theo}\label{chainthm1}
It is an open-dense condition for $H$ such that the set
$$
\arg\min(u^-_{l,\sigma}-u^+_{r,\sigma})\backslash((\Upsilon_{l,\sigma}\cup\Upsilon_{r,\sigma})+\delta)
$$
consists of totally disconnected semi-static curves.
\end{theo}

\subsection{Criterion for strong and weak resonance}
Given a perturbation $\epsilon P(\tilde x,\tilde y)$, it is natural to ask, along the resonant path $\Gamma$, how many many double resonant points need to be treated as strong resonance. Along a segment of resonant path $\Gamma_{\tilde\omega,\ell}$ the resonance condition
$$
\langle\tilde  k,\tilde \omega\rangle=0
$$
is always satisfied and at each double resonant point some other $\tilde k'\in\mathbb{Z}^3$ exists such that $\tilde k'$ is linearly independent of $\tilde k$ and
$$
\langle\tilde  k',\tilde \omega\rangle=0.
$$
Recall the process of KAM iteration, the main part of the resonant term is obtained by averaging the perturbation over a circle determined by these two resonant relations. It takes the form
$$
Z=Z_{\tilde k}(\langle \tilde k,\tilde x\rangle,\tilde y)+Z_{\tilde k,\tilde k'}(\langle \tilde k,\tilde x\rangle,\langle \tilde k',\tilde x\rangle,\tilde y)
$$
where
$$
Z_{\tilde k}=\sum_{j\in\mathbb{Z}\backslash\{0\}}P_{j\tilde k}(\tilde y)e^{j\langle \tilde k,\tilde x\rangle i}, \qquad
Z_{\tilde k,\tilde k'}=\sum_{(j,l)\in\mathbb{Z}^2, l\neq 0}P_{jk+lk'}(\tilde y)e^{(j\langle \tilde k,\tilde x\rangle+l\langle\tilde  k_i,\tilde x\rangle)i}.
$$
Since $P$ is $C^r$-function, the coefficient $P_{j\tilde k+l\tilde k'}$ is bounded by
$$
|P_{j\tilde k+l\tilde k'}|\le 8\pi^3\|P\|_{C^r}\|j\tilde k+l\tilde k'\|^{-r},
$$
which deduces the estimation
\begin{equation}\label{criterioneq1}
\|Z_{\tilde k,\tilde k'}\|_2\le d\|P\|_{C^r}\|\tilde k'\|^{-r+2}
\end{equation}
where $d=d(\tilde k)$ depends on $\tilde k$. The function $Z_{\tilde k}$ is periodic in $\tilde q=\langle \tilde k,\tilde x\rangle$. In virtue of the theorem \ref{appenBthm1} (see Appendix B), the following hypotheses is obviously open and dense:

({\bf H1.1}): For each $\tilde y\in\Gamma_{\ell}$, $Z_{\tilde k}$ is non-degenerate at its maximal point, i.e. $\partial^2_{qq}Z_{\tilde k}(\tilde q)>0$ holds provided $\tilde q$ is a maximal point.

Given some $Z_{\tilde k}$ satisfying the hypothesis ({\bf H1.1}), certain $\lambda>0$ exists such that for each $\tilde y\in\Gamma_{\ell}$, $\partial^2_{qq}Z_{\tilde k}\ge\lambda$ holds at the maximal point. Assume at $\tilde y'\in\Gamma_{\ell}$, the second resonant condition $\langle\tilde  k',\tilde \omega(\tilde y')\rangle=0$ is also satisfied. One thus obtains the normal form (\ref{chaineq2}), by performing the coordinate transformation (\ref{chaineq1}).  The homogenized form of the truncated Hamiltonian takes the form
$$
G=\langle Ay,y\rangle+V_{\tilde k}(x_2)+V_{\tilde k,\tilde k'}(x).
$$
The Hamiltonian flow determined by $\langle Ay,y\rangle+V_{\tilde k}(x_2)$ admits a normally hyperbolic invariant cylinder $\Pi_{\tilde k,\tilde k'}^0=\{y=0,x_2=x^*_2\}\times\mathbb{T}$ if $x^*_2$ is a non-degenerate maximal point of $V_{\tilde k}$. Applying the theorem of normally hyperbolic manifold, one obtains from the estimate (\ref{criterioneq1}) that some positive number $d_1=d_1(\lambda)>0$ exists such that $\Phi_{\bar G}^t$ also admits a normally hyperbolic and invariant cylinder $\Pi_{\tilde k,\tilde k'}$ close to $\Pi_{\tilde k,\tilde k'}^0$ provided
\begin{equation}\label{criterioneq2}
\|\tilde k'\|^{r-2}\ge\frac d{d_1}\|P\|_{C^r}.
\end{equation}
It is a criterion to see whether the double resonance is thought as weak resonance and can be treated in the way for {\it a priori} unstable system. There are only finitely many $\tilde k'\in\mathbb{Z}^3$ not satisfying this condition, thus are treated as strong double resonance.

Therefore, once a perturbation $P$ is chosen so that ({\bf H1.1}) is satisfied, there are finitely many double resonant frequencies which need to be treated as strong double resonance. The number is independent of the size of $\epsilon$. At strong double resonance, it is also open and dense condition that

({\bf H1.2}): at each strong double resonance point, the maximal point of $Z_{\tilde k}+Z_{\tilde k,\tilde k'}$ is non-degenerate, two eigenvalues of the Hessian matrix are positive and different $\lambda_{k,j}>0$ for $j=1,2$. Indeed, there exists $\nu>0$ such that $\lambda_{\tilde k,2}\ge\nu\|\tilde k\|^{r-2}$ and $\lambda_{\tilde k,1}\ge\nu\|\tilde k\|^{r-2}\|\tilde k'\|^{r-2}$.

\subsection{Proof of the main theorem}
Given $y'_0,y'_1,\cdots,y'_k$ we have chosen a resonant path $\Gamma_{\omega}$ so that $\mathscr{L}_{\beta_h}(\Gamma_{\omega})$ passes through each $\delta$-neighborhood of these points. Let $\epsilon P$ satisfy all hypothesis above. In order to make things convenient for readers, we formulate them here again:

({\bf H1}) {\it for each strong double resonance, the potential $V_i$ attains its maximum at one point only, the Hessian matrix of $V_i$ at that point is negative definite. All eigenvalues are different: $-\lambda_2<-\lambda_1<0<\lambda_1<\lambda_2$}. (see {\bf H1} in Subsection 5.1, {\bf H1.1} and {\bf H1.2} in Subsection 8.4);

({\bf H2}) {\it for the Hamiltonian flow $\Phi_{Y_{i,T}}^t$, the stable and unstable manifold of the fixed point intersect transversally along each minimal homoclinic orbit. Each minimal homoclinic orbit approaches to the fixed point along the direction $\Lambda_1$: $\dot\gamma(t)/\|\dot\gamma(t)\| \to\Lambda_{x1}$ as $t\to\pm\infty$.} (see {\bf H2} in Subsection 4.1. The function $Y_{i,T}$ solves the equation $H_{i,T}(\tilde x,y,Y_{i,T})=E$, the transversality is in the sense that, at the intersection points, the tangent space of the stable and unstable manifold span the tangent space of the energy level set.)

({\bf H3}): {\it For each $c\in\partial^*\mathbb{F}_{0,i}$, the Aubry set  does not contain minimal curve homoclinic to the origin $($fixed point$)$.} (see {\bf H2} in Subsection 4.1, each strong double resonance is related to a flat $\mathbb{F}_{0,i}$ corresponding to the Hamiltonian $Y_{i,T}$.)

({\bf H4}): {\it For each $g\in H_1(\mathbb{T}^2,\mathbb{Z})$, there are finitely many $\theta_i\in\mathbb{R}$ such that, for each rotation vector $\theta_i g$, the Mather set consists of two periodic orbits, for other rotation vector $\theta g$, the Mather set consists of one periodic orbit only. All these periodic orbits are hyperbolic.} (see {\bf H4} in Subsection 4.2, also formulated for the Hamiltonian $Y_{i,T}$.)

({\bf H5}): {\it For each $c\in\partial^*\mathbb{F}_{0,i}$ there is a disk disjoint either with the support of $\mu_c$ or of $\mu$, restricted on which, the set $\arg\min(U_c^--U'^+_c)$ is non-empty. The size of the disk is independent of $c$.} (see {\bf H5} in Subsection 4.2, also formulated for the Hamiltonian $Y_{i,T}$.)

Along the resonant path $\Gamma_{\omega}$, the strong double resonance points are denoted by $\omega_0,\omega_1,\cdots, \omega_m$, where the number $m$ depends on $P$. Each flat $\mathscr{L}_{\beta_H}(\omega_i)$ is surrounded by a annulus $\tilde{\mathbb{A}}_i\subset\alpha^{-1}_H(E)$. For each segment of $\Gamma_{\omega}$ connecting $\omega_i$ to $\omega_{i+1}$, denoted by $\Gamma_{\omega,i}$, $\mathscr{L}_{\beta_H}(\Gamma_{\omega,i})$ constitutes a channel connecting $\tilde{\mathbb{A}}_i$ to $\tilde{\mathbb{A}}_{i+1}$.

Split the unit interval into $2m+1$ segments
$$
[0,1]=[0=s_{0,i},s_{0,c}]\cup[s_{0,c},s_{1,i}]\cup\cdots\cup [s_{m-1,i},s_{m,c}] \cup[s_{m,c},s_{m,i}=1]
$$
and let $\Gamma_{j,c}$: $[s_{j,i},s_{j,c}]\to\alpha^{-1}_H(E)$, $\Gamma_{j,i}$: $[s_{j,c},s_{j+1,i}]\to\alpha^{-1}_H(E)$ denote the paths such that $\Gamma_{j,c}(s_{j,c}) =\Gamma_{j,i}(s_{j,c})$, $\Gamma_{j,i}$ falls into the annulus $\tilde{\mathbb{A}}_j$ along which the component $c_3$ keeps constant in the local coordinate system and $\Gamma_{j,c}$ falls into the channel $\mathscr{L}_{\beta_H}(\Gamma_{\omega,j})$ connecting $\tilde{\mathbb{A}}_i$ to $\tilde{\mathbb{A}}_{i+1}$. Let $\Gamma_{0,c}(0)\in\mathscr{L}_{\beta_H}(\nabla h(y'_0))$ and $\Gamma_{m,c}(1)\in\mathscr{L}_{\beta_H}(\nabla h(y'_k))$. The subscript ``$c$" is used to indicate complete intersection and the subscript ``$i$" denotes the incomplete intersection. We choose the conjunction of these curves as candidate of generalized transition chain
\begin{equation}\label{proofeq1}
\Gamma=\Gamma_{0,c}\ast\Gamma_{0,i}\ast\cdots\ast\Gamma_{m-1,i}\ast\Gamma_{m,c}.
\end{equation}
Indeed, restricted on the segment $\Gamma_{j,i}$ ($j=0,\cdots, m-1$) it has been proved satisfying the condition (H2) in the Definition \ref{chaindef1} ($c$-equivalence) by using the hypothesis ({\bf H1$\sim$5}).  To guarantee the condition (H1) in the Definition \ref{chaindef1} when it is restricted on each segment $\Gamma_{j,c}$ ($j=0,\cdots, m$), one need to impose some condition which has been proved to be generic in Subsection 8.3, Theorem \ref{chainthm1}:

({\bf H6}): {\it if the Aubry set covers certain $2$-torus in $\mathbb{T}^3$ for $c\in\Gamma_{j,c}$, then certain finite covering manifold $\check M$ and certain two-dimensional section $\Sigma_c$ exist such that
$$
\mathcal{N}(c,\check M)|_{\Sigma_c}\backslash(\mathcal{A}(c,\check M)+\delta)|_{\Sigma_c}\neq\varnothing.
$$
is totally disconnected.}

Under these hypothesis, namely ({\bf H1$\sim$6}), the path $\Gamma$ defined in (\ref{proofeq1}) is a transition chain. Choose suitably many $c_i\in\Gamma$ ($i=0,1,\cdots,i_m)$ such that

1, each $\tilde{\mathcal{A}}(c_i)$ is connected to $\tilde{\mathcal{A}}(c_{i+1})$ by local minimal orbit either of type-$c$ or of type-$h$;

2, among these classes, some  classes $c_{i_j}$ $(j=0,1,\cdots,k)$ exist such that $c_{i_j}$ is very close to $y'_j$ (the prescribed action variables in Theorem \ref{mainthm}) if one thinks both $c_{i_j}$ and $y_j$ as points in $\mathbb{R}^3$.

Recall the proof of Theorem  \ref{constructionthm1}. Let $\gamma$: $[-K,K']\to\mathbb{T}^3$ be the minimizer of the action (see (\ref{constructioneq23})) satisfying the boundary conditions $\gamma(-K)=x_0$ and $\gamma(K')=x_k$. Dividing the time interval $[-K,K']$ into $2i_m+1$ parts
$$
[-K,K']=[t_0^+,t^-_0]\cup [t^-_0,t^+_1]\cup\cdots\cup [t^+_{i_m},t_{i_m}^-],
$$
imposing some constraints on $\gamma$ at $t=t_i^{\pm}$ and conditions on sufficiently large $t_{i+1}^+-t_i^-$ and $t^-_i-t^+_i$, one then proves that $\gamma$ is a solution of the Lagrange equation determined by $H$. The curve $\gamma$ determines an orbit of the Hamiltonian flow $\Phi_H^t$:
$$
x(t)=\gamma(t),\qquad y(t)=\frac {\partial L}{\partial \dot x}(\gamma(t),\dot\gamma(t)).
$$

For each $x\in M$, the set $V^-(c,x,L)\subset T_xM$ is defined as follows: a vector $v\in V^-(c,x,L)$ if and only if a backward $c$-semi static curve $\gamma^{-}$ for the Lagrangian $L$ exists such that $v=\dot\gamma^{-}(0)$. The set $V^+(c,x,L)$ is defined for forward semi-static curve similarly. Clearly, one has
\begin{pro}
The set-valued map $L\to V^{\pm}_{c,x,L}$ is upper-semi continuous.
\end{pro}
To see that this orbit visits the ball $B_{\delta}(x_0,y_0),B_{\delta}(x_k,y_k)\subset\mathbb{T}^3 \times\mathbb{R}^3$ and the balls $B_{\delta}(y_i)\subset \mathbb{R}^3$ $(i=1,\cdots,k-1)$, we use this proposition. As $h$ is integrable, any backward (forward) $c$-semi static curve is $c$-static for all $t\in\mathbb{R}$. Along any $c$-minimal curve it holds that the action variable $y=c$. Since the perturbation $h\to h+\epsilon P$ is small, $\dot\gamma(-K)$ is close to $V^+_{x,c_0,L}$ if $t_0^--t_0^+$ is sufficiently large. It follows that $\|y(-K)-y_0\|$ is very small provided $\epsilon$ is sufficiently small. In the same way, one can see that $\|y(K')-y_k\|$ is also very small. Since Ma\~n\'e set is upper semi continuous with respect to Lagrangian, at the time $t_i=(t_i^-+t_i^+)/2$, $(\gamma(t_i),\dot\gamma(t_i))$ is very close to $\tilde{\mathcal{A}}(c_i)$, $\|y(t_i)-y_i\|$ is very small. This proves that the Hamiltonian flow $\Phi_H^t$ admits an orbit that visits these balls in turn.

To complete the proof of Theorem \ref{mainthm}, we only need to show the generic property. Towards this goal, let us observe a fact: under the rescaling $y\to\sqrt{\lambda}y$, $t\to\sqrt{\lambda}^{-1}t$, the Hamiltonian equation determined by $\frac12\langle Ay,y\rangle+\lambda V(x)$ is the same as it for the function $\frac12\langle Ay,y\rangle+V(x)$. Therefore, some open-dense set $\mathfrak{O}\subset\mathfrak{S}_1\subset C^r$ exists, some $\epsilon_P>0$ is associated to each $P\in\mathfrak{O}$ such that the Hamiltonian flow $\Phi_H^t$ satisfies the conditions {\bf H1$\sim$5} provided $\epsilon\le\epsilon_P$, because the number of strong double resonant points is independent of the size of $\epsilon$. From the proof of Theorem \ref{chainthm1} in the subsection 8.3, the condition {\bf H6} is required for the intersection of countably many open-dense set contained in $\mathfrak{B}_{\epsilon_0}$: $\cap_i\mathfrak{O}_i$. Clearly, there exists a residual set $\mathfrak{R}_{\epsilon_0}\subset \mathfrak{S}_{\epsilon_0}$, for each $P\in\mathfrak{R}_{\epsilon_0}$ there exists a set $R_P$ residual in $[0,\epsilon_0]$ such that $\{\lambda P:P\in\mathfrak{R}_{\epsilon_0},\lambda\in R_P\}\subset\cap_i\mathfrak{O}_i$. Take the intersection of these sets, we obtain the cusp-residual property. \eop

\appendix
\renewcommand{\theequation}{A.\arabic{equation}}
\section{Normal form}
\setcounter{equation}{0}

In this appendix, we study the normal form of nearly integrable Hamiltonian, from which one obtains some  information about the relevant Mather sets, Aubry sets as well as Ma\~n\'e sets. Here, the system is assumed to have arbitrary $n$-degrees of freedom
$$
H(x,y,t)=h(y)+P_{\epsilon}(x,y,t),\qquad (x,y,t)\in\mathbb{T}^n\times \mathbb{R}^n\times\mathbb{T}.
$$
The perturbation can be autonomous as well as time-1-periodic. As $t$ can be treated as the $(n+1)$-th angle coordinate, we replace $n$ by $n+1$ when we consider time-1-periodic perturbation $P_{\epsilon}(x,y,t)$. Thus, we consider autonomous Hamiltonian only.

\subsection{KAM iteration at strong resonance}

Let $\omega(y)=\nabla h(y)$ denote the frequency vector of the unperturbed system. For autonomous case, a frequency $\omega$ is called rational of (minimal) period $T$ if $T\omega\in\mathbb{Z}^n$
and $t\omega\notin\mathbb{Z}^n$ for each $t\in (0,T)$.

Let the frequency $\omega$ be rational of period $T$. With a function $g(x,y)$ on the torus one associates its time average $[g]$ along the orbits of the linear flow defined by $\omega$:  $x\to x+\omega t$
$$
[g](x,y)=\frac 1T\int_0^Tg(x+\omega t,y)dt.
$$
We say that $g$ is {\it resonant} (with respect to $\omega$) if $g=[g]$, it implies that $g$ is constant along the orbits of the linear flow $(x,y)\to (x+\omega t,y)$.

Let $B_R\subset\mathbb{R}^n$ be the ball of radius $R$ around the origin, then there are positive numbers $M=M(R)\ge m=m(R)>0$ such that
$$
m\|v\|^2\le\langle\nabla^2h(y)v,v\rangle\le M\|v\|^2, \qquad \forall\ y\in B_R,\ v\in\mathbb{R}^n.
$$
Let $\sigma$ and $\varrho$ denote positive number such that
$$
\sigma<\frac 13, \qquad K=K(\epsilon)=K_0\epsilon^{-\varrho},\qquad \varrho=\frac 13(1-3\sigma),
$$
the value of $\sigma$ will be specified later to satisfy certain covering property.

Denoted by $\{\omega_{\lambda}:\lambda\in\Lambda_{K,R}\}\subset B_{MR}$ the set of frequencies which are rational of period $T$ with $T\le K$. Clearly, $\Lambda_K$ is a finite index set. Let $y_{\lambda}=\nabla^{-1} h(\omega_{\lambda})$.

Let $i=(i_1,i_2,\cdots,i_n)\in \mathbb{Z}^n_+$, namely, $i_j$ is non-negative $\forall$ $j\in\{1,2,\cdots,n\}$. Let $|i|=\sum_{j=1}^ni_j$, $Y_i(y)= \prod_{j=1}^ny_{j}^{i_{j}}$. Let $\|\cdot\|_{j,D}$ denote the $C^j$-norm on the domain $D$, we omit the notation $D$ when it is clearly implied.

\begin{theo}\label{normalthm1}
For a nearly integrable Hamiltonian $H(x,y)=h(y)+P_{\epsilon}(x,y)$ we assume that both $h$ and $P_{\epsilon}$ are $C^r$-smooth with $r\ge 8$, and $\|P_{\epsilon}\|_{r,B_R\times\mathbb{T}^{n+1}}\le \epsilon$. Some small $\epsilon_0=\epsilon_0(M,m,n,r)>0$ exists such that for each $\epsilon\le\epsilon_0$ and each rational frequency $\omega_{\lambda} =\omega(y_{\lambda})$ with a period $T\le K(\epsilon)$, a canonical transformation $\mathscr{F}_{\lambda}$ is well defined on
$$
D_{y_{\lambda},\epsilon}=\{(x,y)\in\mathbb{T}^n\times\mathbb{R}^n: \|y-y_{\lambda}\|\le T^{-1}\epsilon^{\sigma}\}
$$
which reduce the Hamiltonian into the normal form
\begin{equation}\label{normaleq1}
H\circ\mathscr{F}_{\lambda}(x,y)=h(y)+Z(x,y)+R(x,y)
\end{equation}
where $Z$ is resonant with respect to $\omega_{\lambda}$ with $\|Z\|_r\le 2\epsilon$,  $R=R_h+R_r$ is a higher order term when it is restricted in $D_{y_{\lambda},\epsilon}:$
\begin{equation*}
R_h=\sum_{|i|=5}Y_i(y-y_{\lambda})R_{h,i}(x,y),
\end{equation*}
\begin{align*}
&\|R_{h}\|_{2,D_{y_{\lambda},\epsilon}}\le D\epsilon^{\frac 13+5\sigma}, \\
&\|R_r\|_{2,D_{y_{\lambda},\epsilon}}\le D\epsilon^{\frac 43+2\sigma}
\end{align*}
where the constant $D>0$ is independent of $\epsilon$. Restricted in the region $\{\|y-y_{\lambda}\|=O(\sqrt{\epsilon})\}$, we have a sharper estimate
$$
\|R_{h}\|_1\le D\epsilon^{\frac 43+5\sigma},\qquad \|R_{h}\|_2\le D\epsilon^{\frac 56+5\sigma}.
$$
\end{theo}
\begin{proof}
The canonical transformation $\mathscr{F}_{\lambda}$ is the composition of five steps of coordinate transformations $\mathscr{F}_{\lambda}=\mathscr{F}_4\circ\cdots\circ\mathscr{F}_1 \circ\mathscr{F}_0$. Each step of transformation $\mathscr{F}_j$  is defined as the time-1-map of $\phi_{W_j}^t$, the Hamiltonian flow determined by the generating function $W_j$.
For the first step of coordinate transformation $\mathscr{F}_0$, we set
$$
W_0(x,y)=-\frac 1T\int_0^TP(x+\omega t,y)tdt
$$
which solves the equation
$$
\Big\langle\omega, \frac{\partial W_0}{\partial x}\Big\rangle=-P+[P].
$$
Let $Z=[P]$, it follows that
$$
H_1=h(y)+Z(x,y)+R_{1,1}(x,y)+R_{1,2}(x,y)
$$
where
\begin{align*}
R_{1,1}&=\Big\langle\frac{\partial h}{\partial y}-\omega,\frac{\partial W_0}{\partial x}\Big\rangle,\\
R_{1,2}&=\int_0^1(1-t)\{\{H,W_0\},W_0\}\circ\phi_{W_0}^tdt.
\end{align*}
Obviously, $[R_{1,1}]=0$ and one has the form $R_{1,1}=\sum_{|i|=1}Y_i(y-y_{\lambda})R_{1,1,i}(x,y)$. If we write $H_j=H_{j,1}+R_{j,2}$ where
$$
H_{j,1}=h(y)+Z(x,y)+R_{j,1}(x,y),
$$
and for the coordinate transformation $\mathscr{F}_j$ ($j=1,\cdots,4$), we set by induction
$$
W_j(x,y)=-\frac 1T\int_0^TR_{j,1}(x+\omega s,y)sds
$$
then the Hamiltonian $H_{j+1}$ takes the form
$$
H_{j+1}=h(y)+Z(x,y)+R_{j+1,1}(x,y)+R_{j+1,2}(x,y)
$$
where
\begin{align*}
R_{j+1,1}=&\Big\langle\frac{\partial h}{\partial y}-\omega,\frac{\partial W_j}{\partial x}\Big\rangle,\\
R_{j+1,2}=&\int_0^1(1-t)\{\{H_{j,1},W_j\},W_j\}\circ\phi_{W_j}^tdt\\
&+R_{j,2}\circ\mathscr{F}_j.
\end{align*}
Also, $[R_{j+1,1}]=0$ and we can write $R_{j+1,1}=\sum_{|i|=j+1}Y_i(y-y_{\lambda})R_{j+1,1,i}(x,y)$.

By the construction, we see that $H_{j,1}$ is $C^{r-j}$-smooth and $R_{j,2}$ is $C^{r-j-1}$-smooth. Some constants $D_{j}>0$, independent of $\epsilon$, exists such that for $(x,y)\in D_{y_{\lambda},\epsilon}:$
\begin{align}\label{estimate}
&\|W_0\|_1\le D_0K_0\epsilon^{1-\varrho},\notag\\
&\|W_j\|_{1}\le D_{j}K_0^{j+1}\epsilon^{1+(j-1)\sigma-2\varrho},\ \ \ (j\ge 1), \notag\\
&\|R_{j+1,1,i}\|_{2}\le D_{j}K_0^{j+1}\epsilon^{1-(j+1)\varrho},\\
&\|R_{j+1,2}\|_{2}\le D_{j}K_0^{2(j+1)}\epsilon^{2-2\varrho},\notag
\end{align}
where we have used the relations that $T\le K_0\epsilon^{-\varrho}$, $\varrho=\frac 13(1-3\sigma)$ and $\sigma<\frac13$. If we write the canonical transformation $\mathscr{F}_j$ in the form
$$
\mathscr{F}_j:(x,y)\Rightarrow (x+U_j(x,y),y+V_j(x,y))
$$
then
$$
(U_j,V_j)=\int_0^1\Big(\frac{\partial W_j}{\partial y},-\frac{\partial W_j}{\partial x}\Big)\circ\phi^t_{W_j}dt
$$
It maps $\mathbb{T}^n\times B_{\delta_{j+1}}\to \mathbb{T}^n\times B_{\delta_j}$ for $\epsilon\le\epsilon_0$ if we set
$$
\delta_j=\Big(2-\frac{j+1}5\Big)\frac {\epsilon^{\sigma}}T,\qquad \epsilon_0\le\max_{j\le 5}\frac{1}{(5D_jK_0^{j+2})^{\frac 1{(j+1)\sigma}}}.
$$
The estimate on $R_r=R_{5,1}$ and $R_h=R_{5,2}$ follows from the formulae (\ref{estimate}).
\end{proof}

If the system is real analytical, the higher order term can be reduced to the order $O(\exp(-\frac 1{\epsilon^\sigma}))$ (see \cite{Lo}), with which one obtains the Nekhoroshev's estimate.

\subsection{Covering property}

Recall that the set of frequencies $\{\omega_{\lambda}:\lambda\in\Lambda_{K,R}\}\subset B_{MR}$  each of which is rational of period $T$ with $T\le K$, and the domains $\{D_{y_{\lambda},\epsilon}:\ \lambda\in\Lambda_{K,R}\}$ where the iteration of KAM is carried (see Theorem \ref{normalthm1} for definition).

\begin{theo}
The following covering property holds
\begin{equation}\label{normaleq2}
\bigcup_{\lambda\in\Lambda_{K,R}}\mathscr{F}_{\lambda}^{-1}D_{y_{\lambda},\epsilon}\supset \mathbb{T}^{n}\times B_R \hskip 0.5 true cm \text{\rm provided}\hskip 0.4 true cm \sigma<\frac 1{3n+3}.
\end{equation}
\end{theo}
\begin{proof}
To show the covering property, we use Dirichlet's approximation theorem. For real $x$ one has
$$
x=[x]+\{x\},
$$
where $[x]\in\mathbb{Z}$ the integer part, and $\{x\}\in(0,1)$. We use notation
$$
\|x\|_{\mathbb{Z}}=\inf\{\{x\},1-\{x\}\}=\text{\rm dist}(x,\mathbb{Z}).
$$
If $x=(x_1,x_2,\cdots,x_n)\in\mathbb{R}^n$ one sets
$$
\|x\|_{\mathbb{Z}}=\sup_{i=1,2,\cdots,n}\|x_i\|_{\mathbb{Z}}.
$$
\begin{pro}\label{normalpro1} {\rm (Dirichlet, see for examples, \cite{Cas, Sch})} Let $\omega\in\mathbb{R}^n$ and $K$ a real number with $K>1$. There exists an integer $k$, $1\le k<K$, such that
$$
\|k\omega\|_{\mathbb{Z}}\le K^{-\frac 1n}.
$$
\end{pro}
For any $\omega\in\mathbb{R}^n$, by applying Dirichlet's theorem, we find some rational vector $\omega^*$ existing such that $K'\omega^*\in \mathbb{Z}^n$ with $K'\le K$ and
$$
\text{\rm dist}(K'\omega,K'\omega^*)\le\sqrt{n}K^{-\frac 1n},
$$
here, $\omega^*$ is a rational vector of period $T$. Since $h$ is assumed strictly convex, there exist two points $y,y^*\in B_R$ such that $\nabla h(y)=\omega$, $\nabla h(y^*)=\omega^*$ and
$$
\text{\rm dist}(y,y^*)\le\frac 1m\text{\rm dist}(\omega,\omega^*).
$$
The condition $\|y-y^*\|\le T^{-1}\epsilon^{\sigma}$ is guaranteed if we choose
\begin{equation}\label{normaleq8}
K^{\frac 1n}=\frac {\sqrt{n}}{m}\epsilon^{-\sigma}.
\end{equation}
As $T\le K$ is required, the following should be satisfied:
$$
K\le K_0\epsilon^{-\frac 13(1-3\sigma)},
$$
that is, referring to (\ref{normaleq8}),
$$
\epsilon^{\frac{1-3\sigma}3-n\sigma}\le \Big(\frac {m}{\sqrt{n}}\Big)^nK_0,
$$
which determines a threshold for $\epsilon$ provided:
$$
\sigma<\frac 1{3n+3}.
$$
As we choose $\sigma$ satisfying this condition, the covering property (\ref{normaleq2}) is proved.
\end{proof}

For the purpose of this paper, the covering property (\ref{normaleq2}) for the whole space is not necessary, instead, it is good enough to cover a neighborhood of a resonant path. Denote by ${\bf k}=(k_1,\cdots,k_{n-1})$ a $n\times(n-1)$ matrix, where $k_1,\cdots,k_{n-1}$ are integer vectors. We consider the $n-1$ resonance line
$$
\Gamma_{\bf k}=\{y\in\mathbb{R}^n:\langle k_i,\partial h(y)\rangle=0\ \forall\ i=1,\cdots n-1\}.
$$
If the covering property (\ref{normaleq2}) in Theorem \ref{normalthm1} is replaced by covering a neighborhood of the line
\begin{equation}\label{normaleq9}
\bigcup_{\lambda\in\Lambda_{K,R}}D_{y_{\lambda},\epsilon}\supset\mathbb{T}^{n}\times\{\|y-y_0\|<\mu
K^{-1}\epsilon^{\sigma}:y_0\in\Gamma_{\bf k}\cap B_R\}
\end{equation}
then it works if
\begin{equation*}
\sigma<\frac 16.
\end{equation*}
Indeed, as all frequencies are on a $(n-1)$-resonance line, by using Dirichlet approximation theorem (Proposition \ref{normalpro1}) for $n=1$  we obtain a threshold $\sigma< 1/6$.

Recall that the term $Z$ in (\ref{normaleq1}) is resonant with respect to $\omega$, some rational frequency of period $T\le K$, namely, it has the form
$$
Z(x,y)=\sum_{\langle k,\omega\rangle=0}Z_{k}(y)e^{i\langle k,x\rangle}.
$$
Note that $T\omega$ is an indivisible integer vector, i.e. $\mu T\omega\notin \mathbb{Z}^n$ for any $\mu\in (0,1)$. There are $n-1$ integer vectors $I_2,I_3,\cdots,I_n$ such that the matrix $(I_2,I_3,\cdots,I_n)$ is indivisible, $\text{\rm rank}(I_2,I_3,\cdots,I_n)=n-1$ and $\langle I_i,\omega\rangle =0$ holds for each $i\in\{2,3,\cdots,n\}$. Clearly, there is another integer vector $I_1$ such that the matrix $I=(I_1,I_2,\cdots,I_n)$ is uni-module. Each integer vector $k\in\mathbb{Z}^n$ with $\langle k,\omega\rangle=0$ uniquely determines an integer vector $\bar k\in\mathbb{Z}^{n-1}$ such that $k=\sum_{j=1}^{n-1} \bar k_jI_{j+1}$.

We introduce a coordinate transformation: $(x,y)\to (p,q)$ such that
\begin{equation}\label{normaleq10}
\tilde q=I^tx,\qquad \tilde p=I^{-1}y.
\end{equation}
This coordinate transformation is symplectic, $H(I^{-t}\tilde q,I\tilde p)$ is also a function defined on $\mathbb{T}^n$ with respect to $\tilde q$. Let $y$ be the point where $\nabla h(y)=\omega$, then the gradient of $\tilde h(\tilde p)=h(I\tilde p)$ satisfies
$$
\tilde\omega=\nabla \tilde h(\tilde p)=(0,\cdots,0,\tilde\omega_n),
$$
and $Z(I\tilde p,I^{-t}\tilde q)$ is independent of $q_n$, thus we can write $\tilde Z(\tilde p,\tilde q) =\tilde Z(p,q,p_n)$ if we use the natation $\tilde p=(p,p_n)$ and $\tilde q=(q,q_n)$.

Next, let us consider the time-1-periodical non-autonomous case. Assume $T(\omega,1)\in \mathbb{Z}^{n+1}$ is an indivisible integer vector. As $Z$ is resonant with respect to $\omega$, we have
$$
Z(x,y,t)=\sum_{\langle k,\omega\rangle+l=0}Z_{k,l}(y)e^{i\langle k,x\rangle+lt}.
$$
Thus, there are $n$ integer vectors $I_1,\cdots,I_n,J\in\mathbb{Z}^n$ such that $\langle I_i, \omega \rangle+J_i=0$ for each $i\in\{1,\cdots,n\}$. For each $(k,l) \in\mathbb{Z}^{n+1}$ with $\langle k, \omega\rangle +l=0$, there is uniquely determined $(\bar k_1,\cdots,\bar k_n)\in\mathbb{Z}^n$ such that
$$
(k,l)=\sum_{i=1}^n\bar k_i(I_i,J_i).
$$
By choosing suitable $I_i$, we can make $I=(I_1,I_2,\cdots,I_n)$ be uni-module. Introduce the coordinate transformation (\ref{normaleq10}), let $y$ be the point where $\nabla h(y)=\omega$, then the gradient of $\bar h(p)=h(Ip)$ satisfies
$$
\bar\omega=\nabla \bar h(p)=-J.
$$
Note that each $(k,l)$ with $\langle k,\omega\rangle+l=0$ uniquely determines $\bar k\in \mathbb{Z}^n$ such that $(k,l)=\bar k(I^t,J)$. As we have
$$
Z(I^{-t}q,Ip,t)=\sum_{\bar k\in\mathbb{Z}^n}Z_{k,l}(Ip)e^{i\langle\bar k, q+Jt\rangle},
$$
in the new coordinates the resonant term $\bar Z=\bar Z(p,q+Jt)$. Let $q'=q+Jt, p'=p$ and let $h'(p')=\bar h(p')+\langle J,p\rangle$, we find the Hamiltonian equation of $h'(p')+\bar Z(p',q')$ is the
same as the Hamiltonian equation of $\bar h (p)+\bar Z(p,q+Jt)$.

\renewcommand{\theequation}{B.\arabic{equation}}
\section{Hyperbolicity of minimal periodic orbits}
\setcounter{equation}{0}

\centerline{{\it by} {\ui Chong-Qing Cheng and Min Zhou}}

In the section 4, we made the hypothesis ({\bf H4}) on the hyperbolicity of minimal periodic orbits in Hamiltonian systems with two degrees of freedom. In the subsetion 8.4, we need the hypothesis ({\bf H1.1}). In $C^r$-topology with $r\ge 4$ these hypotheses are shown generic in \cite{CZ}. For the sake of completeness and convenience of reader, we present the proof in this appendix.

\subsection{Non-degeneracy of global minimum}
\begin{theo}\label{appenBthm1}
Let $F_{\lambda}$: $\mathbb{T}\to\mathbb{R}$ be a family of $C^{r}$-functions depending on the parameter $\lambda\in [\lambda_0,\lambda_1]$ $(r\ge 4)$. If $F_{\lambda}$ is Lipschitz continuous in $\lambda$, then there exists an open-dense set $\mathfrak{O}\subset C^{r}(\mathbb{T},\mathbb{R})$ such that for each $V\in\mathfrak{O}$ and each $\lambda\in[\lambda_0,\lambda_1]$, each global minimum of $F_{\lambda}-V$ is non-degenerate, namely, the second derivative is positive at each global minimizer.
\end{theo}
\begin{proof}
Since the openness is obvious, we only need to show the density. For this goal, we introduce a set of perturbations with four parameters:
$$
\mathfrak{V}=\Big\{V=\sum_{i=1}^2(A_i\cos ix+B_i\sin ix):\ (A_1,B_1,A_2,B_2)\in\mathbb{I}^4\Big\},
$$
where $\mathbb{I}^4=[1,2]\times[1,2]\times[1,2]\times[1,2]$. Let $M=12^{-1}\sup_{x,\lambda}|\partial^4_xF_{\lambda}|$, we are going to show that, for any small numbers $\epsilon,d>0$ there exists $(A_1,B_1,A_2,B_2)\in I^4$ such that
\begin{equation}\label{appenBeq1}
(F_{\lambda}-\epsilon V)(x)-\min_x(F_{\lambda}-\epsilon V)\ge M|x-x^*|^4, \qquad \forall\ |x-x^*|\le d
\end{equation}
holds for each $\lambda\in[\lambda_0,\lambda_1]$ whenever the point $x^*$ is a global minimizer of $F_{\lambda}-\epsilon V$. It implies the second derivative is positive. Indeed, if it equals zero, the third derivative will be zero also. Consequently, the above formula does not hold.

By choosing sufficiently large integer $k$, $\epsilon=\sqrt{\pi/k}$ can be arbitrarily small. Let $x_i=2i\pi/k$, $I_i=[x_i-d,x_i+d]$ and $d=\pi/k$, then $\epsilon=\sqrt{d}$ and
$$
\bigcup_{i=0}^{k-1}I_i=\mathbb{T}.
$$
Restricted on each interval $I_i$, each function $V\in\mathfrak{V}$ is approximated by Taylor series (module constant)
$$
V_i(x)=a_i(x-x_i)+b_i(x-x_i)^2+c_i(x-x_i)^3+O(|x-x_i|^4).
$$
Given two points $(a_i,b_i,c_i)$ and $(a'_i,b'_i,c'_i)$, we have two functions $V_i$ and $V'_i$ in the form of Taylor series. Let $\Delta V=V_i-V'_i$, $\Delta a=a_i-a'_i$, $\Delta b=b_i-b'_i$ and $\Delta c=c_i-c'_i$, we have $\Delta V(x_i)=0$ and
\begin{align*}
&\Delta V(x_i+d)+\Delta V(x_i-d)=2\Delta bd^2+O(d^4),\\
&\Delta V(x_i+d)-\Delta V(x_i-d)=2(\Delta a+\Delta cd^2)d+O(d^4),\\
&\Delta V\Big(x_i\pm\frac 12d\Big)=\Big(\pm\frac 12\Delta a+\frac 14\Delta bd\pm\frac 18\Delta cd^2\Big)d+O(d^4).
\end{align*}
Using the notation
$$
\text{\rm Osc}_{I}V=\sup\{V(x)-V(x'):x,x'\in I\},
$$
it follows from the identities above that
\begin{equation}\label{appenBeq2}
\text{\rm Osc}_{I_i}(V'_i-V_i)\ge\max\Big\{\frac 14|\Delta a|d,|\Delta b|d^2,\frac 12|\Delta c|d^3\Big\}+O(d^4).
\end{equation}
We construct a grid for the parameters $(a_i,b_i,c_i)$ by splitting the domain for $(a_i,b_i,c_i)$ equally into a family of cuboids and setting the size length by
$$
\Delta a_i=8Md^{5/2},\ \ \Delta b_i=2Md^{3/2},\ \ \Delta c_i=4Md^{1/2}.
$$
These cuboids are denoted by $\text{\uj c}_{ij}$ with $j\in\mathbb{J}_i=\{1,2,\cdots\}$, the cardinality of the set of the subscripts is bounded by
$$
\#(\mathbb{J}_i)=K[d^{-9/2}],
$$
where the integer $K$ is independent of $d$. Let $(a_{ij},b_{ij},c_{ij})$ denote the center of each cuboid and let
$$
V_{ij}(x)=a_{ij}(x-x_i)+b_{ij}(x-x_i)^2+c_{ij}(x-x_i)^3+O(|x-x_i|^4).
$$
Define
$$
\ell_{j,j'}=\max\Big\{\frac{|a_{ij}-a_{ij'}|}{8Md^{5/2}},\frac{|b_{ij}-b_{ij'}|}{2Md^{3/2}}, \frac{|c_{ij}-c_{ij'}|}{4Md^{1/2}}\Big\},
$$
we find from the formula (\ref{appenBeq2}) that following holds for suitably small $d>0$
\begin{equation}\label{appenBeq4}
\text{\rm Osc}_{I_i}(\epsilon V_{ij}-\epsilon V_{ij'})\ge 2\ell_{j,j'}Md^{4}.
\end{equation}

Let us define a subset $\mathbb{J}'_i\subset\mathbb{J}_i$ in the following way. A subscript $j\in\mathbb{J}'_i$ if and only the set
\begin{equation*}
\Lambda_j=\{\lambda\in[\lambda_0,\lambda_1]: \text{\rm Osc}_{I_i}(F_{\lambda}-\epsilon V_{ij})< 2Md^4\}\neq\varnothing.
\end{equation*}
is non-empty. By definition, we have
$$
\text{\rm Osc}_{I_i}(F_{\lambda}-\epsilon V_{ij})\ge 2Md^{4}\qquad \forall \ \lambda\in[\lambda_0,\lambda_1]\ \text{\rm and}\ j\in\mathbb{J}_i\backslash\mathbb{J}'_i.
$$
Using the Lipschitz property $\lambda\to F_{\lambda}$, we claim an estimate on the cardinality of this subset
$$
\#(\mathbb{J}'_i)\le 27K_d[d^{-2-\frac 12}], \ \ \ \text{\rm where}\ \ \ \frac{\log K_d}{|\log d|}\to 0 \ \ \text{\rm as}\ \ \ d\to 0.
$$
In order to prove it, let us replace $F_{\lambda}(x)$ by $F_{\lambda}(x)-F_{\lambda}(x_i)=\int_{x_i}^x\partial_xF_{\lambda}(x)dx$, which is still Lipschitz in $\lambda$, and denote the set of functions by
$$
\mathfrak{F}=\{F_{\lambda}:\lambda\in [\lambda_0,\lambda_1]\}.
$$
It follows from the Lipschitz property that the box dimension of the set $\mathfrak{F}$ equals one in $C^0$-topology. Let $\text{\uj C}_{D\epsilon d}(0)$ denote a cube in $C^0$-function space, centered at the origin with the size equal to $D\epsilon d$, where $D>0$ depends on the upper bound of $\{|a_{ij}|, |b_{ij}|, |c_{ij}|\}$. The set $\Lambda_j$  is non-empty only if $F_{\lambda}\in\text{\uj C}_{D\epsilon d}(0)$ holds for $\lambda\in\Lambda_j$. Since the box dimension of the set $\mathfrak{F}$ equals one, we see that the set
$$
\mathfrak{F}\cap\text{\uj C}_{D\epsilon d}(0),
$$
can be covered by as many as $K_d[\epsilon d^{-3}]$ cubes with the size of $2Md^{4}$,
where the number $K_d$ satisfies the condition that $\log K_d/|\log d|\to 0$ as $d\to 0$.

Let us keep in mind that, by the definition, each $j\in\mathbb{J}'_i$ corresponds to a non-empty set $\Lambda_j\subset [\lambda_0,\lambda_1]$. If the cardinality $\#(\mathbb{J}'_i)>28K_d[\epsilon d^{-3}]$, by Pigeonhole principle, there would be at least 28 different subscripts $j_m\in \mathbb{J}'_i$ such that certain $\lambda_{j_m}\in\Lambda_{j_m}$, and the 28 functions $\{F_{\lambda_{j_m}}:m=1,2,\cdots,28\}$ fall into one small cube with the size of $2Md^{4}$. On the other hand, since the parameter space is three dimensional, in these 28 different subscripts, there must be $m\neq m'$ such that
$$
\ell_{j_m,j_{m'}}=\max_{1\le \ell,\ell'\le 28} \Big\{\frac{|a_{ij_{\ell}}-a_{ij_{\ell'}}|}{8Md^{5/2}},\frac{|b_{ij_{\ell}}-b_{ij_{\ell'}}|}{2Md^{3/2}}, \frac{|c_{ij_{\ell}}-c_{ij_{\ell'}}|}{4Md^{1/2}}\Big\}\ge 4,
$$
it follows from (\ref{appenBeq4}) that
\begin{equation*}
\text{\rm Osc}_{I_i}(\epsilon V_{ij_m}-\epsilon V_{ij_{m'}})\ge 8Md^{4}.
\end{equation*}
On the other hand, as both $F_{\lambda_{j_m}}$ and $F_{\lambda_{j_{m'}}}$ fall into the same cube where
$$
\text{\rm Osc}_{I_i}|F_{\lambda_{j}}-\epsilon V_{ij}|<2Md^4, \qquad \text{\rm for}\ j=j_m,j_{m'}
$$
we find that
$$
\text{\rm Osc}_{I_i}(\epsilon V_{ij_m}-\epsilon V_{ij_{m'}})\le 6Md^4.
$$
This  contradiction proves the claim that $\#(\mathbb{J}'_i)\le 27 K_d[\epsilon d^{-3}]$.

By the definition of the cube, we  have that if $(a_i,b_i,c_i)\in\text{\uj c}_{ij}$ then
\begin{equation*}
\text{\rm Osc}_{I_i}(\epsilon V_{ij}-\epsilon V_{i})\le Md^{4}.
\end{equation*}
It follows from the definition for $\mathbb{J}'$ that for $(a_i,b_i,c_i)\in\text{\uj c}_{ij}$ with $j\in\mathbb{J}_i\backslash \mathbb{J}'_i$
\begin{equation}\label{appenBeq5}
\text{\rm Osc}_{I_i}(F_{\lambda}-\epsilon V_{i})\ge Md^4, \ \ \ \ \forall\ \lambda\in[\lambda_0,\lambda_1].
\end{equation}

The gird for $(a_i,b_i,c_i)$ induces a grid for the parameters $(A_1,B_1,A_2,B_2)$ determined by the equation
\begin{equation}\label{appenBeq6}
\left[\begin{matrix}a_i\\ b_i\\ c_i
\end{matrix}\right]=
\left[\begin{matrix}
-\sin x_i &  \cos x_i  & -2\sin 2x_i &  2\cos 2x_i \\
-\cos x_i & -\sin x_i  & -4\cos 2x_i & -4\sin 2x_i \\
 \sin x_i & -\cos x_i  &  8\sin 2x_i & -8\cos 2x_i
\end{matrix}\right]
\left[\begin{matrix} A_1\\  B_1\\  A_2\\  B_2
\end{matrix}\right]
\end{equation}
the coefficient matrix is non-singular for each $x_i\in\mathbb{T}$. Indeed, let ${\bf M}_1$ be the $3\times 3$ matrix formed by first three columns and let ${\bf M}_2$ be the matrix formed by the first, second and the fourth column, we have
$$
\text{\rm det}({\bf M}_1)(x_i)=6\sin 2x_i, \qquad \text{\rm det}({\bf M}_2)(x_i)=-6\cos 2x_i.
$$
Note that $\inf_{x_i}\{|{\rm det}{\bf M}_1(x_i)|,|{\rm det}{\bf M}_2(x_i)|\}=3\sqrt{2}$, the grid for $(a_i,b_i,c_i)$ induces a grid for $(A_1,B_1,A_2,B_2)$. It contains as many as $K[d^{-9/2}]$ 4-dimensional strips, denoted by $\text{\uj s}_{ij}$ with $j\in\mathbb{J}_i$. Each $\text{\uj s}_{ij}$ is mapped onto $\text{\uj c}_{ij}$ by the equation (\ref{appenBeq6}). For each $(A_1,B_1,A_2,B_2)\in\text{\uj s}_{ij}$ with $j\in\mathbb{J}_i\backslash\mathbb{J}'_i$, the inequality (\ref{appenBeq5}) holds for any $\lambda\in [\lambda_0,\lambda_1]$.

Let us consider all intervals $I_i$ with $i=0,1,\cdots, k-1$. Different $I_i$ induces different gird for the parameters $(A_1,B_1,A_2,B_2)$. In general, $\text{\uj s}_{ij}$ is not parallel to $\text{\uj s}_{i'j'}$ if $i\ne i'$. For each $I_i$, one can define two set of subscripts $\mathbb{J}_i\supset\mathbb{J}'_i$ in the way as above. Thus, one obtains the cardinality of the disjoint union set
$$
\#(\vee_{i=0}^{k-1}\mathbb{J}'_i)\le 27K_d\pi[\epsilon d^{-4}]=27K_d\pi[d^{-7/2}]\ll K[d^{-9/2}].
$$
Since $\log K_d/|\log d|\to 0$ as $d\to 0$, we obtain a Lebesgue measure estimate
$$
\text{\rm meas}\Big(\bigcup_{\stackrel {j\in\mathbb{J}'_i}{\scriptscriptstyle 0\le i\le k-1}}\text{\uj s}_{ij}\Big)\le\frac{27K_d\pi}{K}d\to 0\qquad \text{\rm as}\ d\to 0.
$$
Let $\text{\uj S}^c=I^4\backslash\cup_{j\in\mathbb{J}'_i,\, 0\le i\le k-1}\text{\uj s}_{ij}$, we obtain the Lebesgue measure estimate
$$
\text{\rm meas}(\text{\uj S}^c)\ge 1-\frac{27K_d\pi}{K}d\to 1,\qquad \text{\rm as}\ \ d\to 0.
$$
By the definition $\mathbb{J}_i$ and $\mathbb{J}'_i$, one can see that for any $(A_1,B_1,A_2,B_2)\in\text{\uj S}^c$ and any $\lambda\in [\lambda_0,\lambda_1]$  the variation of $F_{\lambda}-\epsilon V$ is bounded from below
\begin{equation*}
\text{\rm Osc}_{I_i}(F_{\lambda}-\epsilon V)\ge Md^4, \qquad \forall\ 0\le i<k.
\end{equation*}
It implies that the inequality (\ref{appenBeq1}) holds. This completes the proof.
\end{proof}

\subsection{Hyperbolicity of minimal periodic orbits}
Let $L\in C^r(T\mathbb{T}^2,\mathbb{R})$ be a Tonelli Lagrangian with two degrees of freedom ($r\ge 4$), let $H$ be the Hamiltonian determined by $L$. Because of topological property of two torus, each ergodic minimal measure is supported on closed orbits if the rotation vector satisfies certain resonant condition. It is a natural question whether these periodic orbits are hyperbolic. Once it is true, one then obtains normally hyperbolic cylinder composed by these periodic orbits.

Given a Lagrangian $L$ and a rotation direction $g\in H_1(\mathbb{T}^2, \mathbb{Z})$, by Fenchel-Legendre transformation, we obtain a channel in $H^1(\mathbb{T}^2,\mathbb{R})$
$$
\mathbb{C}_g=\bigcup_{\lambda>0}\mathscr{L}_{\beta_L}(\lambda g)\subset H^1(\mathbb{T}^2,\mathbb{R}).
$$
Typically, it is foliated into segments of line (flat of the $\alpha$-function), along which the $\alpha$-function keeps constant, all cohomology classes share the same Mather set. Thus, it makes sense to write $\tilde{\mathcal{M}}(c)=\tilde{\mathcal{M}}(E,g)$ with $E=\alpha(c)$ and $c\in\mathbb{C}_g$.

\begin{theo}\label{AppenHyperTh1}
Given a class $g\in H_1(\mathbb{T}^2, \mathbb{Z})$ and a closed interval $[E_a,E_d]\subset \mathbb{R}_+$ with $E_a>\min\alpha$, there exists an open-dense set $\mathfrak{O}\subset C^{r}(\mathbb{T}^2,\mathbb{R})$ with $r\ge 4$ such that for each $P\in\mathfrak{O}$, each $E\in[E_a,E_d]$, the Mather set $\tilde{\mathcal{M}}(E,g)$ for $L+P$ consists of hyperbolic periodic orbits. Indeed, except for finitely many $E_j\in[E_a,E_d]$ where the Mather set consists of two hyperbolic periodic orbits, for all other $E\in[E_a,E_d]$ it consists of exactly one hyperbolic periodic orbit.
\end{theo}

This theorem will be proved by showing the non-degeneracy of the minimal point of certain action function. Toward this goal, let us split the interval into suitably many subintervals $[E_a,E_d]=\cup_{i=0}^{k}[E_i-\delta_{E_i},E_i+\delta_{E_i}]$ with suitably small $\delta_{E_i}>0$. Once the open-dense property holds for each small subinterval, then it hold for the whole interval.

Let us explain how the interval $[E_a,E_d]$ is split. In the channel, one can choose a path along which the $\alpha$-function monotonely increases. Restricted on this path, we obtain a family of Lagrangians with one parameter. By using the method of \cite{BC}, we can see that it is typical that the minimal measure is supported at most on two periodic orbits for each class on this path. Thus, the Mather set $\tilde{\mathcal{M}}(E,g)$ consists of at most two periodic orbits for each $E\in[E_a,E_d]$.

Without of losing generality, we assume $g=(0,1)$, all of these minimal curves are associated with the homological class. Restricted on the neighborhood $\mathbb{S}_{\gamma_{E_i}}\subset \mathbb{T}^2$ of a minimal curve $\gamma_{E_i}\in\mathcal{M}(E_i,g)$ for certain energy $E_i$, we introduce a configuration coordinate transformation $x=X(u)$ such that along the curve $\gamma_{E_i}$ one has $u_1=\text{\rm constant}$. In the new coordinates, the Lagrangian reads
$$
L'(\dot u,u)=L(DX(u)\dot u,X(u))
$$
which is obviously positive definite in $\dot u$. As $\gamma_E(t)$ is a solution of the Euler-Lagrange equation determined by $L$, the curve $X^{-1}(\gamma_E)(t)$ solves the equation determined by $L'$ and is minimal for the action of $L'$. As there are at most two minimal curves for each energy, the neighborhood of these two curves can be chosen not to overlap each other. Therefore, one can extend the coordinate transformation to the whole torus.

Let $H'$ be the Hamiltonian determined by $L'$ through the Legendre transformation. the minimal curve determines a periodic solution for the Hamiltonian equation. By construction, $\partial_{v_2}H'>0$ holds along the periodic solution which entirely stays in the energy level set $H'^{-1}(E)$. We choose suitably small $\delta_{E_i}>0$ such that for $E\in[E_i-\delta_{E_i},E_i+\delta_{E_i}]$ each minimal periodic curve in $\mathcal{M}(E,g)$ falls into the strip $\mathbb{S}_{\gamma_{E_i}}$ and $\partial_{v_2}H'>0$ holds along each minimal periodic orbit.

For brevity of notation, we still use $x$ to denote the configuration coordinates, use $L$ and $H$ to denote the Lagrangian and Hamiltonian, for which the condition $\partial_{y_2}H>0$ holds along minimal periodic orbits for $E\in[E_i-\delta_{E_i},E_i+\delta_{E_i}]$. Under such conditions, the Lagrangian as well as the Hamiltonian can be reduced to a time-periodic system with one degree of freedom when it is restricted on energy level set. The new Hamiltonian $\bar H(x_1,y_1,\tau,E)$ solves the equation $H(x_1,y_1,x_2,\bar H)=E$ with $\tau=-x_2$, from which one obtains a new Lagrangian $\bar L=\dot x_1y_1-\bar H(x_1,y_1,\tau,E)$ where $y_1=y_1(x_1,\dot x_1,\tau)$ solves the equation $\dot x_1=\partial_{y_1}\bar H(x_1,y_1,\tau)$. In the following we omit the subscript ``1", i.e.  let $(x,y,\dot x)=(x_1,y_1,\dot x_1)$ if no danger of confusion occurs.

We introduce a function of Lagrange action $F(\cdot,E)$: $\mathbb{T}\to\mathbb{R}$:
$$
F(x,E)=\inf_{\gamma(0)=\gamma(2\pi)=x} \int_{0}^{2\pi}\bar L(d\gamma(\tau),\tau,E)d\tau.
$$
A periodic curve $\gamma$ is called the minimizer of $F$ if the Lagrange action along this curve reaches the quantity $F(\gamma(0),E)$. As there might be two or more minimizers if $x$ is not a minimal point, the function $F$ may not be smooth in global. However, we claim that it is smooth in certain neighborhood of minimal point.

To verify our claim, we let $T_i=\frac {2\pi i}m$ and define the function of action $F_i(x,x',E)$
$$
F_i(x,x',E)=\inf_{\stackrel {\gamma(T_i)=x}{\scriptscriptstyle \gamma(T_{i+1})=x'}} \int_{T_i}^{T_{i+1}}\bar L(d\gamma(\tau),\tau)d\tau.
$$
There will be two or more minimizers of $F_i(x,x',E)$ if the point $x$ is in the ``cut locus" of the point $x'$. However, the minimizer is unique if $x$ is suitably close to $x'$, denoted by $\gamma_i(\cdot,x,x',E)$. In this case, it uniquely determines a speed $v=v(x,x')$ such that $\dot\gamma_i(T_i,x,x',E)=v(x,x')$. Let $\vec{x}= (x_0,x_1,\cdots,x_{m})$ denote a periodic configuration ($x_0=x_m$), we introduce a function of action
$$
\text{\bf F}(\vec{x},E)=\sum_{i=0}^{m-1}F_i(x_i,x_{i+1},E).
$$
As $T_{i+1}-T_i$ is suitably small and the Lagrangian is positive definite in the speed, the boundary condition $\gamma(T_j)=x_j$  for $j=i,i+1$ uniquely determines the speed $v_j=\dot\gamma(T_j)$ for $j=i, i+1$.  Indeed, the function $F_i$ generates an area-preserving twist map from the time-$T_i$-section to the time-$T_{i+1}$-section $\Phi_i$: $(x_i,y_i)\to(x_{i+1},y_{i+1})$
$$
y_{i+1}=\partial_{x_{i+1}}F_i(x_{i+1},x_i), \qquad y_{i}=-\partial_{x_{i}}F_i(x_{i+1},x_i).
$$
where $y_i=\partial_{\dot x}L(x_i,v_i,T_i)$. As the Lagrangian is positive definite in $\dot x$, it implies that the initial condition $(x_i,v_i)$ smoothly depends on the boundary condition $(x_i,x_{i+1})$ in this case. Because of the smooth dependance of solution of ordinary differential equation on initial condition, the function is smooth. Obviously, each minimal point of $\text{\bf F}(\cdot,E)$ uniquely determines a $c$-minimal measure for $c\in\alpha^{-1}(E)\cap \mathbb{C}_g$, supported on a periodic orbit $(\gamma_E,\dot\gamma_E)$ with $[\gamma_E]=g$. Let $x_i=\gamma_E(T_i)$, it satisfies the discrete Euler-Lagrange equation
$$
\frac{\partial F_i}{\partial x'}(x_{i-1},x_i,E)+\frac{\partial F_{i+1}}{\partial x}(x_{i},x_{i+1},E)=0.
$$

We shall show later that the periodic orbit is hyperbolic if and only if the minimal configuration is non-degenerate, namely, the Jacobi matrix is positive definite:
$$
\text{\bf J}=\left[\begin{matrix}
A_0 & B_0 & 0& \cdots & B_{m-1}\\
B_0 & A_1 & B_1 & \cdots & 0 \\
0 & B_1 & A_2 & \cdots & 0 \\
\vdots & \vdots & \vdots & \ddots & B_{m-2}\\
B_{m-1} & 0 & 0 & B_{m-2} & A_{m-1}
\end{matrix}\right]
$$
where
$$
A_i=\frac{\partial^2F_{i-1}}{\partial x'^2}(x_{i-1},x_i)+\frac{\partial^2F_{i}}{\partial x^2}(x_i,x_{i+1}),\ \  B_i=\frac{\partial^2F_{i}}{\partial x\partial x'}(x_{i},x_{i+1})
$$
and $x_{-1}=x_{m-1}$.

Let $\vec{x}=(x_0,x_1,\cdots,x_m=x_0)$ be a minimal configuration of the function $F(\vec{x},E)$, where the Jacobi matrix is non-negative and the smallest eigenvalue is simple. Indeed, as the Lagrangian is positive definite, the generating function $F_i(x,x',E)$ determines an area-preserving and twist map $\Phi_i$, we have $B_i<0$. Consequently, by using a theorem in \cite{vM}, we find that the smallest eigenvalue is simple. Let $\lambda_i$ denote the $i$-th eigenvalue of the matrix, at the minimal configuration one has
$$
0\le\lambda_0<\lambda_1\le\lambda_2<\cdots\le\lambda_{m-1}.
$$
Let $\xi_i=(\xi_{i,0},\xi_{i,1},\cdots,\xi_{i,m-1})$ be the eigenvector for $\lambda_i$. By choosing $\xi_{0,0}=1$ we have $\xi_{0,i}>0$ for $1\le i<m$ (see Lemma 3.4 in \cite{An}). At the minimal configuration, we find the following matrix is positive definite:
$$
\text{\bf J}_{m-1}=\left[\begin{matrix}
A_1 & B_1 & \cdots & 0 \\
B_1 & A_2 & \cdots & 0 \\
\vdots & \vdots & \ddots & B_{m-2}\\
0 & 0 & B_{m-2} & A_{m-1}
\end{matrix}\right].
$$
If not, there will be a vector $\hat v=(v_1,\cdots,v_{m-1})\in\mathbb{R}^{m-1}\backslash\{0\}$ such that $\hat v^t\text{\bf J}_m\hat v=0$. It follows that $v^t\text{\bf J}v=0$ if we set $v=(0,\hat v)\in\mathbb{R}^m$. As the matrix $\text{\bf J}$ is non-negative, it implies that $v=\mu\xi_0$, but it contradicts the fact that all entries of $\xi_{0}$ have the same sign, either positive or negative.

In a suitably small neighborhood $\vec{U}=U_0\times U_1\times\cdots\times U_{m-1}$ of the minimal configuration, let us consider the equations
\begin{equation}\label{variationeq1}
\frac{\partial\text{\bf F}}{\partial x_i}(x_0,x_1,\cdots,x_{m-1},E)=0,\qquad \forall\ i=1,2,\cdots,m-1.
\end{equation}
Since the matrix $\{\frac{\partial^2\text{\bf F}}{\partial x_i\partial x_j}\}_{i,j=1}^{m-1}=\text{\bf J}_{m-1}$ is positive definite at the minimal point, by the implicit function theorem, this equation has a unique smooth solution $x_i=X_i(x_0,E)$ when $x_0\in U_0$. Let $\gamma$: $[0,2\pi]\to\mathbb{R}$ be a minimizer of $F(x_0)$ with $\gamma(0)=\gamma(2\pi)=x_0$, we obtain a configuration $x_i=\gamma(2i\pi/m)$. Clearly, $\partial_{x_i}\text{\bf F}=0$ holds at this configuration for each $i\ge 1$. It implies the uniqueness of the minimizer of $F(x_0,E)$ for $x_0\in U_0$. The minimal point of $F$ uniquely determines a minimal configuration of $\text{\bf F}$, therefore, the function $F$ is smooth in certain neighborhood of its minimal point.

\noindent{\bf Non-degeneracy of minimizers}

In a neighborhood of the minimal point, let us study what change the function of action undergoes when the Lagrangian is under a perturbation of potential $L\to L+P$, where $P$: $\mathbb{T}^2\to\mathbb{R}$ is a potential. Let $\bar L'$ denote the reduced Lagrangian of $L+P$ and let $G=-(\partial_{y_2}H)^{-1}$, one has
$$
\bar L'=\bar L+GP+O(\|P\|^2).
$$
We denote the minimal curve of $F(x,E)$ by $\gamma(t,x,E)$ such that $\gamma(0,x,E)=x$. Let $\gamma'(t,x,E)$ and $F'(x,E)$ be the quantities defined for $\bar L'$ as the quantities $\gamma(t,x,E)$ and $F(x,E)$ defined for $\bar L$. By the definition of minimizer, we have
\begin{align*}
F'(x,E)-F(x,E)&=\int_0^{2\pi}\bar L'(d\gamma'(\tau))d\tau -\int_0^{2\pi}\bar L(d\gamma(\tau))d\tau\\
&\ge \int_0^{2\pi}G(d\gamma'(\tau),\tau)P(\gamma'(\tau))d\tau+o(\|\gamma'-\gamma\|,\|P\|),
\end{align*}
and
\begin{align*}
F'(x,E)-F(x,E)&=\int_0^{2\pi}\bar L'(d\gamma'(\tau))d\tau-\int_0^{2\pi}\bar L(d\gamma(\tau))d\tau\\
&\le \int_0^{2\pi}G(d\gamma(\tau),\tau)P(\gamma(\tau))d\tau+o(\|\gamma'-\gamma\|,\|P\|).
\end{align*}
Since the distance between these two curves $\gamma$ and $\gamma'$ is due to the perturbation $P$, we finally obtain
\begin{equation}\label{nondegenerate1}
F'(x,E)=F(x,E)+\mathscr{K}_EP(x)+\mathscr{R}_EP(x)
\end{equation}
where $\mathscr{R}_EP=o(\|P\|)$ and
$$
\mathscr{K}_EP(x)=\int_0^{2\pi}G(d\gamma(\tau,x,E),\tau)P(\gamma(\tau,x,E))d\tau.
$$
The field of smooth curves $\{\gamma(\cdot,x,E)\}$ defines an operator $P\to\mathscr{K}_EP$, which maps functions defined on $\mathbb{T}^2$ into the function space defined on $\mathbb{T}$. Obviously, both $\mathscr{K}_EP$ and $\mathscr{R}_EP$ are smooth in $x\in U_0$ and in $E$.

Unless the point $x$ is a minimizer of $F(\cdot,E)$, the curve $\gamma(\cdot,x,E)$ may have corner at $\tau=0\mod 2\pi$.
\begin{lem}\label{hyperlem1}
There exist constants $\varepsilon,\theta>0$ such that if $F(x,E)-\min F(\cdot,E)<\varepsilon$ and if $\gamma:$ $[0,2\pi]\to\mathbb{R}$ is a  minimizer of $F(x,E)$, then
\begin{equation*}\label{hyperbolicorbit2}
\|\dot\gamma(0)-\dot\gamma(2\pi)\|<\theta\sqrt{F(x,E)-\min F(\cdot,E)}.
\end{equation*}
\end{lem}
\begin{proof}
Let us consider the derivative of $F(\cdot,E)$. As the Lagrangian is positive definite, some positive constants $m_L>0$ exist such that
$$
\frac{\partial ^2\bar L}{\partial\dot x^2}\ge m_L, \qquad \forall\ (x,\dot x)\in T\mathbb{T}^2.
$$
Since $\gamma(0,x,E)=\gamma(2\pi,x,E)=x$, one has $\partial_{x}\gamma(0)=\partial_{x}\gamma(2\pi)=1$ and
\begin{align*}
\Big|\frac{\partial F}{\partial x}\Big|&=\Big|\int_{0}^{2\pi}\Big (\frac{\partial\bar L}{\partial\dot x} (d\gamma(\tau),\tau)\frac{\partial\dot\gamma}{\partial x}+\frac{\partial\bar L}{\partial x} (d\gamma(\tau),\tau)\frac{\partial\gamma}{\partial x}\Big )d\tau\Big|\\
&=\Big|\frac{\partial\bar L}{\partial\dot x}(\dot\gamma(0),\gamma(0),0)-\frac{\partial\bar L}{\partial\dot x}(\dot\gamma(2\pi),\gamma(2\pi),2\pi)\Big|\\
&\ge m_L|\dot\gamma(0)-\dot\gamma(2\pi)|,
\end{align*}
where the second equality follows from that $\gamma$ solves the Euler-Lagrange equation. If $\frac{\partial F}{\partial x}>0$ and if the lemma does not hold, by choosing $x'-x=-\sqrt{\Delta}$ ($\Delta=F(x,E)-\min F(\cdot,E)$) we obtain from the Taylor series up to second order that
\begin{align*}
F(x',E)-\min F(\cdot,E)&=F(x',E)-F(x,E)+F(x,E)-\min F(\cdot,E)\\
&\le -\partial_{x}F(x,E)\sqrt{\Delta}+\frac M2\Delta+\Delta<0
\end{align*}
if $\theta>\frac 1{m_L}(1+\frac M2)$, where $M=\max\partial^2_{x}F$. But it is absurd. The case $\frac{\partial F}{\partial x}<0$ can be proved by choosing $x'-x=\sqrt{\Delta}$. This completes the proof.
\end{proof}

Let $x\in(x^*-\delta_{x^*},x^*+\delta_{x^*})$, where $x^*$ is the minimal point of $F(\cdot,E_0)$. As it was shown above, $\gamma(\frac{2i\pi}m,x,E_0)$ smoothly depends on $x$, we have a smooth foliation of curves in a neighborhood of the curve $\gamma(\cdot,x^*,E_0)$. The corner at $\gamma(0,x,E_0)$, i.e. $\dot\gamma(2\pi,x,E_0)-\dot\gamma(0,x,E_0)$ approaches to zero as $F(x,E_0)\downarrow\min F(\cdot,E_0)$. For each $x$, if there is a corner at $\gamma(0,x,E_0)=\gamma(2\pi,x,E_0)$, we construct a curve $\gamma_{x}$ that smoothly connects the point $\gamma(2\pi-\delta,x,E_0)$ to the point $\gamma(\delta,x,E_0)$ with $\gamma_x(0)=x$, where $\delta>0$ is suitably small. Replacing the segment $\gamma(\cdot,x,E_0)|_{[0,\delta]\cup[2\pi-\delta,2\pi]}$ by this curve, we obtain a smooth curve $\gamma_{x}$ such that $\gamma_x(t)=\gamma(t,x,E_0)|_{[\delta,2\pi-\delta]}$ and $\gamma_x(0)=x$. Indeed, as the curve $\gamma(t,x,E_0)$ is $C^3$-smooth in $x$ and $\gamma(t,x^*,E_0)$ is also $C^3$-smooth in $t$, for small number $\varepsilon$ some $\mu_0>0$ exists such that the quantities
$$
\Big|\frac{d^k\gamma}{dt^k}(t,x,E_0)-\frac{d^k\gamma}{dt^k}(t,x^*,E_0))\Big|< \varepsilon,\qquad \forall\ x\in[x^*-\delta_{x^*},x^*+\delta_{x^*}], \ k=0,1,2,3.
$$
Let the curve $\zeta_x(\cdot)$: $[-\delta,\delta]\to\mathbb{R}$ be an interpolation polynomial of degree eight such that
$$
\frac{d^k\zeta_x}{dt^k}(t)=\frac{d^k\gamma}{dt^k}(t,x,E_0)-\frac{d^k\gamma}{dt^k}(t,x^*,E_0) \qquad \forall\ t=\pm\delta,
$$
and $\zeta_x(0)=\gamma(0,x,E_0)-\gamma(0,x^*,E_0)$, then the coefficients of the polynomial are smooth in $x$. Let $\gamma_{x}(t)=\gamma(t,x^*,E_0)+\zeta_x(t)$, we see that the foliation of the curves $\gamma_{x}$ is smooth in $x$ and as a function of $t$, $\gamma_{x}-\gamma(\cdot,x,E_0)$ is small in $C^3$-topology.

For each point $(\tau,x)\in\mathbb{S}$, there is a curve $\gamma_{x_0}$ such that $x=\gamma_{x_0}(\tau)$. It uniquely determines a speed $v=v(x,\tau)=\dot\gamma_{x_0}(\tau)$. By the construction, $v(x,\tau)$ is $C^3$-smooth in $(x,\tau)$. As $-G^{-1}=\partial_{y_2}H>0$ when it is restricted to a neighborhood of the minimal curves, both $v$ and $G$ can be approximated by $C^r$-function $v_s$ and $G_s$ in $C^3$-topology respectively i.e. $\|v-v_s\|_{C^3}<\varepsilon$ and $\|G-G_s\|_{C^3}<\varepsilon$ hold for small $\varepsilon>0$. Given a $C^r$-function $\bar P$: $\mathbb{T}\to\mathbb{R}$ we obtain a $C^r$-function $P=\mathscr{T}_{E_0}\bar P$: $\mathbb{T}^2\to\mathbb{R}$ defined by
\begin{equation}\label{nondegenerate2}
P(x,\tau)=\mathscr{T}_{E_0}\bar P(x_0)=G^{-1}_s(v_s(x,\tau),x,\tau)\bar P(x_0),
\end{equation}
if $x=\gamma_{x_0}(\tau)$. By the definition, we have
\begin{equation}\label{nondegenerate3}
\mathscr{K}_E\mathscr{T}_{E_0}\bar P(x)= \int_0^{2\pi}\frac{G(d\gamma(\tau,x,E),\tau)}{G_s(v_s(\gamma (\tau,x,E),\tau),\gamma (\tau,x,E),\tau)}\bar P(x+\Delta\gamma (\tau,x,E))d\tau
\end{equation}
where $\Delta\gamma(\tau,x,E)$ is defined as follows: passing through the point $\gamma (\tau,x,E)$ there is a unique $x'$ such that $\gamma_{x'}(\tau)=\gamma(\tau,x,E)$. We set $\Delta\gamma(\tau,x,E)=x'-x$.

We introduce a set of perturbations with four parameters:
$$
\bar{\mathfrak{P}}=\Big\{\sum_{\ell=1}^2(A_{\ell}\cos \ell x+B_{\ell}\sin \ell x):\ (A_1,B_1,A_2,B_2)\in\mathbb{I}^4\Big\},
$$
where $\mathbb{I}^4=[1,2]\times[1,2]\times[1,2]\times[1,2]$. By applying the formula (\ref{nondegenerate3}) to the function $\cos\ell x$ and $\sin\ell x$  we find that
\begin{align}\label{nondegenerate4}
\mathscr{K}_E\mathscr{T}_{E_0}\cos\ell x&=u_{\ell}(x,E)\cos\ell x-v_{\ell}(x,E)\sin\ell x,\\
\mathscr{K}_E\mathscr{T}_{E_0}\sin\ell x&=u_{\ell}(x,E)\sin\ell x+v_{\ell}(x,E)\cos\ell x,\notag
\end{align}
where
\begin{align*}
u_{\ell}(x,E)&=\int_0^{2\pi}\frac{G(d\gamma(\tau,x,E),\tau)}{G_s(v_s(\gamma (\tau,x,E),\tau),\gamma (\tau,x,E),\tau)}\cos \ell\Delta\gamma(\tau,x,E)d\tau,\\
v_{\ell}(x,E)&=\int_0^{2\pi}\frac{G(d\gamma(\tau,x,E),\tau)}{G_s(v_s(\gamma (\tau,x,E),\tau),\gamma (\tau,x,E),\tau)}\sin \ell\Delta\gamma(\tau,x,E)d\tau.
\end{align*}
Let us study the dependence of the terms $u_{\ell}(x,E)$ and $v_{\ell}(x,E)$ on the point $x$. We claim that there exists constant $\theta_1>0$ as well as small numbers $\delta_{E_0}>0$ and $\delta_{x^*}>0$ such that for each $E\in (E_0-\delta_{E_0},E_0+\delta_{E_0})$, each $x\in(x^*-\delta_{x^*},x^*+\delta_{x^*})$, $j=0,1,2,3$ and $\ell=1,2$, we have
$$
|u_{\ell}(x,E)|\ge 1-\theta_1\delta, \qquad |v_{\ell}(x,E)|\le\theta_1\delta,
$$
\begin{equation}\label{nondegenerate5}
\max_{j=1,2,3}\Big\{\Big|\frac {\partial^ju_{\ell}}{\partial x^j}(x,E)\Big|,\Big|\frac {\partial^j v_{\ell}}{\partial x^j}(x,E)\Big|\Big\} \le\theta_1\delta.
\end{equation}

By the construction of the curves $\gamma_x$, for $\tau\in\mathbb{T}\backslash (-\delta,\delta)$ and for $x\in(x^*-\delta_{x^*},x^*+\delta_{x^*})$ we have
$$
\Delta\gamma(\tau,x,E_0)=0 \ \ \text{\rm and}\ \ \frac{G(d\gamma(\tau,x,E_0),\tau)}{G(v(\gamma (\tau,x,E_0),\tau),\gamma (\tau,x,E_0),\tau)}=1
$$
and $\partial_{x}^j\Delta\gamma (\tau,x,E_0)$ is small for $\tau\in(-\delta,\delta)$ and for $j=0,1,2,3$. Integrating the them over the set with Lebesgue measure $2\delta$, we find that some small $\theta_1>0$ exists such that the formulae in (\ref{nondegenerate5}) hold for $E=E_0$ with $\theta_1$ being replaced by $\theta_1/4$ if $G_s$ and $v_s$ in the formula (\ref{nondegenerate3}) are replaced by $G$ and $v$ respectively. As both $v$ and $G$ are approximated by $v_s$ and $G_s$ in $C^3$-topology, by choosing $\varepsilon>0$ suitably small, all formulae in (\ref{nondegenerate5}) hold for $E=E_0$ with $\theta_1$ being replaced by $\theta_1/2$.

For other energy $E$, let us recall the solution $x_i=X_i(x_0,E)$ of Eq. (\ref{variationeq1}) is smooth. As the map $\Phi_i$: $(x_i,y_i)\to(x_{i+1},y_{i+1})$ is area-preserving and twist, it uniquely determines the initial speed $v_0=v_0(x_0,E)$, namely, the initial speed smoothly depends on the initial position and such dependence is also smooth in the parameter $E$. As solution of ODE smoothly depends on its initial conditions, the minimal curve $\gamma(\cdot,x,E)$ of $F(x,E)$ smoothly depends on the parameters $x$ and $E$. Thus, the formulae in (\ref{nondegenerate5}) hold if the numbers $\delta_{E_0}>0$ and $\delta_{x^*}>0$ are suitably small.

\begin{theo}\label{theo2}
There exists an open-dense set $\mathfrak{O}\subset C^r(M,\mathbb{R})$ with $r\ge4$  such that for each $P\in\mathfrak{O}$ and each $E\in [E_0-\delta_{E_0},E_0+\delta_{E_0}]$, all minimizers of $F(\cdot,E)$, determined by $L+P$, are non-degenerate.
\end{theo}
\begin{proof}
To show the non-degeneracy of the global minimum of $F(\cdot,E)$ located at the point $x$, we only need to verify that
\begin{equation}\label{nondegenerate6}
F(x+\Delta x,E)-F(x,E)\ge M|\Delta x|^4
\end{equation}
holds for small $|\Delta x|$, where $M=12^{-1}\max\partial^4_{x}F$. Assume $I$ is an interval, we define $\text{\rm Osc}_{I}F=\max_{x,x'}|F(x)-F(x')|$. To show the non-degeneracy, it is sufficient to verify that
$$
\text{\rm Osc}_{I}F(\cdot,E)\ge M|I|^4
$$
if the minimal point $x\in I$, where $|I|$ denotes the length of the interval.

The openness is obvious of $\mathfrak{P}$. To show the density, we are concerned only about the configurations where $F$ takes the value close to the minimum and consider small perturbations from the following set where the parameters $(A_1,B_1,A_2,B_2)$ range over the cube $\mathbb{I}^4=[1,2]\times[1,2]\times[1,2]\times[1,2]$
$$
\mathfrak{V}_E=\Big\{(\mathscr{K}_E+\mathscr{R}_E)\mathscr{T}_{E_0}\sum_{\ell=1}^2\epsilon (A_{\ell}\cos\ell x +B_{\ell}\sin\ell x):(A_1,B_1,A_2,B_2)\in\mathbb{I}^4\Big\}
$$
where each element is a function of $(x,E)$, see the formulae (\ref{nondegenerate4}). Recall that both  operators $\mathscr{K}_E$ and $\mathscr{T}_{E_0}$ are linear and  $\|\mathscr{R}_E(\epsilon P)\|=o(\epsilon)$, see the formula (\ref{nondegenerate1}) and the formula (\ref{nondegenerate2}).

We choose sufficiently large integer $K$ so that $\epsilon=\sqrt[4]{\pi/K}$ can be arbitrarily small. Let $x_k=\frac{2k\pi}K$, $I_k=[x_k-d,x_k+d]$ and $d=\pi/K$, then $\bigcup_{k=0}^{K-1}I_k=\mathbb{T}$. Restricted on each interval $I_k$, each $C^4$-function $V\in\mathfrak{V}_E$ is approximated by the Taylor series (module constant)
$$
V_k(x)=\epsilon\Big(a_k(x-x_{k})+b_k(x-x_{k})^2+c_k(x-x_{k})^3+O(|x-x_{k}|^4)\Big).
$$

Given two points $(a_k,b_k,c_k)$ and $(a'_k,b'_k,c'_k)$, we have two functions $V_k$ and $V'_k$ in the form of Taylor series. Let $\Delta V=V'_k-V_k$, $\Delta a=a'_k-a_k$, $\Delta b=b'_k-b_k$ and $\Delta c=c'_k-c_k$, we have $\Delta V(x_k)=0$ and
\begin{align*}
&\Delta V(x_k+d)+\Delta V(x_k-d)=2\epsilon\Delta bd^2+O(\epsilon d^4),\\
&\Delta V(x_k+d)-\Delta V(x_k-d)=2\epsilon(\Delta a+\Delta cd^2)d+O(\epsilon d^4),\\
&\Delta V\Big(x_k\pm\frac 12d\Big)=\epsilon\Big(\pm\frac 12\Delta a+\frac 14\Delta bd\pm\frac 18\Delta cd^2\Big)d+O(\epsilon d^4).
\end{align*}
It follows that
\begin{equation}\label{nondegenerate7}
\text{\rm Osc}_{I_k}(V'_k-V_k)\ge\epsilon\max\Big\{\frac 13|\Delta a|d,|\Delta b|d^2,\frac 12|\Delta c|d^3\Big\}.
\end{equation}

We construct a grid for the parameters $(a_k,b_k,c_k)$ by splitting the domain for them equally into a family of cuboids and setting the size length by
$$
\Delta a_k=9Md^{\frac{11}4},\ \ \Delta b_k=3Md^{\frac{7}4},\ \ \Delta c_k=6Md^{\frac 34}.
$$
These cuboids are denoted by $\text{\uj c}_{kj}$ with $j\in\mathbb{J}_k=\{1,2,\cdots\}$, the cardinality of the set of the subscripts is up to the order
$$
\#(\mathbb{J}_k)=N[d^{-\frac{21}4}],
$$
where the integer $0<N\in\mathbb{N}$ is independent of $d$. If $\text{\rm Osc}_{I_k}F(\cdot,E)\le Md^4$, we obtain from the formula (\ref{nondegenerate7}) that
$$
\text{\rm Osc}_{I_k}(F(x,E)+V(x))\ge 2Md^4
$$
if $V(x)=\epsilon(a(x-x_k)+ b(x-x_k)^2+c(x-x_k)^3+O(|x-x_k|^4))$ with
$$
\max\Big\{\frac 13|a|d^{-\frac{11}4},|b|d^{-\frac 74}, \frac 12|c|d^{-\frac 34}\Big\}\ge 3M.
$$

The coefficients $(a_k,b_k,c_k)$ depend on the parameters $(A_1,B_1,A_2,B_2)$, the energy $E$ and the position $x_k$. The gird for $(a_k,b_k,c_k)$ induces a grid for the parameters $(A_1,B_1,A_2,B_2)$, determined by the equation
\begin{equation}\label{nondegenerate8}
\left[\begin{matrix} a_k\\  b_k\\  c_k
\end{matrix}\right]=({\bf C_1}{\bf U}+{\bf C_2})
\left[\begin{matrix} A_1\\  B_1\\  A_2\\  B_2
\end{matrix}\right]\Big(1+T_{\epsilon,E,x_{k}}(A_1,B_1,A_2,B_2)\Big)
\end{equation}
where the map $T_{\epsilon,E,x_{k}}$: $\mathbb{R}^4\to\mathbb{R}^3$ is as small as of order $O(\epsilon)$,
$$
{\bf C_1}=\left[\begin{matrix}
-\sin x_{k} &  \cos x_{k}  & -2\sin 2x_{k} &  2\cos 2x_{k}\\
-\cos x_{k} & -\sin x_{k}  & -4\cos 2x_{k} & -4\sin 2x_{k} \\
 \sin x_{k} & -\cos x_{k}  &  8\sin 2x_{k} & -8\cos 2x_{k}
\end{matrix}\right],
$$
$$
{\bf U}=\text{\rm diag}\left\{
\left[\begin{matrix}
u_1(x_{k}) & v_1(x_{k})\\
-v_1(x_{k}) & u_1(x_{k})
\end{matrix}\right],
\left[\begin{matrix}
 u_2(x_{k}) & v_2(x_{k})\\
-v_2(x_{k}) & u_2(x_{k})
\end{matrix}\right]\right\},
$$
each entry of ${\bf C_2}$ is a linear function of $\partial^{j}_{x}u_{\ell}\cos \ell x_{k}$, $\partial^{j}_{x}v_{\ell}\cos\ell x_{k}$, $\partial^{j}_{x}u_{\ell}\sin\ell x_{k}$ and $\partial^{j}_{x}v_{\ell}\sin\ell x_{k}$ with $j=1,2,3$, $\ell=1,2$. Both matrices ${\bf U}$ and ${\bf C}_2$ depend on the energy $E$, ${\bf U}$ is close to the identity matrix.  Let ${\bf M_1}$ be the matrix composed by the first three columns of ${\bf C_1}{\bf U}+{\bf C_2}$, ${\bf M_2}$ be the matrix composed by the first, the second and the fourth column of ${\bf C_1}{\bf U}+{\bf C_2}$. As we are only concerned about those positions where $F$ takes value close to the minimum and about the energy $E$ close to $E_0$, in virtue of (\ref{nondegenerate5}) we obtain
\begin{align*}
\text{\rm det}({\bf M}_1)(x_k)&=6\sin 2x_{k}(1-O(\theta_1\delta)), \\
\text{\rm det}({\bf M}_2)(x_k)&=-6\cos 2x_{k}(1-O(\theta_1\delta)).
\end{align*}
Since $\inf_{x_{k}}\{|{\rm det}{\bf M}_1(x_k)|,|{\rm det}{\bf M}_2(x_k)|\}= 3\sqrt{2}(1-O(\theta_1\delta))$, the grid for $(a_k,b_k,c_k)$ induces a grid for $(A_1,B_1,A_2,B_2)$ which contains as many as $N_1[d^{-\frac{21}4}]$ 4-dimensional strips ($N_1>0$ is independent of $d$). Note that the induced partition for the parameters $(A_1,B_1,A_2,B_2)$ depends on the energy $E$.

Given an energy $E\in[E_0-\delta_{E_0},E_0+\delta_{E_0}]$, if there exist Taylor coefficients $(a_k,b_k,c_k)$ which determines a perturbation $V$ such that
$$
\text{\rm Osc}_{I_k}(F(\cdot,E)+V)\le Md^{4}
$$
then for $(a'_k,b'_k,c'_k)$ which determines a perturbation $\Delta V'$ and satisfies the condition
$$
\max\Big\{\frac{|a_k-a'_k|}{9Md^{\frac{11}4}},\frac{|b_k-b'_k|}{3Md^{\frac 74}}, \frac{|c_k-c'_k|}{6Md^{\frac 34}}\Big\}\ge 1
$$
one obtains from the formula (\ref{nondegenerate7}) that
\begin{equation}\label{nondegenerate9}
\text{\rm Osc}_{I_k}(F(\cdot,E)+V')\ge 2Md^{4}.
\end{equation}

Under the map defined by the formula (\ref{nondegenerate8}), the inverse image of a cuboid $\text{\uj c}_{k}$ with the size $18Md^{\frac{11}4}\times6Md^{\frac{7}4}\times12Md^{\frac{3}4}$ is a strip in the parameter space of $(A_1,B_1,A_2,B_2)$, denoted by $\text{\uj S}_{k}(E)$, with the Lebesgue measure as small as $N_1^{-1}d^{\frac{21}4}$. If the cuboid $\text{\uj c}_{k}$ is centered at $(a_k,b_k,c_k)$, then for $(a'_k,b'_k,c'_k)\notin\text{\uj c}_{k}$ the inequality (\ref{nondegenerate9}) holds.

Splitting the interval $[E_0-\delta_{E_0},E_0+\delta_{E_0}]$ equally into small sub-intervals $I_{E,j}$ with the size $|I_{E,j}|=M_1^{-1}d^4$, we obtain as many as $[M_1d^{-4}]$ small intervals. Since the function $F$ is Lipschitz in $E$, suitably large positive number $M_1$ can be chosen so that
$$
\max_{x\in I_k}|F(x,E)-F(x,E')|<\frac 12Md^4, \qquad \forall\ E,E'\in I_{E,j}.
$$
Therefore, for $V\in\mathfrak{V}_E$ with $(\Delta A_1,\Delta B_1,\Delta A_2,\Delta B_2)\notin\text{\uj S}_{k}(E)$, one has
\begin{equation}\label{nondegenerate10}
\text{\rm Osc}_{I_k}(F(\cdot,E)+\Delta V')\ge Md^{4}.
\end{equation}
Pick up one energy $E_{j}$ in each small interval $I_{E,j}$, there are $[M_1d^{-4}]$ strips $\text{\uj S}_{k}(E_j)$. Finally, by considering all small intervals $I_k$ with $k=0,1,\cdots K-1$, we find
$$
\text{\rm meas}\Big(\bigcup_{k,j}\text{\uj S}_{k}(E_j)\Big)\le M_1N_1^{-1}\sqrt[4]{d}.
$$
Let $\text{\uj S}^c=\mathbb{I}^4\backslash\cup_{j,k}\text{\uj S}_{kj}(E_j)$, we obtain the Lebesgue measure estimate
$$
\text{\rm meas}(\text{\uj S}^c)\ge 1-M_1N_1^{-1}\sqrt[4]{d}\to 1,\qquad \text{\rm as}\ \ d\to 0.
$$
Obviously, for each $(A_1,B_1,A_2,B_2)\in\text{\uj S}^c$, each $E\in [E_0-\delta_{E_0},E_0+\delta_{E_0}]$ and each $k=1,2,\cdots,K$ the formula (\ref{nondegenerate10}) holds. This proves that it is open-dense that all minimal points of $F(\cdot,E)$ are non-degenerate when the energy ranges over the interval $[E_0-\delta_{E_0},E_0+\delta_{E_0}]$.
\end{proof}

\noindent{\bf Hyperbolicity}

Let $x^*$ be a minimal point of the function $F(\cdot,E)$ and let the curve $\gamma(\cdot,x^*,E)$: $\mathbb{T}\to\mathbb{R}$ be the minimal curve of $F(x^*,E)$ which is smooth and determines a periodic orbit $(\tau,\gamma(\tau),\frac d{d\tau}\gamma(\tau))$ of the Lagrange flow $\phi^{\tau}_{\bar L}$. Back to the autonomous system, it determines a periodic orbit $(\gamma_1(t),\dot\gamma_1(t),\gamma_2(t),\dot\gamma_2(t))$ of the Lagrange flow $\phi_L^t$, where $\gamma_2(t)=-\tau$, $\gamma_1(t)=\gamma(\gamma_2(t))$.
\begin{theo}\label{theo3}
If $x^*$ is a non-degenerate minimal point of the function $F(\cdot,E)$, then the periodic orbit $\gamma(\cdot,x^*,E)$ is hyperbolic.
\end{theo}
\begin{proof}
If a periodic orbit is hyperbolic, it has its stable and unstable manifold in the phase space. Consequently, any orbit staying on the stable (unstable) manifold approaches to the periodic orbit exponentially fast as the time approaches to positive (negative) infinity.

In a neighborhood of the minimal periodic curve $\gamma$, each point $x$ on the section $\{\tau=0\}$ determines at least one forward (backward) semi-static curve $\gamma^+_x$: $\mathbb{R}_+\to\mathbb{T}$ ($\gamma^-_x$: $\mathbb{R}_-\to\mathbb{T}$) such that $\gamma^{\pm}_x(0)=x$. These curves determine forward (backward) semi-static orbits $d\gamma^{\pm}_x$ of which the $\omega$-set ($\alpha$-set) is the periodic orbit $d\gamma$. In the configuration space $(x,\tau)\in\mathbb{T}^2$, these two curves intersect with the section $\{\tau=0\}$ infinitely many times at the points $\gamma^+_x(2k\pi)$ and $\gamma^-_x(-2k\pi)$. These points are denoted by $x_i$, they are well ordered $\cdots\prec x_{i+1}\prec x_i\cdots\prec x_0$. It is possible that $\gamma^+_x(2k\pi)=\gamma^-_x(-2k'\pi)$. In this case, we count the point twice. For each point $x_i$, there is a curve joining $(x_i,0)$ to $(x_i,2\pi)$ which is composed by some segments of $\gamma^+_x$ as well as of $\gamma^-_x$. For instance, in the following figure, by starting from the point $(x_2,0)$ and  following a segment of $\gamma^-_x$ to the point $A$, then following a segment of $\gamma^+_x$ to the point $B$ and finally following a segment of $\gamma^-_x$ to the point $(x_2,2\pi)$, we obtain a circle. Clearly, the Lagrange action along this circle is not smaller than the quantity $F(x_2)$.
\begin{figure}[htp] 
  \centering
  \includegraphics[width=5.2cm,height=5.0cm]{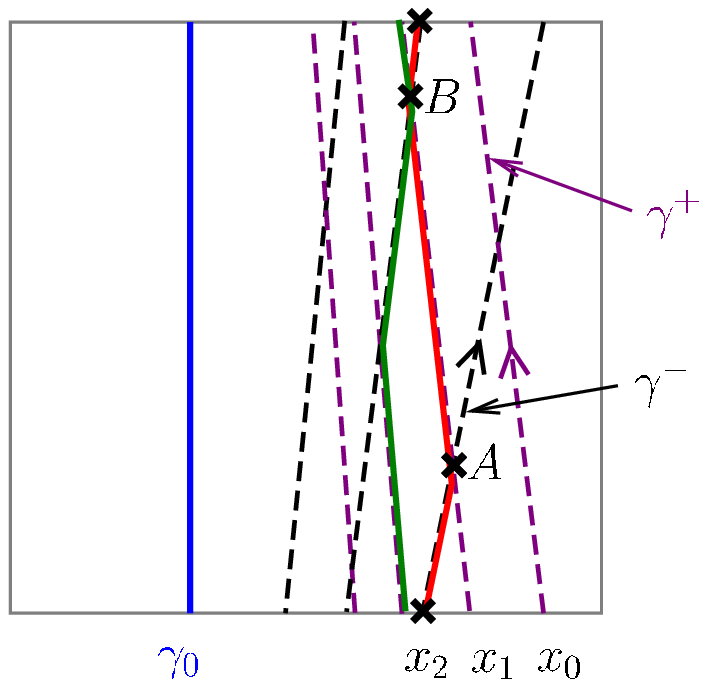}
  \label{}
\end{figure}

Let us consider the whole sequence $\{x_i\}$, we obtain infinitely many circles in that way. Therefore, the sum of the quantities $F(x_i)|_{i=0}^{\infty}$ is obviously not bigger than the total action along all of these circles
\begin{equation}\label{hypereq1}
\sum_{i=0}^{\infty}F(x_i)\le \lim_{k\to\infty}\Big\{\int_{0}^{2k\pi}L(d\gamma^+(\tau),\tau)d\tau+\int_{-2k\pi}^0L(d\gamma^-(\tau), \tau)d\tau\Big\}.
\end{equation}
The right hand side is nothing else but the barrier function valued at $x_0$.

As the periodic orbit supports the minimal measure, both $\gamma^+_x(2k\pi)$ and $\gamma^-_x(-2k\pi)$ approach the point $x^*$ where the periodic curve intersects the section $\{\tau=0\}$ as $k\to\infty$. If the periodic orbit is not hyperbolic, the sequence of $\{x_i\}$ approach $x$ slower than exponentially, i.e., for any small $\lambda>0$ there exists $\delta>0$ such that
\begin{align*}
|\gamma^+_x(2(k+1)\pi)-x^*)|&\ge (1-\lambda)|\gamma^+_x(2k\pi)-x^*)|,\\
|\gamma^-_x(-2(k+1)\pi)-x^*)|&\ge (1-\lambda)|\gamma^-_x(-2k\pi)-x^*)|,
\end{align*}
if $|\gamma^{\pm}_x(0)-x^*|\le\delta$. It follows that $|x_{i+1}-x^*|\ge(1-\lambda)|x_{i}-x^*|$. As the periodic curve is assumed non-degenerate minimizer, some $\lambda_0>0$ exists such that
\begin{equation}\label{hypereq2}
\sum_{i=0}^{\infty}(F(x_i)-F(x^*))\ge\lambda_0\sum_{i=0}^{\infty}(x_i-x^*)^2\ge \lambda_0\frac{(x_0-x^*)^2}{1-(1-\lambda)^2}.
\end{equation}
By subtracting $\min F$ from the Lagrangian $L$ we obtain that $F(x^*)=0$ and
$$
\text{\rm right-hand-side of (\ref{hypereq1})}=u^-(x,0)-u^+(x,0)
$$
where $u^{\pm}$ represents the backward (forward) weak-KAM solution. Since $u^-$ is semi-concave and $u^+$ is semi-convex, $u^--u^+$ is semi-concave. Since $(x^*,0)$ is a minimal point where $u^-(x^*,0)-u^+(x^*,0)=0$, there exists some number $C_L>0$ such that (cf. \cite{Fa2})
$$
u^-(x_0,0)-u^+(x_0,0)\le C_L(x_0-x^*)^2.
$$
Comparing this with the inequality (\ref{hypereq2}), we obtain from (\ref{hypereq1}) a contradiction
$$
\lambda_0\frac{(x_0-x^*)^2}{1-(1-\lambda)^2}\le C_L(x_0-x^*)^2
$$
if $\lambda>0$ is suitably small. This proves the hyperbolicity of the periodic orbit.
\end{proof}

We are now ready to prove the main result.

\noindent{\it Proof of Theorem \ref{AppenHyperTh1}}. According to Theorem \ref{theo2} and \ref{theo3}, for each $E_i\in[E_a,E_d]$, a neighborhood $[E_i-\delta_{E_i},E_i+\delta_{E_i}]$ of $E_i$ and an open-dense set $\mathfrak{O}(E_i)\subset C^r(\mathbb{T}^2,\mathbb{R})$ exist such that for each $P\in\mathfrak{O}(E_i)$ and each $E\in[E_i-\delta_{E_i},E_i+\delta_{E_i}]$ each minimal orbit of $\phi_{L+P}^t$ with homological class $g$ is hyperbolic. As each $\delta_{E_i}$ is positive, there exists finitely many $E_i$ such that $[E_a,E_d]\subset\cup_i[E_i-\delta_{E_i},E_i+\delta_{E_i}]$. We take $P\in\cap\mathfrak{O}(E_i)$, the hyperbolicity for $L+P$ holds for all $E\in[E_a,E_d]$.

Once a minimal point is non-degenerate for certain $E$, by the theorem of implicit function  it has natural continuation to a neighborhood of $E$. Namely, there exists a curve of minimal points passing through this point, it either reaches to the boundary of $[E_a,E_d]$, or extends to some point $E'$ where the critical point is degenerate. Since each global minimal point is non-degenerate, the critical point becomes local minimum when it enters into certain neighborhood of $E'$. As each non-degenerate minimal point is isolated to other minimal points for the same energy $E$, there are finitely many such curves, denoted by $\Gamma_i$.

For a curve $\Gamma_i$: $I_i=(E_i,E'_i)\to\mathbb{T}$, the definition domain $I_i$ contains finitely many closed sub-intervals $I_{i,j}$ such that $F(\Gamma_i(E),E)=\min_xF(\cdot,E)$ for all $E\in I_{i,j}$. By definition, $I_{i,j}\cap I_{i,j'}=\varnothing$ for $j\ne j'$. Let $\Gamma_{i,j}=\Gamma_i|_{I_{i,j}}$, we have finitely many curves $\{\Gamma_{i,j}\}$ such that $F(\cdot, E)$ reaches global minimum at the point $x$ if and only if $x=\Gamma_{i,j}(E)$ for certain subscript $(i,j)$.

For each $E\in\partial I_{i,j}$, by the definition of $I_{i,j}$, some other subscript $(i',j')$ exists such that $F(\cdot,E)$ reaches the global minimum at the points $\Gamma_{i,j}(E)$ and $\Gamma_{i',j'}(E)$. It is obviously an open-dense property that
$$
\frac {dF(\Gamma_{i,j}(E),E)}{dE}\ne \frac {dF(\Gamma_{i',j'})(E),E)}{dE}.
$$
Thus, it is also open-dense that $[E_a,E_d]=\cup I_{i,j}$ and $[E_a,E_d]\backslash\cup\text{\rm int}I_{i,j}$ contains finitely points. This completes the whole proof.\eop

\noindent{\bf Acknowledgement} This work is supported by NNSF of China (Grant 11171146, Grant 10531050), National Basic Research Program of China (973, 2007CB814800), Basic Research Program of Jiangsu Province (BK2008013) and a program PAPD of Jiangsu Province, China.

I would like to thank my colleagues J. Cheng, W. Cheng, X. Cui, J. Yan and M. Zhou for helpful discussions. I also thank Marc Chaperon, Alain Chenciner and Hakan Eliasson for inviting me to give talks on this result at their seminars while I was visiting University of Paris 7 for its hospitality. The main ideas of the proof were presented on the conferences at Edinburgh and at IAS, Princeton in October of 2011.

\end{document}